\documentclass[psamsfonts]{amsart}
\usepackage{amssymb,amsfonts,amsmath}
\usepackage[all,arc]{xy}
\usepackage{enumerate}
\usepackage{mathrsfs}
\usepackage{multicol}
\usepackage{xcolor}
\usepackage{tikz-cd}

\usepackage{comment}
\usepackage{hyperref}

\newtheorem{thm}{Theorem}[section]
\newtheorem{cor}[thm]{Corollary}
\newtheorem{prop}[thm]{Proposition}
\newtheorem{lem}[thm]{Lemma}
\newtheorem{conj}[thm]{Conjecture}

\theoremstyle{definition}
\newtheorem{claim}[thm]{Claim}
\newtheorem{defn}[thm]{Definition}
\newtheorem{exmp}[thm]{Example}

\theoremstyle{remark}
\newtheorem{rem}[thm]{Remark}

\newtheorem{notation}[thm]{Notation}

\DeclareMathOperator{\Mor}{Mor}
\DeclareMathOperator{\Sym}{Sym}

\newcommand{\free}{\operatorname{Free}}
\newcommand{\Nef}{\operatorname{Nef}}
\newcommand{\Br}{\operatorname{Br}}
\newcommand{\NE}{\overline{\operatorname{NE}}}

\makeatletter
\let\c@equation\c@thm
\makeatother
\numberwithin{equation}{section}

\title{Geometric Manin's Conjecture for Fano 3-Folds}

\author{Andrew Burke and Eric Jovinelly}

\begin{document}

\begin{abstract}
We classify families of free rational curves on all smooth Fano threefolds over $\mathbb{C}$.  In particular, we prove the family of very free rational curves representing any fixed numerical curve class is either irreducible or empty.  %
This proves Geometric Manin's Conjecture in dimension three.  
For general Fano threefolds of each deformation type, our results allow us to explicitly count the number of components of the moduli space of irreducible, geometrically rational curves, which may not be free, representing any numerical class.  %

\end{abstract}

\maketitle

\setcounter{tocdepth}{1}

\section{Introduction}\label{Section 1}
Understanding rational curves on Fano varieties is a fundamental task in birational geometry.  However, describing their moduli spaces is subtle; even in the case of Fano surfaces, a description of these spaces was only recently completed \cite{testa2006irreducibility, beheshti2021rational}.  In this paper, we give a detailed description of spaces of rational curves on all smooth Fano threefolds over the complex numbers.  This allows us to prove Geometric Manin's Conjecture \cite{2019} for these varieties.

Let $\Mor(\mathbb{P}^1, X, \alpha)$ be the morphism scheme parameterizing maps $f: \mathbb{P}^1 \rightarrow X$ such that $f_*[\mathbb{P}^1] = \alpha$.  %
We call the map $f$ a \textit{free} (resp.\ \textit{very free}) curve if $f^*\mathcal{T}_X$ is globally generated (resp.\ ample).  For a smooth Fano variety $X$, there are many such free curves because $-K_X := c_1(T_X)$ is ample.  Despite this abundance, it is predicted that moduli spaces of free curves on Fano varieties have a somewhat uniform structure.

An initial prediction about the structure of free curves in Fano varieties appears in heuristic arguments for Manin's Conjecture \cite{batyrev88}.  To arrive at the first version of this conjecture, Batyrev supposed that the locus in $\Mor(\mathbb{P}^1, X, \alpha)$ parameterizing free curves is irreducible for each $\alpha$.  
While this is true in many cases, it is false for threefolds: 
for any integer $n$, there exists a Fano threefold $X$ and a curve class $\alpha$ such that the aforementioned locus in $\Mor(\mathbb{P}^1, X , \alpha)$ has more than $n$ components.  

Our main results describe exactly how Batyrev's heuristic fails for Fano threefolds, and what version of it remains true.  First, we show that very free curves on Fano threefolds satisfy Batyrev's heuristic.  

\begin{thm}\label{thm: skinny intro}
    Let $X$ be a smooth Fano threefold.  For all $\alpha \in H_2(X,\mathbb{Z})$, the locus in $\Mor(\mathbb{P}^1,X,\alpha)$ that parameterizes very free curves is either irreducible or empty.  If $X$ is not the product of $\mathbb{P}^1$ with a del Pezzo surface, then there exist very free rational curves of class $\alpha$ if and only if $-K_X . \alpha \geq 4$ and $\alpha \in \Nef_1(X) \setminus \partial\NE(X)$.
\end{thm}

Second, we classify all families of free but not very free rational curves. 

\begin{thm}\label{thm: skinny intro 2}
      Let $X$ be a smooth Fano threefold other than the product of $\mathbb{P}^1$ with a del Pezzo surface.  
      If a component $M \subset \Mor(\mathbb{P}^1, X, \alpha)$ generically parameterizes free but not very free curves, then either $-K_X .\alpha \leq 3$, or $\alpha \in \partial\NE(X)$, or $M$ parameterizes multiple covers of free $(-K_X)$-conics.
\end{thm}

We present a stronger version of these results in Theorem \ref{Main Result}.  In general, a more restrictive notion than freeness is needed to obtain moduli spaces with predictable behavior.  This restrictive notion is the topic of Geometric Manin's Conjecture.

\subsection{Geometric Manin's Conjecture}  
Manin's Conjecture \cite{Manin, batyrev1990nombre} is a fundamental conjecture in arithmetic geometry about rational points on a smooth Fano variety defined over a global field. 
Batyrev's heuristics \cite{batyrev88} tie an initial version of Manin's Conjecture over global function fields to the behavior of free rational curves; the asymptotic formula in Manin's Conjecture was developed by assuming the number $N(\alpha)$ of components of $\Mor(\mathbb{P}^1,X,\alpha)$ that generically parameterize free curves has a uniform bound.  As previously mentioned, for many Fano threefolds there is no uniform bound on $N(\alpha)$.  This reflects a need to revise the initial version of Manin's Conjecture and Batyrev's heuristics.

To rectify Batyrev's heuristics, Lehmann and Tanimoto proposed a Geometric Manin's Conjecture (GMC) \cite{2019} over arbitrary fields.  This conjecture distinguishes \textit{Manin components} of $\Mor(\mathbb{P}^1,X,\alpha)$ from \textit{accumulating components} (see Definition \ref{manin component}).  Manin components are a subset of the components of $\Mor(\mathbb{P}^1, X,\alpha)$ that generically parameterize free curves.  Geometric Manin's Conjecture predicts the number of Manin components of $\Mor(\mathbb{P}^1,X,\alpha)$ is bounded. 

\begin{conj}[Geometric Manin's Conjecture]\label{GMC Tanimoto}

Let $X$ be a smooth Fano variety with Brauer group $\Br(X)$. There exists $\tau\in \Nef_1(X)_{\mathbb{Z}}$ such that for any integral nef class $\alpha \in \tau + \Nef_1(X)$, $\Mor(\mathbb{P}^1,X,\alpha)$ contains exactly $|\Br(X)|$ Manin components.  

In particular, when $X$ is a smooth Fano threefold, $|\Br(X)|= 1$ so $\Mor(\mathbb{P}^1,X,\alpha)$ contains a unique Manin component.
\end{conj}

We give a geometric characterization of Manin components of rational curves on Fano threefolds in Corollary \ref{cor: classification a-covers}: a component of $ \Mor(\mathbb{P}^1,X,\alpha)$ is Manin if and only if %
its universal family has nonempty, connected fibers over general points in $X$ and $\alpha \notin \partial\NE(X)$. When $-K_X . \alpha \geq 4$, we show this is equivalent to generically parameterizing very free curves.  Thus, Theorem \ref{thm: skinny intro} proves Conjecture \ref{GMC Tanimoto} for all smooth Fano threefolds (see Remark \ref{rem: product varieties} about product varieties).

\begin{thm}\label{thm: skinny intro GMC}
    Conjecture \ref{GMC Tanimoto} holds for all smooth Fano threefolds over $\mathbb{C}$.
\end{thm}

Conjecture \ref{GMC Tanimoto} was previously verified for smooth del Pezzo surfaces in \cite{testa2006irreducibility} and \cite{beheshti2021rational} and for general members of all but two deformation classes of smooth Fano threefolds of Picard rank $\rho = 1$ in \cite{2019, Fanoindex1rank1, JLT2023Err, lastDelPezzoThreefold}. 

We prove Conjecture \ref{GMC Tanimoto} for all other smooth Fano threefolds over $\mathbb{C}$.  This includes the two remaining deformation types with $\rho = 1$, as well as the 88 deformation classes with $\rho > 1$.  Importantly, we resolve the conjecture for every smooth threefold in each deformation class, and not just a general one.  

In fact, our study allows us to understand not just Manin components, but all families of rational curves on general Fano threefolds of each deformation type.

\subsection{Main Results}
Throughout our paper, $\free(X,\alpha)$ denotes the subscheme of the Kontsevich moduli space $\overline{M}_{0,0}(X,\alpha)$ that parameterizes free curves.  %
We say a component of $\free(X,\alpha)$ is Manin if the corresponding component of $\Mor(\mathbb{P}^1,X,\alpha)$ is Manin.  
Our main theorems describe all irreducible components of $\free(X,\alpha)$ when $X$ is a Fano threefold.  Theorem \ref{Main Result} and Remark \ref{rem: product varieties} address the case where $\alpha \not\in \partial \NE(X)$ and $-K_X . \alpha \geq 3$.  %
Theorem \ref{main result 2} addresses the case $\alpha \in \partial \NE(X)$, assuming $X$ is general in moduli.  Theorem \ref{main thm: conics} addresses the case where $\alpha \not\in \partial \NE(X)$ and $-K_X . \alpha = 2$, assuming generality of $X$ for certain deformations types.  As the $(-K_X)$-degree of any free rational curve is at least 2, Theorems \ref{Main Result}, \ref{main result 2}, and \ref{main thm: conics} exhaust all cases. %
Together, these results imply $\free(X,\alpha)$ has, in a certain sense, the minimal number of components possible when $X$ is general in moduli.

\begin{thm}\label{Main Result}
Let $X$ be a smooth Fano threefold.  %
Suppose $X$ is not a product of two varieties.  Assume $\alpha \in \Nef_1(X)_\mathbb{Z}$ satisfies $-K_X . \alpha \geq 3$ and $\alpha \not\in \partial \NE(X)$.
\begin{enumerate}
    \item There exists a unique Manin component $M_\alpha \subset \free(X,\alpha)$.  The general map $f: \mathbb{P}^1 \rightarrow X$ parameterized by $M_\alpha$ is an embedding.  Furthermore, $f$ is a very free curve if and only if $-K_X . \alpha \geq 4$. %
    \item Any other component of $\free(X,\alpha)$ parameterizes multiple covers of $(-K_X)$-conics (which are described by Theorem \ref{main thm: conics}).  
\end{enumerate}
In particular, $\free(X,\alpha)$ is irreducible unless $\alpha = n \beta$ for a $(-K_X)$-conic $\beta$.  Furthermore, the set of Manin components $\{M_\alpha\}_\alpha$ identified in (1) gives a complete list of all Manin components on $X$.

\end{thm}

In light of the following remark on product varieties, Theorem \ref{Main Result}(1) completes the proof of Conjecture \ref{GMC Tanimoto} for all Fano threefolds.

\begin{rem}\label{rem: product varieties}

When $X$ is a product of two varieties, $X \cong \mathbb{P}^1 \times S$ for a del Pezzo surface $S$.  A component of $\free(X,\alpha)$ is Manin if and only if it generically parameterizes very free curves.  By \cite{testa2006irreducibility} and \cite{beheshti2021rational}, there exists a (unique) Manin component of $\free(X,\alpha)$ if and only if $\alpha \in \Nef_1(X)_{\mathbb{Z}}\setminus \partial \NE(X)$ and $-K_S . \pi_* \alpha \geq 3$, where $\pi : X \to S$ is the projection.  
If $(-K_S)^2 \leq 2$, then, for certain $\alpha$, $\free(X,\alpha)$ contains accumulating components that do no fit the structure of Theorem \ref{Main Result}(2); %
such components parameterize curves that cover $-K_S$-conics upon projection to $S$ (see Theorem \ref{del pezzo curves thm} or Theorem \ref{products}). %
\end{rem}

Our remaining theorems on $\free(X,\alpha)$ classify accumulating (non-Manin) components.  These may be divided into two types: $\alpha \in \partial \NE(X)$ and $\alpha \not\in \partial \NE(X)$.  

For $\alpha \in \Nef_1(X) \cap \partial \NE(X)$, each curve of class $\alpha$ is contracted by a fiber type contraction $\pi : X \rightarrow B$, which is either a conic bundle or a del Pezzo fibration.  When $\pi$ is a conic bundle, $\free(X,\alpha)$ is  an irreducible space parameterizing branched covers of $\pi$-fibers.  If $\pi$ is a del Pezzo fibration, the subvariety of $\free(X,\alpha)$ parameterizing curves contained in a general $\pi$-fiber $i: S \hookrightarrow X$ may be identified with $\bigsqcup_i \free(S, \beta_i)$ for various $\beta_i$.  
This demonstrates two important facts.  First, the number $N(\alpha)$ of components of $\free(X,\alpha)$ exhibits polynomial order growth with respect to $-K_X. \alpha$.  %
This is because the volume of the polytope $\NE(S) \cap i_*^{-1}(n\alpha)$ is a polynomial function of $n$.  
Second, the monodromy action induced by $\pi$ on $N_1(S)$ completely determines $N(\alpha)$ from the pushfoward $i_* : N_1(S) \rightarrow N_1(X)$.  %
Therefore, Theorem \ref{main result 2} shows $N(\alpha)$ is minimal when $X$ is general in moduli.

\begin{thm}\label{main result 2}
Let $X$ be a smooth Fano threefold that is general in moduli.  Suppose $\pi : X \to B$ is a del Pezzo fibration.  For a smooth $\pi$-fiber $X_b$, let $W$ be the largest Weyl group that acts on the N\'eron-Severi group $N_1(X_b)$. The monodromy group of $\pi$ acts on $N_1(X_b)$ as the maximal subgroup of $W$ that stabilizes pushforward by inclusion into $X$.
\end{thm}
Theorem \ref{main result 2} cannot be extended to arbitrary Fano threefolds.  
However, as a step in the proof of Theorem \ref{Main Result}, we show $\free(X,\alpha)$ is irreducible if $\alpha \in \Nef_1(X)_{\mathbb{Z}}$, $-K_X. \alpha \in \{2,3\}$, and $X$ is an arbitrary Fano threefold of most deformation types.  Theorem \ref{thm: existence of core} summarizes these findings, which require case-by-case analysis.

For $\alpha  \notin \partial \NE(X)$, a component of $\free(X,\alpha)$ is accumulating if and only if %
the family it parameterizes has disconnected fibers over $X$ \cite[Theorem~5.4]{beheshti2020moduli}.  When $X$ is not a product of two varieties, Theorem \ref{Main Result} limits such families to covers of free anticanonical conics of some class $\beta$.  Theorem \ref{main thm: conics} identifies the number of components of $\free(X,\beta)$ when $X$ is general in moduli.  %
While this number may change upon smooth deformation of $X$, %
we prove $\free(X,\beta)$ is irreducible for all smooth Fano threefolds of most deformation types.  
This extends results of \cite{kuznetsov2018} on conics in prime Fano threefolds with large genus and finishes our classification of families of free curves on Fano threefolds.

\begin{thm}\label{main thm: conics}
Let $X$ be a smooth Fano threefold of index $r$ and Picard rank $\rho$.  %
Suppose $\beta \in \Nef_1(X)_\mathbb{Z}$ satisfies $-K_X . \beta = 2$ and $\beta \not\in \partial \NE(X)$.  Then $\rho \leq 2$, and
\begin{enumerate}
\item $\free(X,\beta)$ is irreducible if $\rho = 2$, if $r > 1$, or if $\rho = r = 1$ and $g(X) > 6$.
\item Suppose $\rho = r = 1$ and $X$ is general in moduli.  If $-K_X$ is very ample, $\free(X,\beta)$ is irreducible; otherwise, $\free(X,\beta)$ has two components.
\end{enumerate}
\end{thm}

It is necessary to assume generality of $X$ in Theorem \ref{main thm: conics}(2).  Double covers of del Pezzo threefolds provide examples of prime Fano threefolds $X$ of each genus $g(X) \leq 6$ wherein the family of free conics has more components than on a general deformation of $X$ \cite[Remark~2.3.7]{kuznetsov2018}.  Thus, Theorem \ref{main thm: conics} is optimal.  %

For general Fano threefolds, we characterize components of $\Mor(\mathbb{P}^1, X, \alpha)$ with nondominant universal family as well.  If $\alpha \not\in \Nef_1(X)$, these are merely families of curves on contractible divisors, which are listed in \cite{matsuki1995weyl, matsuki2023addendumerratum} and classified into five types by \cite[Theorem~3.3]{moriMMP}.  %
If $\alpha \in \Nef_1(X)$ instead, %
Theorem \ref{main result 3} describes each component of $\Mor(\mathbb{P}^1, X, \alpha)$ with nondominant universal family.

\begin{thm}\label{main result 3}
Let $X$ be a smooth Fano threefold that is general in moduli.  Suppose $|-K_X|$ is base-point free.  
Then for all $\alpha \in \Nef_1(X)$, at most one irreducible component of $\Mor(\mathbb{P}^1, X, \alpha)$ does not generically parameterize free curves.  Furthermore, if such a component $M$ exists, then $M$ parameterizes covers of nef $(-K_X)$-lines.  
\end{thm}

See Theorem \ref{main result 4} for a more detailed statement that removes the base-point free hypothesis from $|-K_X|$.  For any Fano threefold $X$, \cite[Theorem~1.2]{beheshti2020moduli} (see Theorem \ref{nonfree curves}) describes the subvariety of $X$ swept out by nondominant components of the universal family over $\Mor(\mathbb{P}^1,X)$.  Theorems \ref{main result 3} and \ref{main result 4} improve this result for general Fano threefolds by specifying what the family must look like as well.  Moreover, Lemma \ref{nonextreme lines} proves that, aside from three explicit exceptions, %
the class $\beta$ of a $(-K_X)$-line spans an extreme ray of $\NE(X)$.  This completes our classification of components of $\Mor(\mathbb{P}^1,X)$ for general Fano threefolds. %

Our proof of Theorem \ref{Main Result} uses the classification of Fano threefolds into $105$ deformation classes by \cite{IskovskihFano3foldsII, mori1981classification, erratumMori}.  Rather than working with $\Mor(\mathbb{P}^1, X)$ directly, we use the Kontsevich moduli space of stable maps to study rational curves on $X$ via their degeneration to reducible curves.  This allows us to apply a movable bend-and-break theorem from \cite{beheshti2020moduli} to study components of free rational curves via chains of free rational curves they contain.  We thus reduce to considering rational curves of low anticanonical degree, which we describe explicitly using monodromy properties of the Mori structures of Fano threefolds (Theorems \ref{Monodromy}, \ref{conic monodromy}, and \ref{monodromy del Pezzo main result}). %
A few cases are not suitable for study via Mori structure, but prove easier to understand on families of Fano threefolds (Lemma \ref{lem: boundary stratum irreducible}, Propositions \ref{prop: terminal away from codim two}, \ref{low degree curves degeneration}) or through Sarkisov links (Lemma \ref{lem: genus 12 Fano curves}).  
Application of these techniques requires individual analysis of most deformation types of Fano threefolds.  This is because the structure of $\Nef_1(X)_{\mathbb{Z}}$ as a monoid is quite varied among Fano threefolds, and our method for proving irreducibility of moduli spaces of high degree free curves (Theorem \ref{Main Method}) requires knowledge of its generators and relations.  Our %
analysis leads to an improvement of the bend-and-break theorem from \cite{beheshti2020moduli}.

\begin{thm}\label{improved MovableBB}

Let $X$ be a smooth Fano threefold.  Suppose $X$ does not contain a contractible divisor of $E5$ type.  Let $M \subset \overline{M}_{0,0}(X)$ be a component that generically paramterizes free curves of $(-K_X)$-degree at least 4.  One of the following holds:
\begin{enumerate}
    \item There exists a birational contraction $\pi : X \rightarrow \mathbb{P}^3$ such that $M$ parameterizes %
    total transforms of lines under $\pi$. %
    \item The boundary of $M$ parameterizes a stable map $f: C_1 \cup C_2 \rightarrow X$ where for each $i$, $C_i$ is irreducible and the restriction $f|_{C_i} : C_i \rightarrow X$ is a free curve.%
\end{enumerate}

\end{thm}

\subsection{Outline}
Section \ref{Section 2} of this paper develops convenient notation and recalls relevant results. Section \ref{Section 3} compares different formulations of Geometric Manin's Conjecture and proves related results. Section \ref{Section 4} discusses the strategies we use to prove irreducibility of spaces of free curves.  Section \ref{Section 5} proves general statements about monodromy related to Mori structures, in particular Theorem \ref{main result 2}.  Section \ref{section: low degree curves} formalizes the study of low-degree curves on Fano threefolds and proves Theorems \ref{main thm: conics} and \ref{main result 3}.  Section \ref{Section 7} outlines our proof of Theorem \ref{Main Result}. Then, in the remaining sections, we apply our strategy to the $105$ deformation classes of Fano threefolds.  The proof of Theorem \ref{improved MovableBB} follows from casework in these sections, while Theorems \ref{thm: skinny intro}, \ref{thm: skinny intro 2}, and \ref{thm: skinny intro GMC} are direct consequences of Theorem \ref{Main Result}.

\subsection{History}
This paper belongs to a long tradition of studying rational curves on Fano varieties.  Koll\'ar, Miyaoka, and Mori revolutionized this practice through their proof of rational connectedness of Fano varieties \cite{mori_fano_rat_connected}.
This spurred a discourse on free rational curves %
that coalesced into Geometric Manin's Conjecture \cite{2019}.  
However, the study of rational curves on Fano varieties dates back to the beginning of algebraic geometry.  Cayley, Salmon, and del Pezzo \cite{Cayley1849, Salmon1849, delPezzo}  %
counted lines on the anticanonical models of Fano surfaces.  %
These lines appear in Clebsch's realization \cite{Clebsch} of cubic surfaces as blow-ups of the projective plane. 
L\"uroth asked if, like the cubic surface, every unirational variety is rational.
Positive answers to L\"uroth's question for curves \cite{Luroth1875} and complex surfaces \cite{CastelnuovoRationality1894} inspired Fano to look for counterexamples among his eponymous threefolds.

Fano's forty years of research on threefolds \cite{Fano08, Fano15, Fano37, Fano47}, %
though somewhat lacking in rigor, laid the foundation for several breakthroughs %
involving rational curves and Fano varieties.  The first of these answered L\"uroth's question.  
While the study of rational curves has evolved since Fano's time, their presence in research from his era is noticeable. %
Some of Fano's earliest research on threefolds \cite{Fano1904} describes the variety of lines on a cubic hypersurface.  This variety was central to Clemens and Griffiths's long-awaited proof of the irrationality of cubic threefolds \cite{ClemensGriffiths1972}.
Along with other Fano-inspired works \cite{IskovskihManin1971, ArtinMumford1972}, this offered one of the first widely-accepted counterexamples to L\"uroth's question.  %

The classification of Fano threefolds into 105 irreducible families, called \textit{deformation types}, was the focal point of the next breakthrough on rational curves and Fano varieties.  Iskovskikh, Mori, and Mukai completed this classification by rectifying Fano's works \cite{IskovskihFano3foldsI, IskovskihFano3foldsII} and applying Mori's new theory on extremal rays \cite{mori1981classification, mori1983classification, mori1985classification, erratumMori}.  Mori's theory \cite{moriMMP} generalized Castelnuovo's contractability criterion \cite{Castelnuovo1901SopraAQ} to families of rational curves on higher dimensional varieties.  His continued research in dimension three \cite{moriFlips, kollarMoriFlips} formed the basis of the modern Minimal Model Program.  This ties the geometry of a Fano variety closely to the behavior of rational curves it contains.  %

Much of Mori's groundbreaking work on rational curves stems from a technique \cite[Theorem~4]{Mori1979} commonly known as Mori's bend-and-break.  Bend-and-break uses morphism schemes to make the following statement precise: irreducible curves in a variety often degenerate to reducible curves.  Mori originally used his technique to prove %
there exist low-degree rational curves through any point on a Fano variety.  %
This had several consequences, including boundedness of smooth Fano varieties \cite{mori_fano_rat_connected}.  
In addition, bend-and-break allows one to study high-degree free rational curves by breaking them into pieces.  This facilitated the beginning of research into rational curves of arbitrarily large degree. %

Following Mori's work, many authors have taken to studying moduli spaces of high-degree free rational curves on Fano varieties.  The earliest such works \cite{thompsen98, KP01} prove irreducibility of Kontsevich moduli spaces of rational curves on homogeneous varieties. \cite{HRS2004, RY2019} prove irreducibility of spaces of rational curves on most Fano hypersurfaces.  These results were extended to arbitrary smooth hypersurfaces of low degree by \cite{CS2009, BV2017} and to general complete intersections by \cite{BK2013}.  \cite{Castravet04, mustopa2020rationalcurvesmodulispaces} classify families of free rational curves on moduli spaces of vector bundles. \cite{testa2006irreducibility, beheshti2021rational} show most spaces of free rational curves on del Pezzo surfaces are irreducible, and \cite{2019, lastDelPezzoThreefold, Okamura2024delPezzo} prove analogous results for higher dimensional del Pezzo varieties.  Spaces of free rational curves on smooth toric varieties are characterized in \cite{Bourqui2016}.  Lastly, \cite{Fanoindex1rank1, JLT2023Err} describe families of free rational curves in many Fano threefolds of Picard rank 1 and index 1.

All of these studies bound the number of components in $\Mor(\mathbb{P}^1,X,\alpha)$ parameterizing free curves when $-K_X . \alpha \gg 0$.  
Although similar bounds appeared in many cases, it was initially unclear what one could expect for arbitrary Fano varieties.  By borrowing ideas from arithmetic geometry \cite{batyrev88}, Lehmann and Tanimoto unified these examples into a Geometric Manin's Conjecture about the structure of $\Mor(\mathbb{P}^1,X)$ for Fano varieties \cite{2019}.  The aforementioned papers verify Geometric Manin's Conjecture for many hypersurfaces, homogeneous varieties, toric varieties, del Pezzo varieties, moduli spaces of vector bundles, and some Fano threefolds.  This paper finishes the proof of Geometric Manin's Conjecture for Fano threefolds.  %

\begin{center}
\textbf{Acknowledgments}
\end{center}

The authors would like to thank Eric Riedl, Sho Tanimoto, Brian Lehmann, Kenji Matsuki, Matthew Scalamandre, Nikolai Konovalov, Shigeru Mukai, Brendan Hassett, Nicholas Salter, Gwyneth Moreland, and Izzet Coskun for countless conversations, feedback, and advice.  In particular, this paper would not have been possible without input from Lehmann, Tanimoto, and Riedl.  Eric Jovinelly was supported by NSF grant 1547292 and an NSF postdoctoral research fellowship, DMS-2303335.

\setcounter{tocdepth}{1}
\tableofcontents

\section{Background and Notation}\label{Section 2}
Let $X$ be a smooth variety defined over $\mathbb{C}$.  Throughout this paper, $N_1(X)$ will refer to the real vector space of numeric equivalence classes of curves on $X$.  We let $\NE(X) \subset N_1(X)$ denote the Mori cone of effective curve classes.  The cone $\Nef_1(X)\subset N_1(X)$ consists of those curve classes pairing nonnegatively with all effective divisors.  If $X$ is a smooth Fano threefold, by Theorem \ref{Representability of Free Curves}, $\Nef_1(X)$ is a rational polyhedron with extremal rays given by classes of free rational curves, as defined below.  Moreover, $\Nef_1(X)\subset N_1(X)$ is locally constant in families of Fano varieties \cite[Corollary~5.1]{Fernex2009RigidityPO}.  

\begin{defn}
A morphism $f:\mathbb{P}^1\rightarrow X$ is \textit{free} (resp.\ \textit{very free}) if $f^*\mathcal{T}_X$ is globally generated (resp.\ ample).
\end{defn}

Every free curve in this paper is rational.  Given a class $\alpha \in N_1(X)$, the space $\Mor(\mathbb{P}^1,X,\alpha)$ parameterizes morphisms $f:\mathbb{P}^1\rightarrow X$ for which $f_*[\mathbb{P}^1]=\alpha$.  The Kontsevich space of stable maps $ \overline{M}_{0,n}(X,\alpha)$ parameterizes $n$-pointed stable curves $f:C\rightarrow X$ of genus $0$ with $f_*[C]=\alpha$ \cite{kontsevich_space}.  As an open subset of $\overline{M}_{0,3}(X)$ embeds in $\Mor(\mathbb{P}^1,X)$, we mostly work with the compact Kontsevich space to obtain limits of maps.  
We are primarily interested in components of $\overline{M}_{0,0}(X)$ that generically parameterize free maps, and so make the following definition.

\begin{defn}\label{def: locus of free curves}
For $\alpha\in N_1(X)$, let $$\free_n(X,\alpha)\subseteq \overline{M}_{0,n}(X,\alpha)$$ be the locus of $n$-pointed maps $f:\mathbb{P}^1\rightarrow X$ such that $f^*\mathcal{T}_X$ is globally generated.  Furthermore, let $\free_n^{bir}(X,\alpha)\subseteq \free_n(X,\alpha)$ denote those maps $f:\mathbb{P}^1\rightarrow X$ which are also birational onto their image.  Denote the closure of these spaces by $\overline{\free}_n^{bir}(X,\alpha)\subseteq \overline{\free}_n(X,\alpha)\subset\overline{M}_{0,n}(X,\alpha)$.  We omit the subscript $n$ when $n=0$.
\end{defn}

\noindent To study $\overline{\free}(X)$, we make use of special maps called \textit{free chains} in the complement $\overline{\free}(X)\setminus \free(X)$.  Our approach is motivated by the following theorem.  

\begin{thm}[Movable Bend-and-Break, \cite{beheshti2020moduli}]\label{MovableBB}
Let $X$ be a smooth Fano threefold.  Let $M\subset \overline{\free}(X)$ be a component parameterizing curves of $(-K_X)$-degree $\geq 6$ (or even $\geq 5$ if $X$ does not admit an $E5$ contraction).  Then $M$ contains a stable map of the form $f : Z \rightarrow X$ with $Z=C_1\cup C_2$, $C_i\cong\mathbb{P}^1$, and $f|_{C_i}$ free for each $i$.  Moreover, if $M\subset \overline{\free}^{bir}(X)$, we may pick $f : Z \rightarrow X$ as above which is birational onto its image.
\end{thm}

\noindent See Theorem \ref{improved MovableBB} for an exact statement as to which quartic free curves degenerate to a \textit{chain of free curves} as in Theorem \ref{MovableBB}.

\begin{defn}[\cite{2019}]
A \textit{chain of curves} of length $n > 1$ on $X$ is a stable map $f : C\to X$ where $C = C_1 \cup \ldots \cup C_n$ is a connected union of irreducible rational curves $C_i$ such that $p_a(C) = 0$ and $C_i \cap C_j \neq \emptyset$ if and only if $|i-j| \leq 1$.  

The \textit{type} of a chain is the sequence $(\alpha_1, \ldots, \alpha_n)$ for $\alpha_i = f_*[C_i] \in N_1(X)$.  A chain is \textit{immersed} if $f$ is unramified at each node.  Lastly, a \textit{free chain} is a chain of curves $f : C \to X$ such that $f_i := f|_{C_i}$ is free for each $i$. 
\end{defn}

\noindent  For simplicity, when $f:C\rightarrow X$ is an embedding, we often omit the map $f:C\rightarrow X$ and consider $C = C_1 \cup \ldots \cup C_n$ as a subvariety of $X$.  %
Thus, we commonly refer to $C$ as a (free) chain of type $([C_1], \ldots , [C_n])$.  %
Chains of free curves are parameterized by \textit{main components} of products of components of $\overline{\free}_n(X)$ over $X$. %

\begin{defn}[\textit{Main Component}, Definition 5.6 \cite{2019}]
Given irreducible components $M_i^j \subset \overline{\free}_j(X,\alpha_i)$, consider the fiber products 

$$M_1^1 \times_X M_2^2 \times_X \ldots \times_X M_{r-1}^2 \times_X M_r^1$$
and
$$\prod_X M_i^2 = M_1^2 \times_X \ldots \times_X M_r^2$$

\noindent obtained by gluing separate markings.  A component $M$ of such a fiber product is a \textit{main component} if it dominates each $M_i^j$ under projection.  Equivalently, the evaluation map $M\rightarrow X$ associated to each glued marking dominates $X$.
\end{defn}
\noindent Main components of $\prod_X M_i^2$ are in bijection with those of $M_1^1 \times_X \ldots \times_X M_r^1$.  A general point in any main component $M$ corresponds to a chain of free curves.  Conversely, any chain of free curves corresponds to a point in some main component.  This follows from the basic fact that the product $M\times_X N$ of any two main components $M\subset \prod_X M_i^2$ and $N \subset \prod_X N_i^2$ contains another main component.  To facilitate discussion of main components and the components of $\overline{\free}(X)$ containing them, we make the following definitions.

\begin{defn}
A free curve $f:C\rightarrow X$ is \textit{freely breakable} if it lies in the same component $M \subset \overline{\free}(X,\alpha)$ as a free chain $g : Z \to X$.  We say $f$ \textit{breaks freely} into $g$.  In this case, we call the component $M$ \textit{freely breakable} as well. If each component of $\overline{\free}(X,\alpha)$ is freely breakable, we call $\alpha \in \Nef_1(X)$ a \textit{freely breakable} class.
\end{defn}

\begin{defn}\label{def core}
A \textit{core (of free curves)} on $X$ is a finite set $\mathscr{C}_X \subset N_1(X)$ of classes represented by free curves such that:
\begin{enumerate}
    \item Every component $M \subset \overline{\free}(X)$ contains a chain of free rational curves whose irreducible components have class in $\mathscr{C}_X$; and
    \item $\overline{\free}(X,\alpha)$ is irreducible for all $\alpha \in \mathscr{C}_X$.
\end{enumerate}
A \textit{weak} core of free curves satisfies $(1)$ but not necessarily $(2)$. Given a weak core $\mathscr{C}_X$, a \textit{core class} is an element of $\mathscr{C}_X$.  %
\end{defn}
\begin{rem}\label{def Sigma}  Suppose $X$ is a general Fano threefold of a fixed deformation type.  Theorems \ref{Main Result}, \ref{main result 2}, \ref{main thm: conics}, and \ref{del pezzo curves thm} imply $X$ has a core of free curves if and only if $X$ is a product variety or none of the following are true: $X$ admits a del Pezzo fibration with degree $\leq 3$ fibers, $X$ is a double cover of $\mathbb{P}^3$ branched over a sextic hypersurface, or $X$ is a double cover of a quadric threefold branched over a degree 8 surface.  In these cases, $X$ has a weak core of free curves $\mathscr{C}_X$ by Theorem \ref{MovableBB}.
\end{rem}

\begin{defn}
Fix a weak core $\mathscr{C}_X$ of $X$. A \textit{separating class} is a core class $\alpha$ such that the generic fiber of the evaluation map $\free^{bir}_1(X,\alpha)\xrightarrow{\text{ev}} X$ is reducible.
\end{defn}

\begin{rem}\label{rem: separating class}
For any separating class $\alpha$, $\overline{\free}_1(X,\alpha) \times_X \overline{\free}_1(X,\alpha)$ contains main components that map to different components of $\overline{\free}(X, 2\alpha)$. Indeed, let $\phi : Z \to \overline{\free}_1(X,\alpha)$ be a smooth resolution and let $Z \to Y\rightarrow X$ be the Stein factorization of the composition $\phi \circ \text{ev} : Z \to X$.  For a general point $p\in X$, consider the smoothing $C$ of a free chain $C_1 \cup C_2$, where $C_1, C_2 \in \text{ev}^{-1}(p)$.  If $C_1, C_2$ lie in the same component of $\text{ev}^{-1}(p)$, $C$ lifts to $Y$.  Otherwise, $C$ does not lift to $Y$, as any lift to $Y$ of a family smoothing $C_1 \cup C_2$ to $C$ as a subscheme of $Y$ would have nonconstant Hilbert polynomial.
\end{rem}

Below, we summarize some basic properties about free chains found in \cite{2019}.

\begin{lem}\label{Basic Properties Chains}
Let $M\subset \overline{\free}(X)$ be a component containing a free chain $f:C\rightarrow X$ with $C=\cup_{i=1}^r C_i$.  
Let $M_i^n\subset \overline{\free}_n(X)$ be the component containing the map $f_i = f|_{C_i}$ with $n$ markings,  
$N\subset M_1^1 \times_X \times M_2^2 \times_X \ldots \times_X M_{r-1}^2 \times_X M_r^1$ be the unique main component containing $(C,f)$, and $N^\circ\subset N$ be the dense open locus parameterizing chains of free curves.  Then 
\begin{enumerate}
    \item $N\subset M$ and no other component of $\overline{\free}(X)$ intersects $N^\circ$.
    \item If any component $M_i$ contains a main component of $\widehat{M}_i^1 \times_X \widetilde{M}_i^1$, then $M$ contains a main component of $M_1^1 \times_X \ldots \times_X \widehat{M}_i^2 \times_X \widetilde{M}_i^2 \times_X \ldots \times_X M_r^1$.
    \item For any permutation $\sigma$ of $\{1, \ldots , r\}$, $M$ contains a main component of $M_{\sigma (1)}^1 \times_X M_{\sigma (2)}^2 \times_X \ldots \times_X M_{\sigma (r-1)}^2 \times_X M_{\sigma (r)}^1$.
    \item Each projection $N^\circ\rightarrow M_i^{n_i} \times_X \ldots \times_X M_j^{n_j}$ is flat and dominates a main component of the product, where $n_k = 1$ if $k \in \{1,r\}$ and $n_k=2$ otherwise.  %
    \item The evaluation map $\text{ev}: N^\circ \rightarrow X$ associated to any marking or node is flat.  For a closed $Z\subsetneq X$ and any two evaluations $\text{ev}_1,\text{ev}_2$, a general map $[g]$ in each component of a general fiber of $\text{ev}_1$ is a free chain with $\text{ev}_2(g)\in X\setminus Z$.
    \item If the reduced image $f(C_i)$ is not a $(-K_X)$-conic or if it differs from the reduced image $f(C_{i+1})$, %
    general maps parameterized by $N$ are immersions along the node $C_i \cap C_{i+1}$.
\end{enumerate}
\end{lem}
\begin{proof}
Properties (1)-(5) were established in \cite{2019}.  For (6), we may suppose $i=1$ and consider $f:Z\rightarrow X$, $Z=C_1\cup C_2$.  If the reduced image $f(C_1)$ is not a $(-K_X)$-conic, there exists a one parameter family of deformations of $f_1 = f|_{C_1}$ through $f(C_1 \cap C_2)$; at least one map in this family has reduced image distinct from $f(C_2)$.  Thus we may assume $f(C_1) \neq f(C_2)$.  

Consider a general $(\dim X-2)$-dimensional family of maps $S\subset \overline{M}_{0,0}(X,f_* [C_1])$ containing $[f|_{C_1}]$.  Since $f(C_1)\neq f(C_2)$ by generality, we may assume the locus $D\subset \overline{M}_{0,1}(X,f_* [C_1])$ lying over $S$ dominates a divisor $\text{ev}_1(D)\subset X$ that does not contain $f(C_2)$.

Let $N\subset \overline{M}_{0,1}(X,f_* [C_2])$ be the component containing pointed maps lying over $[f|_{C_2}]$.  The component $T$ of $D \times_X N$ containing $[f]$ has dimension $\text{dim } N -1$.  Since $T\rightarrow N$ is generically finite onto its image and $\text{ev}_1(D)$ does not contain $f(C_2)$, the natural map $\pi : T\rightarrow \overline{M}_{0,0}(X,f_*[C_2])$ dominates a component of free curves.  By \cite[Proposition~2.8]{beheshti2020moduli}, a general curve parameterized by $\pi(T)$ is not tangent to $\text{ev}_1(D)$ at any point of intersection.  Therefore, a general point of $T$ corresponds to $[f']\in D\times_X N \subset \overline{M}_{0,1}(X,f_* [C_1])\times_X \overline{M}_{0,1}(X,f_* [C_2])$ parameterizing a map $f': Z \rightarrow X$ that is an immersion along the node.
\end{proof}

\begin{rem}
The above proof remains true when $f|_{C_1}$ is replaced by any stable map parameterized by a component of $\overline{M}_{0,1}(X)$ that dominates a divisor $D\subset X$ with $f_*[C_2] . D > 0$.  Thus, when deforming free chains $f:C\rightarrow X$ on a Fano threefold $X$ into chains of curves $g: C' \cup Z \cup C'' \rightarrow X$, where $C',C''$ are free chains, we may assume $g$ is an immersion in a neighborhood of both $C' \cap Z$ and $C'' \cap Z$.  This allows us to apply the standard results from deformation theory below.
\end{rem}

\subsection{Deformation Theory}
Given a map $f:C\rightarrow X$ from a nodal rational curve to a smooth variety, unramified at each node, the tangent space of $[f]\in \overline{\mathcal{M}}_{0,0}(X,f_*[C])$ is given by the global sections of its normal sheaf $\mathcal{N}_f$, which lies in the following sequence:
$$0\rightarrow \mathcal{E}xt^1_{\mathcal{O}_C}(Q,\mathcal{O}_C)\rightarrow \mathcal{N}_f\rightarrow \mathcal{H}om_{\mathcal{O}_C}(K,\mathcal{O}_C)\rightarrow 0.$$
Here, $K$ and $Q$ are the kernel and cokernel of the natural map $f^*\Omega^1_X\rightarrow \Omega^1_C$.  $\mathcal{N}_f$ is locally free iff $f$ is an immersion.  The space of obstructions to deformations is $H^1(C,\mathcal{N}_f)$.  In particular, $[f]$ is a smooth point of $\overline{\mathcal{M}}_{0,0}(X,f_*[C])$ when $H^1(C,\mathcal{N}_f)=0$ (or $H^1(C,f^*\mathcal{T}_X)=0$).  Below, we collect a few well-known results.%

\begin{lem}[Lemma 2.6 \cite{GHSrational}]\label{GHS lemma} Let $X$ be a smooth projective variety and $f:C\rightarrow X$ be a stable map which is an immersion along each node of $C$.  Suppose $C_0$ is an irreducible component of $C$ which meets components $C_1,\ldots , C_n$ at points $p_1\ldots, p_n$.  Letting $f_0 := f|_{C_0}$ we have a short exact sequence
$$0\rightarrow \mathcal{N}_{f_0}\rightarrow \mathcal{N}_f|_{C_0} \rightarrow \oplus_i k(p_i)\rightarrow 0$$
identifying $\mathcal{N}_f|_{C_0}$ with the sheaf of rational sections of $\mathcal{N}_{f_0}$ having at most a simple pole at each $p_i$ in the normal direction tangent to $C_i$.  Moreover, a first-order deformation of $f$ corresponding to a global section $\sigma \in H^0(C,\mathcal{N}_f)$ smooths the node of $C$ at $p_i$ iff $\sigma|_C$ maps nontrivially to $k(p_i)$ in the above sequence.
\end{lem}

In particular, when $C=\cup C_i$ is a decomposition of $C$ into irreducible components, we will let $f_i := f|_{C_i}$ and use $\mathcal{N}_{f_i}$ to determine $\mathcal{N}_{f}|_{C_i}$.

\begin{prop}[Proposition 2.2 \cite{beheshti2021rational}]\label{kontsevich space H1}
Let $E$ be a sheaf on a nodal curve $C$ of arithmetic genus 0 satisfying the following two conditions:
\begin{itemize}
    \item For each component $C_i$ in $C$, $H^1(C_i, E|_{C_i}) = 0$.
    \item For all but at most one $i$, $E|_{C_i}$ is globally generated.
\end{itemize}
Then $H^1(C,E)=0$.
\end{prop}

The components $C_i$ above may either be connected or irreducible.  In particular, Lemma \ref{GHS lemma} implies that maps from nodal rational curves, unramified at each node, are smooth points of $\overline{\mathcal{M}}_{0,0}(X)$ generalizing to irreducible immersed free curves if $\mathcal{N}_f$ is globally generated.  When $C$ is a chain of connected components $C_i$, if in addition $\mathcal{N}_{f_i}$ is globally generated for each $i$, then $f$ generalizes to a free chain smoothing each $C_i$ as above while preserving other nodes.  We use this fact and the well-known propositions below implicitly throughout our paper.

\begin{prop}[Propostion 2.8 \cite{beheshti2020moduli}]\label{prop: general free curve not tangent to divisor}
Let $X$ be a smooth variety, $D\subset X$ a reduced divisor, and $f:C\rightarrow X$ be a free rational curve that is general in its deformation class.  Then $f(C)$ is not tangent to $D$ at any point of intersection.
\end{prop}
\begin{prop}[Propositions 2.9,10 \cite{beheshti2020moduli}, Theorem 3.14.3 \cite{kollar2013rational}]\label{properties very free curves}
Let $X$ be a smooth threefold.  Suppose $M \subset \overline{\free}(X)$ generically parameterizes very free curves.  Then a general map $f : C\to X$ parameterized by $M$ is an embedding.  Moreover, if $D \subset X$ is a smooth divisor with a fibration $\phi : D \to B$ onto a curve, then $f(C)$ meets each fiber of $\phi$ in at most one point.  If $\pi : X \to S$ is a conic fibration, then $\pi \circ f$ is an immersion and $\pi \circ f(C)$ has at worst nodal singularities.  
\end{prop}

\begin{prop}[Theorem~1.4 \cite{ShenVeryFree}]
Let $X$ be a smooth Fano threefold with $\rho(X) =1$.  If a component $M \subset \Mor(\mathbb{P}^1, X, \alpha)$ generically parameterizes very free curves with unbalanced normal bundles, then $X\cong \mathbb{P}^3$ and $M$ parameterizes conics. 
\end{prop}

\section{Geometric Manin's Conjecture}\label{Section 3}
In this section, we first describe the relationship between Conjecture \ref{GMC Tanimoto} and Manin's Conjecture.  This motivates our definition of \textit{Manin components} of curves.  For Fano threefolds, Manin components of large degree curves coincide with families of very free curves.  We conclude with a reformulation of Conjecture \ref{GMC Tanimoto} for smooth Fano threefolds as a statement about families of very free curves.

\subsection{Relationship to Manin's Conjecture}

Franke, Manin, and Tschinkel hypothesized an asymptotic formula \cite{Manin} for counting rational points of bounded height on smooth Fano varieties defined over number fields.  Their original estimate excluded points in a proper, \textit{closed} exceptional set $Z$; however, counterexamples such as varieties fibered by diagonal cubic surfaces \cite{manin_conj_counterexample} show points from a \textit{thin} exceptional set must be excluded instead.  Below, a \textit{thin} subset of $X(F)$ \cite[Definition~1.3]{peyre_manin_conjecture} is a finite union of sets $f(Y(F))$ such that $f: Y \rightarrow X$ is a generically finite morphism over $F$ which does not admit a rational section.

\begin{conj}[Manin's Conjecture]\label{manin conjecture}
    Let $F$ be a number field and $X$ be a geometrically integral, smooth, Fano variety over $F$.  Let $\mathcal{L}$ be a big and nef line bundle with an adelic metrization on $X$.  For $Q \subset X(F)$, let $N(Q, \mathcal{L}, T)$ be the number of points in $Q$ of $\mathcal{L}$-height at most $T$.  Suppose $X(F)$ is not thin. Then there exists a thin exceptional set $Z \subset X(F)$ and constants $c,\ a(X,L),\ b(F,X,L)$ such that  
    $$N(X(F)\setminus Z, \mathcal{L}, T) \sim_{T\rightarrow \infty} c T^{a(X,L)} \mathrm{log}(T)^{b(F,X,L) - 1}.$$
    \end{conj}

Manin's $a$ and $b$-invariants, $a(X,L)$ and $b(F,X,L)$, may be defined geometrically for any smooth projective variety $X$ with a big and nef divisor $L$.

\begin{defn}[Definition 2.1 \cite{tanimoto2021introduction}]
Let $X$ be a smooth projective variety and $L$ be a big and nef $\mathbb{Q}$-Cartier divisor on $X$.  We define the \textit{Fujita invariant} (\textit{a-invariant}) as
$$a(X,L)=\inf \{t\in \mathbb{R}| tL + K_X \in \overline{\text{Eff}}(X)\}.$$
When $L$ is nef but not big, we set $a(X,L)=\infty$.  When $X$ is singular, pick a resolution $\beta:\tilde{X}\rightarrow X$ and define $a(X,L) := a(\tilde{X},\beta^* L)$.  \cite[Proposition~2.7]{balancedlinebundles} shows this is well defined.
\end{defn}

When $a(X, L) > 0$, let $\mathcal{F}(F, X,L)$ denote the face of $\text{Nef}_1(X)$ with vanishing intersection against $K_X + a(X,L)L$. The $b$-invariant is then defined as
\begin{center}
    $b(F,X,L) = \text{dim }\text{Span}(\mathcal{F}(F,X,L)).$%
\end{center}

A recent, conjectural description \cite[Conjecture~5.2]{sengupta_lehmann_tanimoto} of the exceptional set $Z$ in Conjecture \ref{manin conjecture} roughly expects it to be a union of sets $f(Y(F))$ for which $(a(Y, f^*L), b(F, Y, f^*L)) > (a(X,L), b(F,X,L))$ in the lexicographic order.  We refer readers to \cite{Lehmann_2017} and \cite{lehmann_exceptional_set} for further detail.

Batyrev considered an interpretation of Manin's conjecture over global function fields instead.  %
For a Fano variety $X$ defined over a finite field $\mathbb{F}_q$, we may count $\mathbb{F}_q(t)$-points on $\mathcal{X} = X \times \mathbb{P}^1_{\mathbb{F}_q}$, i.e.\ maps $\mathbb{P}^1_{\mathbb{F}_q} \rightarrow X$, %
that have bounded degree with respect to the line bundle $L = -K_X$.
This formulation of Manin's Conjecture connects the geometry of curves on a Fano variety $X$ to point counts.  Explicitly, Batyrev gave heuristic arguments \cite{batyrev88} (summarized in \cite[Section~4.7]{tschinkel_height} and \cite[Section~4]{lehmann2024nonfree}) for Manin's Conjecture in the setting above that rely on three assumptions:

\begin{enumerate}
    \item After removing curves that lie on a fixed closed subset, components of $\Mor(\mathbb{P}^1, X,\alpha)$ have the expected dimension $-K_X . \alpha + \text{dim }X$,
    \item There is an upper bound, independent of $\alpha$, on the number of components of $\Mor(\mathbb{P}^1, X, \alpha)$ whose universal family is dominant, and
    \item For each such component $M\subset \Mor(\mathbb{P}^1, X, \alpha)$, the estimate $\# M(F) \sim q^{\text{dim } M}$ obtained by comparison to affine space is sufficient.
\end{enumerate}

The first assumption has been verified \cite{2019} over characteristic 0, but fails in general over characteristic $p$ due to the existence of inseparable covers $f: Y \rightarrow X$ with $a(Y, -f^*K_X) > a(X, -K_X) = 1$. 

Batyrev's second assumption also fails for many varieties.  Its failure is related to the revision of Manin's Conjecture to allow for a thin exceptional set.  Heuristically, there should be finitely many equivalence classes of covers $f: Y \rightarrow X$ such that, after excluding families of curves that lift to $Y$, there are finitely many components of $\Mor(\mathbb{P}^1, X,\alpha)$ for each $\alpha$.  Geometric Manin's Conjecture makes this precise by counting only \textit{Manin components} of curves which do not lift to certain \textit{a-covers}.

\begin{defn}[Definition 2.12 \cite{tanimoto2021introduction}]\label{a-cover}
Let $X$ be a smooth Fano variety.  An \textit{a-cover} is a dominant, generically finite, morphism $f:Y\rightarrow X$ of degree greater than one from a smooth projective variety such that $a(Y,-f^*K_X) = a(X,-K_X) = 1$.  

An $a$-cover is \textit{face contracting} if $f_* : \Nef_1(Y)\rightarrow \Nef_1(X)$ identifies two distinct classes $\alpha$ with $\alpha.(K_Y - f^*K_X) = 0$.
\end{defn}

For a class $\alpha \in \Nef_1(Y)$ with zero intersection against $K_Y - a(Y, -f^*K_X)f^*K_X$, the $a$-invariant $a(Y,-f^*K_X)$ approximates the ratio between the expected dimensions of $\Mor(\mathbb{P}^1, Y, \alpha)$ and of $\Mor(\mathbb{P}^1, X, f_*\alpha)$.  We therefore expect $a$-covers of $X$ to exhibit families of rational curves which dominate components of $\Mor(\mathbb{P}^1, X)$.  This motivates the following distinction between \textit{accumulating} and \textit{Manin} components.   %

\begin{defn}[Definition 4.3 \cite{tanimoto2021introduction}]\label{manin component}
Let $X$ be a smooth Fano variety and $f:Y\rightarrow X$ be a morphism from a smooth projective variety that is generically finite onto its image.  The map $f$ is a \textit{breaking map} if one of the following holds:
\begin{itemize}
    \item $f:Y\rightarrow X$ satisfies $a(Y,-f^*K_X) > a(X, -K_X)$,
    \item $f$ is an $a$-cover and $\kappa(K_Y - f^*K_X) > 0$ where $\kappa$ is the Iitaka dimension, or
    \item $f$ is an $a$-cover and $f$ is face contracting.
\end{itemize}

A component $M\subset \Mor(\mathbb{P}^1,X)$ is an \textit{accumulating component} if there exists a breaking map $f:Y\rightarrow X$ and a component $N\subset \Mor(\mathbb{P}^1, Y)$ such that $f$ induces a generically finite and dominant map $N\rightarrow M$.  

We say $M$ is a \textit{Manin component} if $M$ is not an accumulating component.
\end{defn}

\begin{rem}
    For suitable $\alpha$, Conjecture \ref{GMC Tanimoto} expects there to be $|\text{Br}(X)|$ Manin components of $\Mor(\mathbb{P}^1,X,\alpha)$.  A refinement expects there to be one Manin component for each algebraic equivalence class of curves with numerical class $\alpha$.  The connection comes from the Integral Hodge Conjecture and an assumption that algebraic and homological equivalence coincide for 1-cycles on $X$ \cite[Remark~4.13]{lehmann2024nonfree}.  %
\end{rem}

Any component of $\Mor(\mathbb{P}^1,X)$ parameterizing a non-dominant family of curves sweeps out a subvariety $Y$ with $a(Y,-K_X) > a(X,-K_X)$ \cite[Theorem~1.1]{2019}.  On smooth Fano threefolds, these subvarieties are well understood.

\begin{thm}[Theorem 1.2 \cite{beheshti2020moduli}]\label{nonfree curves}
Let $X$ be a smooth Fano threefold.  Let $M\subset \Mor(\mathbb{P}^1, X)$ be a component parameterizing curves which sweep out a proper subvariety $Y\subsetneq X$.  Then either: 
\begin{itemize}
    \item $Y$ is swept out by a family of $-K_X$ lines, or
    \item $Y$ is an exceptional divisor for a birational contraction on $X$.
\end{itemize}
\end{thm}

\begin{rem}\label{Remark: Simplifying Conjecture}
For a smooth Fano threefold $X$, every a-cover $f:Y\rightarrow X$ satisfies $\kappa(K_Y - f^*K_X) > 0$ \cite[Theorem~5.4]{beheshti2020moduli}.  Thus, a component $M\subset \Mor(\mathbb{P}^1,X)$ is Manin if and only if it generically parameterizes free curves and there is no $a$-cover $f : Y \to X$ and component $N \subset \Mor(\mathbb{P}^1, Y)$ such that the induced map $f_* : N  \to \Mor(\mathbb{P}^1, X)$ dominates $M$.  Indeed, any dominant generically finite map $f:Y\rightarrow X$ from a smooth projective variety satisfies $K_Y=f^* K_X + R$ for some effective divisor $R$.  Hence $a(Y,-f^*K_X) \leq a(X, -K_X)$ whenever $f$ is dominant.
\end{rem}

\begin{exmp}\label{example: a-covers}
Accumulating components of $\Mor(\mathbb{P}^1, X)$ with dominant universal family include all families of conics and curves of class in the boundary of $\NE(X)$.  Indeed, the base-change of any nontrivial Fano fibration $\pi : X \rightarrow B$ to a general finite cover $g : B' \rightarrow B$ is an $a$-cover $f : Y \rightarrow X$ where the Iitaka fibration for $K_Y - f^*K_X$ is rationally the base change of $g$.  Families of conics whose class lies interior to $\NE(X)$ admit compactifications that are $a$-covers of $X$ (see Proposition \ref{lehmann a-covers} and Example \ref{exmp: conics interior}).  For a smooth del Pezzo surface $X$, this describes all $a$-covers \cite[Theorem~6.2]{Lehmann_2017}.  For a product Fano threefold $X = X_1 \times X_2$ and any $a$-cover $Y_1 \rightarrow X_1$, the base change $Y_1 \times X_2 \rightarrow X$ is an $a$-cover as well. 
\end{exmp}

\subsection{Classifying $a$-covers of Fano Threefolds}
We show Example \ref{example: a-covers} describes all $a$-covers of smooth Fano threefolds by extending results from \cite{beheshti2020moduli}.  Together with Theorem \ref{nonfree curves} and Theorem \ref{MovableBB}, the classification of $a$-covers in Theorem \ref{classification a-covers} makes Geometric Manin's Conjecture more tractable than Conjecture \ref{manin conjecture}.  

We begin by relating $a$-covers to families of curves.  The following proposition shows components of $\Mor(\mathbb{P}^1, X)$ whose universal families have disconnected fibers over general points in $X$ are accumulating components.

\begin{prop}[Proposition 5.15 \cite{2019}]\label{lehmann a-covers}
Let $X$ be a smooth Fano variety.  Consider a component $M\subset \overline{\free}(X)$.  Let $M'\subset \overline{\free}_1(X)$ be the corresponding component of marked curves. Suppose general fibers of $\text{ev}:M'\rightarrow X$ are reducible.  Let $\tilde{M}'\rightarrow M'$ be a resolution, $\tilde{M}'\rightarrow Y \rightarrow X$ be the Stein factorization, and $\tilde{Y}\rightarrow Y$ be a resolution.  Then $\tilde{Y}\rightarrow X$ is an $a$-cover.
\end{prop}
\begin{proof}
General fibers of $\tilde{M}'\rightarrow X$ must be disconnected, so that $f:\tilde{Y}\rightarrow X$ is not birational.  As $Y$ is normal, an open subset $U\subset M$ parameterizes curves which lift to $\tilde{Y}$.  Let $C\rightarrow \tilde{Y}$ be a general curve parameterized by $U$.  Consider a component $\tilde{M}\subset \free(\tilde{Y},C)$ dominating $M$.  Clearly, $\dim \tilde{M} \geq \dim M$.  As
$$-K_{\tilde{Y}}.C = (-f^*K_X - R).C \leq -K_X . f_* C,$$
we see $\text{dim } \tilde{M} = \text{dim } M$.  Thus $(K_{\tilde{Y}} - f^*K_X) . C = 0$, so $(K_{\tilde{Y}} - f^*K_X)$ is not big.  This implies $a(\tilde{Y},-f^*K_X)=1$.
\end{proof}

\begin{exmp}\label{exmp: conics interior}
     Suppose $\beta \in \Nef_1(X) \setminus \partial\NE(X)$ is an anticanonical conic class.  For any component $M' \subset \overline{\free}_1(X, \beta)$, the evaluation map $\text{ev} : M' \to X$ must have disconnected fibers.  If it did not, then $\text{ev}$ would be an isomorphism and the forgetful map $X \cong M' \to \overline{M}_{0,0}(X, \beta)$ would be a Mori fibration contracting $\beta$.  This contradicts $\beta \not\in \partial\NE(X)$.
\end{exmp}

We prove the few remaining open cases of the following classification theorem using our explicit analysis of families of curves on Fano threefolds and results from Section 6.  The idea hinges on a partial converse to Proposition \ref{lehmann a-covers}.

\begin{thm}\label{classification a-covers}
    Let $X$ be a smooth Fano threefold and suppose $X$ is not a product variety.  Let $f : Y \rightarrow X$ be an $a$-cover.  The Iitaka fibration $\psi : Y \dashrightarrow Z$ for $K_Y - f^*K_X$ satisfies one of the following, depending on $\kappa = \kappa(K_Y - f^*K_X).$
    \begin{itemize}
        \item $\kappa = 1:$ $\psi$ is birational to the base-change of a del Pezzo fibration $X \rightarrow B$;
        \item $\kappa = 2:$ $\psi$ is birational to the base-change of a family of conics on $X$.
    \end{itemize}
\end{thm}

\begin{rem}\label{rem: Iitaka dimensions of a-covers}
    Let $f:Y\rightarrow X$ be an $a$-cover of a smooth Fano threefold $X$.  As $(K_Y -f^*K_X)$ is never big, the Iitaka dimension $\kappa(K_Y - f^*K_X)$ is at most $2$.  By \cite[Theorem~5.4]{beheshti2020moduli}, $\kappa(K_Y - f^*K_X)>0$.  Moreover, any component $N\subset \overline{\free}(Y)$ dominating a component $M\subset \overline{\free}(X)$ generically parameterizes irreducible curves $C$ with $C.(K_Y - f^*K_X) = 0$.  It follows that the Iitaka fibration $\varphi_{K_Y - f^*K_X} : Y\dashrightarrow B$ contracts the moving curve $C$.  This demonstrates $C\rightarrow Y \rightarrow X$ cannot be a very free curve.  We will use this to prove Theorem \ref{classification a-covers}.
\end{rem}

As \cite{beheshti2020moduli} proves Theorem \ref{classification a-covers} when $\kappa = 2$ and when both $-K_X$ is very ample and $\kappa = 1$, we address the six deformation types where $-K_X$ is not very ample and assume $\kappa = 1$.

\begin{thm}[\cite{mori1983classification}]\label{thm: not very ample deformation types}
A smooth Fano threefold $X$ such that $-K_X$ is not very ample is isomorphic to one of the following:
\begin{itemize}
    \item $\mathbb{P}^1 \times S_d$ for a del Pezzo surface $S_d$ of degree $d\leq 2$,
    \item A del Pezzo threefold $V_1$ of degree 1,
    \item A blow up of a del Pezzo threefold $V_d$ of degree $d \leq 2$ along the complete intersection of two members of $|-\frac{1}{2} K_{V_d}|$,
    \item A double cover of $\mathbb{P}^1 \times \mathbb{P}^2$ branched over a smooth divisor of bidegree $(2,4)$,
    \item A double cover of $\mathbb{P}^3$ branched over a smooth sextic surface,
    \item A double cover of a quadric $Q\subset \mathbb{P}^4$ branched over its smooth intersection with a quartic hypersurface,
\end{itemize}
\end{thm}

In particular, $|-K_X|$ is base point free unless $X$ is a blow-up of $V_1$ or $\mathbb{P}^1\times S_1$.  We will also use the following fact in the proof of Theorem \ref{classification a-covers}.

\begin{lem}\label{lem: low degree curves not very ample}
    Let $X$ be a smooth Fano threefold which is not a product variety.  Suppose $X$ is general in moduli and $-K_X$ is not very ample.  Suppose $\alpha \in \Nef_1(X)_{\mathbb{Z}}$ satisfies $\alpha \notin \partial \NE(X)$ and $-K_X . \alpha \in \{3,4\}$.  Then $\free^{bir}(X,\alpha)$ is irreducible.  If $-K_X . \alpha = 4$, $\free^{bir}(X,\alpha)$ generically parameterizes very free curves.
\end{lem}

\begin{proof}

This follows directly from Lemmas \ref{del Pezzo fibration Core}, \ref{lem: del Pezzo low degree curves}, \ref{lem: 2.2 low degree curves}, and Corollary \ref{cor: Picard rank 1 irred cubics and quartics}.
\end{proof}

\begin{proof}[Proof of Theorem \ref{classification a-covers}]
We may replace $Y$ with a blow-up and assume the Iitaka fibration $\psi$ is a proper map $Y \to Z$.  By \cite{beheshti2020moduli}, we need only consider the case when $\kappa = 1$ and $-K_X$ is not very ample. We use our classification of families of low-degree free curves on $X$ to conclude the Iitaka fibration $\psi = \psi_{K_Y - f^*K_X}$ is birational to the base-change of a del Pezzo fibration on $X$.  In particular, we will use Lemmas \ref{lem: low degree curves not very ample}, \ref{nonextreme lines}, and \ref{lem: normal bundle lines}, Theorem \ref{main thm: conics} (proven as Theorem \ref{thm: conics}), and Proposition \ref{low degree curves degeneration}.  The proof of each of these results is independent of Theorem \ref{classification a-covers}.  We split the proof of Theorem \ref{classification a-covers} into a series of claims below.

\begin{claim}\label{claim: a-covers 1}
    For some $\alpha \in N_1(Y)$ with $-K_Y . \alpha \in \{3,4\}$, there exists a component $M \subset \free^{bir}(Y, \alpha)$ such that $M$ generically parameterizes free curves contracted by $\psi$, and $M$ dominates a component $N \subset \free^{bir}(X, f_* \alpha)$ under the induced map.
\end{claim}

\begin{proof}
    Following \cite{beheshti2020moduli} we apply a relative $(K_Y - f^*K_X)$-MMP over $Z$ to $Y$ and let $\widehat{Y}$ be the relative anticanonical model of the result.  A general fiber $\widehat{Y}_z$ is a del Pezzo surface with canonical singularities \cite[Lemma~5.3]{Fanoindex1rank1} which still maps regularly to $X$.  By \cite{xu_rat_connected}, the smooth locus of $\widehat{Y}_z$ is covered by very free rational curves $C \subset \widehat{Y}_z^{sm}$.  A general such curve lies in an open subset of $\widehat{Y}$ where the inverse of $Y \dashrightarrow \widehat{Y}$ is well defined.  Thus, we obtain a free rational curve $C \subset Y$ such that $C$ is contracted by $\psi$, i.e.\ $(K_Y - f^*K_X) . C = 0$.
    
    Suppose $-K_Y . C = -f^*K_X . C > 4$.  Theorem \ref{MovableBB} implies $f : C \to X$ breaks freely into a chain of free curves $g : C' \to X$.  There exists an irreducible component $C_i' \subset C'$ such that $-K_X \cdot g_*[C_i'] \geq 3$.  By Lemma \ref{Basic Properties Chains}, we may assume the restriction $g_i : C_i' \to X$ is general as a map in $\free^{bir}(X, g_*[C_i])$.  As the induced map $f_* : \overline{M}_{0,0}(Y, [C]) \to \overline{M}_{0,0}(X, f_*[C])$ is surjective, there exists a stable map $h : C'' \to Y$ whose image under $f_*$ is $g$.  Hence, there exists a component $C_i'' \subset C''$ such that $(f \circ h)|_{C_i''} = g_i$.  By generality of $g_i$, any deformation of $g_i$ is the composition of $f$ with a deformation of $h|_{C_i''}$.  This shows $(K_Y - f^*K_X) . h_*[C_i''] = 0$.  Hence, we may replace $C \subset Y$ with $C_i'' \subset Y$ and repeat this procedure until $-K_Y \cdot C \in \{3,4\}$.  
\end{proof}

\begin{claim}\label{claim: a-covers 2}
    Let $M \subset \free^{bir}(Y,\alpha)$ be as in Claim \ref{claim: a-covers 1}.  If $f_*\alpha \in \partial\NE(X)$, then $\psi$ is birational to the base-change of a del Pezzo fibration $X \to B$.
\end{claim}

\begin{proof}
    Let $\pi : X \rightarrow B$ be the contraction associated to the minimal face of $\NE(X)$ containing $f_*\alpha$.  Since $f_*\alpha \in \partial\NE(X)$, $\dim B \geq 1$.  Moreover, $M$ generically parameterizes maps which are very free on general $\psi$-fibers because these fibers are surfaces, the maps are immersions, and $-K_Y . \alpha \geq 3$.  Hence, if $p \in Y$ is a general point, the fiber $\psi^{-1}(\psi(p))$ is covered by curves parameterized by $M$ that pass through $p$.  In particular, $\pi$ contracts the image of a general $\psi$-fiber under $f$.  This implies $\pi$ is a del Pezzo fibration and there is an induced map $\phi : Z \dashrightarrow B$. It follows that $\psi$ is birational to the base-change $Z \times_B X \to Z$ of $\pi$. 
\end{proof}

\begin{claim}\label{claim: a-covers 3}
    Suppose $\free^{bir}(X,\gamma)$ is irreducible for all $\gamma \in \Nef_1(X)_{\mathbb{Z}} \setminus \partial\NE(X)$ such that $-K_X . \gamma = 3$.  If $f_*\alpha \not\in \partial\NE(X)$, then by replacing $\alpha$ we may additionally assume $-K_X . f_*\alpha = 4$.
\end{claim}

\begin{proof}
    Suppose $-K_X . f_*\alpha = 3$ and set $N = \overline{\free}^{bir}(X, f_*\alpha)$.  We first argue that one may break curves parameterized by $N$ into the union of a free conic and an anticanonical line.  We then construct a family of free rational quartics on $Y$ by smoothing nodal unions of these conics.   
    
    Suppose $\beta \in \Nef_1(X)_{\mathbb{Z}}$ is an anticanonical conic class such that $(f_*\alpha - \beta) \in \NE(X)$.  For any component $N_{\beta} \subset \overline{\free}(X, \beta)$, we claim $N$ parameterizes nodal curves wherein one irreducible component is parameterized by a general point of $N_{\beta}$.  This follows from the existence of a component of $\Mor(\mathbb{P}^1, X, f_*\alpha - \beta)$ that generically parameterizes maps $f : \mathbb{P}^1 \to X$ with $h^1(\mathbb{P}^1, \mathcal{N}_f) \leq 1$.  Lemmas \ref{nonextreme lines} and \ref{lem: normal bundle lines} imply this statement when $X$ is not a blow-up of a del Pezzo threefold of degree one, and the latter case is easily verified.  
    
    Surjectivity of $f_* : M \rightarrow N$ implies the existence of a component $M_{\beta} \subset \free(Y,\tilde{\beta})$ such that $f_*\tilde{\beta} = \beta$ and  $f_* : M_{\beta} \rightarrow N_{\beta}$ is surjective.  When $\rho(X) = 2$, there are two anticanonical conic classes $\beta \in \Nef_1(X)$ such that $(f_*\alpha - \beta) \in \NE(X)$, as $\alpha$ lies interior to $\NE(X)$ and each extreme ray has length one.  When $\rho(X) = 1$, there is a single anticanonical class $\beta$, but at least two choices for components $N_{\beta} \subset \overline{\free}(X, \beta)$ by Theorem \ref{main thm: conics}.  In either case, we obtain two distinct families of anticanonical conics on $Y$, $M_{1} \subset \overline{\free}_1(Y,\tilde{\beta}_1)$ and $M_{2} \subset \overline{\free}_1(Y,\tilde{\beta}_2)$.  Since these families are distinct, a general curve parameterized by $M_1 \times_Y M_2$ maps birationally to $X$.  Hence, a component of $\free(Y, \tilde{\beta}_1 + \tilde{\beta}_2)$ dominates a component $N_{12} \subset \free^{bir}(X, f_*(\tilde{\beta}_1 + \tilde{\beta}_2))$.  Clearly $f_*\tilde{\beta}_i$ may be chosen so that their sum lies in the interior of $\NE(X)$, and we may replace $\alpha$ with $\tilde{\beta}_1 + \tilde{\beta}_2$.
\end{proof}

\begin{claim}\label{claim: a-covers 4}
    Suppose $-K_X . f_* \alpha = 4$.  If $N = \free^{bir}(X, f_*\alpha)$ is irreducible, then $f_*\alpha \in \partial\NE(X)$.
\end{claim}

\begin{proof}
    For the sake of contradiction, assume $f_*\alpha \notin \partial\NE(X)$. Observe that  $N$ cannot parameterize very free curves, as the Iitaka fibration $\psi$ contracts each curve $M$ parameterizes and $\kappa(K_Y - f^*K_X)$ is positive. If $X$ happens to be general in moduli, this contradicts Lemma \ref{lem: low degree curves not very ample}.
    
    For arbitrary $X$, consider $X$ as a fiber of a general family $\pi : \mathcal{X} \to T$ of smooth deformations of $X$.  The monodromy action of $\pi$ on $\NE(X)$ is trivial because it must permute extremal rays and preserve their type.  Hence, for some $\tilde{\alpha} \in \NE(\mathcal{X}/T)$, $\overline{\free}_2^{bir}(X,f_*\alpha)$ is a geometric fiber of an induced map $\pi_* : \overline{\free}_2^{bir}(\mathcal{X}, \tilde{\alpha}) \to T$.  By our assumptions $f_* \alpha \not\in \partial\NE(X)$ and $-K_X . f_* \alpha = 4$, Lemma \ref{lem: low degree curves not very ample} implies the evaluation map $\overline{\free}_2^{bir}(\mathcal{X}, \tilde{\alpha}) \to \mathcal{X} \times_T \mathcal{X}$ is dominant.  Thus $\overline{\free}_2^{bir}(X,f_*\alpha)$ parameterizes curves $C$ passing through general two points of $X$.  Because $-K_X$ is ample, $C$ must be irreducible and therefore very free. Again, we arrive at a contradiction.
\end{proof}

By the preceding claims, Lemma \ref{lem: low degree curves not very ample} implies Theorem \ref{classification a-covers} when $X$ is general in moduli.  Indeed, pick $\alpha \in N_1(Y)$ as in Claim \ref{claim: a-covers 1}.  If $f_*\alpha \in \partial\NE(X)$, then Claim \ref{claim: a-covers 2} implies the Iitaka fibration $\psi : Y \to Z$ is the base change of a del Pezzo fibration on $X$.  If $f_*\alpha \not\in \partial\NE(X)$, then by Lemma \ref{lem: low degree curves not very ample} and Claim \ref{claim: a-covers 3} we may assume $-K_X . f_* \alpha = 4$.  However, as Lemma \ref{lem: low degree curves not very ample} implies $\free(X,f_*\alpha)$ is irreducible, this contradicts Claim \ref{claim: a-covers 4}.  We conclude that Theorem \ref{classification a-covers} holds for general $X$. 

For arbitrary $X$, we claim $\free^{bir}(X, \gamma)$ is irreducible for all $\gamma \in \Nef_1(X)_{\mathbb{Z}}$ satisfying both $\gamma \notin \partial\NE(X)$ and $-K_X . \gamma \in \{3,4\}$.  By the four preceding claims, verifying this statement will prove Theorem \ref{classification a-covers}. 

As before, consider $X$ as a fiber of a general family $\pi : \mathcal{X} \to T$ of smooth deformations of $X$.  We consider $\free^{bir}(X, \gamma)$ as a geometric fiber of some induced map $\pi_* : \free^{bir}(\mathcal{X}, \tilde{\gamma}) \to T$.  We argue in cases depending on whether $-K_{\mathcal{X}_t}$ is very ample for general $t \in T$.

First, suppose that for general $t \in T$, $-K_{\mathcal{X}_t}$ is not very ample.  Together with Proposition \ref{lehmann a-covers}, the preceding proof of Theorem \ref{classification a-covers} for $\mathcal{X}_t$ implies that general fibers of the evaluation map $\free_1^{bir}(\mathcal{X},\tilde{\gamma}) \to \mathcal{X}$ are irreducible.  Proposition \ref{low degree curves degeneration} therefore shows $\free^{bir}(X,\gamma)$ is irreducible.  

The last deformation type in Theorem \ref{thm: not very ample deformation types} is the solitary case where $-K_X$ is not very ample but general deformations of $X$ have very ample anticanonical bundle.  This is because we consider such $X$ as weighted complete intersections of type $(2,4)$ in $\mathbb{P}(1^5,2)$.  In this case, Corollary \ref{cor: Picard rank 1 irred cubics and quartics} and Proposition \ref{low degree curves degeneration} again show $\free^{bir}(X, \gamma)$ is irreducible for $\gamma \in \Nef_1(X)_{\mathbb{Z}}$ satisfying both $\gamma \notin \partial\NE(X)$ and $-K_X . \gamma \in \{3,4\}$.  By the four preceding claims, this finishes our proof.
\end{proof}

Theorem \ref{classification a-covers} implies the following characterization of Manin components of rational curves on Fano threefolds, mentioned in our introduction.

\begin{cor}\label{cor: classification a-covers}
Let $X$ be a smooth Fano threefold.  A component of $ \Mor(\mathbb{P}^1,X,\alpha)$ is Manin if and only if its universal family has nonempty, connected fibers over general points in $X$ and $\alpha \notin \partial\NE(X)$.
\end{cor}

\begin{proof}
    Let $M \subset \Mor(\mathbb{P}^1,X,\alpha)$ be a component.  Consider the universal family $U \to M$ and its evaluation map $\text{ev} : U \to X$.  By Remark \ref{Remark: Simplifying Conjecture}, $M$ is Manin if and only if $\text{ev}$ is dominant and there is no $a$-cover $f : Y \to X$ such that the induced map $f_* : \Mor(\mathbb{P}^1, Y) \to \Mor(\mathbb{P}^1, X)$ dominates $M$.  

    First we show the only if direction. If general fibers of $\text{ev}$ are disconnected, Proposition \ref{lehmann a-covers} constructs an $a$-cover $f : Y \to X$ as above.  Similarly, Example \ref{example: a-covers} shows $M$ is not Manin if $\alpha \in \partial\NE(X)$.

    Suppose, conversely, that $\alpha \notin \partial\NE(X)$ and general fibers of $\text{ev}$ are nonempty and connected.  Then $M$ parameterizes free curves.  Moreover, by Example \ref{exmp: conics interior} the reduced image of a general map parameterized by $M$ cannot be a conic.  Therefore, Theorem \ref{classification a-covers} implies $M$ is Manin if $X$ is not a product variety.
    
    Suppose $X$ is a product variety instead.  Then $X \cong \mathbb{P}^1 \times S$ for some del Pezzo surface $S$.  Let $\pi_1 : X \to \mathbb{P}^1$ and $\pi_2 : X \to S$ be the projections.  Observe that $\pi_{1*}\alpha \neq 0$ and $\pi_{2*} \alpha \not\in \partial\NE(S)$.  Consider a general map $f : \mathbb{P}^1 \to X$ parameterized by $M$ and let $p \in f(\mathbb{P}^1)$ be a general point.  As general fibers of $\pi_2 \circ \text{ev}$ are connected, Example \ref{exmp: conics interior} shows the reduced image of $\pi_2 \circ f$ is not a $(-K_S)$-conic.  Deformations of $f$ through $p$ must dominate $X$ because they dominate $S$, are not contracted by $\pi_1 : X \to \mathbb{P}^1$, and $\text{Aut}(\mathbb{P}^1)$ is two-transitive.  Therefore, $M$ generically parameterizes very free curves and \cite[Theorem~5.4]{beheshti2020moduli} implies $M$ is Manin.
\end{proof}

\subsection{Very Free Curves}

In this section, we establish the equivalence of Conjecture \ref{GMC Tanimoto} and the following conjecture.  This reformulates Geometric Manin's Conjecture for Fano threefolds into a statement about families of very free curves. Though our paper will not use this reformulation explicitly, its statement guides our approach and the results required to prove it (Proposition \ref{very free curves}, Theorem \ref{Representability of Free Curves}) will be essential to our proof of Theorem \ref{Main Result} in Section \ref{Section 7}.  

We caution that Conjecture \ref{GMC} is not in general true for higher dimensional Fano varieties (see Remark \ref{smoothweakfourfold}). However, its statement is equivalent to Conjecture \ref{GMC Tanimoto} for Fano varieties that do not admit adjoint rigid $a$-covers (see Remark \ref{rem: better GMC}).  %

\begin{conj}\label{GMC}
Let $X$ be a smooth Fano threefold.  There exists $\tau \in \Nef_1(X)_{\mathbb{Z}}$ such that for all $\alpha \in \tau + \Nef_1(X)_{\mathbb{Z}}$, there is at most one component of $\Mor(\mathbb{P}^1, X, \alpha)$ that generically parameterizes very free curves. %
\end{conj}

We will use the following propositions to prove the equvalence of Conjecture \ref{GMC} with Conjecture \ref{GMC Tanimoto} in Proposition \ref{EquivConjecture1}.

\begin{prop}[Proposition 3.1 \cite{patel2020moduli}]\label{prop: grauer Mulich}
    Let $E$ be a semistable, torsion-free sheaf on a smooth projective variety $X$.  Suppose $M \subset \overline{\free}_1(X, \alpha)$ is a component such that the evaluation map $\text{ev}: M \rightarrow X$ has irreducible general fiber.  For a general map $[f : \mathbb{P}^1 \rightarrow X] \in M$, we may write $f^*E \cong \oplus \mathcal{O}(a_i)$ with $|a_i - a_{i+1}| \leq 1$.
\end{prop}

\begin{prop}\label{very free curves}
Let $X$ be a smooth Fano threefold.  A component $M \subseteq \overline{\free}(X, \alpha)$ generically parameterizes very free curves if and only if the corresponding component $M'\subseteq\overline{\free}_1(X,\alpha)$ has irreducible fibers over general points in $X$, $-K_X .\alpha \geq 4$, and $\alpha$ lies in the interior of $\NE(X)$.
\end{prop} 

\begin{proof}
Suppose $M \subseteq \overline{\free}(X, \alpha)$ generically parameterizes very free curves and let $[g : \mathbb{P}^1 \to X] \in M$ be a general point.  Since $g^*\mathcal{T}_X$ is ample and $d g : T_{\mathbb{P}^1} \to g^*\mathcal{T}_X$ is nonzero, $-K_X . \alpha 
 = c_1(g^*\mathcal{T}_X) \geq 4$.  Moreover, $\alpha \notin \partial\NE(X)$ because otherwise there would exist a Mori fibration on $X$ contracting the image of $g$, and a general deformation of $g$ could not pass through two general points in $X$.

It remains to show that $\text{ev} : M' \to X$ has irreducible fibers over general points in $X$.  If it did not, then by Remark \ref{Remark: Simplifying Conjecture} and Proposition \ref{lehmann a-covers}, $M$ would be an accumulating component.  However, this contradicts Remark \ref{rem: Iitaka dimensions of a-covers}.

Conversely, suppose $-K_X.\alpha \geq 4$, $\alpha$ lies in the interior of $\NE(X)$, and %
$M \subseteq \overline{\free}(X, \alpha)$ corresponds to a component $M'$ of one-pointed curves such that $\text{ev}:M'\rightarrow X$ has irreducible fibers over general points.  
\cite[Theorem~4.1]{neumann2009decomposition} shows that each term in the Harder-Narasimhan Filtration $0 = \mathcal{F}_0 \subset \ldots \subset \mathcal{F}_r = \mathcal{T}_X$ of $\mathcal{T}_X$ with respect to $\alpha$ is the relative tangent sheaf of a (possibly non-elementary) Mori fibration.  Because $\alpha$ is interior to $\NE(X)$, no such fibration contracts a general map $f:\mathbb{P}^1\rightarrow X$ parameterized by $M$.  

Let the splitting type of $f^*\mathcal{T}_X$ be  $\mathcal{O}(a_1)\oplus \mathcal{O}(a_2)\oplus \mathcal{O}(a_3)$, ordered such that $a_i \geq a_{i+1}$.  Since each $\mathcal{F}_i / \mathcal{F}_{i-1}$ is torsion free and semistable with respect to $\alpha$, we may assume the above splitting type is a refinement of the associated graded module $\oplus_i f^* (\mathcal{F}_i / \mathcal{F}_{i-1})$.  As $\mathcal{T}_{\mathbb{P}^1}$ admits a nontrivial map to each $f^*( \mathcal{T}_X / \mathcal{F}_{i-1})$, $f^* (\mathcal{F}_r / \mathcal{F}_{r-1})$ must have a direct summand of degree at least two.  By Proposition \ref{prop: grauer Mulich}, this implies $a_3 \geq 1$, as the total degree of $f^*\mathcal{T}_X$ is at least 4.
\end{proof}

\begin{rem}\label{rem: better GMC}
The essential condition for a similar statement to Proposition \ref{very free curves} on a smooth Fano variety $X$ is the absence of $a$-covers $f: Y\rightarrow X$ with $\kappa(K_Y - f^*K_X)=0$.  For such $X$, one can see from \cite[Corollary~7.6]{tangentbundles} that when $-K_X . \alpha \gg 0$, there is a bijection between Manin components of $\Mor(\mathbb{P}^1, X, \alpha)$ and components of $\overline{\free}(X,\alpha)$ generically paramterizing very free curves. %
\end{rem}

This establishes a relationship between Conjecture \ref{GMC Tanimoto} and families of very free curves on Fano threefolds.  %
Moreover, for a smooth Fano threefold $X$, the equality $|\text{Br}(X)| =  |H^3(X,\mathbb{Z})_{\textit{tors}}| = 1$ is well known.  %
To relax the requirement that every $\alpha \in \tau + \Nef_1(X)$ be representable by a very free curve, we prove that $\Nef_1(X)$ coincides with the cone generated by classes of free rational curves below.

\begin{thm}\label{Representability of Free Curves}
For a smooth Fano threefold $X$, every extreme ray of $\Nef_1(X)$ contains the class of a free rational curve of anticanonical degree at most $5$.
\end{thm}

\begin{proof}%
Any of \cite[Theorem~3.3]{batyrevcone}, \cite[Corollary~1.2]{araujocone}, or \cite[Theorem~1.4]{lehmanncone} demonstrate that each extreme ray $v\in \Nef_1(X)$ \textit{corresponds} to the class of a curve in a general fiber of some Mori fibration $Y\rightarrow Z$, where $\phi : X \dashrightarrow Y$ is a birational contraction.  Because $X$ is a smooth Fano threefold, this correspondence is particularly well-behaved.  Recall that the contraction of any extreme ray $e$ of $\NE(X)$ on a smooth threefold is either divisorial or a fibration.  In the former case, the exceptional divisor $E$ satisfies $e.E <0$.  Let $\mathcal{E}_X$ be the set of all such exceptional $E$.  On a smooth Fano threefold, $\mathcal{E}_X$ is precisely the set of irreducible, stably rigid divisors, i.e. non-moving extreme rays of $\text{Eff}(X)$, and every movable divisor is nef.  As observed by \cite[Proposition~1.2]{barkowskiFano}, this proves
$$\Nef_1(X) = \{ v \in \NE(X) | v. E \geq 0 \text{ for all } E \in \mathcal{E}_X\}.$$
Let $v \in \Nef_1(X)$ be an extreme ray.  By descending induction on $\rho(X)$, we may suppose there are no $E\in \mathcal{E}_X$ such that $v.E = 0$ and the contraction of $E$ produces a smooth Fano threefold.  Then by \cite[Theorem~3.3,~Corollary~3.4]{moriMMP}, \cite[Proposition~4.5]{mori1983classification}, and \cite[Lemma~2.3]{beheshti2020moduli}, the set $\{E \in \mathcal{E}_{X} | v. E = 0\}$ contains pairwise disjoint divisors. We may contract all of them via a contraction $\pi$ to obtain a $\mathbb{Q}$-factorial and terminal threefold $Y$.  For any $E\in \mathcal{E}_X$ of $E1$ type with $E.v =0$ and $\mathcal{O}_E(E)\cong \mathcal{O}_{\mathbb{P}^1\times \mathbb{P}^1}(-1,-1)$, we contract a ruling of $e$ of $E$ such that the linear span of $\{E \in \mathcal{E}_{X} | v. E = 0\} \cup \{D \in \Nef^1(X)| (v+e).D = 0\}$ has dimension $\rho(X)-1$.  It follows that $\pi_* v \in \NE(Y)$ is extreme and that $Y$ has isolated singularities \cite[Corollary~5.18]{kollarMMP}.  Thus, if $|\{E \in \mathcal{E}_{X} | v. E = 0\}| < \rho(X) -1 $, the contraction of $\pi_* v$ yields a Mori fiber space $Y \rightarrow Z$ with $\text{dim } Z \geq 1$.  The general fiber $Y_z$ of $Y \rightarrow Z$ is a smooth Fano variety and is thus dominated by a family of rational curves %
avoiding the locus where $\pi^{-1}$ is undefined.  The preimage of a general such curve in $X$ is a free rational curve of class proportional to $v$.  Similarly, when $|\{E \in \mathcal{E}_{X} | v. E = 0\}| = \rho(X) -1 $ and $Y$ is smooth, our claim follows immediately.

We may therefore assume $|\{E \in \mathcal{E}_{X} | v. E = 0\}| = \rho(X) -1 $ and that $v.E=0$ for a contractible divisor of type $E3, E4,$ or $E5$.  By \cite[Proposition~4.1.16]{neumann2009decomposition} and an analysis of primitive Fano threefolds with $\rho(X) = 2$ \cite{matsuki1995weyl}, unless $X$ is of type $3.9$, there is at most one such contractible divisor $E$ on $X$, and the contraction of any other (contractible) divisor on $X$ yields a smooth Fano threefold.  We reference \cite{beheshti2020moduli} for a proof of our claim when $X$ has type $3.9$.  Thus we may assume $\rho(X) = 2$ and $v\in \Nef_1(X)$ lies on $\NE(X) \cap E^\perp$, where $E$ is the contractible divisor of type $E3, E4$, or $E5$, corresponding to an extreme ray $e \in \NE(X)$.

Let $C$ be a free rational curve on $X$.  If $C.E = 0$, $[C]$ is proportional to $v$.  Otherwise, $C.E > 0$, and we may glue a representative of $ne$ for $n=1$ or $n=2$ to $-ne.E$ general deformations of $C$ to obtain a nodal curve $f : Z \rightarrow X$ with globally generated normal bundle.  Specifically, if $E$ has $E3$ type, then $E\cong \mathbb{P}^1 \times \mathbb{P}^1$.  In this case, we let $n=1$ and the representative of $e$ be a ruling of $E$ which meets $C$.  We then glue these curves to obtain $Z$.  If $E$ has $E4$ type, then $E$ is the cone over a smooth conic in $\mathbb{P}^2$, and we let $n=2$ and the representative of $2e$ be a general section over the smooth conic.  We glue two deformations of $C$ to this representative at different points to obtain $Z$.  Lastly, if $E$ has $E5$ type, then $E \cong \mathbb{P}^2$, and we let $n=1$ and $e$ be a line in $E$.  We then glue two deformations of $C$ to $e$ to obtain $f:Z \rightarrow X$.

In each case, $\mathcal{N}_f$ is globally generated.  We may smooth $f$ to a free rational curve $C'$.  If $E$ has $E3$ type, by repeating this process we may assume $[C']$ is proportional to $v$.  Otherwise, $C' . E = 2(C.E -1)$.  Repeating this process if necessary, we may assume $C' . (-K_X) > 5$, so that by Theorem \ref{MovableBB}, $C'$ breaks freely into $C_1 \cup C_2$.  We may replace $C$ with $C_i$ such that $\frac{C_i . E}{C_i . (-K_X)}$ attains its minimum.  By iterating, we attain $C_i$ with $C_i.E = 0$.  Theorem \ref{MovableBB} may then be used to bound $(-K_X).C_i$.%
\end{proof}

\begin{cor}[Lemma 7.5 \cite{beheshti2020moduli}]\label{EquivConjecture2}
For smooth Fano threefolds, $\exists \tau \in \Nef_1(X)_{\mathbb{Z}}$ such that %
each $\alpha \in \tau + \Nef_1(X)_{\mathbb{Z}}$ is represented by a very free rational curve.
\end{cor}
\begin{proof}
We claim that $\tau$ may always be chosen such that $\overline{\free}(X,\alpha)$ is nonempty for all $\alpha\in \tau + \Nef_1(X)$.  This is precisely the statement of \cite[Remark~7.6]{beheshti2020moduli}, which uses Theorem \ref{Representability of Free Curves} and \cite[Theorem~1.3]{ChowGp2014} to find a collection of classes of free curves that generate both $\Nef_1(X)$ as a cone and $N_1(X)_{\mathbb{Z}}$ as an abelian group.  The claim then follows directly from \cite[Proposition~3]{Khovanskii1992NewtonPH}, which gives an explicit such $\tau$.  Increasing $\tau$ by the class of a very free curve finishes our proof.
\end{proof}

\begin{prop}\label{EquivConjecture1}
For a smooth Fano threefold $X$, Conjecture \ref{GMC Tanimoto} is equivalent to Conjecture \ref{GMC}.
\end{prop}
\begin{proof}
By Corollary \ref{cor: classification a-covers}, each Manin component of $\Mor(\mathbb{P}^1,X)$ has a universal family whose fibers over general points in $X$ are connected and nonempty.  Choosing $\tau_0$ in the interior of $\Nef_1(X)$ with $-K_X . \tau_0 \geq 4$ ensures by Proposition \ref{very free curves} that every Manin component parameterizing curves of class $\alpha \in \tau_0 + \Nef_1(X)$ parameterizes very free curves.  Conversely, curves in accumulating components correspond to curves on some $a$-cover $f:Y\rightarrow X$ contracted by the Iitaka fibration for $K_Y - f^*K_X$.  Since there are no $a$-covers with $\kappa(K_Y - f^*K_X) = 0$ \cite[Theorem~5.4]{beheshti2020moduli}, every component of $\Mor(\mathbb{P}^1,X)$ parameterizing very free curves is Manin.  Increasing $\tau_0$ to $\tau$ from Corollary \ref{EquivConjecture2} finishes the proof.
\end{proof}

\begin{rem}\label{smoothweakfourfold}
The statement of Conjecture \ref{GMC} does not generalize well because on other varieties, components of the Kontsevich space parameterizing very free rational curves may not be Manin.  For example, consider the smooth weak Fano fourfold $X = \text{Hilb}^{[2]}(S)$, where $S= \mathbb{P}^1 \times \mathbb{P}^1$.  An $a$-cover of $X$ is given by the natural map from the blow up of $S\times S$ along the diagonal \cite[Example~4.2]{tanimoto2021introduction}.  For $n > 2$,  $\text{Hilb}^{[2]}(\mathbb{P}^n)$ is a smooth Fano variety with a similar $a$-cover. %
\end{rem}

\section{Irreducibility of Moduli Spaces of Free Curves}\label{Section 4}%

Section 4.1 outlines a general strategy for proving irreducibility of spaces of free curves on smooth varieties.  %
Section 4.2 simplifies the application of Theorem \ref{Main Method} to smooth Fano threefolds by verifying condition \ref{Main Method}(2) in all appropriate circumstances.  As Theorem \ref{Main Method} requires understanding families of low-degree curves, Section 4.3 outlines a few results which may be used to count their components. %

\subsection{General Strategy}

We will now describe our main method for proving irreducibility of $\overline{\free}(X,\beta)$.  After stating the general technique in Theorem \ref{Main Method}, we outline basic methods of applying it in Lemmas \ref{Relations} and \ref{Gluing}.%

\begin{thm}\label{Main Method}
Let $X$ be a smooth variety with a (weak) core of free curves $\mathscr{C}_X$ and suppose:
\begin{enumerate}[\hspace{.6cm}(1)]
    \item For each relation $\sum \alpha_i = \sum \alpha'_j$ in a generating set of relations in $\mathbb{N}\mathscr{C}_X$, there exists a main component $M \subset \prod_X \overline{\free}_2(X,\alpha_i)$ and a main component $N \subset \prod_X \overline{\free}_2(X,\alpha'_j)$ such that the images of $M$ and $N$ under the appropriate gluing maps lie in the same irreducible component of $\overline{\free}_2(X,\sum \alpha_i)$.
\end{enumerate}
Fix $\beta \in \Nef_1(X)$. If: %
\begin{enumerate}[\hspace{.6cm}(2)]
    \item For any way of writing $\beta$ as a sum $\beta= \alpha_1 + \ldots + \alpha_r$ of elements of $\mathscr{C}_X$, there exists a re-ordering of the $\alpha_i$ such that each main component of $\prod_X \overline{\free}_2(X,\alpha_i)$ maps into the same component of $\overline{\free}_2(X,\beta)$,
\end{enumerate}
then $\overline{\free}(X,\beta)$ is either irreducible or empty.
\end{thm}

\begin{rem}
Lemma \ref{Relations} will provide an algorithm for identifying a generating set of relations as in \ref{Main Method}(1).  Given a relation $\sum \alpha_i = \sum \alpha_j'$ in this list, common strategies for checking \ref{Main Method}(1) include (a) proving that $\free(X,\sum \alpha_i)$ is irreducible, or (b) finding a map $Z\rightarrow \overline{\free}(X,\sum \alpha_i)$ whose domain is a nodal curve $Z$, whose image intersects the two main components, and such that each node of $Z$ maps to a smooth point of $\overline{\mathcal{M}}_{0,0}(X)$.  %
\end{rem}

\begin{exmp}
The following is another strategy for verifying \ref{Main Method}(1) which we use throughout the text. Suppose $X$ is a smooth Fano threefold, $\mathscr{C}_X$ is a core of free curves on $X$, $\ell \in \NE(X)_{\mathbb{Z}}$ is the class of a $(-K_X)$-line, and %
$\alpha_1, \alpha_2 \in \mathscr{C}_X$ satisfy $\alpha_i + \ell  \in \mathscr{C}_X$.  There exists a relation in the monoid $\mathbb{N}\mathscr{C}_X$ of the form $(\alpha_1 + \ell) + \alpha_2 = \alpha_1 + (\alpha_2 + \ell)$.  Suppose some component $M \subset \overline{M}_{0,2}(X,\ell)$ generically parameterizes 2-pointed maps $f: \mathbb{P}^1 \rightarrow X$ with normal bundle $\mathcal{N}_f \cong \mathcal{O} \oplus \mathcal{O}(-1)$.  In this case, there are maps parameterized by $\free_1(X,\alpha_1) \times_X M \times_X \free_1(X, \alpha_2)$ which correspond to smooth points of $\overline{\free}(X, \alpha_1 + \ell + \alpha_2)$ and which generalize to free chains of type $(\alpha_1 + \ell, \alpha_2)$ and $(\alpha_1, \alpha_2 + \ell)$.  Thus, there are main components of $\overline{\free}_1(X,\alpha_1 + \ell) \times_X \overline{\free}_1(X,\alpha_2)$ and $\overline{\free}_1(X,\alpha_1) \times_X \overline{\free}_1(X,\alpha_2 + \ell)$ in the same component of $\overline{\free}(X, \alpha_1 + \ell + \alpha_2)$. %
\end{exmp}

\begin{proof}[Proof of Theorem \ref{Main Method}]
By Definition \ref{def core}, %
every component of $\overline{\free}(X)$ contains a main component of free chains of type $(\alpha_1, \ldots, \alpha_n)$ with $\alpha_i \in \mathscr{C}_X$.
Let $M,N\subset \overline{\free}_2(X,\beta)$ be main components of free chains of type $\alpha_M=(\alpha_1,\ldots , \alpha_r)$ and  $\alpha_N=(\alpha_1',\ldots, \alpha_s')$, respectively.  By Lemma \ref{Basic Properties Chains}(3), we may reorder $\alpha_M$ and $\alpha_N$ to assume $\alpha_i = \alpha_i'$ for $i<n$ and $\{\alpha_{n}, \ldots , \alpha_r\} \cap \{\alpha_{n}', \ldots , \alpha_s'\} = \emptyset$.  If $\alpha_M=\alpha_N$, condition \ref{Main Method}(2) guarantees $M$ and $N$ are contained in the same component of $\overline{\free}(X,\beta)$.  Otherwise, as $\sum_{i=1}^r \alpha_i = \sum_{j=1}^s \alpha_j'$, a sequence of relations described in condition \ref{Main Method}(1) may be used to transform $\alpha_M$ into $\alpha_N$ through substitution and reordering.  By induction, we may assume only a single substitution is required, so that the relation is given by $\sum_{i\geq n} \alpha_i = \sum_{j\geq n} \alpha_j'$.  By condition \ref{Main Method}(1), there exist main components
\begin{align*}
    M'\subseteq \  &\overline{\free}_2(X,\alpha_{1}) \times_X \ldots \times_X \overline{\free}_2(X,\alpha_{n-1})\\
    M''\subseteq \  &\overline{\free}_2(X,\alpha_{n}) \times_X \ldots \times_X \overline{\free}_2(X,\alpha_{r})\\
    N''\subseteq \ &\overline{\free}_2(X,\alpha_{n}') \times_X \ldots \times_X \overline{\free}_2(X,\alpha_{s}')
\end{align*}   
of the corresponding products such that $M'',N''$ lie in the same component $Z \subseteq \overline{\free}_2(X,\sum_{i\geq n} \alpha_i)$.  Note that for any two main components $M_1,M_2$ of free chains with markings on each end, there exists a main component of $M_1\times_X M_2$ dominating $M_1$ and $M_2$ under projection.  Condition (2) therefore allows us to assume $M\subseteq M'\times_X M''$ and $N\subseteq M'\times_X N''$.  It follows from Lemma \ref{Basic Properties Chains}(2) that a main component of $M'\times_X Z$ contains both $M$ and $N$.  
\end{proof}

If $\beta$ is an anticanonical conic in the interior of $\Nef_1(X)$, there are always multiple components of $\overline{\free}(X,n\beta)$ for $n>1$.  Those components which do not parameterize multiple covers of conics may be identified by the types of free chains they contain.

\begin{defn}\label{Def Dagger}
For a chain of free curves, denote by $(\dagger)$ the condition that two of its components have distinct reduced image, or one component has reduced image of anticanonical degree at least 3.  We say a main component \textit{satisfies} $(\dagger)$ if it parameterizes free chains which satisfy $(\dagger)$.
\end{defn}

\begin{cor}\label{bir Main Method}
Under the same hypotheses as Theorem \ref{Main Method}, %
if \ref{Main Method}(1) holds and \ref{Main Method}(2) holds for the subset of main components of $\prod_X \overline{\free}_2(X,\alpha_i)$ that satisfy $(\dagger)$, then $\overline{\free}^{bir}(X,\beta)$ is irreducible or empty.
\end{cor}
\begin{proof}
For main components $M,N\subset \overline{\free}_2^{bir}(X,\beta)$, the proof of Theorem \ref{Main Method} holds verbatim after we note that for any nontrivial relation $\sum_{i\geq n} \alpha_i = \sum_{j\geq n} \alpha_j'$, $-K_X.(\sum_{i\geq n} \alpha_i) > 3$ and the main components $M'',N''$ given by \ref{Main Method}(1) must lie in a component $Z \subseteq \overline{\free}^{bir}_2(X, \sum_{i\geq n} \alpha_i)$.  Thus free chains parameterized by $M''$ and $N''$ satisfy $(\dagger)$.
\end{proof}

To find the relations mentioned in Theorem \ref{Main Method}(1), we apply the following.

\begin{lem}\label{Relations}
Let $c_1, \ldots , c_n$ be nef curve classes on a variety $X$ such that $N_1(X)$ is spanned by $c_{n-\rho(X)}, \ldots , c_n$.  Suppose $D_1, \ldots , D_{n-\rho(X)-1}$ are divisor classes such that $c_i . D_i < 0$ and $c_j . D_i \geq 0$ for $j > i$.  %
A generating set of relations among $c_i$ as generators of an additive monoid include a unique relation among $c_{n-\rho(X)},\ldots , c_n$ and are otherwise of the form $r c_i + \sum_{k>i} a_k c_{k} = \sum_{k>i} b_k c_{k}$, where %
\begin{enumerate}
    \item $r, a_k, b_k \in \mathbb{Z}_{\geq 0}$ and $i \leq n - \rho(X)-1$,
    \item $r>0$ and either $a_k = 0$ or $b_k = 0$ for all $k >i$,
    \item $c_{k} . D_i > 0$ for some $k$ with $a_k > 0$,
    \item with respect to the product order on $\mathbb{Z}_{\geq 0}^{n-i+1}$, $(r, a_{i+1}, \ldots , a_n)$ is a minimal element of the subset of tuples appearing in a relation as above.
\end{enumerate}
\end{lem}
\begin{proof}
Let $\sum_{k\geq 1} a_k c_{k} = \sum_{k\geq 1} b_k c_{k}$ be any relation among $c_1, \ldots c_n$ as additive generators of a monoid.  Note that $a_k, b_k \geq 0$.   By cancellation, we may assume either $a_k = 0$ or $b_k = 0$.  Therefore, the relation above reduces to one of the form $r c_i + \sum_{k>i} a_k c_{k} = \sum_{k>i} b_k c_{k}$ where $r > 0$ and either $a_k = 0$ or $b_k = 0$ for all $k > i$.

First, suppose $i \geq n - \rho(X)$.  Since $c_{n- \rho(X)}, \ldots , c_n$ span $N_1(X)$, we must have $i = n - \rho(X)$.  There is a unique relation involving $c_{n- \rho(X)}, \ldots , c_n$ which satisfies condition \ref{Relations}(2) and whose coefficients have no common factors.

Suppose $i < n - \rho(X)$ instead.  As $(\sum_{k>i} b_k c_{k}).D_i \geq 0$ and $r c_i . D_i < 0$, for some $k > i$ we must have $a_k c_k . D_i > 0$.  Therefore the relation $r c_i + \sum_{k>i} a_k c_{k} = \sum_{k>i} b_k c_{k}$ satisfies conditions \ref{Relations}(1-3). %

Finally, we show the relation $c_i + \sum_{k>i} a_k c_{k} = \sum_{k>i} b_k c_{k}$ can be built from relations where condition (4) is satisfied. Indeed, suppose there exists an expression $r' c_i + \sum_{k>i} a_k' c_{k} = \sum_{k>i} b_k' c_{k}$ as above with $r' \leq r$ and $a_k' \leq a_k$ for all $k$.  We may assume $(r', a_{i+1}', \ldots , a_k') \in \mathbb{Z}_{\geq 0}^{n-i+1}$ is a minimal element satisfying this property.  Substituting, we obtain a relation $(r-r')c_i + \sum_{k>i} (a_k - a_k'+b_k')c_{k} = \sum_{k>i} b_k c_{k}$.  As $r-r' < r$, we can repeat this process until $r-r' = 0$. The claim then follows from induction on $i$.
\end{proof}
\begin{rem}\label{rem: Relations}
It is helpful to note that $c_i$ must be extreme in the cone generated by $\{c_i,\ldots , c_n\}$.  Moreover, we may choose $D_i$ such that $c_j . D_i =0$ for $\rho(X)-1$ values of $j> i$.  Frequently, if $(c_i + c_j).D_i \geq 0$, we find a relation $c_i + c_j = \sum_{k>i} b_k c_{k}$.  This implies each other minimal relation amongst $\{c_i ,c_{i+1}, \ldots , c_n\}$ does not involve both $c_i$ and $c_j$.
\end{rem}

\subsection{Fibers of the Evaluation Map} 
In this subsection, we explore how fibers of the evaluation map control the number of main components of a certain type of free chain.  We also study fibers of the evaluation map for families obtained by smoothing these free chains. This will later allow us to inductively verify condition \ref{Main Method}(2) in Corollary \ref{connected fibers fixall}.

\begin{lem}[Gluing and Smoothing, Lemma 3.2 \cite{tangentbundles}]\label{Gluing}
Let $X$ be a smooth projective variety and $M_1,M_2\subset \overline{\free}_1(X)$ be components.  If general fibers of $\text{ev}:M_1\rightarrow X$ are irreducible, then there is exactly one main component $M\subset M_1 \times_X M_2$.  If general fibers of $\text{ev}:M_2\rightarrow X$ are irreducible as well, then the family of one-pointed curves parameterized by $M$ lies in a component $N\subset \overline{\free}_1(X)$ whose general fiber under $\text{ev} : N\rightarrow X$ is irreducible. 
\end{lem}

\begin{proof}
Suppose $\text{ev}:M_1 \rightarrow X$ has irreducible fibres over general points in $X$.  Since each $M_i$ generically parametrizes free curves, general fibers of $\text{ev}:M_i \rightarrow X$ are nonempty and have the expected dimension.  This implies exactly one component $M\subset M_1\times_X M_2$ dominates $M_2$ under projection.

Suppose general fibers of $\text{ev}: M_2 \rightarrow X$ are also irreducible.  Let $ N\subset \overline{\free}_1(X)$ be the unique component containing the locus $M'$ of one-pointed curves parameterized by $M$.  Let $p \in X$ be a general point.  We will show the fiber $N_p$ of $\text{ev} : N \to X$ over $p$ is irreducible by studying the fiber of $M'$ over $p$.

Let $C_i$ denote a general curve parameterized by $M_i$.  By Lemma \ref{Basic Properties Chains}(5), if $p\in X$ is a general point and $[f] \in M$ is a general free chain passing through $p$, then $f$ maps the node $q \in C_1 \cup_q C_2$ to a point in $X$ over which both $M_1$ and $M_2$ have irreducible fibers.  Therefore, the fiber $M_p' \subset N_p$ of $\text{ev}: M'\rightarrow X$ over $p$ is connected and consists of two irreducible components $M_{1,p}'$ and $M_{2,p}'$.  The component $M_{i,p}'$ parameterizes curves whose marked point lies on $C_i \subset C_1 \cup_q C_2$. 

A general point in the intersection $M_{1,p}' \cap M_{2,p}'$ corresponds to a pointed stable map $f : (C, q) \to X$ where $C = C_1 \cup D \cup C_2$, $q \in D$, and $f(D) = p$.  Since $f|_{C_i}$ is free, we see that $h^1(C,f^*\mathcal{T}_X(-q))=0$ by Proposition \ref{kontsevich space H1}.  This implies $N_p$ is smooth at the point $[f : (C, q) \to X]$.  Hence, both $M_{1,p}'$ and $M_{2,p}'$ lie in the same irreducible component of $N_p$.   As the monodromy action of $N \rightarrow X$ is transitive on irreducible components of $N_p$ but stabilizes $M_p'$, we see there is at most one component of $N_p$.
\end{proof}

If general fibers of some family of free curves are reducible instead, there may be multiple main components of some type of free chain.  For this purpose, the following lemmas extend Lemma \ref{Gluing} slightly to cover circumstances where monodromy of a fibration $\pi : X\rightarrow \mathbb{P}^1$ creates reducible fibers.

\begin{lem}\label{gluing del Pezzo}
Let $X$ be a smooth projective variety with a map $\pi: X \rightarrow \mathbb{P}^1$.  Supppose a component $M_1 \subset \overline{\free}_1(X)$ parameterizes curves contracted by $\pi$, and a component $M_2 \subset \overline{\free}_1(X)$ generically parameterizes finite covers of $\mathbb{P}^1$ under composition with $\pi$.  Suppose for some $i \in \{1,2\}$ and a general $p\in \mathbb{P}^1$, each irreducible component of the fiber $(M_{i})_p = (\pi\circ \text{ev}_i)^{-1}(p)$ has irreducible fibers over general points in $X_p$.  Then there is exactly one main component $M \subset M_1 \times_X M_2$.
\end{lem}
\begin{rem}\label{big divisor remark}
Suppose instead $M_2 \subset \overline{M}_{0,1}(X)$ generically parameterizes sections and maps birationally to an irreducible divisor $D \subset X$.  Suppose the restriction $D_p \subset X_p$ of $D$ to a general fiber of $\pi$ is big and each component of $(M_1)_p$ has irreducible fibers over general points in $D_p$.  In this case, the following proof shows there is one component $M \subset M_1 \times_X M_2$ which dominates $M_2$ under the second projection and dominates the family of $0$-pointed free curves corresponding to $M_1$ under the composition of the first projection with the forgetful map.  %
\end{rem}
\begin{proof}[Proof of Lemma \ref{gluing del Pezzo}]
Let $\tilde{M}_1$ be a smooth resolution of $M_1$, $\tilde{M}_1 \rightarrow Y \xrightarrow{f} X$ be the Stein factorization of the natural map, and $\tilde{M}_1 \xrightarrow{\psi} B \rightarrow \mathbb{P}^1$ be the Stein factorization of $\tilde{M}_1 \rightarrow X \rightarrow \mathbb{P}^1$.  We may assume that $Y$ and $B$ are smooth by taking resolutions.  Consider a main component $M\subset M_1 \times_X M_2$ and let $g: C \rightarrow X$ be a general curve parameterized by $M_2$. 

Let $N_{1,p} \subset (\tilde{M}_1)_p$ and $N_{2,p} \subset (M_2)_p$ be any irreducible components.  There exists $q \in C$, such that the pointed map $[g : (C,q) \to X]$ lies in $N_{2,p}$.  Since $C$ is general, $\pi \circ g : C \rightarrow \mathbb{P}^1$ is \'etale over the branch points of $B \rightarrow \mathbb{P}^1$, and the fiber $C_p \subset X_p$ is a collection of general reduced points in $X_p$. Therefore, $M$ contains curves parameterized by the attachment of $C$ at $q$ to curves parameterized by each irreducible component of $(\tilde{M}_1)_p$.  In other words, for any two irreducible components $N_{1,p} \subset (\tilde{M}_1)_p$ and $N_{2,p} \subset (\tilde{M}_2)_p$, $M$ contains a $N_{1,p} \times_{X_p} N_{2,p}$ that generically parameterizes free chains.

First suppose that $N_{1,p}$ has irreducible fibers over general points in $X_p$.  This implies only one component of $N_{1,p} \times_{X_p} N_{2,p}$ generically parameterizes free chains, and the preceding paragraph shows it must lie in $M$.  Since the monodromy group of $\pi$ acts transitively on irreducible components of $(M_2)_p$,  $M$ is the only main component of $M_1 \times_X M_2$.  

Suppose instead that $N_{1,p}$ has reducible fibers over general points in $X_p$. By our hypotheses, $N_{2,p}$ has irreducible fibers over general points in $X_p$.  Curves parameterized by $N_{1,p}$ map to an irreducible component of $Y_p$ which is a generically finite cover of $X_p$.  Hence, $[g: (C,q) \rightarrow X]$ is contained in a one-parameter family of maps corresponding to a curve in $X_p$ with irreducible preimage in $N_p$.  We again conclude there is only one component of $N_{1,p} \times_{X_p} N_{2,p}$ that generically parameterizes free chains.  It follows that $M \subset M_1 \times_X M_2$ is the unique main component.
\end{proof}

Fibrations on $X$ may also be used to prove certain components of $\overline{\free}_1(X)$ have irreducible fibers over general points in $X$.  Below, an \textit{adjoint rigid} $a$-cover is an $a$-cover $\phi : Y \rightarrow X$ of a smooth Fano variety with $\kappa(K_Y - \phi^* K_X) = 0$.  There are no adjoint rigid $a$-covers of smooth del Pezzo surfaces \cite[Theorem~6.2]{Lehmann_2017}.

\begin{lem}[cf.\ Lemma 7.9 \cite{sectionsDelPezzo}]\label{irreducible fibers del Pezzo}
Let $X$ be a smooth Fano variety with a fibration $\pi : X \rightarrow \mathbb{P}^1$ such that a general fiber $X_p = \pi^{-1}(p)$ admits no adjoint rigid $a$-covers.  Suppose a component $M \subset \overline{\free}_1(X)$ generically parameterizes finite covers of $\mathbb{P}^1$ under composition with $\pi$.  Let $\tilde{M}\rightarrow M$ be a smooth resolution and $g: \tilde{M} \rightarrow B$, $h: B \rightarrow \mathbb{P}^1$ be the Stein factorization of $f : \tilde{M}\rightarrow \mathbb{P}^1$.
\begin{itemize}
    \item If the smooth locus of $M$ parameterizes a stable map $t : C \to X$ such that the restriction of $t$ to some component of $C$ is a section of $\pi$, then $h = \text{id}$.%
    \item If $h = \text{id}$ and general curves in a Manin component $N \subset \overline{\free}(X_p)$ appear as components of some map parameterized by $M$, then $M\rightarrow X$ has irreducible general fiber.
\end{itemize}
\end{lem}

\begin{proof}
Denote by $\psi: \tilde{M} \rightarrow Y$, $\phi : Y \rightarrow X$ the Stein factorization of $\tilde{M}\rightarrow X$.

Suppose that for some stable map $t : C \to X$ parameterized by $M^{sm}$, the restriction of $t$ to a component $C' \subset C$ is a section of $\pi$. We show $h = id$.  Clearly, $g$ factors through $\psi$.  Examining the resulting commutative diagram, we see that it suffices to show $Y \rightarrow \mathbb{P}^1$ has irreducible fibers. Pick a general point $p \in \mathbb{P}^1$ and consider the fiber $Y_p$. Because $Y$ is irreducible, the monodromy group $G$ of $\pi \circ \phi$ acts transitively on irreducible components of $Y_p$.

Let $[t : (C, q) \to X] \in M$ be the unique pointed map such that $q \in C'$ and $t(q) = p$.  As $t|_{C'}$ is a section of $\pi$, varying the marked point in $C'$ induces the full action of $G$ on components of $Y_p$. But, on the other hand, the curve $D \subset \tilde{M}$ obtained by varying the marked point inside of $C'$ maps bijectively onto $t(C') \subset X$ by $\phi \circ \psi$. Therefore, the point $[t : (C, q) \to X]$ lies in every irreducible component of $Y_p$.  We conclude that $Y_p$ is irreducible.

Suppose $h = \text{id}$ so that the fiber $Y_p$ is irreducible, and that general curves parameterized by a Manin component $N \subset \overline{\free}(X_p)$ appear as components of curves parameterized by $M$.  There is a component of $M_Y \subset \overline{\free}(Y)$ dominating $M \subset \overline{\free}(X)$ under composition with $\phi : Y \rightarrow X$.  In particular, components of some curves parameterized by $M_Y$ are parameterized by a component $N_Y \subset \overline{\free}(Y_p)$ which dominates $N$.  The dimension of $N_Y$ is at most the dimension of $N$, as $K_{Y_p} - \phi^* K_{X_p}$ is pseudoeffective.  %
For any curve $C$ parameterized by $N_Y$, this implies $(K_{Y_p} - \phi^* K_{X_p}) . C = 0$.  Thus $K_{Y_p} - \phi^* K_{X_p}$ is pseudoeffective but not big, and so $a(Y_p, -\phi^*K_{X_p}) = 1 = a(X_p, -K_{X_p})$.  Since $N$ is Manin %
and $\kappa(K_{Y_p} - \phi^* K_{X_p}) \neq 0$ by assumption, $\phi$ must be birational. %
\end{proof}

\subsection{General Complete Intersections} To calculate the number of components of $\overline{\free}(X,\alpha)$ for $\alpha$ in a (weak) core of free curves $\mathscr{C}_X$, we often employ results from Section 5 related to the Mori structure of $X$.  When $X$ has Picard rank one, these results are not useful.  Here we outline a technique for proving irreducibility of spaces of low-degree free curves on general Fano complete intersections of very ample divisors inside well-understood ambient spaces.  %
These results will be used to study certain Fano threefolds of Picard rank one and two. %

\textbf{Convention:} In what follows, let $\pi : \mathcal{X} \rightarrow B$ be a flat projective family whose general fiber is a smooth Fano variety.  Suppose in addition the relative cone of curves $\NE(\pi)$ is canonically isomorphic to the Mori cone $\NE(\mathcal{X}_b)$ of a smooth fiber; that is, the monodromy action of $\pi$ on the Neron-Severi group of a smooth fiber is trivial.  We identify curve classes $\alpha \in \NE(\mathcal{X}_b)$ with the corresponding class in $\NE(\pi)$.  Note that every component of $\overline{\free}(\mathcal{X}_b,\alpha)$ lies in a component of $\overline{\free}(\mathcal{X},\alpha)$.  It is often simpler to prove irreducibility of $\overline{\free}(\mathcal{X}, \alpha)$ using regularity properties for curves.  The following results describe conditions on $\mathcal{X}$ and $\alpha$ which allow us to prove irreducibility of $\free^{bir}(\mathcal{X}_b, \alpha)$ given irreducibility of $\free^{bir}(\mathcal{X}, \alpha)$. %

\begin{lem}\label{lem: boundary stratum irreducible}
    Consider a family $\pi : \mathcal{X}\rightarrow B$ as above and let $b \in B$ be general. %
    For $\alpha \in \Nef_1(\mathcal{X}_b)$, consider an irreducible component $\mathcal{M} \subset \overline{\free}(\mathcal{X}, \alpha)$.  If some subvariety $\mathcal{Z} \subset \mathcal{M}$ meeting the smooth locus of $\mathcal{M}$ has irreducible fiber $\mathcal{Z}_b$ over $b$, then $\mathcal{M}_b$ is irreducible.
\end{lem}

\begin{proof}
    The monodromy action of $\pi_* : \mathcal{M} \rightarrow B$ acts transitively on irreducible components of $\mathcal{M}_b$.  By generality of $b$, the induced map $\mathcal{Z} \rightarrow B$ is dominant.  Hence, each irreducible component of $\mathcal{M}_b$ contains $\mathcal{Z}_b$.  Since $\mathcal{Z}$ meets the smooth locus of $\mathcal{M}$, the general point of $\mathcal{Z}_b$ is a smooth point of $\mathcal{M}_b$.
\end{proof}

We apply Lemma \ref{lem: boundary stratum irreducible} when $\mathcal{X}_b$ is a smooth Fano threefold of Picard rank at least two.  To study Fano threefolds of Picard rank one, we instead apply the following result.  This requires additional constraints on the family $\pi : \mathcal{X}\rightarrow B$. %

\begin{prop}\label{prop: terminal away from codim two}

Let $\pi : \mathcal{X}\rightarrow B$ be a family as above.  Suppose $B$ is smooth, simply connected, and for $b\in B$ outside a codimension two subset, $\mathcal{X}_b$ is a terminal Gorenstein Fano threefold with basepoint free anticanonical linear series. Consider a component $\mathcal{M} \subset \overline{\free}(\mathcal{X}, \alpha)$.  If the general map parameterized by $\mathcal{M}$ has balanced normal bundle, then for general $b \in B$ the fiber $\mathcal{M}_b$ is irreducible.   
\end{prop}

More generally, we may replace the hypothesis on $\mathcal{X}_b$ with the following: for $b\in B$ outside a codimension 2 subset, each component of the fiber $\mathcal{M}_b$ parameterizing a dominant family of curves on $\mathcal{X}_b$ is generically reduced.  The proof of this statement follows directly from the beginning of our argument below.

\begin{proof}
    Suppose to the contrary that there exists a component $\mathcal{M} \subset \overline{\free}(\mathcal{X}, \alpha)$ generically parameterizing curves with balanced normal bundle such that the general fiber of $\mathcal{M}$ over $B$ is reducible.  Let $\tilde{\mathcal{M}} \rightarrow \mathcal{M}$ be a resolution of singularities, and consider the finite part $B' \rightarrow B$ of the Stein factorization of $\tilde{\mathcal{M}} \rightarrow B$.  Since $B$ is smooth and simply connected, the branch locus of $B' \rightarrow B$ has codimension one in $B$.  We let $b \in B$ be a general point of the branch locus and consider the fiber $\mathcal{M}_b$ of $\mathcal{M}$ over $b$.

    Since $\tilde{\mathcal{M}} \rightarrow B$ is branched over $b$, a connected component of $\mathcal{M}_b$ is everywhere non-reduced.  We claim that $\mathcal{M}_b$ contains an everywhere non-reduced irreducible component $N$ of dimension at least $-K_{\mathcal{X}/B}. \alpha$.   Furthermore, we may assume $N$ parameterizes curves passing through $\lfloor-\frac{1}{2}K_{\mathcal{X}/B}. \alpha \rfloor$ general points in $\mathcal{X}_b$, and, if $-K_{\mathcal{X}/B}. \alpha$ is odd, a general complete intersection curve in $\mathcal{X}_b$. %
    As the general curve parameterized by $\mathcal{M}$ has balanced normal bundle, these claims follow from upper semicontinuity of the fiber dimension of the evaluation map to $\tilde{X} \times_B B'$.

    We may identify $N$ with a non-reduced component of $\overline{M}_{0,0}(\mathcal{X}_b, \alpha)$. We claim this contradicts the hypothesis that $\mathcal{X}_b$ is terminal Gorenstein Fano threefold with basepoint free anticanonical linear series.  Indeed, let $f : C \rightarrow \mathcal{X}_b$ be the general map parameterized by $N$.  Recall that the image of $C$ passes through a general point of $\mathcal{X}_b$. If $C$ is irreducible, $f(C)$ cannot lie in the smooth locus of $\mathcal{X}_b$, as this implies smoothness of $N$ at $[f]$.  However, if $f(C)$ meets the singular locus of $\mathcal{X}_b$, terminality of $\mathcal{X}_b$ implies $N$ has less than the expected dimension $-K_{\mathcal{X}/B}. \alpha$, contrary to our hypothesis on $N$.  We conclude that the domain $C$ of a general map parameterized by $N$ must be reducible. %

    When $-K_{\mathcal{X}/B}. \alpha$ is even, there are no connected reducible curves of the appropriate degree meeting $-\frac{1}{2}K_{\mathcal{X}/B}.\alpha$ general points in $\mathcal{X}_b$. %
    Hence, we may assume $-K_{\mathcal{X}/B}. \alpha$ is odd.  Let $f : C \rightarrow \mathcal{X}_b$ be a map parameterized by $N$ which passes through $\lfloor-\frac{1}{2}K_{\mathcal{X}/B}. \alpha \rfloor$ general points $p_1, \ldots p_n \in \mathcal{X}_b$ and a general complete intersection curve $T \subset \mathcal{X}_b$.  It follows from terminality of $\mathcal{X}_b$ that $C = C_1 \cup C_2$ has two irreducible components $C_i$ satisfying the following conditions: $-K_{\mathcal{X}/B}. f_*[C_2] = 1$; $f(C_1)$ passes through $p_1, \ldots , p_n$; and $f(C_2)$ meets $T$.  We will show $[f]$ must be a smooth point of $N$, contradicting that $N$ is everywhere non-reduced.

    Consider the subvariety $Y \subset \mathcal{X}_b$ swept out by deformations of the map $f|_{C_2}$.  Since $-K_{\mathcal{X}_b} = -K_{\mathcal{X}/B}$ is Cartier, \cite[Theorem~1.1]{2019} implies $a(Y, -K_{\mathcal{X}_b}|_Y) > 1$.  In fact, since $f(C_2)$ meets both $f(C_1)$ and $T$, it follows from \cite[Proposition~3.17]{lehmann_exceptional_set} that $(Y, -K_{\mathcal{X}_b}|_Y)$ is birationally isomorphic as a polarized variety to $(\mathbb{P}^2, \mathcal{O}(1))$.  As the anticanonical linear series on $\mathcal{X}_b$ is basepoint free, this descends to an isomorphism between $(Y, -K_{\mathcal{X}_b}|_Y)$ and $(\mathbb{P}^2, \mathcal{O}(1))$.  Generality of $p$ and $T$ imply $f$ is an immersion with image contained in the smooth locus of $\mathcal{X}_b$.  A direct calculation shows $h^1(C, \mathcal{N}_f) = 0$, which demonstrates smoothness of $N$ at $f$.
\end{proof}

\section{Monodromy and Mori Structures}\label{Section 5}
This section details consequences the Mori structure of a Fano threefold has on families of rational curves it contains.  Chief among these are statements about monodromy.  We refer to \cite{mori1983classification} for a detailed summary of elementary contractions on smooth threefolds.

\subsection{Blow-Ups}
Here we collect results about the description of smooth Fano threefolds as various blow-ups and implications this has on spaces of free curves.  %
Most smooth Fano threefolds with $\rho(X)\geq 2$ can be realized as a blow-up of either $(\mathbb{P}^1)^{\times 3}$, $\mathbb{P}^1 \times \mathbb{P}^2$, $\mathbb{P}^3$, or a smooth quadric $Q \subset \mathbb{P}^4$. Such descriptions may be obtained from data in \cite[Tables~2-5]{mori1981classification} and \cite{erratumMori} using the following lemma.

\begin{lem}[Lemma 2.1, \cite{mori1985classification}]\label{blow-up numbers}
Let $f:X \rightarrow Y$ be the blow-up of a smooth projective threefold $Y$ along a smooth irreducible curve $C$.  Then
\begin{enumerate}
    \item $-K_X \equiv -f^*K_Y - E$ for the exceptional divisor $E$ of $f$
    \item $(-K_X)^3 = (-K_Y)^3 - 2[(-K_Y . C) - p_a(C) +1]$
    \item $b_3(X) = b_3(Y) + 2p_a(C)$.
\end{enumerate}
\end{lem}
\noindent Blow-ups provide a natural basis for $N_1(X/Y)$.  We remark that when $f:X\rightarrow Y$ is a blow-up map between smooth Fano threefolds, the proof of Theorem \ref{Main Result} for $Y$ follows from the proof for $X$.

\begin{lem}\label{blowup}
Let $\pi: X\rightarrow Y$ be a divisorial contraction between smooth Fano threefolds with (possibly reducible) exceptional divisor $E$.  For any $\alpha \in \Nef_1(Y)$, $\overline{\free}(Y,\alpha)$ is birationally equivalent to $\overline{\free}(X,\pi^*\alpha)$.  In particular, if Theorem \ref{Main Result} holds for $X$, then Theorem \ref{Main Result} holds for $Y$ as well. %
\end{lem}
\begin{proof}
By \cite[Proposition~3.7]{kollar2013rational}, a general free curve avoids finitely many codimension 2 loci.  Therefore, irreducible components of free curves of class $\alpha \in N_1(Y)$ are in natural bijection with irreducible components of free curves on $X$ of class $\pi^* \alpha$.  This bijection induces the desired birational equivalence.  %
Moreover, as the interior of $\pi^*\NE(Y)$ is disjoint from the boundary of $\NE(X)$, if Theorem \ref{Main Result} holds for $X$, then it also holds for $Y$.
\end{proof}

Sometimes a Fano threefold $X$ may be realized as the blow-up of other Fano threefolds in multiple ways. As our arguments often refer to both descriptions of $X$ as a blow-up, we make the following definition.

\begin{defn}[Pseudosymmetry]\label{pseudosymmetry}
Let $X$ be a Fano threefold. Suppose $f:X\rightarrow Y$ and $f:X \rightarrow Y'$ realize $X$ as the blow-ups along smooth irreducible curves $C,C'$ in Fano threefolds $Y,Y'$ such that $Y$ deforms to $Y'$, and the image of $C$ in $Y'$ deforms to $C'$.  Let $e,e'$ (resp. $E,E'$) be the class of a fiber (resp. exceptional divisor) of $f$ and $f'$.  By identifying $N_1(Y) \cong N_1(Y')$ and $\mathbb{R}e \cong \mathbb{R} e'$ via deformation, we may compose with maps obtained from the blow-up descriptions
\begin{align*}
    N_1(Y) \oplus \mathbb{R} e \cong & N_1(X) \cong N_1(Y') \oplus \mathbb{R} e'\\
    N^1(Y) \oplus \mathbb{R} E \cong & N^1(X) \cong N^1(Y') \oplus \mathbb{R} E'
\end{align*}
to obtain automorphisms of $N_1(Y)_\mathbb{Z} \oplus \mathbb{Z} e $ and $N^1(Y)_\mathbb{Z} \oplus \mathbb{Z} E$.  We call such automorphisms \textit{pseudoactions} of the blow-up $X\rightarrow Y$ and the resulting symmetry on $N_1(X)$ a \textit{pseudosymmetry} on $X$.
\end{defn}

In this situation, the genus and anticanonical degree of $C$ and $C'$ are the same by \cite[Lemma~2.1]{mori1985classification}. It is clear that pseudosymmetry is part of the finite monodromy of the local system mentioned in \cite[Theorem~4.2]{Fernex2009RigidityPO}, and thus preserves $\Nef_1(X)$ and $K_X$.  We use \cite{matsuki1995weyl}, \cite{matsuki2023addendumerratum}, and \cite{mori1981classification} to find instances of pseudosymmetry on $X$.

As many of the Fano threefolds $X$ we consider are blow ups of other Fano threefolds $Y$, the following version of Mori's Bend-and-Break lemma \cite[Theorem~4]{Mori1979} is useful for studying curves on these varieties.  Similarly, Theorem \ref{Monodromy} on monodromy and general intersections is used to track intersections of contracted exceptional divisors with divisors containing free curves.

\begin{thm}[Mori's Bend-and-Break]\label{thm: Mori's BB}
Let $X$ be a smooth variety and $\pi : X\rightarrow Y$ be a birational contraction.  Let $p,q \in Y$ be distinct points.  Suppose $f:\mathbb{P}^1\rightarrow X$ is a free rational curve that meets $\pi^{-1}(p)$ and $\pi^{-1}(q)$.  If 
$$-K_X . f_*[\mathbb{P}^1] > \text{dim } X - \text{dim } \pi^{-1}(p) - \text{dim } \pi^{-1}(q) + 1,$$ 
then $f$ deforms to a stable map $g:C\rightarrow X$ such that there are distinct components $C_p, C_q \subset C$ whose images under $\pi \circ g$ are curves containing $p$ and $q$, respectively.
\end{thm}

Below a \textit{simple} cover $C\rightarrow C'$ of curves is a map whose ramification points have valency two and distinct images.  A simple cover is called \textit{elementary} if it does not factor through a nontrivial unramified covering.

\begin{thm}[Monodromy and General Intersections]\label{Monodromy}
Let $X$ be a smooth variety equipped with a linear system $V$ and $C \subset X$ be a smooth curve disjoint from the base locus of $V$. Denote by $U \subset V$ an open subset of divisors transverse to $C$.  Consider the monodromy group $G$ of the incidence correspondence $I \subset C \times U \xrightarrow{\pi} U$, i.e. the image of the group action $\pi_1(U, D) \rightarrow \mathrm{Aut}(D \cap C)$ for a divisor $D \in U$.

\begin{enumerate}[a.]
    \item If the map $C \rightarrow \mathbb{P}V$ is a simple, elementary cover of its normalized image $C'$, then $G$ acts as the wreath product of the full symmetric group on a fiber of $C\rightarrow C'$ with the full symmetric group on a hyperplane section of $C'$.

    \item As long as the map $C \rightarrow \mathbb{P}V$ is not constant, $G$ acts transitively on $\pi$-fibers. 
    
    \item If the map $C \rightarrow \mathbb{P}V$ is generically injective, then for general $D \in U$ any collection of $\dim \mathbb{P}V$ points of $D \cap C$ spans a hyperplane in $\mathbb{P}V$.

\end{enumerate}
\end{thm}

\begin{proof}
This is a collection of well-known facts and restatement of results in \cite{arbarello1985geometry} on pages 109 and 111. We attribute part of \ref{Monodromy}a to the proof of Proposition 1.9 in \cite{FultonMonodromy}.  Though these references deal with embedded curves $C \subset \mathbb{P}V$, the extension to \ref{Monodromy}a follows from observing that small loops around ramification points preserve fibers of $C \rightarrow C'$.
\end{proof}

When each free curve in a family is contained in some del Pezzo surface, the theorem above allows us to apply Theorem \ref{del pezzo curves thm} and conclude many spaces of free curves are irreducible.

\subsection{Conic Fibrations}  We descibe monodromy over curves in the base of a conic fibration on a smooth Fano threefold, showing it is maximal for a general Fano threefold and general pencil of curves.  The following lemma motivates these results.

\begin{lem}\label{3.1 lemma}
Suppose $\pi:X \rightarrow S$ is a conic bundle on a smooth threefold $X$.  For any component $M \subset \overline{\free}(X)$, the induced map $\pi_* : M \rightarrow \overline{\free}(S)$ is either generically finite or dominates a component of $\overline{\free}(S)$.
\end{lem}
\begin{proof}
A general free curve $f : \mathbb{P}^1 \rightarrow X$ parameterized by $M$ is contained in the smooth locus of $\pi$.  Hence, $f^* \mathcal{T}_{X/S}$ is a line bundle, and either $h^0(f^*\mathcal{T}_{X/S}) = 0$ or $h^1(f^*\mathcal{T}_{X/S}) = 0$.
\end{proof}

Let $\pi : X \to S$ be a conic bundle on a smooth threefold $X$.  We call $\pi$ \textit{elementary} if the relative Picard rank $\rho(X/S)$ is 1.  The \textit{discriminant locus} of $\pi$ is the locus $\Delta \subset S$ of singular fibers, see \cite[Section~4]{mori1985classification}.  To study the relationship between families of free curves on $X$ and $S$, we introduce the following notation.

\begin{notation}\label{notation: conic monodromy}
Let $W$ be a pencil of divisors on $S$.  Consider the normalization $\Theta_W$ of the universal family $\{ (x,D) : \pi(x) \in D\} \subset X \times \mathbb{P}W$ and let $\theta : \Theta_W \to \mathbb{P}W$ be the projection map. Suppose $U \subset \mathbb{P}W$ is an open set of irreducible divisors transverse to $\Delta$.  Let $D_0 \in U$ be a general divisor and $\tilde{D}_0 \to D_0$ be its normalization.  

The fiber of $\theta$ over $D_0$ is the normalization $X_{D_0}$ of the surface $X \times_S D_0$.  Observe that $\pi$ induces a fibration $\pi_{D_0} : X_{D_0} \to \tilde{D}_0$ which is a $\mathbb{P}^1$-bundle over $\tilde{D}_0 \setminus \{ D_0 \cap \Delta \}$ and has reducible fibers with two components over $\{D_0 \cap \Delta \}$.  

There is an induced monodromy action of the fundamental group $\pi_1(U, D_0)$ on $N_1(X_{D_0})$. We let $G$ be the image of the map $\pi_1(U, D_0) \rightarrow \mathrm{Aut}( N_1(X_{D_0}) )$.  We will describe the action of $G$ using the set $\mathscr{\bar{S}} = D_0 \cap \Delta$ and the set $\mathscr{S}$ of components of $\pi$-fibers over $\mathscr{\bar{S}}$. 
\end{notation}

Note that $N_1(X_{D_0})$ is generated by the class of an irreducible $\pi$-fiber, the canonical class $K_{X_{D_0}}$, and the classes of a chosen component from each $\pi$-fiber over $\mathscr{\bar{S}}$.  Since $G$ stabilizes the canonical class and the class of a $\pi$-fiber, the action of $G$ on $N_1(X_{D_0})$ is completely described by its action on $\mathscr{S}$.

To further describe this action, let $\varphi : G \to \mathfrak{S}_{\mathscr{\bar{S}}}$ be the morphism induced by the action of $G$ on $\mathscr{\bar{S}}$.  The kernel of $\varphi$ lies naturally in a copy of $(\mathbb{Z}/2\mathbb{Z})^{\oplus \mathscr{\bar{S}} }$ wherein each direct summand swaps irreducible components of the corresponding reducible fiber. Endowing $(\mathbb{Z}/2\mathbb{Z})^{\oplus \mathscr{\bar{S}}}$ with the permutation action by $\mathfrak{S}_{\mathscr{\bar{S}} }$ realizes $G$ as a subgroup of $(\mathbb{Z}/2\mathbb{Z})^{\oplus \mathscr{\bar{S}}} \rtimes \mathfrak{S}_{\mathscr{\bar{S}} }$. The following places a basic restriction on $G$. 

\begin{lem}\label{lem: standard rep}
    With Notation \ref{notation: conic monodromy}, let $V \triangleleft (\mathbb{Z}/2\mathbb{Z})^{\oplus \mathscr{\bar{S}}}$ be the standard representation, i.e.\ the set of vectors whose coordinates sum to zero.  Then $G \leq V \rtimes \mathfrak{S}_{\mathscr{\bar{S}} }$.
\end{lem}

\begin{proof}
Let $s : \tilde{D}_0\rightarrow X_{D_0}$ be a section; the existence of one such section is guaranteed by \cite{GHSrational}.  By contracting components of reducible fibers of $\pi_{D_0} : X_{D_0} \to \tilde{D}_0$ disjoint from $s(\tilde{D}_0)$, we obtain a description of $X_{D_0}$ as the blow-up of some $\mathbb{P}^1$-bundle $\mathbb{P}_{\tilde{D}_0}(E)$.  To fix notation, let $h\in N_1(X_{D_0})$ be the pullback of tautological class of $\mathbb{P}_{\tilde{D}_0}(E) \rightarrow \tilde{D}_0$, $f$ be the class of a smooth fiber, and $e_1,\ldots, e_d$ be the classes of curves contracted by $X_{D_0} \rightarrow \mathbb{P}_{\tilde{D}_0}(E)$.  These form a basis of $N_1(X_{D_0})$.

Suppose an element $\phi \in \ker(\varphi : G \to \mathfrak{S}_{\mathscr{\bar{S}}})$ does not lie in $V$.  Note that components of reducible fibers have class $e_i$ and $f-e_i$.  As mentioned, $G$ must fix $f$ and $-K_{X_{D_0}} = 2h - (h.h + 2g(\tilde{D}_0) - 2)f - \sum e_i$.  Therefore,
\begin{align*}
\phi(2h) &= \phi \big(-K_{X_{D_0}} + (h.h + 2g(\tilde{D}_0) - 2)f + \sum e_i\big)\\
&= -K_{X_{D_0}} + (h.h + 2g(\tilde{D}_0) - 2)f + \phi \big(\sum e_i\big)\\
&= 2h - \sum e_i - \phi(e_i).
\end{align*}
By construction $\phi$ either fixes $e_i$ or sends it to $f - e_i$.  Since $\phi$ does not lie in $V$, $\phi$ must send an odd number of classes $e_i$ to $f - e_i$.  We may suppose $\phi(e_i) = f - e_i$ for $i \leq 2n + 1$ and $\phi(e_i) = e_i$ for $i > 2n + 1$.  By the equation above we obtain $2\phi(h) = 2h - (2n+1)f + \sum_{i = 1}^{2n+1} 2e_i$.  Since $f$ is a primitive element of the lattice $N_1(X)_{\mathbb{Z}}$, this contradicts integrality of the action of $G$.
\end{proof}

Our main result in this subsection identifies conditions on $\Delta$ and $W$ under which $G$ is as large as possible.

\begin{thm}\label{conic monodromy}
Let $X$ be a smooth threefold and $\pi : X\rightarrow S$ be an elementary conic bundle with nonempty, irreducible discriminant locus $\Delta$.  Suppose
\begin{itemize}
    \item $W$ is a pencil on $S$ whose base locus is finite and disjoint from $\Delta$,
    \item the map $\tilde{\Delta} \rightarrow \mathbb{P}W$ from the normalization of $\Delta$ is a simple degree $d$ cover,%
    \item either $d = 2$, $d$ is odd, or each fiber of $\Delta \rightarrow \mathbb{P}W$ has $> \frac{d}{2}$ distinct points.
\end{itemize}
With Notation \ref{notation: conic monodromy}, the action of $G$ on $\mathscr{S}$ gives an isomorphism $G \cong V \rtimes \mathfrak{S}_{\mathscr{\bar{S}} }$.  Moreover, for any section $s : \tilde{D}_0 \to X_{D_0}$, there is a section $\mathfrak{S}_{\mathscr{\bar{S}}} \to G$ whose image fixes the class $s_*[\tilde{D}_0] \in N_1(X_{D_0})$. 
\end{thm}

\begin{proof}
Let $K$ be the kernel of the morphism $\varphi : G \to \mathfrak{S}_{\mathscr{\bar{S}}}$ induced by the action of $G$ on $\mathscr{\bar{S}}$.  Theorem \ref{Monodromy} implies $\varphi$ is surjective while Lemma \ref{lem: standard rep} shows $K \triangleleft V$.  Our goal is to show $K = V$, as the claim about sections follows from this fact. 

Since the preimage of $\Delta$ in $X$ is irreducible \cite[Proposition~4.8]{mori1985classification}, $G$ acts transitively on $\mathscr{S}$.  This implies $d \neq 1$ and $|K| \neq 1$.  Observe that $K$ must be stabilized by the natural action of $\mathfrak{S}_{\mathscr{\bar{S}}}$ on $V$.  This proves our claim if $d = |\mathscr{\bar{S}}|$ is odd, for in this case the orbit under $\mathfrak{S}_{\mathscr{\bar{S}}}$ of any nontrivial element in $V$ is a spanning set.  Hence, we may suppose $d = 2n$ is even.  In this case there is a unique, nontrivial central element $\sigma \in V$, which corresponds to swapping components of all $2n$ reducible fibers of $X_{D_0} \to \tilde{D}_0$. As the $\mathfrak{S}_{\mathscr{\bar{S}}}$-orbit of any $\phi \in V\setminus\{0, \sigma\}$ generates $V$, either $K = \{0, \sigma\}$ or $K = V$.  If $d = 2$, these two possibilities coincide, so we may suppose $d \geq 4$ and prove $|K| \neq 2$.

To prove $|K| \neq 2$, we use the following observation from \cite{sectionsDelPezzo}: $G$ cannot be a split extension of $\mathfrak{S}_{\mathscr{\bar{S}}}$ by $\mathbb{Z}/2\mathbb{Z}$.  To see why, let $B$ be the normalization of the Hilbert scheme of lines in $X$ contracted by $\pi$ and $B \rightarrow \tilde{\Delta}$ be the natural map.  $G$ may be identified with the monodromy group of the branched cover $B \rightarrow \mathbb{P}W$.  Moreover, for every $\mathbb{C}$-point of a fiber $\Delta \rightarrow \mathbb{P}W$, there are $2$ $\mathbb{C}$-points of the corresponding fiber $B\rightarrow \mathbb{P}W$ \cite[Proposition~4.7]{mori1985classification}.  Thus each fiber of $B \rightarrow \mathbb{P}W$ has at least $d + 2$ distinct points.  If $G$ were a split extension of $\mathfrak{S}_{\mathscr{\bar{S}}}$ by $K$ and $|K| = 2$, $B\rightarrow \mathbb{P}W$ would factor through a degree 2 cover $B' \rightarrow \mathbb{P}W$.  Over a branch point of $B' \rightarrow \mathbb{P}W$ there are at most $d$ distinct points in $B$, contradicting our hypotheses.

To derive a contraction suppose $d = 2n$, $n\geq 2$, and $K = \{ 0, \sigma\} \triangleleft V$.  Consider the quotient $V' = V/K$.  The quotient $G/K$ is a section of $\mathfrak{S}_{\mathscr{\bar{S}}}$ inside the semidirect product $V' \rtimes \mathfrak{S}_{\mathscr{\bar{S}}}$.  We will show
$$1 \rightarrow K \rightarrow G \rightarrow \mathfrak{S}_{\mathscr{\bar{S}}} \rightarrow 1$$
must be a split extension by computing surjectivity of $H^1(\mathfrak{S}_{\mathscr{\bar{S}}}, V) \rightarrow H^1(\mathfrak{S}_{\mathscr{\bar{S}}}, V')$.  It follows the section $\mathfrak{S}_{\mathscr{\bar{S}}} \cong G/K$ of $V' \rtimes \mathfrak{S}_{\mathscr{\bar{S}}} \rightarrow \mathfrak{S}_{\mathscr{\bar{S}}}$ lifts to a section of $V \rtimes \mathfrak{S}_{\mathscr{\bar{S}}} \rightarrow \mathfrak{S}_{\mathscr{\bar{S}}}$ contained in $G$.  To compute the indicated surjectivity, we use the standard description of the first group cohomology as crossed homomorphisms modulo principal crossed homomorphisms:
$$H^1(\mathfrak{S}_{\mathscr{\bar{S}}}, V') =\frac{\{f : \mathfrak{S}_{\mathscr{\bar{S}}} \rightarrow V' \ |  f(ab) = f(a) + af(b) \text{ for all } a,b \in \mathfrak{S}_{\mathscr{\bar{S}}}\}}{\{f: \mathfrak{S}_{\mathscr{\bar{S}}} \rightarrow V' \ |  f(g) = gm - m \text{ for some } m \in V'\}}.$$
We use the presentation of $\mathfrak{S}_{\mathscr{\bar{S}}}$ whose generators are transpositions $\tau_1 = (12)$, $\tau_2 = (23)$, $\ldots$, $\tau_{d-1} = ((d-1)d)$.  The relations of this presentation, listed below, impose the following conditions on a set map $f: \{\tau_1, \ldots , \tau_{d-1}\} \rightarrow V$ that defines a crossed homomorphism to $V'$:
\begin{enumerate}
    \item $\tau_i^2 = 1$ imposes $f(\tau_i) + \tau_i f(\tau_i) \in K$;
    \item $(\tau_i \tau_j)^2 = 1$ for $|i-j| > 1$ imposes $f(\tau_i) + f(\tau_j) + \tau_i\tau_j[f(\tau_i) + f(\tau_j)] \in K$;
    \item $(\tau_i\tau_{i+1})^3 = 1$ imposes $(1 + \tau_i\tau_{i+1} + (\tau_i\tau_{i+1})^2)[f(\tau_i) + f(\tau_{i+1})] \in K$.
\end{enumerate}
Since $d = 2n \geq 4$, Conditions (1) and (2) are satisfied if and only if $f(\tau_i) + \tau_i f(\tau_i) = 0$ and $f(\tau_i) + f(\tau_j) + \tau_i\tau_j[f(\tau_i) + f(\tau_j)] = 0$.  Moreover, we may modify the set map $f: \{\tau_1, \ldots , \tau_{d-1}\} \rightarrow V$ by adding $\sigma \in K$ to any of $f(\tau_i)$ and obtain the same crossed homomorphism to $V'$.  If, for each three cycle $\tau_i\tau_{i+1} \in \mathfrak{S}_{\mathscr{\bar{S}}}$, the corresponding element $f(\tau_i\tau_{i+1}): = f(\tau_i) + \tau_i f(\tau_{i+1})$ has an even number of 1's amongst the $i^{th}, (i+1)^{st}$, and $(i+2)^{nd}$ coordinates of the embedding $V\subset (\mathbb{Z}/2\mathbb{Z})^d$ in the permutation representation, then $(1 + \tau_i\tau_{i+1} + (\tau_i\tau_{i+1})^2)[f(\tau_i) + f(\tau_{i+1})] = 0$ for all $i$ and $f$ defines a crossed homorphism to $V$ as well.  The affine map $A : (\mathbb{Z}/2\mathbb{Z})^{d-1} \rightarrow (\mathbb{Z}/2\mathbb{Z})^{d-2}$ describing this problem is surjective.  Thus, each crossed homomorphism $f : \mathfrak{S}_{\mathscr{\bar{S}}} \rightarrow V'$ lifts to a crossed homomorphism to $V$.  This proves $|K| \neq 2$ as $G$ cannot be a split extension of $\mathfrak{S}_{\mathscr{\bar{S}}}$ by $\mathbb{Z}/2\mathbb{Z}$.
\end{proof}

We will now prove an important corollary of Theorem \ref{conic monodromy}.  The following lemma recalls properties about conic bundle structures on Fano threefolds that will be needed in its proof.

\begin{lem}[\cite{mori1983classification}]\label{lem: Mori conic bundles}
Let $X$ be a smooth Fano threefold.  Suppose $\pi : X\rightarrow S$ is a conic bundle with discriminant locus $\Delta_\pi$.
\begin{enumerate}
    \item $S$ is one of $\mathbb{P}^2$, $\mathbb{P}^1 \times \mathbb{P}^1$, or $\mathbb{F}_1$.  If $S \cong \mathbb{F}_1$, the discriminant locus $\Delta_\pi$ is either disjoint from or contains the rigid section of $\mathbb{F}_1 \rightarrow \mathbb{P}^1$ as a connected component.  
    \item The preimage of any smooth, rational, connected component of $\Delta_\pi \subset S$ is the union of two $E1$ divisors $E_1, E_2$.  Let $\pi_i : X \to Y_i$ be the birational contraction that factors $\pi$ and contracts $E_i$.  Then $Y_i$ is Fano for some $i$.
    \item If $\pi$ contracts the class $\alpha$ of a nef $(-K_X)$-line, then either $\rho(X) - \rho(S) = 1$ or $\rho(X) - \rho(S) = 2$, $S \cong \mathbb{F}_1$ and $\Delta_\pi$ is disconnected.%
\end{enumerate}
\end{lem}

\begin{cor}\label{cor: conic monodromy cor}
    Let $X$ be a smooth Fano threefold.  Suppose $\pi : X \to S$ is an elementary conic bundle with nonempty, irreducible discriminant locus.  Consider $\alpha \in \Nef_1(X)_{\mathbb{Z}}$ such that $-K_{X/S} . \alpha \geq -1$.  Suppose the divisor class $\pi_* \alpha$ is very ample, or more generally that a pencil $W \subset |\pi_* \alpha|$ satisfies the hypotheses of Theorem \ref{conic monodromy}. If general members of $|\pi_* \alpha|$ are smooth rational curves, then $\overline{\free}(X,\alpha)$ is nonempty and irreducible.
\end{cor}
\begin{proof}
    Let $W \subset |\pi_* \alpha|$ be a general pencil.  Note that $W$ satisfies the hypotheses of Theorem \ref{conic monodromy}.  Moreover, $\free(S, \pi_* \alpha)$ is irreducible because it is birational to $|\pi_* \alpha|$.  By Lemma \ref{3.1 lemma} the bound on $-K_{X/S} . \alpha$ implies the natural map $\pi_* : \free(X,\alpha) \to \free(S, \pi_* \alpha)$ is dominant on each component of $\free(X,\alpha)$.  Therefore it suffices to prove that the preimage of $W$ under $\pi_* : \free(X,\alpha) \to \free(S, \pi_* \alpha)$ is connected and nonempty.  We will use Theorem \ref{conic monodromy} to establish this claim.

    Consider a general divisor $D \in W$ and the surface $X_D = X \times_S D$.  Let $i : X_D \to X$ be the inclusion map and $\pi_D : X_D \to D$ be the base change of $\pi$.  Since the pushforward map $i_* : N_1(X_D) \to N_1(X)$ has rank two, the preimage $i_*^{-1}(\alpha)$ consists of all classes of $\pi_D$-sections with $(-i^*K_X)$-degree equal to $-K_X . \alpha$.

     Let $\beta, \gamma \in i_*^{-1}(\alpha)$ be any two integral nef classes.  For any reducible $\pi_D$-fiber $F = C_1 \cup C_2$, $\beta$ and $\gamma$ must pair to 0 with one of $C_1, C_2$ and to 1 with the other.  Let $n$ be the number of reducible $\pi_D$-fibers whose components have different intersection pairings against $\beta$ and $\gamma$.  We claim that $n$ must be even.  Indeed, let $\phi : X_D \to Y = \mathbb{P}_D(E)$ be a minimal model of $X_D$ over $D$.  Let $h$ be the tautological class of the projective bundle $Y \to D$ and $f$ be the class of a fiber.  By adjunction, $-K_{X_D} . (\beta - \gamma) = 0$.  Moreover, $\phi$ is a composition of elementary contractions that change the parity of $-K_Y . \phi_* (\beta - \gamma)$ if the intersections of $\beta$ and $\gamma$ against the exceptional locus are different.  Therefore the parity of $-K_Y . \phi_* (\beta - \gamma)$ is the parity of $n$.  As $f . \phi_* (\beta - \gamma) = 0$ and $-K_Y = 2h - (h.h - 2)f$, $n$ must be even.  Hence, $\beta$ and $\gamma$ lie in the same orbit of the monodromy action described by Theorem \ref{conic monodromy}.  It follows that $\free(X,\alpha)$ is smooth and connected if it is nonempty.

     Suppose there exists a $\pi_D$-section $f : D \to X_D$ such that $(i\circ f)_*[D] = \alpha$.  To prove deformations of $i \circ f : D \to X$ dominate $X$, suppose otherwise.  Let $E \subset X$ be the divisor swept out by such deformations.  Because deformations of $i \circ f$ dominate a family of free curves on $S$, the restriction $\pi_E : E \to S$ is an $a$-cover.  However, by the description of $S$ in Lemma \ref{lem: Mori conic bundles}, $E$ would have to be a section of $\pi : X \to S$.  In particular, since $E$ is smooth, the sequence 
    $$0 \to N_{i\circ f(D) / E} \to N_{i\circ f} \to (i \circ f)^*N_E \to 0$$
    is exact.  Hence, $(i \circ f)^*N_E \cong (i \circ f)^*\mathcal{O}(E)$ is not globally generated, and so $\alpha = (i \circ f)_*[D]$ is not nef, a contradiction.

    Lastly, we show there exists a $\pi_D$-section $f : D \to X_D$ such that $(i\circ f)_*[D] = \alpha$.  By \cite{GHSrational}, there exists a $\pi_D$-section $g : D \to X_D$ of some $(-i^*K_X)$-degree.  But every section of $\pi_D$ degenerates to the union of a rigid section $h : D \to X_D$ and components of $\pi_D$-fibers.  We can find a $\pi_D$-minimal model $\phi : X_D \to Y$ which is an isomorphism in a neighborhood of this rigid section.  Since $D$ is a smooth rational curve, $Y$ is a Hirzebruch surface $\mathbb{F}_n$ for some $n > 0$.  From this we see that the self intersection of $h_*[D]$ is the same as $-K_{X/S} . h_*[D] = -n$.  

    If $n = 1$, then we may glue components of reducible fibers to $h(D)$ and smooth the resulting curve to obtain the desired section $f$.  Suppose $n > 1$ instead.  As $D \subset S$ deforms with the expected dimension $-K_S . D - 1$, the stable map $i \circ h : D \to X$ also deforms with dimension at least 
    $$-K_S . D -1 = -K_X . (i\circ h)_*[D] + n - 1.$$  
    Since $n > 1$, this is larger than the expected dimension $-K_X . (i\circ h)_*[D]$ and so deformations of $i \circ h$ cannot form a dominant family of maps.  We have already seen in this case that the divisor $E$ covered by such maps must be a section of $\pi : X \to S$.  As we will show below, there are at least $n$ reducible fibers of $\pi_D : X_D \to D$.  Thus, as $i^*E = h_*[D]$ and $E . \alpha \geq 0$, we may obtain the desired section $f$ by gluing and smoothing curves as in the case $n = 1$.

    We finish the proof by showing $\Delta . D \geq n$, which implies $\pi_D$ has at least $n$ reducible fibers.  Since $\pi : X \to S$ is an elementary conic bundle, Lemma \ref{lem: Mori conic bundles}(2) implies the discriminant locus $\Delta$ is not smooth and rational.  Therefore by adjunction $(K_S + \Delta) \cdot \Delta \geq 0$.  By the restrictions Lemma \ref{lem: Mori conic bundles}(1) places on $(S,\Delta)$, it follows that $K_S + \Delta$ is in fact effective.  Since $D$ is a nef divisor class on $S$ we obtain $\Delta . D \geq -K_S . D$.  But as $X$ is Fano, $-K_S . D = -K_X . (i \circ h)_*[D] + n$ is larger than $n$.  We conclude the desired section $f$ always exists, so $\free(X,\alpha)$ is nonempty.
\end{proof}

To apply Corollary \ref{cor: conic monodromy cor}, we will show the discriminant curve $\Delta$ of an elementary conic bundle structure on a general Fano threefold is smooth and irreducible.  This will also show $\Delta$ is irrational, which verifies part of Theorem \ref{main result 3}.

\begin{lem}\label{lines in C1 fibrations}
Let $X$ be a smooth Fano threefold, general in moduli.  Let $\pi : X\rightarrow S$ be a conic bundle and suppose that $\pi_* \alpha = 0$ for some $(-K_X)$-line class $\alpha \in \Nef_1(X)_{\mathbb{Z}}$.  Then $\mathcal{H} :=\overline{M}_{0,0}(X,\alpha)$ is a smooth, irreducible, irrational curve.
\end{lem}
\begin{proof}
By Lemma \ref{lem: Mori conic bundles}, we may suppose $\pi$ is an elementary conic bundle.  Moreover, elementary conic bundles over $\mathbb{F}^1$ with smooth and Fano total space are the base change of a conic bundle over $\mathbb{P}^2$ \cite{mori1985classification}.  Thus we reduce to considering elementary conic bundles $\pi : X\rightarrow \mathbb{P}^2$ and $\pi : X\rightarrow \mathbb{P}^1 \times \mathbb{P}^1$.  

By the proof of \cite[Proposition~4.7]{mori1985classification}, the normalization of $\mathcal{H}$ is a connected double cover of the normalization $\tilde{\Delta}_\pi$ of the discriminant locus $\Delta_\pi \subset S$, ramified at the preimage of every singular point in $\Delta_\pi$.  %
By Lemma \ref{lem: Mori conic bundles}(2), $\Delta_\pi$ cannot be both smooth and rational.  Hence, it suffices to show $\Delta_\pi$ is smooth.

Suppose first, as a warm-up, that $X$ can be realized as a double cover $X \rightarrow \mathbb{P}^1 \times S$, branched along a general divisor $V(s^2 \otimes f + st \otimes g + t^2 \otimes h)$ of bidegree $(2,2n)$.  In this case, the singularities of $\Delta_\pi$ are $V(f,g,h) \subset S$.  Hence, $\Delta_\pi$ is smooth as $X$ is general in moduli.  A similar argument applies to any double cover of a Fano $\mathbb{P}^1$-bundle over $S$, branched over a general divisor.  

Now suppose, more generally, that $X$ may be realized through the following construction.  Let $E$ be a vector bundle on $S$ of rank 3, and $Y = \mathbb{P}_S(E)$ be the projective bundle of one-dimensional quotients, i.e. $\text{Proj}_S(\text{Sym}^*E)$.  Suppose $X \subset Y$ %
is a general member of a linear series $|\mathcal{O}_{Y/S}(2) \otimes_{\mathcal{O}_Y} \psi^*\mathcal{L}|$, where $\psi : Y \rightarrow S$ is the tautological projection and $\mathcal{L}$ is a line bundle such that $V := \text{Sym}^2(E) \otimes_{\mathcal{O}_S} \mathcal{L}$ is globally generated.  Let $\mathbb{P}H^0(S,V)$ be the projective space of lines in $H^0(S,V)$.

In this case, $\Delta_\pi$ is the fiber of a divisor $D\subset S \times \mathbb{P}H^0(S,V)$ over a general point $v \in \mathbb{P}H^0(S,V)$.  The divisor $D$ is defined by the vanishing of the determinant of the associated map $\sigma_v \in \text{Hom}(E^\vee, E\otimes \mathcal{L})$.  The locus of points $v \in \mathbb{P}H^0(S,V)$ such that the fiber $D_v$ contains a fixed point $p \in S$ is a cubic hypersurface in $\mathbb{P}H^0(S,V)$.  Let $e_1,e_2, e_3$ be local coordinates of $E$ at $p$, $l$ be a local coordinate of $\mathcal{L}$ at $p$, and $v_{ij} = l \otimes e_i \otimes e_j$ for $i \leq j$ be local coordinates of $V$ at $p$.  Global generation of $V$ ensures the map $H^0(S,V) \rightarrow \oplus_{ij} \mathbb{C}v_{ij}$ is surjective.  The affine cone over $D_p \subset \mathbb{P}H^0(S,V)$ is the preimage of the locus in $\oplus_{ij} \mathbb{C}v_{ij}$ defined by 
$$v_{11} v_{22} v_{33} - \frac{1}{4} v_{11} v_{23}^2 - \frac{1}{4} v_{12}^2 v_{33} + \frac{1}{4} v_{12} v_{23} v_{13} - \frac{1}{4} v_{13}^2 v_{22} = 0.$$
In particular, as the cubic hypersurface this defines in $\mathbb{P}^5$ is nonsingular in codimension 1, so is $D_p$.  Hence $D$ is irreducible and nonsingular in codimension 1, so $D_v$ is smooth for general $v \in \mathbb{P}H^0(S,V)$.  Since $\rho(X/S) = 1$, the discriminant locus of $\pi$ is connected.  Hence $D_v$ is irreducible as well.

Thus, all we need to show is that every $C1$ contraction of a smooth Fano threefold can be realized via one of the above constructions.  Since $\rho(X) \in \{2,3\}$, \cite{matsuki1995weyl} may be used to identify each $C1$ contraction to $\mathbb{P}^2$ or $\mathbb{P}^1 \times \mathbb{P}^1$.  The only such conic bundles not immediately realized by the above constructions are on Fano threefolds of deformation type $2.9, 11, 13, 16, 20$ and $3.6, 10$.  We will realize each of these as examples of the latter construction.

By \cite[Proposition~7.5]{mori1985classification}, threefolds of type $2.9$ (resp. $2.13$) are complete intersections in $\mathbb{P}^3 \times \mathbb{P}^2$ (resp. $Q \times \mathbb{P}^2$) of divisors of bidegree $(1,1)$ and $(2,1)$ (resp. $(1,1)$ and $(1,1)$), where $(a,b)$ denotes a member of $|aH_1 + bH_2|$ and $H_i$ is the ample generator on the $i^{th}$ factor.  The projection of each complete intersection to the second factor $\mathbb{P}^2$ realizes the conic bundle structure.  By fixing the divisor of bidegree $(1,1)$, we may realize threefolds of deformation type $2.9$ as divisors inside a bundle $Y = \text{Proj}_S(\text{Sym}^*E)$ via an exact sequence
$$0 \rightarrow \mathcal{O}_{\mathbb{P}^2}(-1) \rightarrow  \bigoplus_4 \mathcal{O}_{\mathbb{P}^2} \rightarrow E \rightarrow 0.$$
Here, $E \cong \mathcal{T}_{\mathbb{P}^2}(-1) \oplus \mathcal{O}_{\mathbb{P}^2}$ and $X$ is a general member of $|\mathcal{O}_{Y/\mathbb{P}^2}(2) \otimes_{\mathcal{O}_Y} \psi^*\mathcal{O}_{\mathbb{P}^2}(1)|$.  Hence, the above analysis shows the discriminant locus of $X \rightarrow \mathbb{P}^2$ is smooth for general $X$.  By fixing two divisors of bidegree $(1,1)$ in $\mathbb{P}^4 \times \mathbb{P}^2$ and varying $Q \subset \mathbb{P}^4$, we may similarly realize threefolds of deformation type $2.13$ as divisors inside a bundle $Y = \text{Proj}_S(\text{Sym}^*E)$, where now $E$ fits into a sequence
$$0 \rightarrow \mathcal{O}_{\mathbb{P}^2}(-1) \rightarrow \mathcal{T}_{\mathbb{P}^2}(-1) \oplus \mathcal{O}_{\mathbb{P}^2}^{\oplus 2} \rightarrow E \rightarrow 0.$$
In this case, $X$ is a general member of $|\mathcal{O}_{Y/\mathbb{P}^2}(2)|$, proving our claim.

Consider instead threefolds of deformation type $2.11, 2.16$ and $2.20$.  In order, these are blow-ups of the following smooth del Pezzo threefolds $V_i$ along a curve $C \subset V_i$ cut out by a net $L \subset |-\frac{1}{2}K_{V_i}|$ of fundamental divisors:
\begin{itemize}
    \item $V_3 \subset \mathbb{P}^4$ is a cubic hypersurface, and $C \subset V_3$ is a line;
    \item $V_4 \subset \mathbb{P}^5$ is a complete intersection of two quadrics, and $C \subset V_4$ is a conic;
    \item $V_5 \subset \mathbb{P}^6$ is a codimension 3 linear section of the Grassmanian $\text{Gr}(2,5) \subset \mathbb{P}^9$ in its Pl\"ucker embedding, and $C\subset V_5$ is a twisted cubic.
\end{itemize}

When $i = 3$, let $Y = \text{Bl}_C \mathbb{P}^{4}  \cong \mathbb{P}_{\mathbb{P}^2}(\mathcal{O}(1) \oplus \mathcal{O}^{\oplus 2})$.  We may realize $X \subset Y$ as a general member of $|\mathcal{O}_{Y/\mathbb{P}^2}(2) \otimes_{\mathcal{O}_Y} \psi^*\mathcal{O}_{\mathbb{P}^2}(1)|$.  When $i = 4$, let $\mathbb{P} = \mathbb{P}_{\mathbb{P}^2}(\mathcal{O}(1) \oplus \mathcal{O}^{\oplus 3}) \rightarrow \mathbb{P}^5$ be the blow-up of the linear span $W$ of the conic $C$.  There exists a unique quadric $Q$ that contains both $V_4$ and $W$.  Since $Q$ must be smooth, its strict transform $Y \subset \mathbb{P}$ will be isomorphic to a projective subbundle $\mathbb{P}_{\mathbb{P}^2}(\mathcal{O}(1) \oplus \mathcal{T}_{\mathbb{P}^2}(-1)) \subset \mathbb{P}$.  We may view $X$ as a divisor of class $|\mathcal{O}_{Y/\mathbb{P}^2}(2)|$.  This verifies our claim when $i = 3,4$.

Suppose $i = 5$ instead.  In this case, one may realize $X$ as a general divisor in $Y = \mathbb{P}_{\mathbb{P}^2}(E)$ of class $|\mathcal{O}_{Y/\mathbb{P}^2}(2) \otimes \psi^*\mathcal{O}_{\mathbb{P}^2}(-1)|$, where $E$ is the cokernel of a general map $\mathcal{O}_{\mathbb{P}^2}^{\oplus 2} \rightarrow \mathcal{O}_{\mathbb{P}^2}(1) \oplus \mathcal{T}_{\mathbb{P}^2}(-1)^{\oplus 2}$.  This description follows from \cite[Section~37]{Coates16} and an explicit analysis of Pl\"ucker relations for $\text{Gr}(2,5) \subset \mathbb{P}^9$.  However, this description, as well as global generation of $\text{Sym}^2(E) \otimes \mathcal{O}_{\mathbb{P}^2}(-1)$, is more cumbersome to verify than the following direct argument.  

Let $M$ be universal family over the Fano scheme of lines on $V_5$.  For general twisted cubics $C \subset V_5$, \cite[Corollary~6]{kollarLefschetz} shows the preimage of $C$ under the evaluation map $\text{ev} : M\rightarrow V_5$ is connected.  It is well-known \cite{furushima_nakayama_1989} that $M$ is smooth and $\text{ev}$ is a triple covering branched over a member $B \in |-K_{V_5}|$, such that the general branch point has two preimages.  Hence, as $\mathcal{H}$ may be identified with $\text{ev}^{-1}(C)$, in general $\mathcal{H}$ is an irreducible triple cover of $C$ with simple branching over six points.  This implies $\Delta_\pi \subset \mathbb{P}^2$ is a smooth elliptic curve.

Fano threefolds of deformation type $3.6$ are blow-ups of $\mathbb{P}^3$ along the disjoint union of a line $\ell$ and an elliptic curve $C$ of degree $4$.  In other words, $X \subset \mathbb{P}^3 \times \mathbb{P}^1 \times \mathbb{P}^1$ is a complete intersection of divisors of tridegree $(1,1,0)$ and $(2,0,1)$.  The conic bundle structure $\pi : X\rightarrow \mathbb{P}^1 \times \mathbb{P}^1$ is induced by the projection. This implies $X$ lies in a projective bundle $Y = \mathbb{P}_{\mathbb{P}^1 \times \mathbb{P}^1}(\mathcal{O}(1,0) \oplus \mathcal{O}^{\oplus 2})$ and is a general member of $|\mathcal{O}_{Y/\mathbb{P}^1 \times \mathbb{P}^1}(2) \otimes_{\mathcal{O}_Y} \psi^*\mathcal{O}_{\mathbb{P}^1 \times \mathbb{P}^1}(0,1)|$, which proves our claim.

Lastly, Fano threefolds of deformation type $3.10$ are blow-ups $X\rightarrow Q \subset \mathbb{P}^4$ along two disjoint conics $C_1, C_2$.  As before, we may realize $X$ as a complete intersection in $\mathbb{P}^4 \times \mathbb{P}^1 \times \mathbb{P}^1$ of divisors of tridegree $(2,0,0)$, $(1,1,0)$, and $(1,0,1)$.  Hence, $X \subset Y = \mathbb{P}_{\mathbb{P}^1 \times \mathbb{P}^1}(\mathcal{O}(1,0) \oplus \mathcal{O}(0,1) \oplus \mathcal{O})$ is a general member of $|\mathcal{O}_{Y/\mathbb{P}^1 \times \mathbb{P}^1}(2)|$.  Global generation of $\text{Sym}^2(\mathcal{O}_{\mathbb{P}^1 \times \mathbb{P}^1}(1,0) \oplus \mathcal{O}_{\mathbb{P}^1 \times \mathbb{P}^1}(0,1) \oplus \mathcal{O}_{\mathbb{P}^1 \times \mathbb{P}^1})$ proves our claim.
\end{proof}

\subsection{Del Pezzo Fibrations}  This section proves Theorem \ref{main result 2}.  We include results from \cite{testa2006irreducibility} and \cite{beheshti2021rational} characterizing families of free curves on smooth del Pezzo surfaces, and orbits of their classes under actions by Weyl groups.

The following theorem is an extension of results in \cite{testa2006irreducibility}.  By blowing up general points in del Pezzo surfaces of larger degree, we obtain similar statements for free rational curves on these surfaces.

\begin{thm}[Theorem 6.6, \cite{beheshti2021rational}; \cite{testa2006irreducibility}]\label{del pezzo curves thm}
Let $S$ be a smooth del Pezzo surface of degree $1$ over an algebraically closed field of characteristic 0.  Suppose $\alpha \in \Nef_1(S)_{\mathbb{Z}}$ satisfies $-K_S . \alpha \geq 2$. Then $\overline{\free}(S,\alpha)$ is nonempty, and
\begin{enumerate}
    \item If $\alpha$ is not an integral multiple of an integral $(-K_S)$-conic class, then $\overline{\free}(S,\alpha) = \overline{\free}^{bir}(S,\alpha)$ is irreducible.
    \item If $\alpha$ is the class of a $(-K_S)$-conic, then $\overline{\free}(S,\alpha) = \overline{\free}^{bir}(S,\alpha)$ is irreducible provided $\alpha \neq -2K_S$.  If $S$ is general, $\overline{\free}(S, -2K_S)$ is irreducible.
    \item If $\alpha = -2nK_S$ for some $n>1$, then $\overline{\free}^{bir}(S,\alpha)$ is irreducible and nonempty, while $\overline{\free}(S,\alpha)\setminus \overline{\free}^{bir}(S,\alpha)$ has at least one component.
    \item If $\phi: S \rightarrow S'$ is the contraction of a $(-1)$-curve and $\alpha = -n\phi^* K_{S'}$ for some $n>1$, then $\overline{\free}^{bir}(S,\alpha)$ is irreducible and nonempty, while $\overline{\free}(S,\alpha)\setminus \overline{\free}^{bir}(S,\alpha)$ has exactly one component.
    \item Otherwise, $\alpha = n[C]$ for some smooth rational $(-K_S)$-conic $C\subset S$, and $\overline{\free}(S,\alpha)$ has exactly one component, which parameterizes $n$-fold covers of deformations of $C$.
\end{enumerate}
Components of $\overline{\free}^{bir}_1(S)$ with reducible fibers over general points are components of $\overline{\free}(S, -2K_S)$ and $\overline{\free}(S, -\phi^* K_S')$ for contractions $\phi: S \rightarrow S'$ of a $(-1)$-curve.
\end{thm}

In particular, on a del Pezzo surface $S$ of degree at least $2$, for each $\alpha \in \Nef_1(S)_\mathbb{Z}$ $\overline{\free}(S,\alpha)$ is irreducible and nonempty.  Automorphisms of the Picard lattice of a del Pezzo surface $S$ which preserve $K_S$ and the intersection pairing $.$ form a Weyl group $\text{Aut}(\text{Pic}(S), K_S, \ . \ )$.  We use the following facts about $\text{Aut}(\text{Pic}(S), K_S, \ . \ )$.

\begin{lem}[Lemma 3.3.3 \cite{testa2006irreducibility}, Lemma 2.9 \cite{sectionsDelPezzo}]\label{testa theorem}
Let $S$ be a smooth del Pezzo surface of degree $d$.  Suppose $C_1, C_2 \subset S$ are nef curves such that $-K_S . C_1 = -K_S . C_2 \leq 3$.  Then the classes $[C_1], [C_2]$ lie in the same $\text{Aut}(\text{Pic}(S), K_S, \ . \ )$-orbit if and only if $p_a(C_1) = p_a(C_2)$.  Furthermore,
\begin{itemize}
    \item If $-K_S . C_1 = 1$, then $d = 1$ and $[C_1] = -K_S$;
    \item If $-K_S . C_1 = 2$, then $p_a(C_1) \leq  \max(0, 3 - d)$;
    \item If $-K_S . C_1 = 3$, then $p_a(C_1) = 0$ if $d \geq 4$; $p_a(C_1) \leq 1$ if $d \in \{2,3\}$; and $p_a(C_1) \leq 4$ if $d = 1$.  
\end{itemize}
\end{lem}

\begin{proof}
    The references provided prove our claim when $-K_S . C_1 \leq 2$ or when $d \geq 2$ and $-K_S . C_1 = 3$.  Assume $d = 1$ and pick $\alpha \in \Nef_1(S)_{\mathbb{Z}}$ such that $-K_S . \alpha = 3$.  We will describe curves of class $\alpha$ as the union of anticanonical curves with lower degree nef curves.  These descriptions imply our claim.
    
    By the Hodge Index Theorem, $d \alpha^2 \leq 9$.  Moreover, if $g$ is the arithmetic genus of curves of class $\alpha$, then $ \alpha^2 - 3 = 2g - 2$, and hence $\alpha^2$ is odd.
    
    Let $r \geq 0$ be the minimal pairing between $\alpha$ and a $(-1)$-curve; this is also the largest integer $r$ such that $\alpha + rK_S$ is nef.  Because $S$ is Fano, $\alpha + rK_S$ is semiample.  Let $\phi : S \to B$ be the associated map.  Note that $(\alpha + rK_S)^2 = \alpha^2 + r^2 - 6r \geq 0$.  As well, $\alpha^2 \geq 1$ because it is odd, and $\alpha + rK_S$ is big iff $(\alpha + rK_S)^2 > 0$.
    
    If $r = 0$, then $B$ is a del Pezzo surface, $(-K_B)^2 \geq 2$, and $\mathcal{O}_S(\alpha) = \phi^*\mathcal{O}_{B}(\alpha')$ for some $-K_{B}$-cubic $\alpha'$.  In this case, $(-K_B)^2  \alpha^2 \leq 9$ and so $\alpha^2 \leq 3$.  It follows that either $(-K_B)^2 = 3$ and $\mathcal{O}_{S}(\alpha) \cong \phi^*\mathcal{O}_{B}(-K_B)$ or $B = \mathbb{P}^2$ and $\mathcal{O}_{S}(\alpha) \cong \phi^*\mathcal{O}_{B}(1)$.  
    
    If $r = 1$ and $(\alpha + rK_S)^2 = 0$, then $\alpha^2 = 5$ and $\alpha + rK_S$ is the class of a smooth $(-K_S)$-conic.  If instead $r = 1$ and $(\alpha + rK_S)^2 > 0$, then $B$ is a del Pezzo surface of degree 2.  In this case, $\mathcal{O}_S(\alpha + rK_S) = \phi^*\mathcal{O}_{B}(-K_B)$ and so $\alpha^2 = 7$.  
    
    If $r \geq 2$, then $\alpha + 2K_S$ is the class of a nef $(-K_S)$-line.  The only such class is the anticanonical class itself.  Hence $r = 3$ and $\alpha = -3K_S$.
\end{proof}

A \textit{del Pezzo fibration} on a Fano threefold $X$ is a map $\pi : X \rightarrow \mathbb{P}^1$ with connected fibers.  Let $U \subset \mathbb{P}^1$ be the open subset where $\pi$ is smooth. The monodromy group of $\pi$ is the image $G$ of the monodromy map $\pi_1(U,p) \rightarrow \mathrm{Aut}(N_1(X_p))$.  Below, we obtain a characterization of this monodromy group for general Fano threefolds. 

\begin{thm}[Restatement of Theorem \ref{main result 2}]\label{monodromy del Pezzo main result}
Let $X$ be a smooth Fano threefold that is general in moduli.  Suppose $\pi : X \to B$ is a del Pezzo fibration.  For a smooth $\pi$-fiber $X_b$, let $W$ be the largest Weyl group that acts on the N\'eron-Severi group $N_1(X_b)$. The monodromy group of $\pi$ acts on $N_1(X_b)$ as the maximal subgroup of $W$ that stabilizes pushforward by inclusion into $X$.
\end{thm}
\begin{proof}
Let $U \subset \mathbb{P}^1$ be the open subset over which $\pi$ is smooth, $p \in U$, and $X_p = \pi^{-1}(p)$.  The homotopy group $\pi_1(U,p)$ acts on $N_1(X_p)$ via parallel transport.  Let $\nu \leq 3$ be the minimal $(-K_X)$-degree of a curve contracted by $\pi$.  If $\nu = 3$, $\pi$ is elementary ($\rho(X/\mathbb{P}^1) = 1$) and  our claim is trivial; if $\nu = 2$, it is an application of the invariant cycle theorem \cite[Theorem~4.18]{voisin_2003}.  Otherwise, $\nu = 1$ and the action of $\pi_1(U,p)$ is determined by the permutation action on smooth $(-K_X)$-lines in $X_p$.  
Suppose first that $\pi : X \rightarrow \mathbb{P}^1$ is an elementary del Pezzo fibration.  %
There are only eight deformation types of Fano threefolds with $D1$-fibrations: double covers of $\mathbb{P}^1 \times \mathbb{P}^2$ branched over divisors of bidegree $(2,4)$ and blow-ups $X \rightarrow Y$ for $Y \in \{ V_1, \ldots , V_5, \mathbb{P}^3, Q\}$ along the complete intersection of a pencil in $|\frac{r-1}{-r}K_Y|$, where $r$ is the index of $Y$.  The monodromy action of $\pi_1(U,p)$ is well-known to be the full Weyl group in these cases.  For instance, when $|\frac{-1}{r} K_Y|$ is very ample, the works of \cite{kollarLefschetz} \cite{reid1972thesis} and \cite{beauville_monodromy} imply our claim (see also \cite[Section~4.4]{DeligneMonodromy} and \cite[Remark~2.2.16]{larsenDeligne}).  %
The remaining cases follow from the application of \cite[Corollary~6]{kollarLefschetz} to Lemma \ref{lem: V1 monodromy}, \cite[Theorem~1.1]{V2monodromy}, and \cite[Proof~of~Proposition~7.7]{2019}.

Suppose instead that $\pi : X \rightarrow \mathbb{P}^1$ is not an elementary del Pezzo fibration.  We may obtain a factorization $f : X \rightarrow Y$, $g: Y \rightarrow S$,  $h  : S \rightarrow \mathbb{P}^1$, where $f$ is birational, $Y$ is smooth and Fano, and $g$ is either trivial or an elementary conic bundle.  If $g$ is trivial, $h : Y \rightarrow \mathbb{P}^1$ is an elementary del Pezzo fibration with $\nu \geq 2$.  The monodromy action of $\pi_1(U,p)$ on $N_1(X_p)$ in this case is an extension of its action on $N_1(Y_p)$.  Components $E$ of the exceptional locus of $f: X \rightarrow Y$ lying over points in $Y$ are disjoint from $X_p$, while for general $X$ the components of $E1$-type lie over curves which are either contracted by $\pi$ or have simple and general branching over $\mathbb{P}^1$.  Theorem \ref{Monodromy} then implies our claim.

If $g : Y \rightarrow S$ is an elementary conic bundle, we may assume $S \cong \mathbb{F}_1$ or $S \cong \mathbb{P}^1 \times \mathbb{P}^1$.  If $g$ is a $C2$-fibration, or, more generally, if the monodromy action of $\pi_1(U,p)$ on $N_1(Y_p)$ is maximal, a reiteration of our prior argument applies.  Similarly, if $S \cong \mathbb{F}_1$, then as $g$ is the base change of a conic-bundle structure over $\mathbb{P}^2$, for general $X$ Theorem \ref{Monodromy} again applies.  Hence we may assume $S \cong \mathbb{P}^1 \times \mathbb{P}^1$ and $g : Y \rightarrow S$ is a $C1$ contraction.  %
Moreover, we may assume every factorization of $\pi$ through a sequence of maps $f, g, h$ as above has this form.  This limits $X = Y$ to one of two deformation types of Fano threefolds, distinguished by $\rho(X) = 3$ and $(-K_X)^3 \in \{12, 14\}$.  When $(-K_X)^3 = 14$, the second example in \cite[Section~11]{noheighttwo16} proves $\pi_1(U,p)$ acts as $W(D_5)$.  When $(-K_X)^3 = 12$, $X \rightarrow \mathbb{P}^1 \times \mathbb{P}^1 \times \mathbb{P}^1$ is a double cover branched over a smooth divisor $B$ of tridegree $(2,2,2)$.  For general $X$, $B$ is a $K3$ surface whose projection to any factor has twenty-four singular fibers.  If $\Delta \subset \mathbb{P}^1 \times \mathbb{P}^1$ is the discriminant locus of a conic bundle $X \rightarrow \mathbb{P}^1 \times \mathbb{P}^1$, the singular fibers of $B$ lie over the branch points of the projection of $\Delta$ to a given factor.  Thus, for general $X$, each projection $\Delta \rightarrow \mathbb{P}^1$ is simply branched, so we may apply Theorem \ref{conic monodromy}.%
\end{proof}

The following Lemma was required for the proof of Theorem \ref{monodromy del Pezzo main result}.  This result is mentioned in \cite[Proof~of~Lemma~3.6]{lastDelPezzoThreefold} without proof.  We provide more details for the reader's convenience as we are unable to find a proof in the literature.

\begin{lem}\label{lem: V1 monodromy}
    Let $V_1$ be a general del Pezzo threefold of degree one.  Consider the linear system $|-\frac{1}{2}K_{V_1}|$ of fundamental divisors on $V_1$ and let $U$ be the open set of smooth fundamental divisors. If $H \in U$,
     then the image of the monodromy map $\pi_1(U, H) \rightarrow \mathrm{Aut}(N_1(H))$ is the Weyl group of $E_8$ type.
\end{lem}
\begin{proof}
    We consider $V_1$ as a double cover of $\mathbb{P}(1,1,1,2)$ branched along a general sextic hypersurface $B \subset \mathbb{P}(1,1,1,2)$.  Let $H$ be a smooth fundamental divisor on $V_1$ and $G$ be the monodromy group acting on $N_1(H)$.  Each $(-1)$ curve in $H$ is a smooth conic on $V_1$.  \cite{Tihomirov_1982} shows the space of conics on $V_1$ is irreducible.  As no other fundamental divisor contains any $(-1)$ curve in $H$, this implies the action of $G$ is transitive on $(-1)$-curves in $H$.  We will show the stabilizer of one such curve $C \subset H$ is isomorphic to the Weyl group of $E_7$.  The Orbit-Stabilizer Theorem then implies our claim.

    First we regard $H$ as the double cover of a quadric cone $Q \subset \mathbb{P}^3$ branched over a general member of $|\mathcal{O}_Q(3)|$.  Let $\mathbb{P} = \mathbb{P}(1,1,1,2)$ and $U \subset |\mathcal{O}_{\mathbb{P}}(6)|$ be the locus of irreducible sextic hypersurfaces.  We have a rational map
    $$\psi : U \times \text{PGL}_3 \dashrightarrow |\mathcal{O}_Q(3)|$$
    mapping $f(x_0, x_1, x_2, y) \times A$ to $(A \cdot f)|_{x_0 = 0}$.  The map $\psi$ surjects the locus $V \subset |\mathcal{O}_Q(3)|$ of irreducible complete intersection curves, and the fiber over a point $v \in V$ is an affine bundle over $\text{PGL}_3$.  Similarly, the rational images of orbits $\psi(u \times \text{PGL}_3)$ connect any two general points in $V$.  As $V_1$ is general, by \cite[Theorem~5]{kollarLefschetz} this demonstrates the monodromy action on $H$ as a member of $|H|$ is equivalent to the monodromy action on $H$ among all branched double covers of $Q \subset \mathbb{P}^3$.  

    Next, we degenerate $H$ to a general singular fundamental divisor $\overline{H}$.  One may realize this degeneration in two equivalent ways.  Viewing $H$ as a divisor in $V_1$, $\overline{H}$ is the preimage under $V \rightarrow \mathbb{P}(1,1,1,2)$ of a divisor tangent to $B$.  The argument of \cite[Proposition 5.2]{Tihomirov_1982} shows $\overline{H}$ has a unique $A_1$ singularity, i.e. the minimal resolution $\widetilde{H} \rightarrow \overline{H}$ is a weak del Pezzo surface of degree one with a unique $(-2)$ curve $E$.  Viewing $H$ as a blow-up of $\mathbb{P}^2$ along $8$ general points instead, the degeneration $\overline{H}$ may be seen as the anticanonical model of the blow-up $\widetilde{H}$ of $\mathbb{P}^2$ along $8$ points $p_1, \ldots, p_8$ general among the collection of points such $p_1, p_2, p_3$ are collinear.

    Let $\mathcal{H} \rightarrow \mathbb{C}$ be a family realizing the degeneration of $H = \mathcal{H}_1$ to $\overline{H} = \mathcal{H}_0$.  \'Etale locally, we may interpret the family as a generalizing $p_1, p_2, p_3$ to linearly independent points.  Picard-Lefschetz theory \cite{voisin_2003} shows the action of the fundamental group $\pi_1(\mathbb{C}\setminus\{0\}, 1)$ on the Neron-Severi group of $H$ does not stabilize the class $C$ of the strict transform of a conic in $\mathbb{P}^2$ passing through $p_4, \ldots p_8$.  Let $\widetilde{C}$ denote the class of the corresponding conic on $\widetilde{H}$.
    
    Recall that $\widetilde{H}$ has a unique $(-2)$-curve $E$.  A direct calculation shows $D = -2K_{\widetilde{H}} - E$ is nef and $\widetilde{C}$ is the unique $(-1)$-curve in $\tilde{H}$ such that $\widetilde{C} . E = 2$.  Moreover, the complete linear series $|2D|$ determines a birational map $\phi : \widetilde{H} \rightarrow S$ onto a smooth del Pezzo surface $S$ of degree $2$ such that $\phi^*K_S = -D$.  This may be seen from the anticanonical model of $\widetilde{H}$, which is the double cover of a quadric cone $Q \subset \mathbb{P}^3$ branched over its intersection with a cubic hypersurface $T$ that has a simple double point along $Q$.  %
    The 56 $(-1)$-curves in $S$ are in bijection with the 56 $(-1)$-curves in $\widetilde{H}$ which pair to 0 with $\widetilde{C}$.  
    Moreover, the map given by the complete linear series $|-K_S|$ is a double covering $S \rightarrow \mathbb{P}^2$ whose branch locus is the canonical embedding of the normalization of $Q \cap T \subset \mathbb{P}^3$.  

    Since $\widetilde{C} \subset \widetilde{H}$ is uniquely determined, the stabilizer of $C \subset H$ in the monodromy group $G$ may be identified with the stabilizer of $\widetilde{C} \subset \widetilde{H}$ under the monodromy action induced by varying the corresponding singular fundamental divisor $\overline{H}$.  By the equivalence established earlier with \cite[Theorem~5]{kollarLefschetz}, this is equivalent to varying the complete intersection curve $Q \cap T \subset \mathbb{P}^3$ among all nodal complete intersection curves. %
    We will show the normalization of $Q \cap T$ varies among all smooth genus $3$ curves with very ample canonical bundle.  As the del Pezzo surface $S$ is uniquely determined by the normalization of $Q \cap T$, it then follows from \cite[Theorem~1.1]{V2monodromy} that the stabilizer of $C$ in $G$ is isomorphic to the Weyl group of $E_7$.

    Let $C$ be a smooth genus 3 curve with very ample canonical bundle $K$.  For any two points $u,v \in C$, $h^0(C, \mathcal{O}(K + u + v)) = 4$.  For general points $u,v \in C$, the corresponding map $C \rightarrow \mathbb{P}^3$ realizes $C$ as the normalization of the complete intersection $C'$ of a smooth quadric $Q'$ and a tangent cubic hypersurface $T'$.  The preimage in $C$ of the unique node in $C'$ consists of $u$ and $v$.  Since $Q' \cong \mathbb{P}^1 \times \mathbb{P}^1$, projection onto either factor is a $g^1_3$ on $C$.  We may write $\mathcal{O}_C(K) \cong \mathcal{O}_C(u + v + x_1 + x_2)$, where $|\mathcal{O}_C(u + v + x_i)|$ is the linear series corresponding to the restriction of the $i^{th}$ projection $Q' \rightarrow \mathbb{P}^1$ to $C$.  For general $u,v \in C$, the line through $u,v$ under the canonical embedding $C \rightarrow \mathbb{P}^2$ does not meet $C$ along any bitangent line to $C$.  We may therefore assume 
    $$\mathcal{O}_C(K - 2x_2) \cong \mathcal{O}_C(u + v + x_1 - x_2) \cong \mathcal{O}_C(u' + v')$$ 
    for distinct points $u', v' \in C$.  As before, $h^0(C, \mathcal{O}(K + u' + v')) = 4$, and the corresponding map $\theta : C \rightarrow \mathbb{P}^3$ is an immersion which sends $u'$ and $v'$ to the same point.  However, since 
    $$\mathcal{O}(K + u' + v') \cong \mathcal{O}_C(K - 2x_2)\cong \mathcal{O}_C(2u + 2v + 2x_1),$$
    the projection of $\theta(C) \subset \mathbb{P}^3$ from a unique point in $\mathbb{P}^3 \setminus \theta(C)$ realizes $\theta(C)$ as a triple cover of a conic in $\mathbb{P}^2$.  In other words, $\theta(C)$ is the nodal complete intersection of a quadric cone $Q$ with a cubic hypersurface $T$.  This finishes our proof.
\end{proof}

We now discuss a corollary of Theorem \ref{monodromy del Pezzo main result}.  This result will be used in Section \ref{section: low degree curves} to prove Theorem \ref{main result 3} and Theorem \ref{main result 4}.

\begin{cor}\label{D1 lines irred}
Let $X$ be a smooth Fano threefold, general in moduli.  %
Suppose $\alpha \in \NE(X)_\mathbb{Z}$ is the generator of an extreme ray associated to an elementary $D1$-contraction.  %
All components of $\overline{M}_{0,0}(X,\alpha)$ are geometrically irrational curves.  Moreover, $\overline{M}_{0,0}(X,\alpha)$ is irreducible if $|-K_X|$ is base-point free.  If instead $|-K_X|$ is not base-point free, $\overline{M}_{0,0}(X,\alpha)$ has exactly two irreducible components.
\end{cor}
\begin{proof}
When $X \rightarrow Y$ is a blow-up and $|\frac{-1}{r}K_Y|$ is very ample, it is well-known \cite{eisenbud20163264} that a general pencil of hypersurface sections will be a Lefschetz pencil \cite{voisin_2003}.  In these cases, fibers of $\pi : X \rightarrow \mathbb{P}^1$ are del Pezzo surfaces of degree $\geq 3$ with at most one $A_1$ singularity, and therefore cannot be cones.  Thus each component of $\overline{M}_{0,0}(X,\alpha)$ parameterizes finitely many curves in any fiber of $\pi$, and Theorem \ref{monodromy del Pezzo main result} implies a single component of $\overline{M}_{0,0}(X,\alpha)$ contains such curves.  

If instead $\pi : X \rightarrow \mathbb{P}^1$ has degree 2 fibers, either $X$ is a blow-up of $V_2$ along the preimage of a general complete intersection of hyperplanes in $|\frac{-1}{2}K_{V_2}| \cong \mathbb{P}^3$, where $V_2 \rightarrow |\frac{-1}{2}K_{V_2}|$ is a double cover branched over a smooth quartic hypersurface, or $X \rightarrow \mathbb{P}^1 \times \mathbb{P}^2$ is a branched double cover.  In both cases, for general $X$, each fiber of $\pi$ will be a double cover of $\mathbb{P}^2$ branched over a quartic curve with at most one node.  Theorem \ref{monodromy del Pezzo main result} again proves $\overline{M}_{0,0}(X,\alpha)$ is irreducible.  

Lastly, suppose $X$ is a blow-up of a del Pezzo threefold $V_1$.  The blow-up $\tilde{V}_1$ of $V_1$ at the unique basepoint of $|\frac{-1}{2}K_{V_1}|$ admits a morphism $\phi : \tilde{V}_1 \rightarrow |\frac{-1}{2}K_{V_1}| \cong \mathbb{P}^2$ with elliptic fibers.  For a general $V_1$, the locus of singular fibers of $\phi$ is an irreducible, geometrically irrational curve of degree twelve \cite{Tihomirov_1982}.  These correspond to geometrically rational anticanonical curves in smooth fundamental divisors on $V_1$.  As well, the space of conics on $V_1$ is an irreducible, branched cover of the dual space to $|\frac{-1}{2}K_{V_1}|$ \cite{Tihomirov_1982}.  Strict transforms of these conics are smooth lines in fibers of $\pi$.  Thus, for a general pencil of lines in $|\frac{-1}{2}K_{V_1}|$, $\overline{M}_{0,0}(X,\alpha)$ has two components with the preceding descriptions.

It remains to show the unique component $M \subset \overline{M}_{0,0}(X, \alpha)$ parameterizing embedded curves is geometrically irrational.  We apply the Reimann-Hurwitz formula for this purpose.  As $X$ is general, each singular fiber of $\pi : X \rightarrow \mathbb{P}^1$ will have a unique $A_1$ singularity.  This implies $M$ is smooth.  
Since $\pi$ contracts each curve parameterized by $M$, it induces a branched covering $\overline{\pi} : M \rightarrow \mathbb{P}^1$.  In each case, we may use \cite[Proposition~7.17]{eisenbud20163264} to relate the number of singular fibers of $\pi : X \rightarrow \mathbb{P}^1$ to topological invariants.  We may compute the degree of the ramification divisor $R$ of $\overline{\pi}$ from this and the number of lines in each singular fiber.  It follows that $K_M = \overline{\pi}^*K_{\mathbb{P}^1} + R$ is effective, proving our claim. %

The preceding argument has the advantage of being uniform.  However, for specific deformation types of Fano threefolds, there are simpler geometric constructions that show $M \subset \overline{M}_{0,0}(X, \alpha)$ from the preceding paragraph is an irrational curve.  For instance,  Let $V_d$ be a general del Pezzo threefold of degree $d \leq 5$.  Suppose $X$ is the blow-up of $V_d$ along a general complete intersection $C \subset V_d$ of fundamental divisors.  If $T \subset X$ is a curve of class $\alpha$, then the image of $T$ in $V_d$ meets $C$ at exactly one point.  As we vary $T$ among lines parameterized by $M$, the induced map $M \to C$ is surjective.  Hence, $M$ cannot be rational, as $C$ is a smooth elliptic curve.  Arguments similar to \cite[Example~8.5]{balancedlinebundlesFano} address all other cases.
\end{proof}

\section{Low Degree Curves and Degeneration to Special Threefolds}\label{section: low degree curves}

In this section, we describe families of lines and conics on smooth Fano threefolds.  We conclude our analysis of families of nef lines in Fano threefolds and prove Theorem \ref{main result 3} via Theorem \ref{main result 4}.  We then consider two-parameter families of non-free conics on smooth Fano threefolds, and prove they lie in specific types of contractible divisors when $\rho(X) \geq 2$.  We also prove Theorem \ref{main thm: conics} in Theorem \ref{thm: conics}.  Using these descriptions of low degree curves, we show Manin components of cubic and quartic curves specialize well in families of Fano threefolds.  This allows us to extend certain results from general Fano threefolds to arbitrary Fano threefolds.

\subsection{Lines, Conics, and Nondominant Families of Nef Curves}
We expand upon Theorem \ref{nonfree curves} and the following lemma from \cite{beheshti2020moduli}.  In this subsection, \textit{lines} and \textit{conics} refer to irreducible, reduced, geometrically rational curves embedded in a smooth variety $X$.  A conic is \textit{non-free} if the normal sheaf $\mathcal{N}_f$ of its normalization $f: \mathbb{P}^1 \rightarrow X$ is not globally generated.  Given a Hirzebruch surface $\pi: \mathbb{F}_r \rightarrow \mathbb{P}^1$, we let $C$ be the class of a rigid section and $F$ be the class of a fiber of $\pi$.

\begin{lem}[\cite{beheshti2020moduli}]\label{lem: E5 lines}
Let $X$ be a smooth Fano threefold. %
\begin{enumerate}
    \item Every subvariety $Y\subset X$ swept out by lines is the image of a birational map $f: S \rightarrow Y$ from a ruled surface $\pi : S \rightarrow C$ whose fibers map birationally to lines in $X$.
    \item Any family of lines of dimension $>1$ sweeps out an $E5$ divisor.
    \item Any family of conics of dimension $>2$ is contained in a contractible divisor of type $E3$, $E4$, $E5$, or $E1$ with normal bundle $\mathcal{O}_{\mathbb{P}^1\times \mathbb{P}^1}(-1,-1)$.
\end{enumerate}
\end{lem}
In fact, remarks after \cite[Lemma~4.6]{beheshti2020moduli} show each family of degree $d$ rational curves with dimension $\geq 2d -1$ lies in a contractible divisor of type $E3, E4, E5$ or $E1$ with normal bundle $\mathcal{O}_{\mathbb{P}^1\times \mathbb{P}^1}(-1,-1)$.  Proposition \ref{properties nonfree conics} elaborates on these remarks.  First, we identify classes of lines in Fano threefolds which are not extreme in the respective Mori cone, and prove Theorem \ref{main result 3}.

\begin{lem}\label{nonextreme lines}
Let $X$ be a smooth Fano threefold and $\alpha \in \NE(X)_\mathbb{Z}$ satisify $-K_X . \alpha = 1$. Then $\Mor(\mathbb{P}^1, X, \alpha)$ is nonempty, and if $\alpha$ is not extreme in $\NE(X)$, then $(X,\alpha)$ is one of three possible pairs listed below:
\begin{itemize}
\item $X \cong \mathbb{P}^1 \times S_1$, where $S_1$ is a del Pezzo surface of degree 1, and $\alpha$ is the class of an anticanonical curve in a fiber of $X \rightarrow \mathbb{P}^1$.
\item $X \rightarrow \mathbb{P}^1 \times \mathbb{P}^2$ is the blow-up of a smooth curve $C$ of bidegree $(5,2)$, and $\alpha$ is the class of the strict transform of a bisecant of $C$ lying in a fiber of $X\rightarrow \mathbb{P}^1$.
\item $\rho(X) = 3$, $(-K_X)^3 = 14$, and $\alpha$ is the pullback of an extremal effective class on a weak Fano contraction $X\rightarrow Y$ associated to a $D1$ fibration $Y\rightarrow \mathbb{P}^1$ with degree $4$ fibers.
\end{itemize}
\end{lem}
\begin{proof}
Applying \ref{Gordan's Lemma} to generators of $\NE(X)$, found in \cite{matsuki1995weyl} and \cite{matsuki2023addendumerratum}, implies our claim.  Alternatively, the following argument may be used to first narrow possibilities.  If $\rho(X) \geq 2$, then $-K_X = D_1 + D_2$ for two integral, nef divisors $D_i$.  Therefore, every anticanonical line is contained in the boundary of $\NE(X)$.  If $\alpha$ is not nef, it must be extreme in $\NE(X)$.  Alternatively, every nef $(-K_X)$-line is contracted by some morphism $\pi : X \rightarrow B$.  If $\text{dim } B = 2$, $\pi$ factors through an elementary contraction contracting $\alpha$.  %
Thus, we may assume $\alpha \in \Nef_1(X)$ and there is no morphism $\pi : X\rightarrow B$ contracting $\alpha$ with $\text{dim } B \geq 2$.  It follows that $\alpha$ is contained in the interior of a facet $F \subset \NE(X)$ associated to a del Pezzo fibration $\pi : X \rightarrow \mathbb{P}^1$.  Hence, we may assume $\rho(X) > 2$.  Moreover, by contracting divisors on $X$, we may assume $\alpha$ does not pair trivially with any $E1$ divisor whose contraction yields a Fano variety.  Since $E2, E3, E4, E5$ divisors are components of reducible fibers of any morphism $X \rightarrow \mathbb{P}^1$ contracting $\alpha$, $\alpha$ must pair trivially with any such divisor.  An analysis of tables of extreme rays of $\NE(X)$ in \cite{matsuki1995weyl} and \cite{matsuki2023addendumerratum} proves our claim.
\end{proof}

The following theorem proves Theorem \ref{main result 3}.

\begin{thm}\label{main result 4}
Let $X$ be a smooth Fano threefold that is general in moduli.  Suppose $X$ is not the product of $\mathbb{P}^1$ with a del Pezzo surface of degree one.  Then for all $\alpha \in \Nef_1(X)$, every component of $\Mor(\mathbb{P}^1, X, \alpha)$ with nondominant universal family parameterizes multiple covers of nef $(-K_X)$-lines.  

For the class $\beta \in \Nef_1(X)_\mathbb{Z}$ of a nef $(-K_X)$-line, $\Mor(\mathbb{P}^1, X ,\beta)$ is irreducible if and only if $X$ is not the blow-up of a del Pezzo threefold $V_1$ of degree one.  When $X$ is a blow-up of $V_1$, $\Mor(\mathbb{P}^1, X ,\beta)$ has two components.  %
\end{thm}

Theorem \ref{main result 4} does not apply to products $X \cong \mathbb{P}^1 \times S_1$ where $S_1$ is del Pezzo surface of degree one.  For such $X$ instead consider the projection $\pi : X \to S_1$.  For $\alpha \in \Nef_1(X)$, if $\Mor(\mathbb{P}^1, X, \alpha)$ has a component $M$ with a nondominant universal family, then $\alpha \notin \NE(\pi)$.  Therefore $\pi$ induces a natural map $\pi_* : \Mor(\mathbb{P}^1, X, \alpha) \to \Mor(\mathbb{P}^1, Y, \pi_*\alpha)$.  The image of $M$ under this map will be a component that parameterizes covers of a geometrically rational anticanonical curve.

\begin{proof}[Proof of Theorem \ref{main result 4}]
Let $\alpha \in \Nef_1(X)$ and suppose the universal family $U$ over a component $M\subset \Mor(\mathbb{P}^1, X, \alpha)$ does not dominate $X$.  Let $D \subset X$ be the closure of the image of $U$.  We first claim that $D$ is a nef divisor.  %
Suppose $D$ is not nef instead. Since $D$ is irreducible, $D$ must be the exceptional divisor for some elementary contraction of $X$.  As $E4$ divisors are specializations of $E3$ divisors and $X$ is general in moduli, this implies $D$ is smooth.  By adjunction, $-K_D = (-K_X - D)|_D$.  However, as $M$ parameterizes a family of curves that dominates $D$ and $\alpha \in \Nef_1(X)$, $M$ must have the expected dimension $(-K_X - D) . \alpha + \dim D < -K_X . \alpha + \dim X$, a contradiction.  Thus $D$ is nef.

By Theorem \ref{nonfree curves} and Lemma \ref{lem: E5 lines}, $D$ is the image of a birational map $f : S \rightarrow D$ from a ruled surface $\phi : S \rightarrow C$ whose fibers map birationally to nef lines in $X$.  If $C$ is irrational, the component $M\subset \Mor(\mathbb{P}^1, X, \alpha)$ must parameterize multiple covers of the fibers of $\phi$.  Therefore, it suffices to prove %
irrationality of all components of $\overline{M}_{0,0}(X,\beta)$ for each $\beta \in \Nef_1(X)_{\mathbb{Z}}$ such that $-K_X . \beta = 1$.  For a proof of this statement and irreducibility of $\overline{M}_{0,0}(X,\beta)$ when $\rho(X) = 1$, see \cite[Chapter~4]{FanoV_Shafarevich}.  %
We address the case $\rho(X) \geq 2$ below.

Suppose $\beta \in \Nef_1(X)_{\mathbb{Z}}$ satisfies $-K_X . \beta = 1$.  Let $\pi : X \rightarrow B$ be the fiber-type contraction corresponding to the smallest face of $\NE(X)$ containing $\beta$.  If $\beta$ spans an extreme ray of $\NE(X)$, then Lemma \ref{lines in C1 fibrations} and Corollary \ref{D1 lines irred} prove our claim.  Otherwise, by Lemma \ref{nonextreme lines}, there are only three possibilities for the pair $(X,\beta)$.  We address these individually below.

If $X \cong \mathbb{P}^1 \times S_1$, there are twelve components of $\overline{M}_{0,0}(X,\beta)$, all of which are rational curves.  If $X \rightarrow \mathbb{P}^1 \times \mathbb{P}^2$ is the blow-up of a general curve $C$ of bidegree $(5,2)$, $\pi : X\rightarrow \mathbb{P}^1$ is induced by the first projection, and the induced map $\pi : \overline{M}_{0,0}(X,\beta) \rightarrow \mathbb{P}^1$ is finite.  Lemma \ref{monodromy del Pezzo main result} implies the space of bisecant lines to $C$ contained in fibers of $\pi$ is irreducible.  Lastly, if $\rho(X) = 3$ and $(-K_X)^3 = 14$, \cite[Theorem~2.11.2]{weak_fano_del_pezzo} identifies $X$ as the blow-up of a complete intersection $Y$ in a $\mathbb{P}^4$-bundle over $\mathbb{P}^1$ along a unique flopping curve.  Because each line of class $\beta$ must be disjoint from the blown-up section, the induced map $\overline{M}_{0,0}(X,\beta) \cong \overline{M}_{0,0}(Y,\beta) \rightarrow \mathbb{P}^1$ must be finite.  Lemma \ref{monodromy del Pezzo main result} implies irreducibility of $\overline{M}_{0,0}(X,\beta)$.  In each case, the Riemann-Hurwitz formula proves irrationality of the unique component of $\overline{M}_{0,0}(X,\alpha)$ from \cite[Proposition~7.17]{eisenbud20163264}, as in the proof of Corollary \ref{D1 lines irred}.
\end{proof}

The following proposition identifies all 2-dimensional families of non-free conics in smooth Fano threefolds.

\begin{prop}\label{properties nonfree conics}
Let $X$ be a smooth Fano threefold, $\alpha \in \NE(X)$ satisfy $-K_X . \alpha = 2$, and $M \subset \overline{M}_{0,0}(X,\alpha)$ be a component whose corresponding one-pointed family $M'$ is nondominant.    
Suppose $M$ does not parameterize multiple covers of $(-K_X)$-lines.  Let $Y$ be the image of $M'$ in $X$.
\begin{itemize}
    \item If $\alpha \in \Nef_1(X)$, then $M$ parameterizes reducible conics and $\alpha$ is extreme in $\NE(X)$.  Moreover, either $\rho(X) = 1$ and $g(X) \leq 3$, or $\rho(X) = 2$ and $Y$ is a singular fiber of an elementary del Pezzo fibration of degree at most 3.
    \item If $\alpha \not\in \Nef_1(X)$, an irreducible component $E\subset Y$ is the exceptional divisor of an elementary birational contraction $X \rightarrow \overline{X}$ lying over a point or curve of $-K_{\overline{X}}$-degree $\leq 1$.  If $E$ is not an $E5$ divisor, $Y = E$ is irreducible.

\end{itemize}
\end{prop}

\begin{proof}
Suppose $\alpha \in \Nef_1(X)$.  To derive a contradiction, suppose general maps parameterized by $M$ have irreducible domain.  Any component of $\Mor(\mathbb{P}^1, X)$ which parameterizes curves contained in some exceptional divisor $E \subset X$ must parameterize non-nef curves.  By Theorem \ref{nonfree curves}, every non-dominant 2 dimensional family of nef conics must lie in the image of some Hirzebruch surface $f : \mathbb{F}_r \rightarrow X$ whose fibers $\pi : \mathbb{F}_r \rightarrow \mathbb{P}^1$ parameterize lines.  As \cite[Theorem~4.1]{beheshti2020moduli} and \cite[Lemma~4.5]{beheshti2020moduli} construct such a family as a contraction of a resolution of a subvariety in $X$, $f$ may be chosen to be birational onto its image.  Since $-f^*K_X$ is big, nef, and pairs to one with each fiber of $\pi$, $-f^* K_X = C + m F$ for some $m \geq r$.  

Every moving curve in $\mathbb{F}_r$ maps to a curve of degree at least $m$ in $X$.  Hence, $r \leq 2$.  If $m = r \leq 2$, then $f$ contracts the rigid section $C$, and the image of $f$ is either a line ($r = 0$), an $E5$ divisor ($r = 1$), or an $E4$ divisor ($r=2$) \cite[Theorem~4.1]{beheshti2020moduli}.  If $r = 0$ and $m = 1$, the image of $f$ must be an $E3$ divisor or $E1$ divisor with normal bundle $\mathcal{O}_{\mathbb{P}^1\times \mathbb{P}^1}(-1,-1)$ \cite[Theorem~4.1]{beheshti2020moduli}.  If $r = 0$ and $m = 2$, there are no two parameter families of embedded $(-K_X)$-conics contained in $\mathbb{F}_r$.  Hence, $r = 1$ and $f^*(-K_X) = C + 2F$.  

As $f_*[\mathbb{F}_1]. (-K_X)^2 = 3$ is odd, we must have $\rho(X) \geq 2$.  Moreover the conics must have class $C + F$ in $\mathbb{F}_1$.  Supposing $\alpha \in \Nef_1(X)_\mathbb{Z}$ is such a conic, we will show $\alpha$ is twice the class of an effective anticanonical line, and the image of $f : \mathbb{F}_1 \rightarrow X$ must be a non-normal divisor.  Lemma \ref{nonextreme lines} then implies our claim, as either $\alpha$ is contracted by a conic fibration, or $\frac{\alpha}{2}$ is the class of a nef line contracted by a del Pezzo fibration $\psi : X \to \mathbb{P}^1$.  In the latter case, $f(\mathbb{F}_1)$ would be contracted by $\psi$, but this contradicts \cite[Remark~4.10]{Fujita_del_Pezzo} because $\psi$ cannot have degree 6 fibers. %

Suppose $-K_X$ is base-point free.  Since $C.(C + 2F) = F.(C + 2F) =1$, the map $\phi: X \rightarrow |-K_X|$ embeds the rigid section of $\pi : \mathbb{F}_1 \rightarrow \mathbb{P}^1$ and each fiber.  Moreover as $\phi$ is at most a degree two cover, the map $\phi \circ f$ either embeds each conic or maps some to multiple covers of lines.  %
Note that $\phi \circ f$ is birational onto its image, as $f_*[\mathbb{F}^1]$ has odd $(-K_X)^2$-degree and $f$ is birational onto its image.  If the restriction of $f$ to each conic is an embedding, then $f$ is injective and adjunction shows $f(\mathbb{F}_1)$ must be an exceptional divisor of $E1$ type.  Hence, some (possibly reducible) conic is a double cover of a line in $X$.

If $-K_X$ has basepoints, $X$ is either $\mathbb{P}^1 \times S_1$ or the blow-up of $V_1$ along an elliptic curve.  If $X \cong \mathbb{P}^1 \times S_1$, there are no non-dominant, two-dimensional families of nef conics. If $X$ is a blow-up of $V_1$, the are only two nef classes of conics.  One of these is twice the class of a nef line.  General, irreducible conics of the other class must be disjoint from the exceptional divisor of $X \rightarrow V_1$, so our claim reduces to the preceding argument.

Hence, general maps parameterized by $M$ must have reducible domain.  Let $f: C_1 \cup C_2 \rightarrow X$ be a general map parameterized by $M$. Because $\alpha$ is nef, no irreducible component of $f(C_1 \cup C_2)$ may be contained in an $E5$ divisor.  Hence, $\overline{M}_{0,0}(X,f_*[C_i])$ is one-dimensional for each $i$.  Let $M_i \subset \overline{M}_{0,0}(X, f_*[C_i])$ parameterize deformations of $f|_{C_i}$ and $M_i' \rightarrow M_i$ be the corresponding family of one-pointed curves.  We may consider $M$ as a component of $M_1' \times_X M_2'$.  Any two distinct curves in $X$ meet at finitely many points.  Hence, as $\text{dim } M = 2$, any two curves parameterized by $M_1$ and $M_2$ intersect.  Therefore $M_1'$ and $M_2'$ have the same image $Y \subset X$.  By taking the base change of $M_i' \to M_i$ along a smooth resolution $B_i \to M_i^{red}$, we may realize $Y$ as the image of a ruled surface $M_i' \times_{M_i} B_i$.  This shows $M_1$ and $M_2$ must be the same component of $\overline{M}_{0,0}(X)$, as otherwise $Y$ would be an $E3$ or $E1$ divisor.  

It follows from Lemma \ref{nonextreme lines} that $\alpha \in \NE(X)$ is extreme.  If $\rho(X) = 1$, then by \cite[Lemma~2.3.3]{kuznetsov2018} $g(X) \leq 3$.  Suppose $\rho(X) \geq 2$ instead.  Let $\psi : X \to B$ be the elementary contraction associated to the extreme ray of $\NE(X)$ containing $\alpha$.  Because $\alpha$ is nef, $\psi$ is of fiber-type.  As any two lines parameterized by $M_1 = M_2$ intersect, $Y$ must be contracted by $\psi$.  Hence, the contraction $\psi$ is an elementary del Pezzo fibration and $\rho(X) = 2$.  In particular, every fiber of $\psi$ is irreducible, which implies $Y$ is a $\psi$-fiber.  By the classification of Fano threefolds, $Y . (-K_X)^2 \neq 6$, and therefore $Y$ is normal \cite[Remark~4.10]{Fujita_del_Pezzo}.  However, the map $M_1' \times_{M_1} B_1 \to Y$ cannot be an isomorphism because any two fibers of $M_1' \to M_1$ are disjoint in $M_1'$ but their images in $Y$ intersect.  It follows from the definition of $M_{0,0}(X)$ that a section of $M_1' \to M_1$ must be contracted by the evaluation map $M_1' \to X$.  Hence, $Y$ is a singular fiber of $\psi$.  When $-K_X$ is very ample, the image of $Y$ in the anticanonical model of $X$ will be the cone over a smooth curve.  From this description of $Y$ and \cite[Remark~2.2.7]{kuznetsov2018}, we see that $\psi$ has fibers of degree $\leq 3$.

Suppose $\alpha \notin \Nef_1(X)$.  If general maps parameterized by $M$ have reducible domains, then one component must sweep out an $E5$ divisor.  Otherwise, $Y$ must be an irreducible, contractible divisor, and general maps parameterized by $M$ have irreducible domain.  Unless $Y$ is a Hirzebruch surface $\mathbb{F}_r$, the classes of such curves are extreme in $\NE(X)$ and associated to the contraction of $Y$.  When $Y \cong \mathbb{F}_r$, $-K_X|_Y = C + mF$ for some $m > r$, so that every moving curve in $Y$ is either of class $F$ or has $(-K_X)$-degree at least $m$.  %
As $m \leq 2$ and $Y|_Y= -C - (2 + r -m)F$ by adjunction, $Y$ lies over a curve $\ell \subset \overline{X}$ of $(-K_{\overline{X}})$-degree $2 - Y^3 =  -2 - r + 2m \leq 2$ \cite[Lemmas~2.1,2]{mori1985classification}.  However, if $r = 0$ and $m = 2$, then $Y|_Y \cong \mathcal{O}_{\mathbb{P}^1 \times \mathbb{P}^1}(-1,0)$, and the only curves of $(-K_X)$-degree $2$ in $Y$ are a one-parameter family of free conics and a three parameter family of multiple covers of lines.  Hence, $\ell . (-K_{\overline{X}}) \leq 1$.
\end{proof}

\begin{rem}

Numbered as in \cite{mori1981classification}, deformation types of Fano threefolds with
\begin{itemize}
    \item $E5$ divisors are $2.28$, $2.36$, $3.9$, $3.14$, $3.22$, and $3.29$;
    \item $E3$ or $E4$ divisors are $2.8$, $2.15$, and $2.23$;
    \item $E1$ divisors lying over curves of anticanonical degree 0 in weak Fano threefolds are $3.2$, $3.5$, $3.21$, $3.31$, $4.2$, $4.5$, $4.8$, $4.11$, $4.12$, $5.1$, and $5.2$;
    \item $E2$ divisors are $2.30$, $2.35$, $3.14$, $3.19$, and $3.26$; %
    \item $E1$ divisors lying over lines in Fano threefolds are $3.11$, $3.16$, $3.18$, $3.23$, $3.30$, $4.4$, $4.9$, $4.12$, $5.1$, and $5.2$.
\end{itemize}
The last type of $E1$ divisor is the strict transform of an $E2$ divisor on a smooth, elementary, Fano contraction.  There is at most one divisor of the first three types on a Fano threefold $X$ unless $X$ has deformation type $3.9$ or $4.2$.  %
\end{rem}

We conclude our analysis of lines and conics with a proof of Theorem \ref{main thm: conics}.    We recall the Theorem's statement below.

\begin{thm}[Restatement of Theorem \ref{main thm: conics}]\label{thm: conics}
Let $X$ be a smooth Fano threefold of index $r$ and Picard rank $\rho$.  %
Suppose $\alpha \in \Nef_1(X)_\mathbb{Z}$ satisfies $-K_X . \alpha = 2$ and $\alpha \not\in \partial \NE(X)$.  Then $\rho \leq 2$, and
\begin{enumerate}
\item $\free(X,\alpha)$ is irreducible if $\rho = 2$, if $r > 1$, or if $\rho = r = 1$ and $g(X) > 6$.
\item Suppose $\rho = r = 1$ and $X$ is general in moduli.  If $-K_X$ is very ample, $\free(X,\alpha)$ is irreducible; otherwise, $\free(X,\alpha)$ has two components.
\end{enumerate}
\end{thm}

\begin{proof}
\cite{Fanoindex1rank1} proves (2) when $-K_X$ is very ample, while the remaining cases follow from \cite{ceresa86} (see Remark \ref{rem: Ceresa-Verra}), \cite{monodromyK3}, and \cite{doublequadrics}.  Note that \cite{doublequadrics} ignores double covers of lines tangent to the branch locus of $X \rightarrow Q \subset \mathbb{P}^4$, which constitute another component of free conics on $X$.  %

\cite{kuznetsov2018} proves (1) when $\rho(X) = 1$ and $X$ is not a del Pezzo threefold of degree $\leq 2$.  Lemma \ref{interior conic lemma} proves the bound $\rho(X) \leq 2$.  We prove irreducibility of $\free(X,\alpha)$ in three separate cases, according to the Mori structure of $X$.   %

\textbf{Case 1:}  Suppose $\rho(X) = 2$ but $X$ does not admit a conic bundle structure over $\mathbb{P}^2$.  Then there exists a blow-up map $\pi : X \rightarrow Y$ where $Y$ is a Fano threefold of index $r \geq 2$.  The exceptional divisor $E$ of $\pi$ maps to a smooth curve $C \subset Y$.  If $r = 2$, then $E . \alpha = 0$, and our claim follows from either \cite{kuznetsov2018} or Case 3.  Otherwise, $\alpha$ must be the class of a strict transform of a line in $Y$ meeting $C$ with multiplicity $r-2$.  The space of such curves is clearly irreducible.

\textbf{Case 2:}  Suppose $X$ has a conic bundle structure $\pi : X \rightarrow \mathbb{P}^2$.  In this case, $\pi_* \alpha$ must be the class of a line in $\mathbb{P}^2$.  Consider a general pencil $V \subset |\mathcal{O}_{\mathbb{P}^2}(1)|$.  For general $C \in V$, the preimage $S_C = \pi^{-1}(C)$ contains finitely many curves of class $\alpha$.  To show $\overline{\free}(X,\alpha)$ is irreducible, we will prove that monodromy over $V$ acts transitively on the classes of these rigid curves.

Let $f \in N_1(S_C)$ be the class of a smooth fiber of $\pi$ and consider $\overline{\free}(X, \alpha + i_*f)$, where $i : S_C \rightarrow X$ is the inclusion.  Lemma \ref{lines in C1 fibrations}, Corollary \ref{cor: conic monodromy cor}, and Proposition \ref{low degree curves degeneration} show $\overline{\free}(X,\alpha + i_*f)$ is irreducible.  Hence, as $V$ is a general pencil, monodromy over $V$ acts transitively on classes in $N_1(S_C)$ representing $\alpha + i_*f$.  But this action also stabilizes the class of $f$, and hence must act transitively on classes of curves representing $\alpha$.

\textbf{Case 3:} Suppose $X$ is a del Pezzo threefold of degree $d(X) \leq 2$.  Our claim is well-known when $X$ is general in moduli, see \cite{Tihomirov_1982} and \cite{2019}.  Consider the moduli $\pi : Y \rightarrow B$ of del Pezzo threefolds of degree $d(X)$.  Let $i: X \rightarrow Y$ be the inclusion and consider $M = \overline{M}_{0,0}(Y, i_*\alpha)$.  Irreducibility of $\overline{M}_{0,0}(X,\alpha)$ is equivalent to irreducibility of the fiber of $M$ over $p = \pi\circ i(X)$.  Since the fiber of $M$ over a general point in $B$ is irreducible, $M_p$ must be connected.  When $d(X) = 2$, we will show $M_p$ is nonsingular in codimension 1 and has local embedding dimension at most 3 around each singular point.  As $\text{dim } M_p = 2$, this will prove $M_p$ is irreducible when $d(X) = 2$.  When $d(X) = 1$, $M_p$ may not be non-reduced in codimension 1.  However, its singular locus is contracted by a generically finite morphism $M_p \rightarrow \mathbb{P}^2$ constructed in \cite{Tihomirov_1982}, which allows us to apply a similar argument.  %

First, suppose $f: \mathbb{P}^1 \rightarrow X$ is not an embedding and $f_*[\mathbb{P}^1] = \alpha$.  Necessarily $d(X) = 1$, $C = f(\mathbb{P}^1)$ is the complete intersection of two fundamental divisors, and $C$ has at most one cusp or node.  Hence, $h^1(C,\mathcal{N}_{f/X}) = 0$.

Consider instead an embedding $f: C \rightarrow X$ with $f_* C = \alpha$.  Let $\phi : X \rightarrow Z$ be the double cover structure induced by $\frac{-1}{d(X)} K_X$.  There is an exact sequence %
$$0 \rightarrow \mathcal{T}or_1^{\mathcal{O}_X}(f_*\mathcal{O}_C, \mathcal{O}_R(2R)) \rightarrow \mathcal{N}_{f/X} \rightarrow \mathcal{N}_{\phi \circ f/Z} \rightarrow f^*\mathcal{O}_R(2R) \rightarrow 0$$ %
where $R \subset X$ is the ramification locus of $\phi$.  %
Since $\phi \circ f$ is an embedding, its image is a smooth complete intersection with normal bundle
$$\mathcal{N}_{\phi \circ f/Z} \cong \mathcal{O}_{\mathbb{P}^1}(1) \oplus \mathcal{O}_{\mathbb{P}^1}(3 - d(X)).$$

\noindent If $f(C) \subset R$, $f^*\mathcal{O}_R(2R) \cong \mathcal{O}_C(2R) \cong \mathcal{O}_C(8- 2d(X))$ and $\mathcal{T}or_1^{\mathcal{O}_X}(f_*\mathcal{O}_C, \mathcal{O}_R(2R)) \cong \mathcal{O}_C(4-d(X))$.  %
Thus $h^1(C,\mathcal{N}_{f/X}) \leq 3 - d(X)$.  Note that there are at most finitely many rational curves in $ f(C) \subset R$ of bounded $K_R$-degree.

Otherwise, if $f(C) \not\subset R$, then $f^*\mathcal{O}_R(2R)$ is torsion of length $4- d(X)$ and $\mathcal{T}or_1^{\mathcal{O}_X}(f_*\mathcal{O}_C, \mathcal{O}_R(2R))$ is 0.  Thus, $\mathcal{N}_{f/X}$ is generated in degree at most $3 -d(X)$.

When $d(X) = 2$, this proves $M_p$ is singular only along isolated points corresponding to curves $f(C) \subset R$.  As the embedding dimension of $M_p$ at each singular point is $3$, $M_p$ must be irreducible as it is connected. 

When $d(X) = 1$, $h^1(C,\mathcal{N}_{f/X}) \leq 2$ and equality is obtained only for the finitely many stable maps such that $f(C) \subset R$.  However, there may be one-parameter families of maps $f: C \rightarrow X$ such that $h^1(C,\mathcal{N}_{f/X}) = 1$.  The proof of \cite[Proposition~6.11]{balancedlinebundlesFano} completely characterizes such families: they are the rulings of fundamental divisors $H$ isomorphic to a cone 
$$H = V(ax^6_1 + bx^4_1y + y^3 + z^2) \subset \mathbb{P}(1,1,2,3), \text{ such that } 4b^3 + 27a^2 \neq 0,$$%
over an elliptic curve, where $\text{deg}(x_i) = 1$, $\text{ deg}(y) = 2$, and $\text{ deg}(z) = 3$.  The projection $H \subset \mathbb{P}(1,1,2,3) \dashrightarrow \mathbb{P}(1,1,2)$ is the restriction of the double-cover structure $\phi : X \rightarrow Z = \mathbb{P}(1,1,1,2)$ to $H$.  The branch locus $B \subset Z$ of $\phi$ intersects $\phi(H)$ along the union of three smooth curves, the irreducible components of $V(ax_1^6 + bx_1^4y + y^3) \subset \mathbb{P}(1,1,2)$.  The one-parameter family of conics in $H$ such that $h^1(C, \mathcal{N}_{f/X}) = 1$ are curves that map under $\phi$ to elements of the pencil $\langle x_1^2, y \rangle \subset H^0(\mathbb{P}(1,1,2), \mathcal{O}(2))$ generated by irreducible components of $B \cap \phi(H)$.

There are finitely many divisors $H$ as above, because any such divisor is covered by a one-parameter family of non-free conics.  As we will show below, outside of these one-parameter families there are finitely many conics $f: C \rightarrow X$ with $h^1(C,\mathcal{N}_{f/X}) = 1$, and each such curve corresponds to a higher-order flex line or multi-tangent of a cuspidal degree 12 curve $C_0 \subset \mathbb{P}^2$.

Following the detailed analysis of \cite{Tihomirov_1982}, we may consider the rational projection $\psi : X \rightarrow \mathbb{P}(1,1,1,2) \dashrightarrow \mathbb{P}(1,1,1)$.  This induces regular morphisms $\psi|_B : B \rightarrow \mathbb{P}^2$ and $\psi_* : M_p \rightarrow (\mathbb{P}^{2})^{*}$.  The branch locus of the triple cover $\psi|_B$ is a degree 12 curve $C_0 \subset \mathbb{P}^2$ with 24 cusps for singularities, since $B$ is smooth \cite{Tihomirov_1982}.  

From out description of $\phi \circ f(C)$ as a complete intersection, it is simple to see that curves $f: C \rightarrow X$ with $h^1(C,\mathcal{N}_{f/X}) > 0$ must lie in singular fundamental divisors $H$ whose singular locus is not a single $A_1$ point, and that such fundamental divisors correspond to singularities of the dual curve $C_0^* \subset (\mathbb{P}^2)^*$ \cite[Page~386]{Tihomirov_1982}.  Thus, $\psi_*$ contracts the singular locus of $M_p$ to a finite set of points.  The preimage $\psi_*^{-1}(\ell) \subset M_p$ of a general line $\ell \subset (\mathbb{P}^2)^*$ is therefore strictly contained in the smooth locus of $M_p$, and meets each component of $M_p$.  It thus suffices to prove $\psi_*^{-1}(\ell)$ is connected.  However, when $X$ is general this follows immediately from the irreduciblity of $M_p$; thus, by specializing $X$ we obtain our result.
\end{proof}

\subsection{Degeneration to Special Threefolds}  
When proving Theorem \ref{Main Result} for certain deformation types of Fano threefolds, we assume generality of $X$ to apply monodromy results. Our goal in this subsection is to prove the following proposition, which allows us to extend our proof of Theorem \ref{Main Result} to arbitrary smooth threefolds.  %

\begin{prop}\label{low degree curves degeneration}
Consider a family $\pi : \mathcal{X} \rightarrow B$ of smooth Fano threefolds.  Let $\alpha \in \NE(\pi)$ be a class such that $-K_{\mathcal{X}} . \alpha \leq 5$.  Suppose that for some component $\mathcal{M}' \subset \overline{\free}^{bir}_1(\mathcal{X}, \alpha)$ general fibers of the evaluation map $\text{ev}: \mathcal{M}' \to \mathcal{X}$ are irreducible.  Then for each $b \in B$ the intersection $\mathcal{M}' \cap \overline{\free}^{bir}_1(\mathcal{X}_b, \alpha)$ is irreducible.  In particular, if $\overline{\free}^{bir}_1(\mathcal{X}, \alpha)$ is irreducible, then so is $\overline{\free}^{bir}_1(\mathcal{X}_b, \alpha)$.
\end{prop}

Our proof requires the following lemmas.

\begin{lem}\label{lem: normal bundle lines}
Let $X$ be a smooth Fano threefold.  Suppose $X$ is not a blow-up of a del Pezzo threefold of degree one.  There are at most finitely many anticanonical lines $f: \mathbb{P}^1 \rightarrow X$, up to reparameterization, whose normal bundle $\mathcal{N}_f$ has $h^1(\mathbb{P}^1, \mathcal{N}_f) > 1$.  Moreover, each such line meets finitely many other lines in $X$.

\end{lem}

\begin{proof}
Our claim is trivial if $-K_X$ is very ample or if $X$ is a product variety.  Otherwise, as $-K_X$ is basepoint free, the complete linear system defines a double cover structure $\pi: X \rightarrow Y \subset \mathbb{P}|-K_X|$ with a smooth branch locus $B \subset Y$.  Explicit descriptions of $B$ imply $K_B$ is effective, and thus $B$ contains finitely many $-K_X$ lines.  Every other anticanonical line $\ell$ on $X$ is an irreducible component of a line in $Y \subset \mathbb{P}|-K_X|$ with a reducible preimage.  There is an inclusion
$$0 \rightarrow \mathcal{N}_{\ell/X} \rightarrow \mathcal{N}_{\ell\cup \ell'/X}$$
of the normal bundle of $\ell$ into the normal bundle $\mathcal{N}_{\ell\cup \ell'/X}$ of the union of $\ell$ with its Galois conjugate $\ell'$, considered as a subscheme of $X$.  As the latter bundle is generated in degree $-1$ and $\ell$ is smooth, this proves $h^1(\mathcal{N}_f) \leq 1$.

Hence, the ramification locus $\pi^{-1}(B) \subset X$ contains every anticanonical line $f: \mathbb{P}^1 \rightarrow X$ whose normal bundle $\mathcal{N}_f$ has $h^1(\mathcal{N}_f) > 1$.  Suppose such a line $f(\mathbb{P}^1)$ met a one parameter family of other lines.  This one parameter family may be parameterized by a ruled surface $\rho : S \rightarrow C$, $\phi : S \rightarrow X$ mapping birationally onto its image in $X$.  We claim there exists a $\rho$-section $T \subset S$ such that $\phi(T)$ is a point and  $f(\mathbb{P}^1)$ contains $\phi(T)$.

Suppose there is no section $T \subset S$ as above.  Let $U \subset S$ be a maximal open subset such that $\phi|_U$ is an embedding.  Since the restriction of $\phi$ to each $\rho$-fiber is an embedding, $S \setminus U$ is the union of finitely many $\rho$-fibers and at most one $\rho$-section, which would be contracted by $\phi$.  Thus, as $f(\mathbb{P}^1)$ meets the image of every $\rho$-fiber, by our hypothesis $f(\mathbb{P}^1) \subset \phi(S)$ and $f(\mathbb{P}^1)$ is not the image of a $\rho$-fiber.  Hence, $C \cong \mathbb{P}^1$ and $S \cong \mathbb{F}_r$ for some $r$.  But, since $X$ is Fano and $-K_X . f_*[\mathbb{P}^1] = 1$, we must have $r \leq 1$.  If $r = 1$, the rigid section of $\rho$ must be contracted by $\phi$.  In this case, $\phi(S)$ would be an $E5$ divisor.  If $r = 0$ instead, then $\phi(S)$ would be an $E1$ or $E3$ divisor.  Both options contradict $h^1(\mathcal{N}_f) > 1$, proving our claim.

The image of each fiber of $\rho$ therefore meets $f(\mathbb{P}^1)$ at a common point $p \in f(\mathbb{P}^1) \subset \pi^{-1}(B)$.  However, this implies $\pi \circ \phi(S)$ is a cone with vertex $p$.  However, this is impossible, as each line in $X$ must be totally tangent to the smooth surface $B \subset Y$ upon projection.
\end{proof}

\begin{lem}\label{terminal threefold lemma}
    Let $X$ be a smooth Fano threefold.  Let $D \subset X$ be the union of exceptional divisors on $X$ which contain reduced $(-K_X)$-conics.  There is a contraction $\pi : X \rightarrow \overline{X}$ to a terminal weak Fano threefold with execptional locus $D$.
\end{lem}

\begin{proof}
Recall that $D\subset X$ is the collection of exceptional divisors whose elementary contractions are the blow-ups of points or curves of degree $\leq 1$ by Proposition \ref{properties nonfree conics}.  Let $D_0 \subset D$ be the union of all $E2$, $E3$, $E4$, $E5$, and $E1$ divisors with normal bundles $\mathcal{O}_{\mathbb{P}^1 \times \mathbb{P}^1}(-1,-1)$.  It is well known that $D_0$ contains mutually disjoint divisors \cite[Lemma~2.3]{beheshti2020moduli}.  Let $D_1$ be the closure of $D\setminus D_0$.  Each component $E \subset D_1$ is a Hirzebruch surface $\mathbb{F}_1$ whose rulings are the exceptional lines of a contraction $X \rightarrow X_E$ to a smooth Fano threefold.  Moreover, by \cite[Corollary~4.8,~Proposition~4.9]{mori1983classification} and \cite[Proposition~5.3]{mori1985classification}, the image of $E$ in $X_E$ must be an exceptional line on an $E1$ divisor $F_E$.  For any such divisor $E$, we let $\tilde{F}_E \subset X$ denote the strict transform of $F_E$ and $C_E \subset E \cong \mathbb{F}_1$ be the rigid section.  Note that 
\begin{itemize}
    \item $\tilde{F}_E \cap E$ is the rigid section $C_E \subset E \cong \mathbb{F}_1$, 
    \item $C_E . E = 0$ and $C_E . \tilde{F}_E = -1$,  
    \item $C_E$ is the only curve contained in $E$ which pairs nonnegatively with $E$.
\end{itemize}
For any divisor $E \subset D_1$, we let $X \rightarrow \tilde{X}_E$ denote the contraction of $[C_E] \in \NE(X)$.  Observe that $\tilde{X}_E$ is a smooth threefold, and the image of $E$ in $\tilde{X}_E$ is an $E2$ divisor.

Any two intersecting divisors in $D$ must both be of $E1$ type.  Indeed, suppose $E_1, E_2 \subset D$ are two components with nonempty intersection.  We may assume $E_1 \subset D_1$.  Let $C$ be any irreducible component of $E_1 \cap E_2$.  Suppose $E_2 \in D_0$.  Since every curve in $E_1$ aside from $C_{E_1}$ moves in $E_1$, and every curve in $E_2$ pairs negatively with $E_2$, we must have $C = C_{E_1}$.  Were $E_2$ not of $E1$ type, deformations of $C \subset E_2$ would pair nontrivially with $E_1$, contradicting $C_{E_1} . E_1 = 0$.

Suppose $E_1, E_2 \in D$ intersect nontrivially, and assume $E_1 \in D_1$.  We will show that either $E_2 = \tilde{F}_{E_1}$, or $E_2 \in D_1$ and $E_1 = \tilde{F}_{E_2}$.  Indeed, if $E_2 \in D_0$ then by the preceding paragraph $E_1 \cap E_2 = C_{E_1}$.  Since $E_2|_{E_2}$ is anti-ample, the exceptional locus of $X \to \tilde{X}_{E_1}$ contains $E_2$.  This proves $E_2 = \tilde{F}_{E_1}$ if $E_2 \in D_0$.  

Suppose $E_2 \subset D_1$ instead and let $C$ be a component of $E_1 \cap E_2$.  Since every curve in $E_i$ aside from $C_{E_i}$ moves in $E_i$ and pairs negatively with $E_i$, $C$ must be $C_{E_i}$ for some $i$, say $i = 1$.  If $C = C_{E_2}$ as well, then $E_2 \neq \tilde{F}_{E_1}$.  Under the contraction $X \rightarrow \tilde{X}_{E_1}$ of $C \in \NE(X)$, the images $\overline{E}_1, \overline{E}_2 \subset \tilde{X}_{E_1}$ of $E_1$ and $E_2$ would be intersecting $E2$ divisors.  This cannot happen, because the intersection $\overline{E}_1 \cap \overline{E}_2$ would be a curve that deforms to cover $\overline{E}_1$, but $\overline{E}_2|_{\overline{E}_2}$ is anti-ample.  Thus $C \neq C_{E_2}$.  By similar reasoning, the existence of an elementary contraction $f: X \rightarrow \tilde{X}_{E_1}$ to a smooth variety, contracting $C_{E_1} = C\subset E_2$, proves $E_2 = \tilde{F}_{E_1}$.  

Suppose $E_1 \subset D_1$ and $E_2 = \tilde{F}_{E_1} \subset D$.  We claim there is no other component $E_3 \subset D$ that meets $E_1$.  Indeed, if $E_3$ were such a component, then $E_3 \not\subset D_0$ as this would imply $E_3 = \tilde{F}_{E_1} = E_2$.  Therefore, as $E_3 \subset D_1$ and $E_3 \neq \tilde{F}_{E_1}$, the preceding paragraph shows $E_1 = \tilde{F}_{E_3}$.  But then $E_1 \cap E_3$ is a ruling of $E_1 \cong \mathbb{F}_1$, and $E_1 \cap E_2$ is the rigid section $C_{E_1} \subset E_1$.  Hence $E_3 \cap E_2 \neq \emptyset$.  Since $E_1 = \tilde{F}_{E_3}$, $E_2 \neq \tilde{F}_{E_3}$.  It follows from previous arguments that $E_2 \in D_1$ and $E_3 = \tilde{F}_{E_2}$.

Consider the face $F\subset \NE(X)$ spanned by curves contained in $E_1 \cup E_2 \cup E_3$.  The extreme rays of $F$ are generated by the three $(-K_X)$-line classes $e_i \in \NE(X)$ such that $E_i . e_i = -1$.  Hence, $F$ intersects $\Nef_1(X)$ precisely along the ray spanned by $r = e_1 + e_2 + e_3$.  The contraction $\pi_r : X \rightarrow \mathbb{P}^1$ would be a del Pezzo fibration with special fiber $E = E_1 \cup E_2 \cup E_3$.  Moreover, one may factor $\pi_r$ first through the contraction of $E_1$, $X \rightarrow X_{E_1} = \tilde{X}_{E_3}$, and then through the contraction $\tilde{X}_{E_3} \rightarrow Y$ of the image of $E_3$, which is an $E2$ divisor.  This shows $Y$ is a smooth Fano threefold of Picard rank 2.  In addition, the image of $E$ in $Y$ is a fiber $F$ of $Y \to \mathbb{P}^1$.  Since this fiber is the image of $E_2$, $F \cong \mathbb{P}^2$.  The only Fano $\mathbb{P}^2$-bundles over $\mathbb{P}^1$ are $\mathbb{P}^1 \times \mathbb{P}^2$ and the blow-up $\text{Bl}_l \mathbb{P}^3$ of $\mathbb{P}^3$ along a line $l \subset \mathbb{P}^3$.  However, by construction $\tilde{X}_{E_3} \rightarrow Y$ is the blow-up of $Y$ along a point $p \in F$, and $X \rightarrow \tilde{X}_{E_3}$ is the blow-up of the strict transform of a line $\ell \subset F$ through $p$.  If $Y \cong \mathbb{P}^1 \times \mathbb{P}^2$, then the strict transform of the curve $\mathbb{P}^1 \times p \subset Y$ would have $(-K_X)$-degree 0, a contradiction.  Similarly, if $Y \cong \text{Bl}_l \mathbb{P}^3$, then $\ell \subset F$ would meet the exceptional divisor $E_l$ of $Y \to \mathbb{P}^3$.  Since $E_l$ is swept out by $(-K_Y)$-lines, this also contradicts the hypothesis that $X$ is Fano.

Suppose once more that $E_1 \subset D_1$ and $E_2 = \tilde{F}_{E_1} \subset D$.  If $E_2 \subset D_1$, then as no other component $E_i \subset D$ meets $E_1$, for every component $E_i \neq E_2$ meeting $E_2$, $E_2 = \tilde{F}_{E_i}$.  Therefore, we may contract $E_2$ to reduce the number of components of $D$.  Similarly, suppose $E_2 \subset D_0$ instead.  As we have shown $E_2$ has type $E1$, necessarily $\mathcal{O}_{E_2}(E_2) \cong \mathcal{O}_{\mathbb{P}^1 \times \mathbb{P}^1}(-1,-1)$, $C_{E_1} = E_1 \cap E_2$ is a ruling of $E_2 \cong \mathbb{P}^1 \times \mathbb{P}^1$, and any nonempty intersection of a component of $D_1$ with $E_2$ must be linearly equivalent to $C_{E_1}$.  As this must hold for each component $E_2 \subset D_0$ meeting a component of $D_1$, there is a contraction of all divisors in $D$.%
\end{proof}

\begin{cor}\label{general curve and point}
    Let $X$ be a smooth Fano threefold.  Suppose $X$ is not a blow-up of a del Pezzo threefold of degree one.  Then there exists a big and base-point free linear series $V$ on $X$ such that the following statement holds.
    
    Let $p \in X$ be a general point, $L : T_p X \to \mathbb{C}$ be a general linear function, and $Q \subset X$ be a general complete intersection of two divisors in $V$.  Suppose $f : (C, q) \rightarrow X$ is a stable, pointed, rational curve such that $-K_X. f_*[C] = 5$, $f(C)$ meets $Q$, $f(q) = p$, and $(L \circ df_q) : T_q C \to \mathbb{C}$ is the zero map.  Then $[f : (C,q) \to X]$ is a smooth point of $\overline{\mathcal{M}}_{0,1}(X, f_*[C])$.
\end{cor}

\begin{rem}\label{rem: general linear function}
    Let $X$, $p$, and $L$ be as in Corollary \ref{general curve and point}.  Consider a stable, pointed, rational curve $f : (C, q) \rightarrow X$ such that $f(q) = p$ and $L \circ df_q$ is the zero map.  Let $\mathcal{F} \subset \overline{\mathcal{M}}_{0,1}(X, f_*[C])$ be a family of deformations of $f$ preserving these two conditions.  If $C$ is irreducible, then a general deformation of $f$ is immersive at its marked point.  Thus, generality of $p$ and $L$ imply $\dim \mathcal{F} \leq -K_X . f_*[C] - 3$.  For example, $f(C)$ cannot be a $(-K_X)$-conic because only finitely many geometrically rational conics pass through $p$, and all of these curves must be smooth at $p$.
    
    More generally, let $C_q \subset C$ be the component containing $q$.  By the paragraph above, either $f(C_q) = p$ or $f(C_q)$ is a curve of $(-K_X)$-degree at least 3.  If $f(C_q) = p$, let $C' \subset C$ be the union of components which meet $f^{-1}(p)$.  At least two components of $C'$ are not contracted by $f$.  Thus, $-K_X . f_*[C'] \geq 4$.  If equality holds, then every infinitesimal deformation of $f$ in $\mathcal{F}$ either smooths a node of $C$ or fixes $f|_{C'}$.
\end{rem}

\begin{proof}[Proof of Corollary \ref{general curve and point}]
Let $V$ be the linear system defining the morphism $\pi : X \rightarrow \overline{X}$ from Lemma \ref{terminal threefold lemma}.  A general complete intersection $Q$ of two divisors in $V$ does not intersect the exceptional locus of $\pi$.  Thus, by Lemma \ref{lem: E5 lines} and Proposition \ref{properties nonfree conics}, there are only finitely many lines and irreducible non-free conics meeting $Q \subset X$.  Let $L(Q) \subset X$ be the union of $Q$ with such lines and irreducible non-free conics.  

Consider a pointed map $f:(C,q) \rightarrow X$ be as in the statement of Corollary \ref{general curve and point}.  In order to show $f$ represents a smooth point of $\overline{\mathcal{M}}_{0,1}(X, f_*[C])$, we will analyze the map by its behavior on the irreducible components of $C$.  To begin, let $C_1 \subset C$ be the component containing $q$.  If $f(C_1) = p$, we claim every component of $C$ meets $f^{-1}(p)$.  If not, the union $C' \subset C$ of components meeting $f^{-1}(p)$ would have $(-K_X)$-degree at most 4.  By Remark \ref{rem: general linear function}, $f(C')$ cannot meet $L(Q)$, a contradiction.  Thus, $H^1(C, f^*\mathcal{T}_X) = 0$ and so $\overline{\mathcal{M}}_{0,1}(X, f_*[C])$ is smooth at $f$.

Suppose instead that $f$ does not contract $C_1$.  By Remark \ref{rem: general linear function}, $-K_X . f_*[C_1] \geq 3$.  If $-K_X . f_*[C_1] = 3$, then as $p$ and $L$ are general there are finitely many choices for $f|_{C_1}$.  In this case, $f(C_1)$ must be disjoint from $L(Q)$ by freeness of $f|_{C_1}$.  Moreover, if $-K_X . f_*[C_1] = 4$ and $f(C_1)$ meets $Q$, then $f|_{C_1}$ must be a general free quartic curve.  It follows that $C$ has at most three components; we list combinatorial types for $C$ and incidence conditions with $Q$ below:
\begin{enumerate}
    \item $C = C_1$ is irreducible;
    \item $C = C_1 \cup C_2$ has two components and $f(C_1)$ meets $Q$;
    \item $C = C_1 \cup C_2$ has two components and $f(C_2)$ meets $Q$;
    \item $C = C_1 \cup C_2 \cup C_3$ has three components and $f(C_3)$ meets $Q$.
\end{enumerate}

In each case, $f_1 = f|_{C_1} : C_1 \rightarrow X$ is free. So, we may assume $C$ is reducible and $f|_{C_2}$ is not free.  By Proposition \ref{properties nonfree conics} and the definition of $Q$, this eliminates cases (1) and (3).  It follows that $f(C_i)$ must be an anticanonical line for each $i \geq 2$.  In particular, $-K_X . f_*[C_1] = 3$ if and only if $C$ has three components.  Therefore, prior arguments show $f(C_1)$ is a general free curve.  Proposition \ref{prop: general free curve not tangent to divisor} ensures $f(C_1)$ is not tangent to any divisor swept out by anticanonical lines in $X$.  

Consider the maps $f_i = f|_{C_i}$ for each $i$.  For $i \geq 2$, $f(C_i)$ must either meet $Q$ or the image $f(C_1)$ of a general free curve.  Lemma \ref{lem: normal bundle lines} therefore implies $h^1(C_i, \mathcal{N}_{f_i}) \leq 1$.  If $-K_X . f_*[C_1] = 4$, then $\mathcal{N}_{f}|_{C_1}$ is globally generated and $h^1(C_2, \mathcal{N}_f|_{C_2}) = 0$.  Similarly, if $-K_X . f_*[C_1] = 3$, it follows from Proposition \ref{properties nonfree conics} that $\mathcal{N}_{f}|_{C_i}$ is globally generated for $i = 1,2$ and $h^1(C_3, \mathcal{N}_f|_{C_3}) = 0$.  This proves our claim.
\end{proof}

We finish this section with a proof of Proposition \ref{low degree curves degeneration}.

\begin{proof}[Proof of Proposition \ref{low degree curves degeneration}]
Let $\mathcal{X}_b$ be any fiber of $\pi : \mathcal{X} \to B$.  Consider the fiber $\mathcal{M}'_b$ of $\pi \circ \text{ev} : \mathcal{M}' \to B$ over $b$.  The intersection $\mathcal{M}' \cap \overline{\free}_1^{bir}(\mathcal{X}_b, \alpha)$ is a union of components of $\overline{\free}_1^{bir}(\mathcal{X}_b, \alpha)$; we will show this intersection is irreducible.  If $\mathcal{M}' = \overline{\free}^{bir}_1(\mathcal{X}, \alpha)$, this implies $\overline{\free}_1^{bir}(\mathcal{X}_b, \alpha)$ is irreducible as $\overline{\free}_1^{bir}(\mathcal{X}_b, \alpha) \subset \mathcal{M}'$.

By \cite[Theorem~3.5]{Fernex2009RigidityPO}, after passing to a base change of $\pi$ we may assume there are natural isomorphisms $N_1(\mathcal{X}/B) \cong N_1(\mathcal{X}_{b'})$ sending $\NE(\pi) \subset N_1(\mathcal{X}/B)$ to $\NE(\mathcal{X}_{b'})$ for every $b' \in B$.  If $\alpha \in \partial \NE(\pi)$, then there exists a fiber-type contraction $\psi : \mathcal{X} \to Z$ over $B$ associated to the minimal face of $\NE(\pi)$ containing $\alpha$.  Since general fibers of $\text{ev} : \mathcal{M}' \to \mathcal{X}$ are irreducible, $\mathcal{M}'$ parameterizes curves of a unique class in the N\'eron-Severi group of any $\psi$-fiber.  Irreducibility of $\mathcal{M}' \cap \overline{\free}^{bir}(\mathcal{X}_b, \alpha)$ reduces to a statement about families of curves in the Fano fibers of $\psi : \mathcal{X} \to Z$.  As this is addressed by Theorem \ref{del pezzo curves thm}, we may assume $\alpha \notin \partial\NE(\mathcal{X}_b)$.  

By Example \ref{exmp: conics interior}, $-K_{\mathcal{X}_b} . \alpha \geq 3$ because otherwise $\text{ev} : \mathcal{M}' \to \mathcal{X}$ would have reducible fibers.  If $\mathcal{X}_b$ is not a blow-up of a del Pezzo threefold of degree one, irreducibility of $\mathcal{M}' \cap \overline{\free}^{bir}_1(\mathcal{X}_b, \alpha)$ follows immediately from the next claim.

\begin{claim}\label{claim: specializing claim 1}
    Suppose $\mathcal{X}_b$ is not a blow-up of a del Pezzo threefold of degree one.  There exists a connected curve $\mathcal{F}_b \subset \mathcal{M}'_b$ contained in the smooth locus of $\mathcal{M}'_b$ such that $\mathcal{F}_b$ meets every irreducible component of $\mathcal{M}' \cap \overline{\free}^{bir}_1(\mathcal{X}_b, \alpha)$.
\end{claim}

\begin{proof} 
Consider a general point $p \in \mathcal{X}_b$.  Let $L : T_p \mathcal{X}_b \to \mathbb{C}$ be a general linear function and let $Q \subset \mathcal{X}_b$ be a general complete intersection curve as in Corollary \ref{general curve and point}.  After a suitable base change of $\pi$, we may realize 
\begin{itemize}
    \item $p$ as the fiber of a general section $\sigma : B \to \mathcal{X}$;
    \item $L$ as the restriction of a general differential form $\mathcal{L} : \mathcal{T}_\mathcal{X}|_{\sigma(B)} \to \mathbb{C}$; and
    \item $Q$ as the fiber of a general complete intersection $\mathcal{Q} \subset \mathcal{X}$ flat over $B$.
\end{itemize}
To study $\mathcal{M}' \cap \overline{\free}^{bir}_1(\mathcal{X}_b, \alpha)$, we let $\mathcal{F}_{\sigma} \supseteq \mathcal{F}_{\sigma, \mathcal{L}} \supseteq \mathcal{F}_{\sigma, \mathcal{L}, \mathcal{Q}}$ denote the nested closures of the following loci of maps $f : (\mathbb{P}^1,q) \to \mathcal{X}$ in $\mathcal{M}'$:

\begin{itemize}
    \item $\mathcal{F}_\sigma = \text{cl}\big( \{ f : (\mathbb{P}^1,q) \to \mathcal{X} \text{ such that } f(q) \in \sigma(B) \}\big)$;
     \item $\mathcal{F}_{\sigma, \mathcal{L}} = \text{cl}\big( \{ f \in \mathcal{F}_{\sigma} \text{ such that } (\mathcal{L}_{b'} \circ d_q f) : T_q \mathbb{P}^1 \to \mathbb{C} \text{ is the zero map} \}\big)$;
     \item $\mathcal{F}_{\sigma, \mathcal{L}, \mathcal{Q}} = \text{cl}\big( \{ f \in \mathcal{F}_{\sigma, \mathcal{L}}   \text{ such that } f(\mathbb{P}^1) \text{  meets } \mathcal{Q} \} \big)$.
\end{itemize}
Let $\mathcal{F}$ denote one of these three varieties and suppose that, for a general $b' \in B$, the fiber $\mathcal{F}_{b'}$ of $\mathcal{F}$ over $b' \in B$ has dimension $t$.  Then every irreducible component of $\mathcal{M}' \cap \overline{\free}_1^{bir}(\mathcal{X}_b, \alpha)$ also contains a $t$-dimensional family of pointed maps that satisfy the corresponding conditions, as generality of $\sigma$, $\mathcal{L}$, and $\mathcal{Q}$ ensure the existence of such curves in each component only depends on the anticanonical degree $-K_{\mathcal{X}} . \alpha$. 

Suppose in addition that $t > 0$.  Proposition \ref{very free curves} shows $\mathcal{M}'$ generically parameterizes curves with balanced normal bundles in $\pi$-fibers.  Hence, generality of $\sigma$, $\mathcal{L}$, and $\mathcal{Q}$ ensures $\mathcal{X}$ is covered by curves in $\mathcal{F}$ if and only if $t \geq 2$.  Irreducibility of general fibers of $\text{ev} : \mathcal{M}' \to \mathcal{X}$ implies $\mathcal{F}$ itself is irreducible \cite[Exercise~III.11.3]{hartshorne2013algebraic}.  It follows that the fiber $\mathcal{F}_b$ is connected \cite[Exercise~III.11.4]{hartshorne2013algebraic}.  We now show that when $t = 1$, $\mathcal{F}_b$ is contained in the smooth locus of $\mathcal{M}_b'$.  This proves our claim, as $\mathcal{F}_b$ is connected and meets every component of $\mathcal{M}' \cap \overline{\free}_1^{bir}(\mathcal{X}_b, \alpha)$.

\textbf{Case 1:} Suppose $-K_{\mathcal{X}_b}. \alpha = 3$ and set $\mathcal{F} = \mathcal{F}_\sigma$.  By generality of $p \in \mathcal{X}_b$, $\mathcal{F}_b$ is a curve.  Any singular point of $\mathcal{F}_b$ must correspond to a map $f: (C,q) \rightarrow \mathcal{X}_b$ whose domain $C = C_1 \cup C_2$ is reducible.  We may assume $f(C_1)$ meets $p$ and is therefore free.  Hence, $-K_{\mathcal{X}_b} . f_* C_2 = 1$.  Since $\mathcal{X}_b$ is not a blow-up of a del Pezzo threefold of degree one, it follows from Lemma \ref{lem: normal bundle lines} that $h^1(C, \mathcal{N}_f) = 0$.

\textbf{Case 2:} Suppose $-K_{\mathcal{X}_b} . \alpha = 4$ and set $\mathcal{F} = \mathcal{F}_{\sigma, \mathcal{L}}$.  Let $f : (C, q) \to \mathcal{X}_b$ be a stable map contained in $\mathcal{F}_b$.  Since $p$ and $L$ are general, by Remark \ref{rem: general linear function} the marked point $q$ cannot lie on a component of $C$ mapping to a $(-K_{\mathcal{X}_b})$-conic.  If $q$ lies in a component $C_2 \subset C$ contracted by $f$, then $f(C_2) = p$, $C = C_1 \cup C_2 \cup C_3$ has three components, and $f|_{C_1}$, $f|_{C_3}$ are free conics.  Otherwise, $q$ lies in a component of $C$ which maps to a general free $(-K_{\mathcal{X}_b})$-cubic.  As in Case 1, Lemma \ref{lem: normal bundle lines} implies $h^1(C, \mathcal{N}_f) = 0$.

\textbf{Case 3:} Suppose $-K_{\mathcal{X}_b} . \alpha = 5$ and set $\mathcal{F} = \mathcal{F}_{\sigma, \mathcal{L}, \mathcal{Q}}$.  Corollary \ref{general curve and point} shows any stable map $f : (C, q) \to \mathcal{X}_b$ contained in $\mathcal{F}_b$ is a smooth point of $\overline{\free}^{bir}_1(\mathcal{X}_b, \alpha)$.
\end{proof}

This finishes our proof when $\mathcal{X}_b$ is not a blow-up of a del Pezzo threefold of degree one.  Suppose instead that $\phi : \mathcal{X}_b \to V_1$ is a such a blow-up.  In this case, there is also a del Pezzo fibration $\psi : \mathcal{X}_b \to \mathbb{P}^1$.  Our argument for Claim \ref{claim: specializing claim 1} breaks down because $\mathcal{X}_b$ may contain divisors that are covered by $(-K_{\mathcal{X}_b})$-lines $f : \mathbb{P}^1 \to \mathcal{X}_b$ with $h^1(\mathbb{P}^1, \mathcal{N}_f) > 1$.  However, the del Pezzo fibration $\psi$ contracts all such divisors.

Let $E \subset \mathcal{X}_b$ be the exceptional divisor of the blow-up $\phi$ and $D \subset \mathcal{X}_b$ be a general fiber of $\psi$.  Note that $(2D + E) . \alpha = -K_{\mathcal{X}_b} . \alpha$.  We may assume $E . \alpha > 0$; otherwise, $E . \alpha = 0$ and we may apply Lemma \ref{blowup} to reduce to the case $\mathcal{X}_b = V_1$.  Since $\alpha \notin \partial\NE(\mathcal{X}_b)$, we have $D . \alpha > 0$ as well.

We will show $\overline{\free}^{bir}(\mathcal{X}_b, \alpha)$ is irreducible.  The following lemma, Lemma \ref{lem: deg 1 del Pezzo curves}, proves irreducibility of $\overline{\free}^{bir}(\mathcal{X}_b, \alpha)$ when $D . \alpha = 1$.  Using this result, we briefly prove irreducibility of $\overline{\free}^{bir}(\mathcal{X}_b, \alpha_0)$ for the unique class $\alpha_0 \in \Nef_1(\mathcal{X}_b)_{\mathbb{Z}}$ such that $D . \alpha_0 > 1$ and $-K_{\mathcal{X}_b} . \alpha_0 \leq 5$.  This class is specified by the intersection pairings $D . \alpha_0 = 2$ and $-K_{\mathcal{X}_b} . \alpha_0 = 5$.

By Theorem \ref{MovableBB}, every component of $\overline{\free}^{bir}(\mathcal{X}_b, \alpha_0)$ contains a free chain $f : C_1 \cup C_2 \to \mathcal{X}_b$ such that $D . f_*[C_i] = 1$ and $-K_{\mathcal{X}_b} . f_*[C_i] = i + 1$ for each $i$.  Theorem \ref{thm: conics} and Lemma \ref{lem: deg 1 del Pezzo curves} prove $\overline{\free}(\mathcal{X}_b, f_*[C_i])$ is irreducible for each $i$.  Hence, by Lemma \ref{Gluing} it suffices to show general fibers of $\text{ev} : \overline{\free}_1(\mathcal{X}_b, f_*[C_2]) \to \mathcal{X}_b$ are irreducible.  This follows immediately from Theorem \ref{classification a-covers}, whose proof does not assume irreducibility of $\overline{\free}^{bir}(\mathcal{X}_b, \alpha_0)$, as it only uses facts about curves of $(-K_{\mathcal{X}_b})$-degree at most 4.
\end{proof}

\begin{lem}\label{lem: deg 1 del Pezzo curves}
    Let $X$ be a smooth Fano blow-up of a del Pezzo threefold of degree one.  Let $\psi : X \to \mathbb{P}^1$ be the unique del Pezzo fibration on $X$.  If $\alpha \in \Nef_1(X)$ is the class of a section of $\psi$ and $-K_X . \alpha \leq 5$, then $\overline{\free}(X, \alpha)$ is irreducible.
\end{lem}

\begin{proof}
Note that $\overline{\free}(X, \alpha) = \overline{\free}^{bir}(X, \alpha)$ since $\alpha$ is the class of a section of $\psi$.  To prove $\overline{\free}(X, \alpha)$ is irreducible, let $M \subset\overline{\free}(X, \alpha)$ be any component.  We will show $M$ must contain certain stable maps with reducible domains, and prove all such maps lie in the same component of $\overline{\free}(X, \alpha)$.  

To fix notation, let $\phi : X \to V_1$ be the blow-up map and $E$ be the exceptional divisor of $\phi$.  Let $D \subset X$ be a general fiber of $\psi$.  Denote by $e \in N_1(X)$ the class of a fiber of $\phi|_E$.  Write $\gamma \in N_1(X)$ for the class of a $(-K_{X})$-line contracted by $\psi : X \to \mathbb{P}^1$.  A unique component $N \subset \overline{M}_{0,0}(X, \gamma)$ parameterizes geometrically rational anticanonical curves in $\psi$-fibers \cite{Tihomirov_1982}.  We analyze nodal unions of free curves with lines parameterized by $N$.

\begin{claim}\label{claim: deg 1 del pezzo claim 1}
    Let $i_* : \NE(D) \to \NE(X)$ be the pushforward of cycles by inclusion into $X$.  Suppose $\beta \in \NE(D)$ be the class of a smooth, connected, rational curve in $D$ such that $i_* \beta = \alpha - \gamma - e$.  Then a single component of $\overline{\free}^{bir}(X, \alpha)$ contains the image of $\overline{\free}^{bir}_1(D, -K_D + \beta) \times_{X} \overline{M}_{0,1}(X, e)$ under the natural gluing map. 
\end{claim}

\begin{proof}
    By Theorem \ref{del pezzo curves thm}, $\overline{\free}^{bir}_1(D, -K_D + \beta)$ is irreducible. Moreover, the evaluation map $\text{ev} : \overline{M}_{0,1}(X, e) \to X$ induces an isomorphism between $ \overline{M}_{0,1}(X, e)$ and $E$.  We may therefore identify the fiber product  $\overline{\free}^{bir}_1(D, -K_D + \beta) \times_{X} \overline{M}_{0,1}(X, e)$ with the preimage $\text{ev}^{-1}(D \cap E)$ of $D \cap E$ under $\text{ev} : \overline{\free}^{bir}_1(D, -K_D + \beta) \to D$.

 Since $D \cap E$ is a smooth anticanonical curve in $D$, any curve parameterized by $\overline{\free}^{bir}(D, -K_D + \beta)$ meets $D \cap E$ at exactly $-K_{X} . \alpha - 1$ points, counted with multiplicity.  Thus, the forgetful map induces a finite covering $\text{ev}^{-1}(D \cap E) \to \overline{\free}^{bir}(D, -K_D + \beta)$.  In particular, as the image of a general map parameterized by $\overline{\free}^{bir}(D, -K_D + \beta)$ intersects $D \cap E$ only at general points in $D \cap E$, every irreducible component of $\text{ev}^{-1}(D \cap E)$ dominates $D \cap E$.

 Let $p \in D \cap E$ be a very general point and suppose $-K_{X} . \alpha \geq 4$.  In this case, $-K_D . \beta \geq 2$ and Theorem \ref{del pezzo curves thm} shows general fibers of $\text{ev} : \overline{\free}^{bir}_1(D, -K_D + \beta) \to D$ are irreducible.  Thus, the fiber $\text{ev}^{-1}(p)$ is connected in codimension one.  If $\text{ev}^{-1}(p)$ is irreducible, our claim is trivial.  Otherwise, let $f : (C,q) \to D$ be a general map parameterized by the intersection of any two components of $\text{ev}^{-1}(p)$.  Recall that $f$ corresponds to a stable map $g : C' \to X$ parameterized by $\overline{\free}^{bir}_1(D, -K_D + \beta) \times_{X} \overline{M}_{0,1}(X, e)$. Since $D \cap E$ is irrational and $D$ is a smooth surface, $p$ is not contained in any irreducible, non-free rational curve in $D$.  If $C_q \subset C$ is the component containing $q$, it follows that $f|_{C_q}$ is a free curve.  One may use this and connectivity of $\text{ev}^{-1}(p)$ in codimension one to see that $h^1(C', \mathcal{N}_g) = 0$.

 Suppose instead that $-K_{X} . \alpha = 3$.  In this case, $\beta$ is a $(-1)$-curve class.  The map $\phi_\beta : D \to D'$ contracting $\beta$ realizes $-K_D + \beta$ as the pullback $\phi_\beta^*(-K_{D'})$.  Moreover, the anticanonical linear series $|-K_{D'}|$ induces a double covering $\varphi : D' \to \mathbb{P}^2$ branched over a smooth quartic curve $Q \subset \mathbb{P}^2$.  In this setting, $D \cap E$ is the $(\varphi \circ \phi_\beta)$-preimage of a line $L \subset \mathbb{P}^2$ everywhere transverse to $Q$.  Furthermore, we have an isomorphism $\theta : \overline{\free}^{bir}(D, -K_D + \beta) \to Q$ identifying a point $q \in Q$ with the $(\varphi \circ \phi_\beta)$-preimage of its tangent line $T_q Q \subset \mathbb{P}^2$.
 
 There is a degree twelve map $h : Q \to L$ sending $q \in Q$ to the intersection point $L \cap T_q Q$.  Under the isomorphism $\theta$, $\text{ev}^{-1}(D \cap E)$ corresponds to the fiber product $Q \times_L (D \cap E)$; a point $(q, p) \in Q \times_L (D \cap E)$ corresponds to a stable map $f : C \to X$ which glues $(\varphi \circ \phi_\beta)^{-1}(T_q Q)$ to a curve of class $e$ along $p \in (\varphi \circ \phi_\beta)^{-1}(L \cap T_q Q)$.  Hence, $\text{ev}^{-1}(D \cap E)$ is a double cover of $Q$ branched over $h^{-1}(Q \cap L)$.  It suffices to show that the preimage of one of these branch points corresponds to a point of $\overline{\free}^{bir}(X, \alpha)$ lying in exactly one irreducible component of the space.
 
Let $q \in h^{-1}(Q \cap L) \subset Q$ and consider the stable map $f_q : C_q \to X$ corresponding to the reduced fiber of $\text{ev}^{-1}(D \cap E)$ over $q$.  Let $C_{q,1}, \ldots C_{q,n} \subset C_q$ be the irreducible components and $p_{q,1}, \ldots , p_{q,n-1} \in C_q$ be the nodes.  We may study the tangent space of $\overline{\free}^{bir}(X, \alpha)$ at $f_q$ using the natural exact sequence
$$0 \to f_q^*\mathcal{T}_X \to \oplus_i f_q^*\mathcal{T}_X|_{C_i} \to \oplus_j f_q^*\mathcal{T}_X|_{p_j} \to 0.$$
If $T_q Q$ is an ordinary tangent line or bitangent line to $Q$, i.e.\ $\text{ord}_q ( T_q Q \cap Q ) = 2$, then $f_q$ is a smooth point of $\overline{\free}^{bir}(X, \alpha)$ as $h^1(C_q, f_q^*T_X) = 0$.  If $T_q Q$ is an ordinary flex line to $Q$, i.e.\ $\text{ord}_q ( T_q Q \cap Q ) = 3$, then $h^1(C_q, f_q^*T_X) = 1$.  Lastly, the preimage of $E$ under the evaluation map $\text{ev}: \overline{M}_{0,1}(X, \gamma) \to X$ is a reduced curve.  Since $D$ is a general $\psi$-fiber, infinitesimal deformations of lines in $D$ are not tangent to $\phi$-exceptional curves in $E$.  This shows that if $T_q Q$ is a hyperflex line to $Q$, i.e.\ $\text{ord}_q ( T_q Q \cap Q ) = 4$, then $q \in Q \cap L$ and $h^1(C_q, f_q^*T_X) = 1$.  
%Alternatively, note that $\deg h^*(Q \cap L) = 48$.  By construction, if $q \in Q$ lies in the support of $h^*(Q \cap L)$, then $\text{ord}_q \big (h^*(Q \cap L)\big ) = \text{ord}_q \big ( T_q Q \cap Q\big )$ if $q \in Q \cap L$, while if $q \not\in Q \cap L$ then $\text{ord}_q  (h^*(Q \cap L) ) = \text{ord}_q ( T_q Q \cap Q ) - 1$.  Thus, $h^{-1}(Q \cap L)$ must contain a point $q$ such that $T_q Q$ is an ordinary flex line to $Q$.

The embedding dimension of $\overline{\free}^{bir}(X, \alpha)$ at a map $f_q$ corresponding to a flex line is therefore four.  If $f_q$ lay in the intersection of two components $M_1, M_2 \subset \overline{\free}^{bir}(X, \alpha)$, then there would exist a two-dimensional family of deformations of $f_q$ lying in $M_1 \cap M_2$.  Any such deformation would still be a singular point of $\overline{\free}^{bir}(X, \alpha)$.  However, since $Q$ has finitely many flex lines, the family $F$ of deformations of $f_q$ that do not smooth a node of $C_q$ has dimension one.  This family must be contained in $M_1 \cap M_2$, as $D$ is a general $\psi$-fiber.  But the images under deformations in $F$ of distinct irreducible, non-contracted components of $C_q$ sweep out distinct divisors in $X$.  Therefore, any two-parameter family of curves containing $F$ will sweep out $X$.  Such a family would contain a smooth point of $\overline{\free}^{bir}(X, \alpha)$.  This shows $f_q$ cannot lie in $M_1 \cap M_2$, a contradiction.
\end{proof}

\begin{claim}\label{claim: deg 1 del pezzo claim 2}
    Suppose $\overline{\free}^{bir}(X, \alpha - \gamma)$ is irreducible.  Then the smooth locus of $M$ contains a stable map $f : C_1 \cup C_2 \to X$ such that $f|_{C_1}$ is a map parameterized by $N$ and $f|_{C_2}$ is a general free curve of class $\alpha - \gamma$.
\end{claim}

\begin{proof}
    Consider a general point $p \in D$ of the general $\psi$-fiber $D \subset X$.  Let $F_p \subset M$ be the locus of curves that pass through $p$.  Since $\dim F_p = -K_{X} . \alpha - 2 \geq 1$, there exists another point $q \in D \setminus \{p\}$ that is contained in a curve parameterized by $F_p$.  But as $D . \alpha = 1$, any curve $C \subset X$ of class $\alpha$ meeting both $p$ and $q$ has a component $C_0 \subset C$ contained in $D$.
    
    If $-K_{X} . \alpha = 3$, $C_0$ must be a general free $(-K_{X})$-conic and $C \setminus C_0$ is an exceptional curve of the blow-up $\phi$.  By Theorem \ref{del pezzo curves thm}, the map $s|_{C_0} : C_0 \to D$ degenerates to a reducible curve.  This allows us to obtain a chain of type $(\gamma, \gamma , e)$ in $M$.  Smoothing the second node of this chain produces a curve $h : C_1 \cup C_2 \to X$ such that $h|_{C_1}$ is a $(-K_{X})$-line in $D$ and $h|_{C_2}$ is a general free curve of class $\alpha - \gamma$.  If $h|_{C_1}$ is a map parameterized by $N$, we may set $f = h$.  Otherwise, irreducibility of $\overline{\free}^{bir}(X, \alpha - \gamma)$ allows us to break $h|_{C_2}$ and obtain a chain of type $(\gamma, \gamma , e)$ where the middle component is a curve parameterized by $N$.  By Claim \ref{claim: deg 1 del pezzo claim 1}, we may smooth and re-break the first two curves in this chain to obtain a chain of type $(\gamma, \gamma, e)$ where instead the first component is a curve parameterized by $N$.  This allows us to obtain $f$.

    If instead $-K_{X} . \alpha = 5$, we may find a free chain in $M$ by Theorem \ref{MovableBB}.  Applying a similar argument to a component of this free chain proves our claim.  So, we may suppose $-K_{X} . \alpha = 4$.  If $M$ parameterizes very free curves, then we may suppose the pair $(p,q) \in D \times D$ is general.  Any curve of class $\alpha$ through $p$ and $q$ will have exactly two components, and our arugment proceeds as before.  If instead $M$ does not parameterize very free curves, Theorem \ref{thm: Mori's BB} shows there exists a free chain in $M$, and a similar argument applies once more. 
\end{proof}

Notice that, when $D . \alpha = 1$, $\alpha = e + n\gamma$ for some $n \in \{2,3,4\}$. By Theorem \ref{thm: conics}, $\overline{\free}(X, e + \gamma)$ is irreducible.  Thus, if $-K_{X} . \alpha = 3$ then $\overline{\free}^{bir}(X, \alpha - \gamma)$ is irreducible.  Claim \ref{claim: deg 1 del pezzo claim 3} therefore inductively verifies irreducibility of $\overline{\free}^{bir}(X, \alpha)$.

\begin{claim}\label{claim: deg 1 del pezzo claim 3}
    Suppose $D . \alpha = 1$ and $\overline{\free}^{bir}(X, \alpha - \gamma)$ is irreducible.  Then every stable map $f : C_1 \cup C_2 \to X$ as in Claim \ref{claim: deg 1 del pezzo claim 2} lies in the same irreducible component of $\overline{\free}^{bir}(X, \alpha)$.
\end{claim}
\begin{proof}
    Let $N' \subset \overline{M}_{0,1}(X,\gamma)$ be the one-pointed family over $N$.  Consider the natural fiber product $Z = N' \times_{X} \overline{\free}^{bir}_1(X, \alpha - \gamma)$ parameterizing chains of type $(\gamma, \alpha - \gamma)$. Suppose $Z_1, Z_2 \subset Z$ are any two components that dominate $\overline{\free}^{bir}(X, \alpha - \gamma)$ under the composition of the projection and forgetful maps.  It suffices to show $Z_1$ and $Z_2$ map into the same component of $\overline{\free}^{bir}(X, \alpha)$.

    By hypothesis, $\overline{\free}^{bir}(X, \alpha - \gamma)$ is irreducible.  A description of the image of $N'$ in $X$ \cite{Tihomirov_1982} shows the covering $\varphi: Z \to \overline{\free}^{bir}(X, \alpha - \gamma)$ is generically finite of degree $12(-K_{X} . \alpha - 2)$.  To describe this more precisely, consider a stable map $g : C_2' \cup C_2'' \to X$ parameterized by $\overline{\free}^{bir}(X, \alpha - \gamma)$ such that $g|_{C_2'}$ is a general smooth free curve of class $\alpha - \gamma - e$ contained in $D$. Each point of the fiber $\varphi^{-1}([g])$ corresponds to a chain of type $(\gamma, n\gamma, e)$ lying in the smooth locus of $\overline{\free}^{bir}(X, \alpha)$.  By Claim \ref{claim: deg 1 del pezzo claim 1}, all points in $\varphi^{-1}([g])$ lie in the same component of $\overline{\free}^{bir}(X, \alpha)$.  Since $Z_1, Z_2$ both contain a point in $\varphi^{-1}([g])$, this proves our claim.
\end{proof}

Therefore, by Claims \ref{claim: deg 1 del pezzo claim 2} and \ref{claim: deg 1 del pezzo claim 3}, $\overline{\free}^{bir}(X, \alpha)$ is irreducible.
\end{proof}

\section{Proof of Theorem \ref{Main Result}}\label{Section 7}
We outline our proof of Theorem \ref{Main Result} in Subsection \ref{subsection 7.1}.  The full proof of Theorem \ref{Main Result} occupies Sections 8-13.  In the remainder of Section \ref{Section 7}, we gather facts that will be used repeatedly in this proof.  Subsection \ref{subsection 7.2} exists to prove Corollary \ref{connected fibers fixall}, which controls the behavior of main components of chains of free curves.  Subsections \ref{subsection 7.3} and \ref{subsection 7.4} analyze families of curves on Fano blow-ups of $\mathbb{P}^1 \times \mathbb{P}^2$ and $\mathbb{P}^3$, respectively, as many Fano threefolds can be expressed as such a blow-up.

\subsection{Outline of Proof}\label{subsection 7.1}  We prove Theorem \ref{Main Result} by grouping Fano threefolds into the $105$ deformation types identified by \cite{mori1981classification} \cite{erratumMori}.  If a Fano threefold admits a smooth Fano blow-up, then so do all of its deformations.  Thus, by Lemma \ref{blowup} it suffices to consider only those deformation types of Fano threefolds which do not admit smooth Fano blow-ups.  

Let $X$ be a Fano threefold of one such deformation type.  
We will find a curve class $\tau \in \Nef_1(X)_{\mathbb{Z}}$ such that either $\overline{\free}(X, \alpha)$ or $\overline{\free}^{bir}(X,\alpha)$ is irreducible and nonempty for all $\alpha \in \tau + \Nef_1(X)_{\mathbb{Z}}$.  The proof of the existence of $\tau$ and of Theorem \ref{Main Result} for the given threefold is identical, and proceeds as follows:

\begin{enumerate}
    \item[(1)] Identify a weak core $\mathscr{C}_X$ of free curves on $X$ and the separating classes $\alpha \in \mathscr{C}_X$.  
\end{enumerate}
By Theorem \ref{MovableBB}, we may limit the anticanonical degree of $\alpha \in \mathscr{C}_X$ to five.  When $X$ does not have an $E5$ divisor, this limit may be reduced to four.  Thus, we list all low degree classes $\alpha \in \Nef_1(X)_\mathbb{Z}$, and calculate the number of irreducible components of $\overline{\free}(X, \alpha)$.  The following theorem summarizes this analysis.

\begin{thm}\label{thm: existence of core}
    Let $X$ be a smooth Fano threefold.  Suppose $X$ is not a product variety.  Then $X$ has a core of free curves if $X$ does not have one of the following deformation types:
    \begin{enumerate}
        \item $X$ has Picard rank 1, index 1, and genus at most 6;
        \item $X$ has Picard rank 2 and an elementary $D1$ del Pezzo fibration;
        \item $X$ has Picard rank 3 and deformation type $3.1$, $3.2$, $3.4$, or $3.5$.
    \end{enumerate}
    If in addition $X$ is general in moduli, then $X$ has a core of free curves if $-K_X$ is very ample and $X$ does not admit a del Pezzo fibration with degree $3$ fibers.

    Furthermore, suppose $X$ has one of the deformation types listed in (1)-(3).  If $\alpha \in \Nef_1(X)_\mathbb{Z} \setminus \partial\NE(X)$, then $\free^{bir}(X, \alpha)$ is irreducible if $-K_X . \alpha \in \{3,4\}$ or if $\rho(X) > 1$ and $-K_X . \alpha = 2$.  
\end{thm}

\begin{proof}
    See Sections 8 through 13.  See also Theorem \ref{thm: conics}, Theorem \ref{classification a-covers}, Corollary \ref{cor: classification a-covers}, and Proposition \ref{low degree curves degeneration}.
\end{proof}

We also determine which $\alpha \in \mathscr{C}_X$ are separating classes.  For $\alpha \in \mathscr{C}_X \setminus \partial\NE(X)$, Theorem \ref{classification a-covers} (which extends results in \cite{beheshti2020moduli}) proves $\alpha$ is a separating class if and only if either $-K_X . \alpha = 2$ or $\free^{bir}(X, \alpha)$ is reducible.  In most cases, however, there are simpler geometric arguments.  We include these to cross-verify results from \cite{beheshti2020moduli} and to determine which $\alpha \in \mathscr{C}_X \cap \partial\NE(X)$ are separating classes.

\begin{enumerate}
    \item[(2)] Find relations amongst generators of $\mathbb{N}\mathscr{C}_X$ and verify the hypotheses of Theorem \ref{Main Method}(1).
\end{enumerate}
To find relations, we apply Lemma \ref{Relations}.  We begin with a divisor class $D \in N^1(X)$ that pairs negatively with one class $c_1 \in \mathscr{C}_X$ and nonnegatively with all other $\alpha \in \mathscr{C}_X\setminus \{c_1\}$.  This allows us to identify those relations in a generating set which involve $c_1$.  We then pick a new divisor class $D$ such that $D. c_2 < 0$ for a unique $c_2 \in \mathscr{C}_X\setminus \{c_1\}$, and repeat the procedure.  After finding all relations in a generating set, we verify the hypotheses of Theorem \ref{Main Method}(1) by considering two chains of curves corresponding to each side of the relation.  We verify these chains of curves lie in the same component of $\overline{\free}(X)$ by deforming one to the other without leaving the smooth locus of $\overline{\free}(X)$.

\begin{enumerate}
    \item[(3)] Show for each $\alpha \in \Nef_1(X)_\mathbb{Z}$ such that $-K_X . \alpha \geq 3$ and $\alpha \not\in \partial\NE(X)$, $\free^{bir}(X, \alpha)$ is nonempty, irreducible, and Manin. %
\end{enumerate}  
As we will show in Corollary \ref{connected fibers fixall}, Theorem \ref{Main Method}(2) or the corresponding condition required by Corollary \ref{bir Main Method} is always satisfied outside the relative cones of del Pezzo fibrations.  Thus, (2) implies $\overline{\free}^{bir}(X,\alpha)$ is either irreducible or empty for such $\alpha$.  If it is nonempty, Theorem \ref{classification a-covers} proves $\free^{bir}(X, \alpha)$ is Manin because $-K_X . \alpha \geq 3$.  To prove $\free^{bir}(X, \alpha)$ is nonempty, we use the following well-known result.

\begin{lem}[Gordan's Lemma]\label{Gordan's Lemma}
Let $\mathcal{C}\subset \mathbb{Q}^n$ be a polyhedral cone with extreme rays $\{v_1, \ldots , v_m \}\subset \mathbb{Z}^n$.  Let $K = \{ \sum_i r_i v_i | 0 \leq r_i \leq 1\}$.  The semigroup (monoid) of integer points in $\mathcal{C}$ is generated by the finite set $S = \mathbb{Z}^n \cap K$.
\end{lem}

This result allows us to check that we have found all generators of $\Nef_1(X)_{\mathbb{Z}}$ as a monoid.  From this it follows that $\free^{bir}(X,\alpha)$ is nonempty for all $\alpha\in \Nef_1(X)_\mathbb{Z}$ aside from anticanonical line classes and multiples of classes of conic-bundle fibers.  

Having shown (3), Theorem \ref{classification a-covers} implies that the varieties $\free^{bir}(X, \alpha)$, for all $\alpha \in \Nef_1(X)_\mathbb{Z}\setminus \partial\NE(X)$ of degree $-K_X . \alpha \geq 3$, form an exhaustive list of Manin components on $X$.  That $\free^{bir}(X, \alpha)$ generically parameterizes embeddings when $-K_X . \alpha = 3$ and $\alpha \in \Nef_1(X)_\mathbb{Z}\setminus \partial\NE(X)$ is deduced from the casework proving Theorem \ref{thm: existence of core}.  The remaining claims in Theorem \ref{Main Result} follow from Corollary \ref{cor: classification a-covers}, Proposition \ref{very free curves}, and well-known facts from deformation theory.

\subsection{Reducible General Fiber}\label{subsection 7.2}
Throughout this subsection, let $X$ be a smooth Fano threefold which is not a product variety.  By Theorem \ref{MovableBB}, there exists a weak core $\mathscr{C}_X$ on $X$. Lemma \ref{Gluing} guarantees $\prod_X \overline{\free}_2(X,\alpha_i)$ has a unique main component if for each $\alpha_i$ the evaluation map $\text{ev} : \overline{\free}_1(X,\alpha_i)\to X$ has irreducible fibers over general points in $X$.  In other words, if $\mathscr{C}_X$ does not contain a separating class, then condition \ref{Main Method}(2) is always satisfied.  On the other hand, if general fibers of $\text{ev} : \overline{\free}_1(X,\alpha_i)\to X$ are reducible for some $i$, condition \ref{Main Method}(2) will fail for some $\beta \in \Nef_1(X)$.

\begin{exmp}
Let $X\rightarrow \mathbb{P}^3$ be the blow-up of a genus 3 curve $C \subset \mathbb{P}^3$ of degree 6.  There exists a core $\mathscr{C}_X$ on $X$ that contains the class $\alpha \in \Nef_1(X)$ of the strict transform of a bisecant line to $C$.  Since the image of $C$ under projection from a general point in $\mathbb{P}^3$ has 7 nodes, $\alpha$ is a separating class. There are two components of $\overline{\free}(X, 2\alpha)$: a component of double covers and a component of very free curves.
\end{exmp}

This subsection (Corollary \ref{connected fibers fixall}) proves condition \ref{Main Method}(2) or its analogue in Corollary \ref{bir Main Method} always holds for $\beta \in \Nef_1(X) \setminus \partial\NE(X)$, where $\partial\NE(X)$ is the boundary of the Mori cone $\NE(X)$.  We will use the following result from \cite{beheshti2020moduli}.

\begin{thm}[Theorem 1.3, \cite{beheshti2020moduli}]\label{reducible fibers: 4 author result}
Let $X$ be a smooth Fano threefold with very ample anticanonical bundle.  Let $M$ be a component of $\overline{\free}_1(X)$.  Suppose the evaluation map $\text{ev}: M \rightarrow X$ has generically reducible fibers.  Then either:
\begin{itemize}
    \item $M$ parameterizes stable maps whose reduced images are $(-K_X)$-conics, or
    \item $M$ parameterizes curves contracted by a del Pezzo fibration $\pi: X \rightarrow Z$.
\end{itemize}
\end{thm}

Theorem \ref{classification a-covers} and Corollary \ref{cor: classification a-covers} directly extend Theorem \ref{reducible fibers: 4 author result} to all Fano threefolds which are not product varieties.  We will use this extension and Theorem \ref{thm: existence of core} throughout the remainder of this subsection.  The following is a corollary of Theorem \ref{thm: existence of core}.

\begin{cor}\label{connected fibers fixall}
Let $X$ be a smooth Fano threefold with a weak core $\mathscr{C}_X$ of free curves.  Suppose $X$ is not a product variety.  For any expression $\beta = \sum \alpha_i$ with $\alpha_i \in \mathscr{C}_X$ and $\beta \notin \partial\NE(X)$, each main component of $\prod_X \overline{\free}_2(X,\alpha_i)$ satisfying $(\dagger)$ maps to the same component of $\overline{\free}_2(X,\beta)$. %
\end{cor}

We will prove Corollary \ref{connected fibers fixall} from results in this section.  First, in Remark \ref{rem: classify separating components} and Lemma \ref{interior conic lemma} we identify which classes in $\mathscr{C}_X$ may be separating classes. 

\begin{rem}\label{rem: classify separating components}
Suppose $X$ is a smooth Fano threefold which is not a product variety.  Let $M \subset \free_1^{bir}(X, \alpha)$ be a component.  Suppose the evaluation map $\text{ev} : M \to X$ has reducible fibers over general points in $X$.  If $\alpha$ lies in the interior of $\NE(X)$, then by Theorem \ref{classification a-covers} $-K_X . \alpha = 2$.  If $\alpha$ lies in the boundary of $\NE(X)$, there must exist a del Pezzo fibration $\pi : X \to \mathbb{P}^1$ such that the relative Mori cone $\NE(\pi)$ contains $\alpha$ in its interior.  Otherwise, there would exist a conic fibration $\psi : X \to S$ such that $\alpha \in \NE(\psi)$.  This would force $\alpha$ to be the class of a general fiber of $\psi$.  Consequently, the map $\overline{\free}_1(X, \alpha) \to X$ would be an isomorphism, contradicting our hypothesis on $M$.  This argument moreover shows that if $\free(X,\alpha)$ is reducible and $\alpha \in \partial\NE(X)$, then $\alpha$ lies in the interior of a facet of $\NE(X)$.
\end{rem}

\begin{lem}\label{interior conic lemma}
Let $X$ be a smooth Fano threefold.  If there exists an integral nef $(-K_X)$-conic $\alpha$ interior to $\NE(X)$, then $\rho(X) \leq 2$ and $\alpha$ is unique.  Moreover, $\free^{bir}(X, 2\alpha)$ is irreducible. 
\end{lem}

\begin{proof}
Sections \ref{Section 8} through \ref{Section 13} prove our claims through case-by-case analysis.  
\end{proof}

In fact, when $\rho(X) = 2$, such a conic $\alpha$ exists if and only if there are no Fano blow-ups $X' \rightarrow X$. 
The uniqueness of $\alpha$ in Lemma \ref{interior conic lemma} allows us to prove the following result.  This will be a key step in our proof of Corollary \ref{connected fibers fixall}.

\begin{prop}\label{unique main component}
Let $X$ be a smooth Fano threefold which is not a product variety.  Let $M \subset \overline{\free}^{bir}_1(X,\alpha)$ and $N \subset \overline{\free}^{bir}_1(X,\gamma)$ be components.  If $\alpha + \gamma \notin \partial\NE(X)$, then all main components of $M\times_X N$ satisfying $(\dagger)$ map to the same component $P \subset \overline{\free}(X,\alpha + \gamma)$, and $P$ generically parameterizes very free curves.
\end{prop}

\noindent This statement is nontrivial and fails when $X$ is a product variety and $\rho(X) > 8$.

\begin{exmp}
Let $X \cong \mathbb{P}^1 \times S_2$ be the product of $\mathbb{P}^1$ with a degree 2 del Pezzo surface $S_2$.  Consider the component $N \subset \overline{\free}_1(X)$ parameterizing sections of $X \rightarrow \mathbb{P}^1$ which project to geometrically rational anticanonical curves under $X \rightarrow S_2$.  If $M \subset \overline{\free}_1(X)$ parameterizes geometrically rational anticanonical curves in fibers of the projection to $\mathbb{P}^1$, there are two main comoponents of $M \times_X N$.  One lies in an accumulating component of curves, while the other lies in a component of very free curves.
\end{exmp}

\begin{proof}[Proof of Proposition \ref{unique main component}]
Theorem \ref{classification a-covers} and Proposition \ref{very free curves} imply every main component of $M \times_X N$ lies in a component of very free curves.  By Lemma \ref{Gluing} we may assume the evaluation maps for both $M$ and $N$ have reducible fibers over general points in $X$.   If both $\alpha$ and $\gamma$ are nef conic classes interior to $\NE(X)$, then by Lemma \ref{interior conic lemma} both $\alpha = \gamma$ and $\overline{\free}^{bir}(X, \alpha + \gamma)$ is irreducible.  Hence, Theorem \ref{classification a-covers} allows us to assume $\alpha$ lies in the relative Mori cone $\NE(\pi)$ of a del Pezzo fibration $\pi : X \rightarrow \mathbb{P}^1$.  But together with Lemma \ref{gluing del Pezzo}, Theorems \ref{classification a-covers} and \ref{del pezzo curves thm} allow us to further assume $-K_X . \alpha = -K_X . \gamma = 2$.  When $|-K_X|$ is very ample, any del Pezzo fibration on $X$ has fibers of degree at least three.  Hence, Theorem \ref{del pezzo curves thm} and Lemma \ref{gluing del Pezzo} imply there is only one main component of $M \times_X N$. When $|-K_X|$ is not very ample, $\free^{bir}(X, \alpha + \gamma)$ is irreducible by Theorem \ref{thm: existence of core}. %
\end{proof}

Lastly, to prove Corollary \ref{connected fibers fixall} we will need the following lemma about breaking cubic curves in del Pezzo surfaces.

\begin{lem}\label{lem: breaking curves del Pezzo fibration}
    Let $S$ be a smooth del Pezzo surface.  Suppose $\pi : S \to \mathbb{P}^1$ is a conic fibration and $\alpha \in \Nef_1(S)_{\mathbb{Z}}$ satisfies $\pi_* \alpha \neq 0$ and $-K_S . \alpha = 3$. Then $\overline{\free}(S,\alpha)$ is irreducible and its smooth locus contains a stable map $f : C_1 \cup C_2 \to S$ such that $f|_{C_1}$ is a free $(-K_S)$-conic and $(\pi \circ f)_*[C_2] \neq 0$.
\end{lem}
\begin{proof}
    Observe that $S$ is a nontrivial blow-up of $\mathbb{P}^2$ because it admits a conic bundle structure and $-K_S . \alpha = 3$.  Theorem \ref{del pezzo curves thm} proves $\overline{\free}(S,\alpha)$ is irreducible.  To see that its smooth locus contains the desired map, note that any $(-K_S)$-line $\gamma \in \NE(S)_{\mathbb{Z}}$ is the class of a rational curve $h : C_2 \to S$ such that $h^1(C_2, \mathcal{N}_h) \leq 1$.  Similarly, for any $(-K_S)$-conic $\beta \in \Nef_1(S)_{\mathbb{Z}}$, there exists a free rational curve $g : C_1 \to S$ of class $\beta$ such that $g(C_1)$ is transverse to $h(C_2)$.  If $\alpha = \beta + \gamma$, then $\beta . \gamma > 0$ because $\alpha$ is nef and $\gamma$ is either self-negative or ample.  In particular, the corresponding stable map $f : C_1 \cup C_2 \to S$ would be a smooth point of $\overline{\free}(S,\alpha)$.  Therefore, we need only to find classes $\beta, \gamma$ as above such that $\alpha = \beta + \gamma$ and $\pi_* \gamma \neq 0$.  To show this, we will use the descriptions of curves of class $\alpha$ obtained in the proof of Lemma \ref{testa theorem}.  These descriptions depend only on $g = \frac{\alpha^2 - 1}{2}$.
    
    If $g = 0$, then by the proof of Lemma \ref{testa theorem} there exists a birational map $\phi : S \to \mathbb{P}^2$ which realizes curves of class $\alpha$ as the total transforms of lines in $\mathbb{P}^2$.  Notice that $\pi$ cannot factor through $\phi$.  Hence, for some $\phi$-exceptional $(-1)$-curve $\gamma$, $\pi_* \gamma \neq 0$.  The class $\beta = \alpha - \gamma$ is nef, proving our claim.

    If $g = 1$, then there exists a contraction $\phi : S \to B$ such that $B$ is a del Pezzo surface, $(-K_{B})^2 = 3$, and $\mathcal{O}_S(\alpha) = \phi^*(-K_{B})$.  For some $(-1)$-curve $C_2 \subset B$, $\gamma = \phi^{*}[C_2]$ satisfies $\pi_* \gamma \neq 0$.  Once again, the class   $\beta = \alpha - \gamma$ is nef.

    Lastly, if $g \geq 2$ the proof of Lemma \ref{testa theorem} shows $(-K_S)^2 = 1$ and $\alpha + K_S$ is nef.  In this case we may set $\gamma = -K_S$ and $\beta = \alpha + K_S$.  This finishes our proof.
\end{proof}

We finish the proof of Corollary \ref{connected fibers fixall} below.

\begin{proof}[Proof of Corollary \ref{connected fibers fixall}]
We will prove Corollary \ref{connected fibers fixall} through a series of claims.  Our first claim reduces the this proof to the verification of a simpler statement. 

\begin{claim}\label{claim: connected 1}
    Corollary \ref{connected fibers fixall} is equivalent to the following statement: for some weak core $\mathscr{C}_X$ of free curves on $X$, any pair of classes $\alpha_1, \alpha_2 \in \mathscr{C}_X$ such that $\alpha_1 + \alpha_2 \notin\partial \NE(X)$, and any choice of components $M_1,M_1' \subset \overline{\free}_2(X,\alpha_1)$ and $M_2 \subset \overline{\free}_2(X,\alpha_2)$, there exists a component of $\overline{\free}_2(X,\alpha_1 + \alpha_2)$ that contains the image of a main component of both $M_1 \times_X M_2$ and $M_1' \times_X M_2$.
\end{claim}

\begin{proof}
Corollary \ref{connected fibers fixall} directly implies the statement above.  To prove the converse, first, suppose $\mathscr{C}_X$ and $\mathscr{C}_X'$ are two different weak cores of free curves on $X$.  By Definition \ref{def core} and Lemma \ref{Basic Properties Chains}(2), every main component of chains of free curves in $X$ will contain a chain whose components have class in $\mathscr{C}_X \cap \mathscr{C}_X'$.  Therefore, Corollary \ref{connected fibers fixall} does not depend on the choice of $\mathscr{C}_X$. 

Let $\beta = \sum_{i = 1}^r \alpha_i$ with $\alpha_i \in \mathscr{C}_X$ and $\beta \notin \partial\NE(X)$.  By Lemma \ref{Basic Properties Chains}(2), every component of $\overline{\free}(X)$ contains a free chain $f : C_1 \cup \ldots \cup C_s \to X$ whose restriction to each irreducible component $C_i$ is a map parameterized by $\free^{bir}(X, f_*[C_i])$.  Therefore, it suffices to show each main component of $\prod_X \overline{\free}^{bir}_2(X,\alpha_i)$ satisfying $(\dagger)$ maps to the same component of $\overline{\free}(X,\beta)$.  

Let $S$ be the set of such main components.  Consider some $M \in S$.  If for each $1 \leq i \leq r$ the map $\text{ev}_i : \free^{bir}_1(X, \alpha_i) \to X$ has irreducible fibers over general points in $X$, then $S =\{M\}$ is a singleton by Lemma \ref{Gluing}.  

Suppose instead that for some $i \in \{1, \ldots , r\}$, $\text{ev}_i$ has reducible fibers over general points in $X$.  For any such $i$, Theorem \ref{thm: existence of core}, Remark \ref{rem: classify separating components}, and the hypothesis $\beta \notin \partial \NE(X)$ imply there exists $j \neq i$ such that $\alpha_i + \alpha_j \notin \partial\NE(X)$.  By Lemma \ref{Basic Properties Chains}(3) the component of $\overline{\free}_2(X,\beta)$ that contains the image of $M$ also parameterizes a free chain whose irreducible components are any reordering of the irreducible components of a free chain in $M$.  Thus, we may reindex the $\alpha_1, \ldots , \alpha_r$ so that $i = 1$ and $j = 2$.  Hence for any components $M_i \subset \overline{\free}_2(X,\alpha_i)$, Proposition \ref{unique main component} proves by induction on $r$ that every main component of $\prod_X M_i$ satisfying $(\dagger)$ lies in the same component of $\overline{\free}_2(X,\beta)$.  While remaining in the same component of $\overline{\free}_2(X,\beta)$, we may generalize a free chain parameterized by $\prod_X M_i$ to a free chain $f : C \to X$ parameterized by $\overline{\free}_2(X, \alpha_1 + \alpha_2) \times_X M_3 \times_X \ldots \times_X M_r$.  
Assuming the aforementioned claim, we may then specialize $f$ to a free chain parameterized by $M_1' \times_X M_2 \times_X \ldots \times_X M_r$.  Iterating this process proves our desired equivalence.
\end{proof}

Let $\mathscr{C}_X$ be a weak core of free curves that does not contain any freely breakable classes.  If $\mathscr{C}_X$ is a core of free curves, the statement in Claim \ref{claim: connected 1} is trivial.  Hence, we may assume $X$ has one of the deformation types appearing in Theorem \ref{thm: existence of core}.  

Since $X$ does not admit an $E5$ contraction, by Theorem \ref{MovableBB} $-K_X . \alpha \leq 4$ for all $\alpha \in \mathscr{C}_X$.  Furthermore, if $\pi : X \to B$ is a Mori fibration and $\alpha \in \NE(\pi)$, $-K_X . \alpha \leq 3$ because free curves of anticanonical degree $\geq 4$ on Fano varieties of dimension $\leq 2$ are freely breakable.  Lastly, if $\alpha \in \Nef_1(X)_\mathbb{Z} \setminus \partial \NE(X)$ is a $(-K_X)$-conic, then any component of $\overline{\free}(X, 2\alpha) \setminus \overline{\free}^{bir}(X, 2\alpha)$ is freely breakable, and $\overline{\free}^{bir}(X,2\alpha)$ will be irreducible by Theorem \ref{thm: existence of core}.  Therefore any component of $\overline{\free}(X,2\alpha)$ contains free chains of type $(\alpha, \alpha)$ and $2\alpha \notin \mathscr{C}_X$.

Let $\alpha_1, \alpha_2 \in \mathscr{C}_X$ be any pair such that $\alpha_1 + \alpha_2 \notin \partial\NE(X)$ and $\overline{\free}_2(X,\alpha_1)$ is reducible.  Let $M_1,M_1'$ be two of its components and consider the fiber products $M_1 \times_X M_2$ and $M_1' \times_X M_2$, where $M_2$ is a component of $\overline{\free}_1(X, \alpha_2)$. If $-K_X . (\alpha_1 + \alpha_2) = 4$, Theorem \ref{thm: existence of core} states that $\overline{\free}^{bir}(X,\alpha_1 + \alpha_2)$ is irreducible.  The statement in Claim \ref{claim: connected 1} is again trivial in this case, so we may suppose $-K_X . (\alpha_1 + \alpha_2) > 4$.  We analyze the remaining cases by the Picard rank of $X$.

\begin{claim}
    If $\rho(X) = 1$, then there exist main components of $M_1 \times_X M_2$ and $M_1' \times_X M_2$ that map to the same component of $\overline{\free}(X, \alpha_1 + \alpha_2)$.
\end{claim}

\begin{proof}
Let $l \in \NE(X)$ be the class of a $(-K_X)$-line.  By Theorems \ref{MovableBB} and \ref{thm: existence of core}, we may assume $\mathscr{C}_X = \{2l, 3l\}$, and that $\free(X, 3l)$ is irreducible.  As $-K_X . (\alpha_1 + \alpha_2) > 4$, necessarily $\alpha_1 = 2l$ and $\alpha_2 = 3l$.  Consider a general map $f : C \rightarrow X$ parameterized by a component of $M_1 \times_X \overline{M}_{0,2}(X, l) \times_X M_1'$.  By Lemma \ref{lem: normal bundle lines}, $h^1(C, \mathcal{N}_f) = 0$, so that $[f] \in \overline{M}_{0,0}(X, 5l)$ is a smooth point.  However, the map $f$ is a chain of type $(2l, l ,2l)$ wherein each irreducible component deforms with the expected dimension \cite{beheshti2020moduli}.  Hence, $f$ generalizes to a free chain of type $(2l, 3l)$ or $(3l, 2l)$ parameterized by a main component of $M_1 \times_X \overline{\free}_1(X, 3l)$ or $\overline{\free}_1(X, 3l) \times_X M_1'$, respectively.  This shows both main components lie in the same component of $\overline{\free}(X, 5l)$.
\end{proof}

Suppose $\rho(X) \geq 2$ instead.  By Theorem \ref{thm: existence of core}, $\overline{\free}(X, \alpha)$ is irreducible for every $\alpha \in \mathscr{C}_X \setminus \partial \NE(X)$.  Thus, by Theorem \ref{classification a-covers} and Remark \ref{rem: classify separating components} there exists a del Pezzo fibration $\pi : X \rightarrow \mathbb{P}^1$ such that $\alpha_1$ lies in the interior of $\NE(\pi)$.  Since $\NE(\pi) \subset \partial \NE(X)$, $\alpha_2 \notin \NE(\pi)$.  At least one of $\alpha_1, \alpha_2$ has anticanonical degree at least 3.  Since $\alpha_1 \in \NE(\pi)$, we may assume $-K_X . \alpha_1 \leq 3$.

\begin{claim}
    Suppose $\rho(X) \geq 2$ and $-K_X . \alpha_2 \geq 3$. Then $M_2$ contains a map $f: C_1 \cup C_2 \rightarrow X$ such that $\ell : = f_*[C_1]$ satisfies $\pi_* \ell \neq 0$, $\alpha_2' : = f_* [C_2] \in \mathscr{C}_X$, and $h^1(C_1 \cup C_2, \mathcal{N}_f) = 0$.  Moreover, $f|_{C_2}$ is a free curve, and if $f|_{C_1}$ is not a free curve then the divisor $D$ swept out by its deformations satisfies $D . \alpha_1 > 0$.
\end{claim}

\begin{proof}
    Our hypothesis on $\mathscr{C}_X$ implies $\alpha_2$ is not contracted by conic bundle.  Therefore if $\alpha_2 \in \partial\NE(X)$, there must exist another del Pezzo fibration $\psi : X \to \mathbb{P}^1$ such that $\alpha_2 \in \NE(\psi)$.  In this case, the map $\pi \times \psi : X \to \mathbb{P}^1 \times \mathbb{P}^1$ is a conic fibration.  In particular, $\rho(X) = 3$ and the restriction of $\pi$ to a fiber of $\psi$ is a conic bundle.  Hence, Lemma \ref{lem: breaking curves del Pezzo fibration} proves $M_2$ contains such a map $f: C_1 \cup C_2 \rightarrow X$, though it does not verify $D . \alpha_1 > 0$.  Since $\alpha_1$ lies in the interior of $\NE(\pi)$, if $D . \alpha_1 = 0$ then $D$ cannot be nef and $\pi$ factors through a contraction $\phi_2 : X \to Y_2$ of $D$.  But as $\phi_2$ does not contract $f(C_1)$ and $-K_X . \ell = 1$, Lemma \ref{nonextreme lines} implies $D. \ell = -1$ and $D$ is a section of $\pi \times \psi$.  Therefore, $\psi$ factors through a different contraction $\phi_1 : X \to Y_1$ of $D$.  This description of the Mori structure of $X$ shows it has deformation type 3.2. One can apply Lemma \ref{lem: breaking curves del Pezzo fibration} to a fiber of the induced map $Y_1 \to \mathbb{P}^1$ instead.

    Similarly, if $\alpha_2 \notin \partial\NE(X)$, then $\overline{\free}(X, \alpha_2)$ is irreducible by Theorem \ref{thm: existence of core}.  Thus, $M_2 = \overline{\free}(X, \alpha_2)$ contains any reducible curve that generalizes to a free curve.  For each deformation type of $X$ in question, the Mori cone $\NE(X)$ is simplicial and only has extreme rays of length one.  Hence, Lemmas \ref{interior conic lemma} and \ref{Gordan's Lemma} show the monoid of integer points in $\NE(X)$ is generated by $(-K_X)$-lines, and Lemma \ref{nonextreme lines} identifies these classes.  The analysis of families of lines and conics in Lemma \ref{lem: E5 lines}, Proposition \ref{properties nonfree conics}, and Lemma \ref{lem: normal bundle lines} proves we may find such a map $f : C_1 \cup C_2 \to X$.  
\end{proof}

The condition $D . \alpha_1 > 0$ allows us to form stable maps $g : C_0 \cup C_1 \cup C_2 \to X$ where $g|_{C_0}$ is a free curve of class $\alpha_1$ and $g|_{C_1 \cup C_2}$ is a general deformation of $f$.  This is essential to the proof below.

\begin{claim}\label{claim: connected 2}
    If $\rho(X) \geq 2$ and $-K_X . \alpha_2 \geq 3$, then there exist main components of $M_1 \times_X M_2$ and $M_1' \times_X M_2$ that map to the same component of $\overline{\free}(X, \alpha_1 + \alpha_2)$.
\end{claim}

\begin{proof}
    Let $M_2'\subset \overline{\free}_1(X, \alpha_2')$ be the component containing $[f|_{C_2}]$.  Both $M_1 \times_X M_2$ and $M_1' \times_X M_2$ must lie in a component of $\overline{\free}(X, \alpha_1 + \alpha_2)$ which contains chains of type $(\alpha_1, \ell, \alpha_2)$.  General chains of this type will be smooth points of the corresponding moduli space by Lemma \ref{lem: normal bundle lines}.  Therefore, $M_1 \times_X M_2$ and $M_1' \times_X M_2$ must lie in components of $\overline{\free}(X, \alpha_1 + \alpha_2)$ that contain free chains of type $(\alpha_1 + \ell, \alpha_2')$.  These are parameterized by a main component of $\overline{\free}^{bir}_1(X, \alpha_1 + \ell) \times_X M_2'$.  As $\alpha_1 + \ell$ lies in the interior of $\NE(X)$ and has $(-K_X)$-degree at most 4, $\overline{\free}^{bir}(X, \alpha_1 + \ell)$ is irreducible by Theorem \ref{thm: existence of core}. This proves our claim.
\end{proof}

To analyze the remaining case $-K_X . \alpha_2 = 2$ and $-K_X . \alpha_1 = 3$, we first deform $X$ within a family of Fano varieties.  Let $\widetilde{X}$ be a general deformation of $X$ as a smooth Fano variety.  We may consider $X$ and $\widetilde{X}$ as fibers of some family $\psi : \mathcal{X} \to B$ of smooth Fano threefolds.  By \cite[Theorem~4.2]{Fernex2009RigidityPO}, after a suitable base change there are natural isomorphisms $N_1(X) \cong N_1(\mathcal{X}/B) \cong N_1(\widetilde{X})$.  We use these below.

\begin{claim}\label{claim: generalizing X}
    Suppose $-K_X . \alpha_2 = 2$ and $-K_X . \alpha_1 = 3$.  Assume that, for any choice of components $\widetilde{M}_1,\widetilde{M}_1' \subset \overline{\free}_2(\widetilde{X},\alpha_1)$ and $\widetilde{M}_2 \subset \overline{\free}_2(\widetilde{X},\alpha_2)$, there exists a component of $\overline{\free}_2(\widetilde{X},\alpha_1 + \alpha_2)$ that contains the image of a main component of both $\widetilde{M}_1 \times_{\widetilde{X}} \widetilde{M}_2$ and $\widetilde{M}_1' \times_{\widetilde{X}} \widetilde{M}_2$.  Then there exist main components of $M_1 \times_X M_2$ and $M_1' \times_X M_2$ that map to the same component of $\overline{\free}(X, \alpha_1 + \alpha_2)$.
\end{claim}
\begin{proof}
    There exist components $\mathcal{M}_1,\mathcal{M}_1' \subset \overline{\free}_2(\mathcal{X},\alpha_1)$ and $\mathcal{M}_2 \subset \overline{\free}_2(\mathcal{X},\alpha_2)$ whose fibers over $\psi(X) \in B$ contain $M_1, M_1'$, and $M_2$ respectively.  Since $\widetilde{X}$ is a general deformation of $X$, the fibers of these three components over $\psi(\widetilde{X})$ contain components $\widetilde{M}_1,\widetilde{M}_1' \subset \overline{\free}_2(\widetilde{X},\alpha_1)$ and $\widetilde{M}_2 \subset \overline{\free}_2(\widetilde{X},\alpha_2)$, respectively.  Let $\widetilde{M} \subset \overline{\free}_2(\widetilde{X},\alpha_1 + \alpha_2)$ be a component which contains the image of a main component of both $\widetilde{M}_1 \times_{\widetilde{X}} \widetilde{M}_2$ and $\widetilde{M}_1' \times_{\widetilde{X}} \widetilde{M}_2$.

    By taking a base change of $\psi$ if necessary, we may assume $\widetilde{M}$ is the fiber of a component $\mathcal{M} \subset \overline{\free}_2(\mathcal{X},\alpha_1 + \alpha_2)$ over the point $\psi(\widetilde{X}) \in B$.  By Remark \ref{rem: classify separating components}, general fibers of an evaluation map $\text{ev}: \mathcal{M} \to \mathcal{X}$ are irreducible.  Therefore, by Proposition \ref{low degree curves degeneration} the fiber of $\mathcal{M}$ over $\psi(X)$ contains exactly one component $M \subset \overline{\free}(X, \alpha_1 + \alpha_2)$.  The fiber of $\mathcal{M}$ over $\psi(X)$ contains the fiber of both $\mathcal{M}_1 \times_{\mathcal{X}} \mathcal{M}_2$ and $\mathcal{M}_1' \times_{\mathcal{X}} \mathcal{M}_2$ over $\psi(X)$, proving our claim.
\end{proof}

To finish our proof, we need only to verify the assumption in Claim \ref{claim: generalizing X}.  If $\overline{\free}_2(\widetilde{X},\alpha_1)$ is irreducible, this assumption is trivial.  Otherwise, let $\widetilde{M}_1,\widetilde{M}_1' \subset \overline{\free}_2(\widetilde{X},\alpha_1)$ and $\widetilde{M}_2 \subset \overline{\free}_2(\widetilde{X},\alpha_2)$ be three distinct components.  By Theorem \ref{thm: existence of core} and Lemma \ref{lem: 3.2 weak core}, there exists a del Pezzo fibration $\widetilde{\pi} : \widetilde{X} \to \mathbb{P}^1$ of degree at most 3 such that $\widetilde{\pi}_* \alpha_1 = 0$.  Such a fibration exists only if $\rho(\widetilde{X}) = 2$ or if $\widetilde{X}$ has deformation type $3.2$.  

Let $i : \widetilde{X}_p \to \widetilde{X}$ be a general fiber of $\widetilde{\pi}$.  We claim $\widetilde{M}_1$ and $\widetilde{M}_1'$ contain maps $f : C_1 \cup C_2 \to \widetilde{X}_p$ and $f' : C_1' \cup C_2' \to \widetilde{X}_p$, respectively, such that $f|_{C_1}$ and $f'|_{C_1'}$ are free conics and $(i \circ f)_*[C_i] = (i \circ f')_*[C_i']$.  Indeed, let $\phi : \widetilde{X}_p \to \mathbb{P}^1$ be a conic bundle.  Lemma \ref{lem: breaking curves del Pezzo fibration} allows us to find candidate maps $f, f'$ but does not verify the condition $(i \circ f)_*[C_i] = (i \circ f')_*[C_i']$.  If $\rho(\widetilde{X}) = 2$, this condition is immediate because $i_* : N_1(\widetilde{X}_p) \to N_1(\widetilde{X})$ has rank one.  If $\widetilde{X}$ has deformation type $3.2$, there is a unique class $l \in \NE(\widetilde{X})_{\mathbb{Z}}$ such that $-K_{\widetilde{X}} . l = 1$ and $l$ lies in the interior of $\NE(\widetilde{\pi})$.  It follows from Lemma \ref{lem: 3.2 weak core} that $\widetilde{M}_1, \widetilde{M}_1'$ contain maps $f,f'$ as above with $(i \circ f)_*[C_2] = (i \circ f')_*[C_2'] = l$.

Set $\ell = (i \circ f)_*[C_2]$, $\alpha_1' = \alpha_1 - \ell$, and $\alpha_2' = \alpha_2 + \ell$.  By construction, $\alpha_2'$ lies in the interior of $\NE(\widetilde{X})$ and satisfies $-K_{\widetilde{X}} . \alpha_2' = 3$.  Theorem \ref{thm: existence of core} implies $\overline{\free}(\widetilde{X}, \alpha_2')$ is irreducible.  By attaching free curves of class $\alpha_2$ to the domains of $f$ (resp.\ $f'$) along $C_2$ (resp.\ $C_2'$), we obtain chains of type $(\alpha_1', \ell, \alpha_2)$.  Lemma \ref{lem: breaking curves del Pezzo fibration} shows we may assume these chains lie in the smooth locus of $\overline{\free}(\widetilde{X}, \alpha_1 + \alpha_2)$.  As they are contained inside a main component of $\overline{\free}(\widetilde{X}, \alpha_1') \times_{\widetilde{X}}\overline{\free}(\widetilde{X}, \alpha_2')$, we may replace $\alpha_1$ with $\alpha_1'$ and $\alpha_2$ with $\alpha_2'$ to apply Claim \ref{claim: connected 2}.  This finishes our proof.
\end{proof}

\subsection{Blow-ups of $\mathbb{P}^1 \times \mathbb{P}^2$}\label{subsection 7.3}
For several Picard rank $3$ cases, we will require the following results about free curves on blow-ups $X$ of $\mathbb{P}^1\times \mathbb{P}^2$ along a smooth curve $c$.  We let $\pi_i : X\rightarrow \mathbb{P}^i$ be the composition of the blow-up map and corresponding projection.  Note that since $X$ is Fano, points of intersection between $c$ and any $\mathbb{P}^2$-fiber of the first projection must be linearly general.

\begin{lem}\label{P1 x P2 monodromy}
Let $X$ be a smooth Fano blow-up of $\mathbb{P}^1\times \mathbb{P}^2$ along a curve $c$ of bidegree $(d_1, d_2)$ with each $d_i > 0$.  The map $c\subset \mathbb{P}^1\times \mathbb{P}^2 \xrightarrow{\pi_2} \mathbb{P}^2$ is an embedding, and the monodromy group of $c\subset \mathbb{P}^1\times \mathbb{P}^2 \xrightarrow{\pi_1} \mathbb{P}^1$ is the full symmetric group whenever $d_1 \leq 2$ or $X$ is general in its deformation class.
\end{lem}

\begin{proof}
First observe that because $X$ is Fano, each fiber of $\pi_2$ may meet $c$ with multiplicity at most one.  %
We invoke Theorem \ref{monodromy del Pezzo main result} to finish our proof.
\end{proof}

When $X$ is a blow up of $\mathbb{P}^1 \times \mathbb{P}^2$, we often use the basis $\ell_1, \ell_2, e$ for $N_1(X)$. Here, $l_1$ is the class of a fiber of $X$ over $\mathbb{P}^2$, $l_2$ is the class of a general line in a $\mathbb{P}^2$-fiber of $X$ over $\mathbb{P}^1$, and $e$ is the exceptional curve of the blow up.  This notation is used for the following four lemmas.

\begin{lem}\label{P1 x P2 effective curves}
Suppose $\pi : X\rightarrow \mathbb{P}^1 \times \mathbb{P}^2$ realizes a smooth Fano threefold $X$ as the blow up of $\mathbb{P}^1\times \mathbb{P}^2$ along a curve $c$ of bidegree $(d_1,d_2)$ with each $d_i >0$.  Then $\NE(X)$ has three extreme rays, two of which are always $e$ and $l_1-e$.  The third extreme ray is either $l_2 -e$ (if $d_1 =1$), $l_2 -2e$ (if $2\leq d_1 \leq 4$), or $2l_2 -5e$ (if $d_1 =5$).
\end{lem}
\begin{proof}
This follows directly from \cite{matsuki1995weyl}.  Note that by \cite{mori1981classification} and Lemma \ref{blow-up numbers}, the anticanonical degree of the blown up curve in $\mathbb{P}^1 \times \mathbb{P}^2$ is at most 14 whenever $X$ does not have deformation type $3.3, 3.5, 3.7$, implying our degree bound on $d_1$.
\end{proof}

\begin{rem}
By \ref{Gordan's Lemma}, the only effective $-K_X$ lines have class $e, l_1 -e, l_2 -2e$, and $2l_2 -5e$.  Moreover, if $c\rightarrow \mathbb{P}^1$ is a degree 5 cover, then $X$ has deformation type $3.5$.  If $c\rightarrow \mathbb{P}^1$ is degree one, then $X$ has deformation type $3.17$ or $3.24$.  If $c$ embedds as a line in $\mathbb{P}^2$, then $X$ has deformation type 3.21 or 3.24.  %
\end{rem}

\begin{lem}\label{curves1 in P1 x P2}
Let $X$ be a smooth Fano blow-up of $\mathbb{P}^1\times \mathbb{P}^2$ along a curve $c$ of bidegree $(d_1,d_2)$ with each $d_i > 0$.  
Suppose $\alpha = dl_2 -ne$.  
Each component $M \subseteq \overline{\free}(X,\alpha)$ 
generically parameterizes strict transforms of curves in $\mathbb{P}^1\times \mathbb{P}^2$ that meet $d_1$ points of $c$ with multiplicities $P(M) = (n_1, \ldots , n_{d_1})$, ordered such that $n_i \geq n_{i+1}$.  Components $N\subset \overline{\free}(X,\alpha)$ with $P(N) = P(M)$ are in bijection with assignments of these multiplicities to the $d_1$ points of $c$ on a general fiber of $\pi_1: X\rightarrow \mathbb{P}^1$ up to the monodromy of $\pi_1 |_c$.  Fibers of $\overline{\free}_1(X,\alpha)\rightarrow X$ are irreducible iff $n_1 = n_{d_1}$.

\end{lem}
\begin{proof}
Clearly each curve parameterized by a component $M\subset \overline{\free}(X,\alpha)$ is contained in a fiber of $\pi_1$.  Since the intersections of $c$ with any such fiber are linearly general, general fibers of $\pi_1$ are del Pezzo surfaces of degree $9-d_1$.  Our claim follows immediately from Theorem \ref{del pezzo curves thm}.
\end{proof}

\begin{lem}\label{curves2 in P1 x P2}
Let $X$ be a smooth Fano blow-up of $\mathbb{P}^1\times \mathbb{P}^2$ along a curve $c$ of bidegree $(d_1,d_2)$ with each $d_i > 1$.  
For every $\alpha = l_1 + dl_2 - ne$ with %
$n < \frac{3d + 3}{2}$ and $d > 0$, $\overline{\free}(X,\alpha)$ is irreducible and nonempty.  %
\end{lem}
\begin{proof}
Let $\pi: X\rightarrow \mathbb{P}^1 \times \mathbb{P}^2$ be the blow-up of $c$.  Suppose $\alpha = l_1 + dl_2 -ne$ satisfies $n < \frac{3d + 3}{2}$.  Otherwise, as $d \geq 1$, $n \leq 2d \leq d \cdot d_2$ and $\alpha = (l_1 -e) + d(l_2 -2e) + (2d +1 -n) e$, each component of $\overline{\free}(X,\alpha)$ parameterizes very free curves, and we may think of points in $\overline{\free}(X, \alpha)$ as parametrizing maps $f: \mathbb{P}^1 \rightarrow \mathbb{P}^2$ whose graphs $\tilde{f}$ meet $c\subset \mathbb{P}^1 \times \mathbb{P}^2$ transversely at $n$ of $d \cdot d_2$ possible points.  For each component $M\subset \overline{\free}(X,\alpha)$, this induces a natural map $M \dashrightarrow \text{Sym}^n(c)$, which we show must be dominant.

Indeed, let $\tilde{f} : \mathbb{P}^1 \rightarrow \mathbb{P}^1 \times \mathbb{P}^2$ be a map whose strict transform $g: \mathbb{P}^1 \rightarrow X$ corresponds to a point in $M$.  Note that $\mathcal{N}_{\tilde{f}} \cong f^*\mathcal{T}_{\mathbb{P}^2}$, which is semistable.  Therefore, $\mathcal{N}_{\tilde{f}} \cong \mathcal{O}( \lfloor \frac{3d}{2} \rfloor) \oplus \mathcal{O}( \lfloor \frac{3d +1}{2} \rfloor)$ by \cite[Proposition~3.1]{patel2020moduli}.  If we blow-up any $n-1$ of the exceptional fibers of $\pi : X \rightarrow \mathbb{P}^1 \times \mathbb{P}^2$ meeting $g(\mathbb{P}^1)$, the strict transform of $\tilde{g}$ of $g$ has a normal bundle $\mathcal{N}_{\tilde{g}} \cong \mathcal{O}(a) \oplus \mathcal{O}(b)$ with $(a,b) \in \{(\lfloor \frac{3d}{2} \rfloor - n, \lfloor \frac{3d+1}{2} \rfloor - n +1) , (\lfloor \frac{3d}{2} \rfloor - n+1, \lfloor \frac{3d+1}{2} \rfloor - n)\}$.  In either case, there is a positive dimensional family of deformations of $\tilde{g}$, whose projection to $\mathbb{P}^2$ either dominates $\mathbb{P}^2$ or corresponds to reparameterizations of the same map.  However, any family of such reparameterizations has dimension at most $3-n$, and a short analysis of $n\leq 2$ shows these cannot constitute all such deformations of $\tilde{g}$.  Therefore, we may fix any $n-1$ intersections of $\tilde{f}$ with $c$ while moving the $n^{th}$, proving $M \dashrightarrow \text{Sym}^n(c)$ must be dominant.

It is sufficient to show that the general fiber of $\overline{\free}(X, \alpha) \dashrightarrow \text{Sym}^n(c)$ is irreducible.  Let $p_1 = [1: r_1] \times [1 : x_1 : y_1], \ldots , p_n =  [1 : r_n] \times [1 : x_n : y_n]$ a general collection of $n$ points in $c$ corresponding to a general point $p \in \text{Sym}^n(c)$.  Parameterize maps $f: \mathbb{P}^1 \rightarrow \mathbb{P}^2$ described in the first paragraph above by $[s : t] \rightarrow [a_0 s^d + \ldots + a_d t^d:  b_0 s^d + \ldots + b_d t^d : c_0 s^d + \ldots + c_d t^d]$.  The conditions that $\tilde{f}$ meets $c$ at the $n$ points above corresponds to $f([1: r_i]) = [1 : x_i : y_i]$.  In turn, these correspond to linear conditions on the $a_j, b_j, c_j$ with coefficients described by polynomials in the $r_i$, $x_i$, and $y_i$.  For $p_{i+1}$ general with respect to $p_1, \ldots , p_i$, these linear conditions are nondegenerate, and determine a linear space of maps whose general member has no base points.  %
This proves our claim.
\end{proof}

\begin{lem}\label{curves3 in P1 x P2}
Consider a Fano blow-up $X$ of $\mathbb{P}^1\times \mathbb{P}^2$ along a smooth curve $c$ of bidegree $(d_1,d_2)$ with each $d_i >1$.  For every $\alpha = a l_1 + b l_2 - ne$ with $0 \leq n \leq 2a + 1$ and $n \leq bd_2$, $\overline{\free}(X,\alpha)$ is irreducible and nonempty.
\end{lem}
\begin{proof}
By Lemma \ref{interior conic lemma}, Proposition \ref{very free curves}, and Theorem \ref{reducible fibers: 4 author result}, if $\alpha$ lies in the interior of $\NE(X)$ and has anticanonical degree at least 4, every component of $\overline{\free}(X,\alpha)$ will parameterize very free curves.  If $\alpha$ lies on the boundary of $\NE(X)$ or has anticanonical degree less than 4, we see $\alpha \in \{l_1, l_2, l_2 -e, l_1 + l_2 -2e, l_1 + l_2 -3e, 2l_1 + l_2 -4e \}$.  The last option, $\alpha = 2l_1 + l_2 -4e$, may only happen when $d_2 \geq 4$, which would imply that $X$ has deformation type $3.3$.  In this case, $2l_1 + l_2 -4e$ is pseudosymmetric to $l_2$.  Otherwise, If $d_2 \geq 3$, $\overline{\free}(X,l_1 + l_2 -3e)$ is a degree ${d_2 \choose 3}$ cover of the space of lines in $\mathbb{P}^2$.  A monodromy argument proves this space is irreducible.  Similar arguments apply to each other possibility for $\alpha$.  Thus, we may assume each component of $\overline{\free}(X,\alpha)$ parameterizes very free curves.

Consider the map $\pi_{2*} : \overline{\free}(X,\alpha)\rightarrow \overline{\free}(\mathbb{P}^2, b)$ induced by $\pi_2$.  %
Since $\pi_2 : X\rightarrow \mathbb{P}^2$ is smooth away from a codimension 2 locus, over a general free curve $g: \mathbb{P}^1\rightarrow X$ with $\pi_2 \circ g = f$, the map $g^* \mathcal{T}_X \rightarrow f^* \mathcal{T}_{\mathbb{P}^2}$ is surjective.  It follows that the natural map $\mathcal{N}_g \rightarrow \mathcal{N}_{f}$ is surjective on global sections, as the kernel is a line bundle of degree $2a -n$.  This shows $\pi_{2*}$ is dominant on any component of $\overline{\free}(X,\alpha)$.

For $[f] \in \free(\mathbb{P}^2, b)$ whose image contains a general point in $\mathbb{P}^2$, we may consider the fiber product $X_f = X \times_{\mathbb{P}^2} \mathbb{P}^1$, which is isomorphic to the blow up of $\mathbb{P}^1 \times \mathbb{P}^1$ along points in $c$ lying over each point of $f(\mathbb{P}^1)\cap \pi_2(c)$.  A general map $[g] \in \free(X,\alpha)$ lying over $[f]$ must be a curve of class $(a,1)$ in $\mathbb{P}^1 \times \mathbb{P}^1$ that meets $n$ of the blown-up points in $X_f$ with multiplicity one (see \cite[Propositions~2.9-10]{beheshti2020moduli}).  %

Since $\pi_{2*}$ is dominant when restricted to any component of $\overline{\free}(X,\alpha)$, for general $[f]$, each component of the fiber $\pi_{2*}^{-1}[f]$ has the expected dimension $2a  + 1 -n$.  It is easily seen that these components are in bijection with choices of $n$ of the $bd_2$ points in $f(\mathbb{P}^1) \cap \pi_2(c)$, as each gives a linear system of rational curves in $X_f$.  

We claim that the monodromy of intersections $f(\mathbb{P}^1) \cap \pi_2(c)$ as $f$ varies is $n$-transitive, which shows that each component of the fiber belongs to the same component of $\overline{\free}(X,\alpha)$.  
When $d_2 \leq 3$, we may fix any $bd_2 -2$ of the intersection points $f(\mathbb{P}^1) \cap \pi_2(c)$ while perturbing the remaining two in a positive dimensional family.  This proves the monodromy of intersections of $f(\mathbb{P}^1)$ with $\pi_2(c)$ over $\overline{\free}(\mathbb{P}^2, b)$ is the full symmetric group.  When $d_2 = 4$, $X$ has deformation type 3.3.  In this case, if $b\leq 2$, $\overline{\free}(\mathbb{P}^2,b)$ is a very ample linear system of divisors on $\mathbb{P}^2$, which shows the appropriate monodromy group is the full symmetric group.  We shall only use the case $b\leq 2$ in our proof of Theorem \ref{3.3thm}, and the other cases follow.
\end{proof}

\subsection{Blow-ups of $\mathbb{P}^3$}\label{subsection 7.4}

Lastly, we prove a result about curves in blow-ups of $\mathbb{P}^3$.

\begin{lem}\label{lines_in_P3}
Suppose $\pi: X \rightarrow \mathbb{P}^3$ realizes a smooth Fano threefold $X$ as an iterated blow up of $\mathbb{P}^3$ along points and strict transforms of smooth curves in $\mathbb{P}^3$. Let $E_1, \ldots , E_n$ be the exceptional divisors of $\pi$. %
Suppose $\alpha\in N_1(X)_\mathbb{Z}$ satisfies $\alpha . (-K_X) \geq 2$ and $\alpha . E_i \geq 0$ for all $i$.

\begin{enumerate}
    \item If $\pi_* \alpha$ is the class of a line, then $\overline{\free}(X, \alpha)$ is irreducible and nonempty iff $\alpha . \sum_{i \in S} E_i \leq 1$ for all collections $S$ of coplanar $\pi(E_i)$.  If $\alpha . (\sum_{i\leq n} E_i) \leq 1$, then the evaluation map $\text{ev} : \overline{\free}_1(X,\alpha)\rightarrow X$ has irreducible fibers.
    \item If $\pi_* \alpha$ is the class of a conic, $\overline{\free}^{bir}(X, \alpha)$ is irreducible and nonempty if 
    \begin{itemize}
        \item $\alpha . (\sum_{i \in S} E_i) \leq 2$ for all collections $S$ of coplanar $\pi(E_i)$, 
        \item $\alpha . (\sum_{i\leq n} E_i) \leq 4$, 
        \item $\alpha . E_i \leq d_i$, where $d_i$ is the degree of $\pi(E_i)$ as a point or curve in $\mathbb{P}^3$, 
        \item and $\alpha . E_i$ is either $0$ or $d_i$ for all but at most one $i$.  %
    \end{itemize} 
    Unless $X$ has deformation type 5.1 and $\alpha.E_i > 0$ for two distinct $E_i$ contracted to a point by $\pi$, the evaluation map $\text{ev} : \overline{\free}^{bir}_1(X,\alpha)\rightarrow X$ has irreducible fibers.  Moreover, if $-K_X . \alpha \geq 4$, then $\alpha$ is freely breakable unless $\alpha . E_i = 2$ for a planar curve $\pi(E_i)$ and $\alpha . E_j = 1$ for a point $\pi(E_j)$.
\end{enumerate}
If $\alpha_1 + \alpha_2 = \alpha_3 + \alpha_4$ and all $\alpha_i$ satisfy \ref{lines_in_P3}(1), then a main component of the product $\overline{\free}_1(X,\alpha_1) \times_X \overline{\free}_1(X,\alpha_2)$ lies in the same component of $\overline{\free}(X)$ as a main component of $\overline{\free}_1(X,\alpha_3) \times_X \overline{\free}_1(X,\alpha_4)$.

\end{lem}

\begin{rem}
For any $\alpha$ satisfying the hypotheses of \ref{lines_in_P3}(2), if $-K_X . \alpha \geq 5$ and $\alpha$ does not meet any colinear set of $\pi(E_i)$ with multiplicity $2$ (i.e. a blown-up line twice, or a blown-up line through a blown-up point), $\overline{\free}^{bir}(X,\alpha)$ necessarily parameterizes very free curves.  This follows from observing that we may find a representative of $\alpha$ through two general points in $X$.
\end{rem}

\begin{proof}
Suppose $\pi_* \alpha$ is the class of a line.  Since $\alpha . (-K_X) \geq 2$, $\alpha . E_i \neq 0$ for at most two $i$, in which case $\alpha . E_i = 1$.  If $\alpha . E_1 = \alpha . E_2 = 1$, then $\pi(E_1)$ and $\pi (E_2)$ are non-coplanar curves.  No hypersurface contains all lines joining $\pi(E_1)$ and $\pi(E_2)$, as planes are the only surfaces that contain 2 parameter families of lines.  Hence, $\alpha$ is nef, and $\overline{\free}(X,\alpha)$ is dominated by an open subset of $\pi(E_1)\times \pi(E_2)$.  The normal bundle of a general such curve will be globally generated.

If instead $\alpha. E_1 = 2$, then $\pi(E_1)$ is a non-planar smooth curve whose general point does not meet $\pi(E_i)$ for all $i\neq 1$.  Again, it follows that $\alpha$ is nef, and $\overline{\free}(X,\alpha)$ is dominated by an open subset of $\pi(E_1)\times \pi(E_1)$.  When $\alpha . E_1 = 1$ and $\alpha . E_i = 0$ for all other $i$, the space of lines through $\pi (E_1)$ is irreducible with irreducible fibers over general points in $\mathbb{P}^3$.

Suppose $\pi_* \alpha$ is the class of a conic and satisfies the hypotheses written above.  It follows that $\alpha . E_i = 1$ for at most two $E_i$ with $\pi(E_i)$ a point.  Note that if $\pi(E_i)$ is a point, $d_i = 1$, so $\alpha . E_i \leq 1$.  Let $E_1, \ldots , E_k$ ($k\leq 4$) be the exceptional divisors for which $\alpha . E_i > 0$.

Suppose $\alpha . E_i > 0$ implies $\pi(E_i)$ is a curve.  By Theorem \ref{Monodromy}, there is an open subset of planes $U\subset |H|$ which intersect each $\pi(E_i)$ transversely, with intersection points in linearly general position.  Theorem \ref{Monodromy} and our hypotheses ensures the space of conics passing through appropriate points of intersection of $\pi(E_i)$ with a plane parameterized by $U$ is irreducible.  Fixing a general point $p \in \mathbb{P}^3$, to show $\text{ev}^{-1}(p)$ is irreducible, we only need to ensure that if $\alpha . E_i \neq 0, d_i$ for some $i$, say $i = 1$, the monodromy action of $|H-p|$ is appropriately transitive on intersections $H \cap \pi(E_1)$.  %
This follows immediately from Theorem \ref{Monodromy} if a general $p$ lies on only finitely many secant and tangent lines of $\pi(E_1)$.  This is equivalent to the statement that a generic projection of a curve in $\mathbb{P}^3$ to $\mathbb{P}^2$ is birational \cite[Section~IV.3]{hartshorne2013algebraic}.

To see that $\alpha$ is freely breakable, it is sufficient to see that $\alpha = \alpha_1 + \alpha_2$ for $\alpha_i$ satisfying the hypotheses of \ref{lines_in_P3}(1).  This implies our claim because $\overline{\free}^{bir}(X,\alpha)$ is irreducible.  If $\alpha.(\sum_{i\leq n} E_i) \leq 3$, we may take $\alpha_1$ to be a line meeting one or two non-coplanar $\pi(E_i)$, and $\alpha_2$ to be a line meeting $\pi(E_i)$ such that $\alpha.E_i > \alpha_1 .E_i$.  If $\alpha.E_i \geq 3$ for some $i$, we may let $\alpha_1$ meet $E_i$ twice.  If $\alpha . E_i = 2$ for some $i$, we may let $\alpha_1$ and $\alpha_2$ meet $E_i$ once.  Otherwise, $\alpha . E_i = 1$ for four different $E_i$, so $\rho(X) \geq 5$.  There are no such smooth Fano threefolds with four different exceptional $E_i$ contracted by $\pi$.

Suppose, for $i=1,2$, that $\alpha . E_i = 1$ and $\pi(E_i)$ is a point.  It follows that since $X$ is Fano, the strict transform of the line between $\pi(E_1)$ and $\pi(E_2)$ must also be blown up by $\pi$.  Thus, $X$ is a blow up of 4.12.  This limits $X$ to $4.12$, $5.1$, and $5.2$.  Let $E_3$ be the exceptional divisor over the line containing $\pi(E_1)$ and $\pi(E_2)$, and $E_4$ the exceptional divisor over a conic (when $X$ is 5.1) or another line (when $X$ is 5.2).  Our hypotheses imply $\alpha.E_3 = 0$, and $\alpha .E_4 \leq 1$.  If $\alpha.E_4 = 0$, the claim is clear.  When $\alpha.E_4 = 1$, $\overline{\free}(X,\alpha)$ is dominated by a family of conics in planes parameterized by $p \in \pi(E_4)$ that contain $\pi(E_1)$, $\pi(E_2)$, and $p$.  When $E_4$ lies over a line and $q\in \mathbb{P}^3$ is general, there is exactly one $p \in \pi(E_4)$ meeting the plane containing $\pi(E_1)$, $\pi(E_2)$, and $q$.  If $-K_X . \alpha \geq 4$, $\alpha$ breaks freely into two lines containing $\pi(E_1)$ and $\pi(E_2)$.

Suppose $\alpha . E_1 = 1$ and $\pi(E_1)$ is a point, but $\pi(E_i)$ is not a point for all other $E_i$ such that $\alpha.E_i > 0$.  An open subset $U$ of the linear system of planes $|H-E_1|$ intersects $\pi(E_2), \ldots, \pi(E_k)$ transversely, and contains no line through any three points on the $\pi(E_i)$ such that $\alpha.E_i > 0$: unless these lines sweep out a plane, there are only 1 dimensional families of them; if they sweep out a plane, $\alpha$ would meet a collection of coplanar $E_i$ in more than 2 points.
If $\alpha . E_i = 0,d_i$ for all $i$, then each plane parameterized by $U$ contains a linear system of generically irreducible conics whose strict transforms have class $\alpha$.  If $\alpha . E_i < d_i$ for some $i$, we may suppose $i=2$, so that each plane contains ${d_2 \choose \alpha. E_2}$ such linear systems of conics.  Since $X$ is Fano, there are no lines through $\pi (E_1)$ and any two points on other exceptional loci, including $E_2$.  This implies $\pi(E_2)\rightarrow \mathbb{P}|H-E_1|$ is injective.  By Theorem \ref{Monodromy}, this implies the space of such conics is irreducible.  

Each conic through a general point $p\in \mathbb{P}^3$ whose strict transform has class $\alpha$ lies in plane through $p$ and $\pi(E_1)$.  The pencil of such planes gives a morphism to $\mathbb{P}^1$.  Irreducibility of $\pi(E_2)$ implies the monodromy action on the fibers of this morphism is 1-transitive.  Thus, the evaluation map $\text{ev} : \overline{\free}^{bir}_1(X,\alpha)\rightarrow X$ has connected fibers unless $1 < \alpha . E_2 < d_2 -1$.  %
We address $\alpha . E_2 = 2,3$ separately.  If $\alpha . E_2 = 3$, $d_2 \geq 4$ and $\alpha$ pairs positively with no other $E_i$.  Thus $\alpha$ may be realized as a curve class on one of $3.11$, $3.14$, $3.16$, $3.19$, $3.23$, $3.26$, $3.29$, or $3.30$, as these are the only Fano threefolds of Picard rank 3 which are blow ups of 2.35.  %
None of these have a divisorial contraction onto 2.35 with exceptional divisor $E_2$ lying over a curve in $\mathbb{P}^3$ of degree $d_2 > 3 + 1$.  If $\alpha . E_2 = 2$, $\alpha . E_i > 0$ for at most one other $E_i$, say $E_3$.  Thus $\alpha$ may be realized as a curve on one of the varieties $X$ above, or one of $4.4$, $4.9$, $4.12$.  Again, only $3.11$ has a divisorial contraction onto 2.35 with exceptional divisor $E_2$ lying over a curve in $\mathbb{P}^3$ of degree $d_2 > 2 + 1$.  In this case, the curve $\pi(E_2)$ has degree 4 and passes through $\pi(E_1)$, so as $\pi(E_2)\rightarrow \mathbb{P}|H-E_1|$ is injective, Theorem \ref{Monodromy} implies the fiber $\text{ev}^{-1}(p)$ is irreducible.

When $-K_X . \alpha \geq 4$, by our assumptions $\alpha . (\sum_{i\neq 1} E_i) \leq 2$.  Since no two distinct $\pi(E_i)$ are coplanar in cases 4.4, 4.9, and 4.12, $\alpha$ must break freely into a line $\alpha_1$ through $\pi(E_1)$ and a free curve of class $\alpha - \alpha_1$.

Lastly, suppose $\alpha_1 + \alpha_2 = \alpha_3 + \alpha_4$ is a relation between classes $\alpha_i$ satisfying the hypotheses of \ref{lines_in_P3}(1).  Our claim follows from irreducibility of $\overline{\free}^{bir}(X,\alpha_1 + \alpha_2)$ when $(\alpha_1 + \alpha_2) .E_i \leq d_i$ for all $i$.  If $(\alpha_1 + \alpha_2) . E_i > d_i$ for some $i$, then $d_i <4$.  In such a case, we may assume $(\alpha_1 + \alpha_2) . E_j = 0$ for any $E_j$ contracted to a point $\pi(E_j)$.  If $d_i =1$, we may assume $\pi(E_i)$ is a line, in which case there are no nontrivial relations $\alpha_1 + \alpha_2 = \alpha_3 + \alpha_4$.  If $d_i = 2$, $\pi(E_i)$ must be a conic, and there are no nontrivial relations as before.  If $d_i = 3$, all $\alpha_i$ must be a line passing through $\pi(E_i)$ twice.
\end{proof}

\section{Picard Rank $\geq$ 5}\label{Section 8}

\subsection*{5.1}

\textbf{Blow-up of a quadric $Q \subset \mathbb{P}^4$:} Let $X$ be the blow up of a smooth quadric $Q$ in $\mathbb{P}^4$ along a conic and three exceptional fibers.  This the same as the blow up of $Q$ first along the three points, then along the strict transform of the conic.  Projection from one of these blown up points shows $X$ is isomorphic to the blow up of $\mathbb{P}^3$ along two points, the strict transform of the line through them, and a conic disjoint from the line, which is also not coplanar with either blown up point.  Let the points be $p,q \in \mathbb{P}^3$, the line be $\ell$, and the conic be $c$.

\textbf{Generators for $N^1(X)$ and $N_1(X)$:}
Let $H$ be the hyperplane section in $\mathbb{P}^3$ and denote by $E_p$, $E_q$, $E_\ell$, and $E_c$ the exceptional divisors over $p,q,\ell$, and $c$.  Let $l$ be the class of a general line in $\mathbb{P}^3$, and $e_p,e_q,e_\ell,$ and $e_c$ be the classes of lines in each exceptional divisor.

\begin{thm}\label{GMC5.1}
For each $\alpha\in l + \Nef_1(X)_{\mathbb{Z}}$, %
$\overline{\free}(X,\alpha)$ is irreducible, nonempty, and generically parameterizes very free curves.

\end{thm}

\textbf{Intersection Pairing:} 
\begin{center}
\begin{tabular}{lllll}
    $H \cdot l = 1$ &  $H \cdot e_p = 0$ &  $H \cdot e_q = 0$ &  $H \cdot e_\ell = 0$ &  $H \cdot e_c = 0$ \\
    $E_p \cdot l = 0$ &  $E_p \cdot e_p = -1$ &  $E_p \cdot e_q = 0$ &  $E_p \cdot e_\ell = 0$ &  $E_p \cdot e_c = 0$ \\
    $E_q \cdot l = 0$ &  $E_q \cdot e_p = 0$ &  $E_q \cdot e_q = -1$ &  $E_q \cdot e_\ell = 0$ &  $E_q \cdot e_c = 0$ \\
    $E_\ell \cdot l = 0$ &  $E_\ell \cdot e_p = 0$ &  $E_\ell \cdot e_q = 0$ &  $E_\ell \cdot e_\ell = -1$ &  $E_\ell \cdot e_c = 0$ \\
    $E_c \cdot l = 0$ &  $E_c \cdot e_p = 0$ &  $E_c \cdot e_q = 0$ &  $E_c \cdot e_\ell = 0$ &  $E_c \cdot e_c = -1$ \\
\end{tabular}
\end{center}

\textbf{Anticanonical Divisor:} 
\begin{align*}
    -K_X = & 4H - 2E_p -2E_q -E_\ell -E_c \\
    = & 2(H-E_p -E_q -E_\ell) + (H - E_c) + H + E_\ell
\end{align*}

\textbf{Effective Divisors:} The extremal effective divisors are $H-E_p -E_q -E_\ell$, $H-E_c$, $2H-2E_p -E_c$, $2H - 2E_q -E_c$, $E_p$, $E_q$, $E_\ell$, and $E_c$.

\textbf{Pseudosymmetry:}  The projection of $Q$ from a different blown up point results in an equivalent description of $X$ as a blow-up of $\mathbb{P}^3$.  To obtain this description, note that the exceptional divisor lying over the blown up points correspond to $E_p, E_q$, and $H-E_c$. $\overline{\text{NE}}(X)$ is generated by $e_c$, $e_\ell$, $l-e_\ell -2e_c$, $l-e_p -e_q + e_\ell$, $e_p - e_\ell$, $e_q -e_\ell$, $l-e_p -e_c$, and $l-e_q -e_c$.  In order, these sweep out the divisors $E_c$, $E_\ell$, $H-E_c$, $E_\ell$, $E_p$, $E_q$, $2H-2E_p -E_c$, and $2H-2E_q -E_c$.  Projection from a different point leaves $E_\ell$ invariant, swaps $H-E_c$ with $E_p$ (or $E_q$) and $E_c$ with $2H - 2E_p -E_c$ (resp. $2H - 2E_q -E_c$).  Extending linearly, we find $E_\ell \rightarrow E_\ell$, $E_p \rightarrow H-E_c$, $E_q \rightarrow E_q$, $E_c \rightarrow 2H -2E_p -E_c$, and $H\rightarrow 2H -E_p -E_c$.  This corresponds to the following, nontrivial pseudoactions on $N_1(X)$:
\begin{enumerate}
    \item $e_\ell \rightarrow e_\ell, \ e_q \rightarrow e_q, \ e_p \rightarrow l-2e_c, \ e_c \rightarrow l - e_p -e_c, \ l \rightarrow 2l -e_p -2e_c$,
    \item $e_p \rightarrow e_q, \ e_q \rightarrow e_p, \ l \rightarrow l, \ e_\ell \rightarrow e_\ell, \  e_c \rightarrow e_c.$
\end{enumerate}

\begin{lem}
A core of free curves on $X$ is given by 
\begin{align*}
    \mathscr{C}_X = &\{ l, \ l-e_p, \ l-e_q, \  l-e_\ell, \  l-e_c, \  l-e_\ell -e_c, \  2l -e_p -e_q -e_c, \ 2l -e_p -2e_c,\\
        &  2l -e_q - 2e_c, \ 2l -e_p -e_q -2e_c,\ 2l - e_p - e_\ell - 2e_c,\ 2l -e_q - e_\ell -2e_c\}.
\end{align*}
The only separating classes $\alpha \in \mathscr{C}_X$ are $l-e_\ell -e_c$ and $2l -e_p -e_q -e_c$.
\end{lem}
\begin{proof}
\textbf{Nef Curve Classes of Anticanonical Degree Between $2$ and $4$:} 
Note that $-K_X = 2(H-E_p -E_q -E_\ell) + (H - E_c) + H + E_\ell$.  
The only curve classes of appropriate anticanonical degree, which pair nonnegatively with all divisors above, are of class $\alpha = dl - m_p e_p - m_q e_q - m e_\ell - n e_c$ with $d, m_p, m_q, m, n \geq 0$ and $d\leq 4$. If $d=4$, then $n=4$, $m = 0$, and $m_p + m_q = 4$.  The corresponding quartic curve must be planar.  However, it would need to be double at two points of intersection with $c$, and double at both $p$ and $q$, which is impossible, unless the curve is a double cover or reducible.  Similarly, if $d=3$, then either $m=1$ or $m=0$.  If $m=1$, then $n=3$ and $m_p = m_q = 1$.  Such cubics exist and are always planar (indeed for any plane containing $\ell$ there are 3 dimensions of them), however, we easily break it into free classes $l-e_\ell -e_c$ and $2l -e_p -e_q -2e_c$.  If $m=0$, then $m_p + m_q = 3$ and $n = 2,3$.  No such irreducible curves exist for $n=3$, while for $n=2$, we easily break it as $l-e_{p/q}$ and $2l -e_p -e_q -2e_c$.  Thus, the relevant curve classes to consider all have $d\leq 2$.  In additon to the classes listed above, the classes $2l -e_p -e_q$, $2l -2e_{p/q}$, $2l -e_{p/q} -e_\ell -e_c$, and $2l -2e_\ell -2e_c$ are also classes of free curves, but these are easily broken into the union of two other free curves.

\textbf{Irreducible spaces}: by Lemma \ref{lines_in_P3}, $\overline{\free}_1(X,\alpha)$ is irreducible for all $\alpha \in \mathscr{C}_X$ and has generically irreducible fibers over $X$ whenever $\alpha \notin \{ l-e_c -e_\ell, \ 2l - e_p -e_q -e_c, \ 2l -e_p -e_q -2e_c\}$.  There is a unique curve of class $2l -e_p -e_q -2e_c$ through a general point, while there are two irreducible components of free curves of class $\alpha$ through a general point of $X$ when $\alpha = l-e_c -e_\ell, \ 2l - e_p -e_q -e_c$.  
\end{proof}

\begin{lem}
Up to pseudosymmetry, relations in the monoid $\mathbb{N}\mathscr{C}_X$ are generated by the following list:
\begin{enumerate}
    \item $l + (l-e_\ell -e_c) = (l-e_\ell) + (l-e_c)$
    \item $l + (2l -e_p -e_q -e_c) = (l-e_p) + (l-e_q) + (l-e_c)$
    \item $l + (2l -e_{p/q} -2e_c) = (l-e_{p/q}) + 2(l-e_c)$
    \item $l + (2l -e_p -e_q -2e_c) = (l-e_c) + (2l -e_p -e_q -e_c)$
    \item $l + (2l -e_{p/q} -e_\ell -2e_c) = (l-e_\ell - e_c) + (l -e_{p/q}) + (l -e_c)$.
    \item $(l-e_\ell) + (2l -e_p -e_q -e_c) = (l-e_\ell -e_c) + (l-e_p) + (l-e_q)$
        \item $(l-e_\ell) + (2l-e_{p/q} -e_\ell -2e_c) = 2(l-e_c -e_\ell) + (l-e_{p/q})$
        \item $(l-e_\ell) + (2l -e_p -e_q -2e_c) = (l-e_\ell -e_c) + (2l -e_p -e_q - e_c)$
        \item $2(2l -e_p -e_q -e_c) = (2l -e_p -e_q -2e_c) + (l-e_p) + (l-e_q)$.
\end{enumerate}
For each relation $\sum \alpha_i = \sum \alpha_j'$, a main component of $\prod_X \overline{\free}_2(X,\alpha_i)$ lies in the same component of free curves as a main component of $\prod_X \overline{\free}_2(X,\alpha_j')$.
\end{lem}
\begin{proof}
We note that action of $\text{Aut}(X)$ on $\mathscr{C}_X$ decomposes into the following orbits:
\begin{itemize}
    \item $\{ l, \ 2l -e_p -2e_c, \ 2l - e_q -2e_c\}$,
    \item $\{ 2l -e_p -e_q -2e_c, \ l-e_q, \  l-e_p\}$,
    \item $\{ l-e_\ell, \  2l-e_p -e_\ell -2e_c, \ 2l-e_q -e_\ell -2e_c\}$,
    \item $\{ l-e_c \}, \ \{l-e_\ell -e_c\}, \ \{ 2l -e_p -e_q -e_c\}$.
\end{itemize}

\textbf{Relations}: %
We apply Lemma \ref{Relations} to find a generating set of relations.

\begin{itemize}
    \item $D=-H + E_p + E_q + E_\ell + E_c$: The relevant nonzero pairings are $D.l = -1$, $D.l-e_c -e_\ell = 1$, $D.(\text{any conic})\in \{1,2\}$.  We obtain the relations 
    \begin{itemize}
    \item $l + (l-e_\ell -e_c) = (l-e_\ell) + (l-e_c)$
    \item $l + (2l -e_p -e_q -e_c) = (l-e_p) + (l-e_q) + (l-e_c)$
    \item $l + (2l -e_{p/q} -2e_c) = (l-e_{p/q}) + 2(l-e_c)$
    \item $l + (2l -e_p -e_q -2e_c) = (l-e_c) + (2l -e_p -e_q -e_c)$
    \item $l + (2l -e_{p/q} -e_\ell -2e_c) = (l-e_\ell - e_c) + (l -e_{p/q}) + (l -e_c)$.
    \end{itemize}
    By symmetry, we eliminate all relations involving $2l - e_{p/q} -2e_c$ as well.
    \item $D= -H + E_p + E_q + E_c$: $D$ only pairs negatively with $l-e_\ell$ and positively with any conic.  We obtain the relations
    \begin{itemize}
        \item $(l-e_\ell) + (2l -e_p -e_q -e_c) = (l-e_\ell -e_c) + (l-e_p) + (l-e_q)$
        \item $(l-e_\ell) + (2l-e_{p/q} -e_\ell -2e_c) = 2(l-e_c -e_\ell) + (l-e_{p/q})$
        \item $(l-e_\ell) + (2l -e_p -e_q -2e_c) = (l-e_\ell -e_c) + (2l -e_p -e_q - e_c)$
    \end{itemize}
    By symmetry, we eliminate all relations involving $2l -e_{p/q} -e_\ell -2e_c$ as well.  Between the remaining 6 elements of $\mathscr{C}_X$, the unique relation is $2(2l -e_p -e_q -e_c) = (2l -e_p -e_q -2e_c) + (l-e_p) + (l-e_q)$.  
\end{itemize} 

\textbf{Main Components:} The claim for Relations $(1) - (2)$ follows from Lemma \ref{lines_in_P3}.  For relations $(3)-(5)$, let $C=C_1\cup C_2$ be a chain of type $(l,\alpha)$ where $\alpha$ is given by the right-hand side of the relation.  By deforming the component of class $l$, we may move the intersection of the two components to lie on $c$.  The resulting chain is a smooth point of $\overline{\mathcal{M}}_{0,0}(X)$ and has type $(l-e_c, e_c, \alpha)$.  We may smooth the latter two components, and break the resulting curve using Lemma \ref{lines_in_P3} to obtain a chain of free curves of type given by the left-hand side.  For relations $(6)-(8)$, after substituting $l-e_\ell$ for $l$ in the above argument, we may use an identical procedure.  %

For relation $(9)$, we may smooth a free chain of type $(2l -e_p -e_q -2e_c, l-e_p, l-e_q)$ to a free chain of type $(2l -e_p -e_q -2e_c, 2l -e_p -e_q)$.  Then, by moving the intersection point to lie on $c$, we prove our claim.
\end{proof}

\begin{lem}
For all nonzero $\alpha \in \Nef_1(X)_\mathbb{Z}$, $\overline{\free}(X,\alpha)$ is nonempty.
\end{lem}

\begin{proof}
The extreme rays of $\Nef_1(X)$ are spanned by elements of $C=\mathscr{C}_X\setminus\{2l -e_p -e_q -e_c\}$ (see Theorem \ref{Representability of Free Curves}).  Let $\alpha \in \Nef_1(X)_\mathbb{Z}$.  By Lemma \ref{Gordan's Lemma}, it suffices to check $\overline{\free}(X,\alpha)$ is nonempty when $\alpha = \sum_{\alpha_i \in C} a_{\alpha_i} \alpha_i$ for $0\leq a_{\alpha_i} < 1$.  Write $\alpha = c_l l - c_p e_p -c_q e_q -c_\ell e_\ell - c_c e_c$.  It follows that $0\leq c_p,c_q,c_\ell \leq 3$ while $0\leq c_l \leq 15$ and $0\leq c_c \leq 11$.  Using pseudosymmetry, we may assume $\alpha.(H-E_c) \leq \alpha.E_q \leq \alpha.E_p$, i.e. $ c_l - c_c \leq c_q \leq c_p$.  As $\alpha$ is nef, we must have $c_l \geq c_p + c_q + c_\ell$, $c_l \geq c_c$, and $2c_l \geq 2c_p + c_c$.  It follows that $2c_c \geq c_l + c_\ell$ and $c_c \geq c_p + c_\ell$.  

If $c_p = 0$, then $a_{\alpha_i} = 0$ unless $\alpha_i \in \{l - e_c, l -e_\ell -e_c\}$, and so $\alpha =0$ is the only possibility.  If $c_p \geq 1$ but $c_q = 0$, then $a_{\alpha_i} = 0$ unless $\alpha_i \in \{l - e_c, l -e_\ell -e_c, 2l -e_p -2e_c, 2l -e_p -e_\ell -2e_c\}$, and the only nonzero integral possibility is $\alpha = 3l -e_p -e_\ell -3e_c = (l -e_\ell - e_c) + (2l -e_p -2e_c)$.  If $c_q \geq 1$ but $c_l = c_c$, then $a_{\alpha_i} = 0$ whenever $\alpha_i \in \{l, l -e_p, l -e_q, l -e_\ell\}$, and the only possibilities for $\alpha$ have $c_\ell, c_p, c_q \leq 2$.  For each possibility, as $2c_l \geq 2c_p + c_c$ implies $c_l \geq 2c_p$ and $c_l \geq c_p + c_q + c_\ell$, we may write $\alpha$ as a sum of $l -e_\ell -e_c$, $2l - e_p -e_q -2e_c$, $2l -e_p -2e_c$, and $l-e_c$ with nonnegative integer coefficients.  Thus we may assume $c_l > c_c$.

If $c_p \geq 2$, note that we always have $c_l \geq c_p + c_q + c_\ell \geq 2c_q + c_\ell$, $2c_l \geq 2c_p + c_c$, and $c_c \geq c_p + c_\ell$.  Since $c_c \geq 2$, we may subtract $(2l-e_p -e_q -2e_c)$ from $\alpha$ to obtain another nef class.  By repeating this process if necessary and using pseudosymmety, we may assume $c_p=1$.  Thus, we may assume $1 = c_l -c_c = c_q = c_p$.  

If $c_\ell > 0$, then as $c_l = 1 + c_c \geq 1 + c_p + c_\ell$, by subtracting integer multiples of $(l-e_c -e_\ell)$ from $\alpha$, we may reduce to the case $c_\ell = 0$.  Then if $c_c > 1$, we may subtract multiples of $(l-e_c)$ instead until $c_c = 1$.  Thus we may assume $\alpha = 2l - e_p -e_q -e_c$.
\end{proof}

\subsection*{5.2}

\textbf{Blow-up of two points and the strict transform of two lines in $\mathbb{P}^3$:}
Let $\ell_1$ and $\ell_2$ be disjoint lines in $\mathbb{P}^3$ and let $p$ and $q$ be points on $\ell_2$. $X$ is obtained from taking the blow-up of $\mathbb{P}^3$ along $\ell_1$ and $\ell_2$ and then blowing-up along the exceptional lines over $p$ and $q$. This is equivalent to taking the blow-up $g:Y \rightarrow \mathbb{P}^3$ along $p$ and $q$ and letting $f:X \rightarrow Y$ be the blow-up along the strict transforms $\ell_1'$ and $\ell_2'$ of $\ell_1$ and $\ell_2$.

\begin{thm}\label{5.2thm}
For all nonzero $\alpha \in \Nef_1(X)_\mathbb{Z}$, $\overline{\free}(X,\alpha)$ is irreducible and nonempty.
\end{thm}

\textbf{Generators for $N^1(X)$ and $N_1(X)$:}

\begin{center}
\begin{tabular}{ll}
 $H$ = hyperplane in $\mathbb{P}^3$, & $l$ = line in $\mathbb{P}^3$ \\ 
 $E_1$ = $f^{-1}(\ell_1')$, & $e_1$ = an $f$-fiber of a point in $\ell_1'$  \\  
 $E_2$ = $f^{-1}(\ell_2')$, & $e_2$ = an $f$-fiber of a point in $\ell_2'$  \\  
 $E_p$ = $(g \circ f)^{-1}(p)$ & $e_p$ = a line in $E_p$ \\
 $E_q$ = $(g \circ f)^{-1}(q)$ & $e_q$ = a line in $E_q$ \\
\end{tabular}
\end{center}

\textbf{Intersection Pairing:}
\begin{center}
\begin{tabular}{ll}
    $H \cdot l = 1$ &  $H \cdot e_j = 0$ \\
    $E_i \cdot l = 0$ &  $E_i \cdot e_j = -\delta_{ij}$ \\
\end{tabular}
\end{center}

\textbf{Anticanonical Divisor:}
\begin{align*}
    -K_X = 4H-2E_p-2E_q-E_1-E_2
\end{align*}

\textbf{Effective Divisors:}
The divisors $H-E_1-E_p$, $H-E_1-E_q$, $H-E_2-E_p-E_q$, $E_1$, $E_2$, $E_p$, and $E_q$ are all effective.

\begin{lem}\label{5.2core}
A core of free curves on $X$ is given by
\begin{align*}
    \mathscr{C}_X = \{ l, l-e_1, l-e_2, l-e_p, l-e_q, l-e_1-e_2, 2l-e_1-e_p-e_q \}
\end{align*}
\noindent There are no separating classes in $\mathscr{C}_X$.
\end{lem}

\begin{proof}

\textbf{Nef Curve Classes of Anticanonical degree between $2$ and $4$:}
If $\alpha = d l - m_1 e_1 - m_2 e_2 - m_p e_p - m_q e_q$ is a nef curve class of anticanonical degree between $2$ and $4$, then
\begin{align*}
    0\leq  m_1, m_2, m_p, m_q, & \hspace{.5cm}  m_1+m_p  \leq d, \hspace{.5cm} m_1+m_q \leq d, \hspace{.5cm} m_2+m_p+m_q \leq d, \\ & 2 \leq 4d-2m_p-2m_q-m_1-m_2 \leq 4.
\end{align*}

\noindent Solving these inequalities, we obtain the classes shown above in $\mathscr{C}_X$, as well as $2l-2e_1-2e_2$, $2l-e_1-e_2-e_p$, $2l-e_1-e_2-e_q$, $2l-e_1-e_2-e_p$, and $2l-e_1-e_p-e_q$.  However, the curves of class $2l-2e_1-2e_2$ are all double covers, and are therefore freely breakable. The curves of class $2l-e_1-e_2-e_p$ break freely into curves of class $l-e_1-e_2$ and $l - e_p$, while $2l-e_1-e_2-e_p$ breaks freely into $l-e_1-e_2$ and $l - e_p$, and $2l-e_p-e_q$ breaks freely into $l-e_p$ and $l - e_q$.

\textbf{Irreducible Spaces and Fibers:}  %
By Lemma \ref{lines_in_P3}, $\overline{\free}(X,\alpha)$ is irreducible for each $\alpha \in \mathscr{C}_X$.  Lemma \ref{lines_in_P3} also shows for each $\alpha \in \mathscr{C}_X\setminus \{l-e_1 -e_2\}$, $\overline{\free}_1(X,\alpha)\xrightarrow{\text{ev}} X$ has irreducible general fiber.  The fiber of $\overline{\free}_1(X,l -e_1 -e_2)\xrightarrow{\text{ev}} X$ over a general point $x\in X$ is the unique point corresponding to the line passing through $x$ that meets $\ell_1$ and $\ell_2$.
\end{proof}

\begin{lem}\label{5.2relations}
The relations in the monoid $\mathbb{N} \mathscr{C}_X$ are generated by:

\begin{enumerate}
    \item $(2l-e_1-e_p-e_q)+(l-e_2) = (l-e_1-e_2)+(l-e_p)+(l-e_q)$,
    \item $(2l-e_1-e_p-e_q)+(l) = (l-e_1)+(l-e_p)+(l-e_q)$, 
    \item $(l-e_1-e_2)+(l) = (l-e_1)+(l-e_2)$.
\end{enumerate}

The corresponding moduli spaces of free curves of class $3l-e_1-e_2-e_p-e_q$, $3l-e_1-e_p-e_q$, and $2l-e_1-e_2$ are irreducible.
\end{lem}

\begin{proof}
Consider a relation of the form $\sum \alpha_i = \sum \beta_j$ in $\mathbb{N} \mathscr{C}_X$. Note that, out of all curve classes in the core $\mathscr{C}_X$, the divisor $H-E_1-E_p-E_q$ pairs negatively with $2l-e_1-e_p-e_q$, positively with $l$ and with $l-e_2$, and to zero otherwise. Thus, we may use relations $(1)$ and $(2)$ to remove all instances of $2l-e_1-e_p-e_q$ from the relation. The remaining six curve class in $\mathscr{C}_X$ span a five-dimensional space, and so satisfy the single relation $(3)$.

Each curve of class $3l-e_1-e_2-e_p-e_q$ projects to a degree three curve in $\mathbb{P}^3$ intersecting $l_2$ three times, so it must be planar and hence singular at a point $r \in \mathbb{P}^3$. %
For a general point $r\in \mathbb{P}^3$, there is a 3 dimensional linear system of planar cubics singular at $r$ whose strict transforms have class $3l-e_1-e_2-e_p-e_q$.  In particular, the space of curves of class $3l-e_1-e_2-e_p-e_q$ is irreducible

No irreducible curve of class $3l -e_1 -e_p -e_q$ can be planar, for if one was, the plane it lies in would contain the line $\ell_2$, so the curve would need to intersect $\ell_2$ at a third point. Thus, each irreducible representative is a twisted cubic through $p,q$, and which intersects $\ell_1$.  %
We may therefore parameterize a dominant family of such twisted cubics by an open subset of $\ell_1 \times \mathbb{P}^3 \times \mathbb{P}^3 \times \mathbb{P}^3 $, where each copy of $\mathbb{P}^3 $ corresponds to the choice of a point for the cubic to pass through.

The space of free curves of class $2l-e_1-e_2$ is irreducible by Lemma \ref{lines_in_P3}.
\end{proof}

\begin{lem}
For all nonzero $\alpha \in \Nef_1(X)_\mathbb{Z}$, $\overline{\free}(X,\alpha)$ is nonempty.
\end{lem}
\begin{proof}
By Theorem \ref{Representability of Free Curves} and Lemma \ref{Gordan's Lemma}, we only need to check that if $\alpha = \sum_{\alpha_i \in \mathscr{C}_X} a_{\alpha_i} \alpha_i$ is an integral class with $0\leq a_{\alpha_i} <1$, $\alpha$ is represented by a free curve.  Writing $\alpha = c_l l - c_p e_p -c_q e_q -c_1 e_1 -c_2 e_2$, it follows that $c_p = c_q \leq 1$.  The only possibilities are easily shown to be represented by free curves.
\end{proof}

\subsection*{5.3-8, 4.10}

\textbf{$\mathbb{P}^1 \times S_d$:}
Let $X = \mathbb{P}^1 \times S_d$ for $1 \leq d \leq 7$ where $S=S_d$ is a del Pezzo surface of degree $d$. %

\begin{thm}\label{products}
If $X=\mathbb{P}^1 \times S$, then Conjecture \ref{GMC Tanimoto} holds for $X$.  Moreover, when %
\begin{itemize}
    \item $d\geq 3$: There is exactly one component of $\overline{\free}( \mathbb{P}^1 \times S, \alpha)$ for each nonzero $\alpha \in \Nef_1(X)_{\mathbb{Z}}$
    \item $d=2$: Every nonzero $\alpha \in \Nef_1(X)_{\mathbb{Z}}$ is represented by a free rational curve. $\overline{\free}(\mathbb{P}^1\times S, \alpha)$ is irreducible unless $\pi_{2*}\alpha = -nK_S$ for $n>1$, where $\pi_2$ is projection onto $S$.  In this latter case, there are exactly two components of $\overline{\free}(\mathbb{P}^1\times S, \alpha)$, only one of which parameterizes very free curves.
    \item $d=1$: Let $\pi_2:X\rightarrow S$ be the projection.  A nonzero $\alpha \in \Nef_1(X)_{\mathbb{Z}}$ is represented by a free rational curve iff $\pi_{2*} \alpha \neq -K_S$.  $\overline{\free}(\mathbb{P}^1\times S, \alpha)$ is irreducible unless $\pi_{2*}\alpha = -2nK_S$ for $n\geq 1$ or $\pi_{2*}\alpha = -n \phi^* K_{S'}$ for $n > 1$ and $\phi : S \rightarrow S'$ the contraction of a $(-1)$ curve.  In these cases, when $n > 1$ exactly one component of $\overline{\free}(\mathbb{P}^1 \times S, \alpha)$ parameterizes very free curves.
\end{itemize}
\end{thm}
\begin{proof}
As $\Mor(\mathbb{P}^1, \mathbb{P}^1) \times \Mor(\mathbb{P}^1, S) \cong \Mor(\mathbb{P}^1, X)$, Our claim follows from Theorem \ref{del pezzo curves thm}.  The Manin components of $\Mor(\mathbb{P}^1,X,\alpha)$ are precisely those which generically parameterize very free curves \cite[Theorem~6.2]{Lehmann_2017}.  A (unique) Manin component exists if and only if $-K_X . \alpha \geq 5$  and $\alpha \in \Nef_1(X)_{\mathbb{Z}}\setminus \partial \NE(X)$.
\end{proof}

\section{Picard Rank 4}\label{Section 9}

\subsection*{4.1}
\textbf{Blow-up of $(\mathbb{P}^1)^3$:}
 Let $X$ be the blow up of $(\mathbb{P}^1)^3$ along the intersection of two divisors of class $\mathcal{O}(1,1,1)$, an elliptic curve $c$.  This curve meets each fiber of a projection to $\mathbb{P}^1$ twice, and embeds into $\mathbb{P}^1\times \mathbb{P}^1$ as a curve of class $(2,2)$ under the product of any two projections.  Equivalently, this is a smooth divisor in $(\mathbb{P}^1)^4$ of class $(1,1,1,1)$.
 
\begin{thm}\label{GMC4.1}
For all $\alpha \in (l_1 + l_2 + l_3 -2e) + \Nef_1(X)$, $\overline{\free}(X,\alpha)$ is irreducible, nonempty, and parameterizes very free curves.
\end{thm}
 
\textbf{Generators for $N^1(X)$ and $N_1(X)$:} $H_i = \pi_i^*(\mathcal{O}(1))$ for $i=1,2,3$, where $\pi_i :X\rightarrow \mathbb{P}^1$ is the $i^{th}$ projection, and the exceptional divisor $E$ form a basis for $N^1(X)$.  Equivalently, $H_1 + H_2 + H_3 - E$ is the pullback of $\mathcal{O}(1)$ under the fourth projection to $\mathbb{P}^1$.  We make use of the resulting pseudosymmetry.

Similarly, let $l_i$, for $i=1,2,3$ be classes of section of $\pi_i$ and $e$ be the fiber of the blow up.  These form a basis for $N_1(X)$.  The pseudosymmetry mentioned above swaps $e$ with $l_i -e$ and stabilizes $l_j-e$ and $l_k -e$.

\textbf{Intersection Pairing:} $H_i \cdot l_j = \delta_{ij}$, $H_i \cdot e = 0$, $E \cdot l_j = 0$, $E \cdot e = -1$.

\textbf{Anticanonical Divisor:} $2(H_1 + H_2 + H_3) - E$

\textbf{Effective Divisors:} $H_1 + H_2 + H_3 -E$, $2(H_i + H_j) -E$, $E$, $H_i$.

\textbf{Effective Curves:} $l_i -e$ and $e$ generate the Mori cone.

\begin{lem}
A core set of free curves on $X$ is given by
\begin{align*}
    \mathscr{C}_X = \{ l_i , l_i + l_j -e, l_i + l_j - 2e, l_1 + l_2 + l_3 -3e \},
\end{align*}
where $i\neq j$. The only separating classes $\alpha \in \mathscr{C}_X$ are $l_i + l_j -e$ and $l_1 + l_2 + l_3 -3e$.
\end{lem}
\begin{proof}
\textbf{Nef Curve Classes of Anticanonical Degree Between $2$ and $4$:}
Note that the generators of $\overline{\text{NE}}(X)$ are linearly independent and of anticanonical degree one.  Thus, by Gordan's Lemma (Lemma \ref{Gordan's Lemma}) these generate the semigroup of integer points in $\overline{\text{NE}}(X)$.  To determine which integer points represent nef classes, we merely track intersections with divisors. Note curves of class $e$ sweep out $E$, and curves of class $l_i - e$ sweep out $2(H_j + H_k) - E$.  To obtain a nef class, we must add at least two different curves $e$ and $l_i -e$ (or $l_i -e$ and $l_j -e$).  Up to pseudosymmetry, these are identical cases.  We may add a third curve class to obtain $l_i + l_j -e$ or $l_1 + l_2 + l_3 -3e$.  Up to pseudosymmetry, adding four curve classes creates $2l_i$, $l_i + l_j$, $l_1 + l_2 + l_3 -2e$, or a class that is not nef.

\textbf{Freely Breakable Classes:} $2l_i$ is a double cover of a general fiber of $X\rightarrow \mathbb{P}^1 \times \mathbb{P}^1$, and therefore breaks.  Similarly, $l_i + l_j$ breaks into $l_i$ and  $l_j$.  Note $l_1 + l_2 + l_3 - 2e = (l_1) + (l_2 + l_3 -2e)$ is the class of a section of $\pi_1$, $f: \mathbb{P}^1\rightarrow X$, that embeds as a curve of class $(1,1)$ under composition with $\pi_2 \times \pi_3$.  Two of the 4 intersection points of $\pi_2 \times \pi_3 \circ f(\mathbb{P}^1)$ with $\pi_2 \times \pi_3 (c)$ are intersections of $f(\mathbb{P}^1)$ with $c$.  Let these two points be $\pi_2 \times \pi_3 \circ f(p)$ and $\pi_2 \times \pi_3 \circ f(q)$.  Since $c\rightarrow \pi_2 \times p_3 (c)$ is an isomorphism, this fixes the images of $p,q\in \mathbb{P}^1$ under $f$.  We may move $\pi_2 \times \pi_3 \circ f(p)$ and $\pi_2 \times \pi_3 \circ f(q)$ along $\pi_2 \times \pi_3 (c)$ until $f(p)$ and $f(q)$ lie in the same fiber of $\pi_1$.  Since a degeneration of $f$ to a map $g:\mathbb{P}^1\rightarrow X$ cannot meet two points in the same fiber of $\mathbb{P}^1$, this breaks $f$ into a map $g$ from a nodal curve $C$.  Generically, $C$ will have two components, one of class $l_1$ and the other of class $l_2 + l_3 -2e$.

\textbf{Irreducible Spaces and Fibers:}  By pseudosymmetry, we only need to examine $\alpha = l_1$ and $\alpha = l_1 + l_2 -e$.  Clearly, $\overline{\free}_1(X,l_1) \cong X$, and the forgetful map $\overline{\free}_1(X,l_1) \rightarrow \overline{\free}(X,l_1)$ is the projection $\pi_2 \times \pi_3$.  Similarly, $\overline{\free}(X,l_1 + l_2 -e)$ parameterizes maps into a fiber of $\pi_3$, meeting one of its two intersection points with $c$.  The monodromy group of $\pi_3 : c\rightarrow \mathbb{P}^1$ is transitive, so the choice of intersection point does not produce two irreducible components of $\overline{\free}(X,l_1 + l_2 -e)$.  However, this does mean $\text{ev}_{l_1 + l_2 -e}$ has disconnected fibers.  Hence, so does $\text{ev}_{l_1 + l_2 + l_3 - 3e}$, as $l_1 + l_2 + l_3 - 3e$ is pseudosymmetric to $l_1 + l_2 -e$.
\end{proof}

\begin{lem}
Up to pseudosymmetry, relations in the monoid $\mathbb{N}\mathscr{C}_X$ are generated by: %
\begin{enumerate}
    \item $l_1 + (l_2 + l_3 -2e) = (l_1 + l_2 -2e) + l_3$,
    \item $l_1 + (l_2 + l_3 -e) = (l_1 + l_2 -e) + l_3$,
    \item $l_1 + l_2 + (l_1 + l_3 -2e) = (l_1 + l_2 -e) + (l_1 + l_3 -e)$,
    \item $l_1 + l_2 + (l_1 + l_2 -2e) = 2(l_1 + l_2 -e)$.
\end{enumerate}
For each relation $\sum \alpha_i = \sum \alpha_j'$, a main component of $\prod_X \overline{\free}_2(X,\alpha_i)$ lies in the same component of free curves as a main component of $\prod_X \overline{\free}_2(X,\alpha_j')$.
\end{lem}
\begin{proof}
\textbf{Relations:}
As stated above, each element of $\mathscr{C}_X$ is a sum of two or three distinct elements of $\{ e, l_1 -e, l_2 -e, l_3 -e\}$, the generators of the Mori cone of $X$.  There is a natural, faithful action of $S_4$ on $\{ e, l_1 -e, l_2 -e, l_3 -e\}$ that extends to $N_1(X)$.  Partition $\mathscr{C}_X$ into its orbits $\mathscr{C}_X^2 = \{l_i , l_i + l_j -2e\}$ and $\mathscr{C}_X^3 = \{ l_i + l_j -e, l_1 + l_2 + l_3 -3e\}$ under this action.  We claim all relations between elements of $\mathscr{C}_X$ decompose into sums 
\begin{enumerate}
    \item $l_1 + l_2 + l_3 -2e = \alpha_1 + \alpha_2 = \alpha_3 + \alpha_4$ for  $\alpha_i \in \mathscr{C}_X^2$,
    \item $l_1 + l_2 + l_3 -2e + c = \alpha_1 + \beta_1 = \alpha_2 + \beta_2$ for $\alpha_i \in \mathscr{C}_X^2, \ \beta_i \in \mathscr{C}_X^3,$ and $c\in\{ e, l_1 -e, l_2 -e, l_3 -e\},$
    \item $l_1 + l_2 + l_3 -2e + c_1 + c_2 = \beta_1 + \beta_2 = \alpha_1 + \alpha_2 + \alpha_3$ for $\beta_i \in \mathscr{C}_X^3, \ \alpha_i \in \mathscr{C}_X^2,$ and distinct $c_i \in \{ e, l_1 -e, l_2 -e, l_3 -e\}$, or
    \item $l_1 + l_2 + l_3 -2e + 2c = \beta_1 + \beta_2 = \alpha_1 + \alpha_2 + \alpha_3$ for $\beta_i \in \mathscr{C}_X^3, \ \alpha_i \in \mathscr{C}_X^2,$ and $c \in \{ e, l_1 -e, l_2 -e, l_3 -e\}$.
\end{enumerate}
Indeed, the four elements of $\mathscr{C}_X^3$ are linearly independent, so there is no relation between them.  Conversely, any two distinct $\beta_1, \beta_2 \in \mathscr{C}_X^3$ sum to an expression of type (3), while $2\beta_1$ is only expressible as a sum of type (4).  Hence, we may assume at most one $\beta_i$ appears on each side any other relation.  As each curve class in $\mathscr{C}_X^d$ has $(-K_X)$-degree $d$, the remaining relations involving $\beta_i \in \mathscr{C}_X^3$ are of the form $\beta_1 + \sum a_i \alpha_i = \beta_2 + \sum b_i \alpha_i$, with $\alpha_i \in \mathscr{C}_X^2$, $a_i, b_i \geq 0$, and $\sum a_i = \sum b_i$.  If there is no $a_i > 0$ such that $\beta_1 + \alpha_i = l_1 + l_2 + l_3 -2e + c$ for some $c\in\{ e, l_1 -e, l_2 -e, l_3 -e\},$ it would follow that $\beta_1 = \beta_2$.  Thus, (2) may be used to reduce all relations involving at least one $\beta_i \in \mathscr{C}_X^3$ to a relation $\sum a_i \alpha_i = \sum b_i \alpha_i$ between elements of $\mathscr{C}_X^2$, with $a_i, b_i \geq 0$ and $a_i b_i =0$.  In such an expression, there must exist $a_i, a_j > 0$ such that $\alpha_i + \alpha_j = l_1 + l_2 + l_3 -2e$.  Otherwise, there exists $c \in \{ e, l_1 -e, l_2 -e, l_3 -e\} $ such that $\sum a_i \alpha_i = (\sum a_i ) c + \sum a_i (\alpha_i - c)$, with $\alpha_i - c \in \{ e, l_1 -e, l_2 -e, l_3 -e\}$ for all $a_i > 0$.  This would imply $a_i = b_i$.  This proves our claim.  It follows that up to symmetry, each relation is of the indicated form.

\textbf{Main Components:}  The claim for Relations (1) - (3) follows from the fact that the corresponding space of free curves is irreducible.  For Relation (4), the component of $\overline{\free}(X, 2(l_1 + l_2 -e))$ whose general points correspond to maps $\mathbb{P}^1 \rightarrow X$ meeting $c$ at two distinct points contains free chains of both indicated types.
\end{proof}

\begin{lem}
Each nonzero $\alpha \in \Nef_1(X)_{\mathbb{Z}}$ is the class of a free rational curve.
\end{lem}
\begin{proof}
This follows from the observation each $\alpha \in \overline{\text{NE}}(X)_\mathbb{Z}$ is the class of a nodal rational curve of class $a_0 e + \sum_{i = 1}^3 a_i (l_i - e)$.  Those $\alpha \in \Nef_1(X)_{\mathbb{Z}}$ are represented by nodal rational curves with $a_i\leq \sum_{j\neq i} a_j$ for all $i$.  We may always obtain a free rational curve by smoothing pairs or triples of these components.
\end{proof}

\subsection*{4.2}
\textbf{Blow-up of the Cone over a Quadric Surface:} Let $S$ be a smooth quadric surface in $\mathbb{P}^3$, $C\subset \mathbb{P}^4$ the cone over $S$, and $X$ be the blow up of $C$ along its vertex $p$ and an elliptic curve $c$ on $S$.  We let $C'\cong \mathbb{P}_S(\mathcal{O}\oplus \mathcal{O}(1,1))$ be the blow up of $C$ along $p$.

\begin{thm}\label{GMC4.2}
Let $\tau = l_1 + l_2 -ke$ for $k= 0,1,$ or $2$.  For each $\alpha \in \tau + \Nef_1(X)$, $\overline{\free}(X,\alpha)$ is irreducible, nonempty, and parameterizes very free curves.
\end{thm}

\textbf{Generators for $N^1(X)$ and $N_1(X)$:}We let $f$ denote the class of a general fiber of $X\rightarrow S$, $e$ denote the class of a fiber of $X\rightarrow C'$ over $c\subset S$, and $l_1$ and $l_2$ denote the classes of general lines in a moving section not equal to $S$.  These generate $N_1(X)$.

Let $H_1, H_2$ denote the pullbacks from $S \cong \mathbb{P}^1 \times \mathbb{P}^1$ of $\mathcal{O}_S(1,0)$ and $\mathcal{O}_S(0,1)$.  These are the strict transforms of planes in $C$.  Let $E_0$ denote the rigid section of $C'\rightarrow S$ (the exceptional divisor over $p$), and $E_\infty$ the strict transform of $S\subset C'$.  Let $E$ be the exceptional divisor over the elliptic curve $c$.  The other component of the preimage of $c$ under $X\rightarrow S$ we denote by $E'$.  These divisors generate $N^1(X)$, though not freely.  The two relations are $E+E' = 2(H_1 + H_2)$ and $E_0 + H_1 + H_2 = E_\infty + E$, which may be seen from the embedding of $C$ in $\mathbb{P}^4$.

\textbf{Intersection Pairing:}
\begin{tabular}{llll}
    $H_1 \cdot l_1 = 0$ &  $H_1 \cdot l_2 = 1$ & $H_1 \cdot e = 0$ & $H_1 \cdot f = 0$, \\
    $H_2 \cdot l_1 = 1$ &  $H_2 \cdot l_2 = 0$ & $H_2 \cdot e = 0$ & $H_2 \cdot f = 0$, \\
    $E \cdot l_1 = 0$ &  $E \cdot l_2 = 0$ & $E \cdot e = -1$ & $E \cdot f = 0$, \\
    $E' \cdot l_1 = 2$ &  $E' \cdot l_2 = 2$ & $E' \cdot e = 1$ & $E' \cdot f = 0$, \\
    $E_0 \cdot l_1 = 0$ &  $E_0 \cdot l_2 = 0$ & $E_0 \cdot e = 0$ & $E_0 \cdot f = 1$, \\
    $E_\infty \cdot l_1 = 1$ &  $E_\infty \cdot l_2 = 1$ & $E_\infty \cdot e = 1$ & $E_\infty \cdot f = 1$
\end{tabular}

\textbf{Anticanonical Divisor:} $-K_X = 2(H_1 + H_2) + E_0 + E_\infty$

\textbf{Effective Divisors:}  The extreme effective divisors on $X$ are $H_1, H_2, E_0, E_\infty, E, E'$

\textbf{Effective Curves:} $\NE(X)$ is generated by $e, f-e, l_i - f, l_i - 2e$

\textbf{Pseudosymmetry:} The contraction of either $e$ or $f-e$ realizes $X$ as the blow up of a variety isomorphic to $C'$ \cite{matsuki1995weyl} \cite{mori1981classification}.  This pseudoaction of $X\rightarrow C'$ swaps $E$ with $E'$ ($e$ with $f-e$) and swaps $E_0$ with $E_\infty$ ($l_i -f$ with $l_i -2e$).  This corresponds to the following pseudosymmetry on $N_1(X)$:
$$e\rightarrow f-e, \ f\rightarrow f, \ l_i \rightarrow l_i + f - 2e$$
For our purposes, the roles of $l_1$ and $l_2$ are interchangeable as well.

\begin{lem}
A core of free curves on $X$ is given by 
\begin{align*}
    \mathscr{C}_X =& \{ l_i -e, \ f , \ l_i + f -2e , \ l_i \}
\end{align*}
The only separating classes $\alpha \in \mathscr{C}_X$ are $l_i -e$.%
\end{lem}
\begin{proof}

\textbf{Nef Curve Classes of Anticanonical Degree Between $2$ and $4$:}
The only curve classes of appropriate anticanonical degree, which pair nonnegatively with all divisors above, appear below:
\begin{align*}
    & l_i -e, \hspace{.5cm} f, \hspace{.5cm} l_i + f - 2e, \hspace{.5cm} l_i \\&
    2f, \hspace{.5cm} 2(l_i -e), \hspace{.5cm} l_1 + l_2 -2e, \hspace{.5cm} l_i + f -e
\end{align*}
\textbf{Freely Breakable Classes:}
Of these, $2f$, $2(l_i -e)$, $l_1 + l_2 -2e$, and $l_i + f - e$ always break into a union of two free curves: $2f$ is never irreducible; $2(l_i -e)$ is a conic in the plane $H_i \subset C$ lying over a line in $S$, which meets $c$ twice; $l_1 + l_2 -2e$ is the class of a conic in $\mathbb{P}^4$ passing through $c$ twice--it is the complete intersection of $C$ with a plane containing two points of $c$--the space of such planes is irreducible, so we may deform a general free curve of class $(l_1 + l_2 -2e)$ to a free chain of type $(l_1 - e, l_2 -e)$; and $l_i + f -e$ is the class of a conic in $H_i$ meeting $p$ and a point of intersection $H_i \cap c$, which we split into two curves of class $l_i-e$ and $f$ by deforming the conic in $H_i$.

\noindent \textbf{Irreducible Spaces and Fibers:}
\begin{itemize}
    \item $l_i-e$ is the class of a line in the plane $H_i$ through one of its two intersection points with $c$.  By \ref{Monodromy}, since $c\rightarrow \mathbb{P}|H_1|$ is a $2-1$ covering of the line $\mathbb{P}|H_1|$, the two choices of points in $H_i \cap c$ belong to the same component of $\overline{\free}(X, l_i -e)$.  However, since fixing a point $x \in X$ also fixes the plane $H_i$, the fiber of $\text{ev}_{l_i -e}^{-1}(x)$ has two components;
    \item $f$ is the class of a generic fiber of $X\rightarrow S$, parameterized by the open subset $S\setminus c$.  The fiber $\text{ev}_{f}^{-1}(x)$ is a single point;
    \item $l_i +f - 2e$ is the class of a conic in $H_i$ passing through the two intersection points $H_i \cap c$ and the vertex of $p$.  The fiber $\text{ev}_{l_i + f  -2e}^{-1}(x)$ is a net of conics;
    \item $l_i$ is the class of a general line in $H_i$. The fiber $\text{ev}_{l_i}^{-1}(x)$ is a pencil of lines.
\end{itemize}  
\end{proof}

\begin{lem}
Relations in the monoid $\mathbb{N}\mathscr{C}_X$ are generated by the following list:
\begin{enumerate}
    \item $l_i + (l_j -e) = (l_i -e) + l_j$,
    \item $l_i + (l_i + f -2e) = 2(l_i -e) + f$,
    \item $l_i + (l_j + f -2e) = (l_i - e) + (l_j -e) + f$,
    \item $(l_i + f -2e) + (l_j -e) = (l_j + f - 2e) + (l_i -e)$.
\end{enumerate}
For each relation $\sum \alpha_i = \sum \alpha_j'$, a main component of $\prod_X \overline{\free}_2(X,\alpha_i)$ lies in the same component of free curves as a main component of $\prod_X \overline{\free}_2(X,\alpha_j')$.
\end{lem}

\begin{proof}
\textbf{Relations:}
In accordance with \ref{Relations}, order the nef curves above as $l_1, l_1 + f -2e, \ldots$, and consider the following divisors:
\begin{itemize}
    \item $D = -H_2 + E:$ The relevant nonzero pairings are $l_1 . D = -1$, $l_2 - e . D = 1$, $l_1 + f -2e . D = 1$, $l_2 + f - 2e . D = 2$.  We find the relations $l_1 + (l_2 -e) = (l_1 - e) + l_2$; $l_1 + (l_1 + f - 2e) = 2(l_1 -e) + f$; $l_1 + (l_2 + f -2e) = (l_1-e) + (l_2 -e) + f$.
    \item $D= 2H_1 + H_2 - E$: The relevant nonzero pairings are $D.(l_1 + f - 2e) = -1$, $D. l_2 = 2$, $D. l_2 -e = 1$.  We obtain the relations $(l_1 + f -2e) + l_2 = (l_1 -e) + (l_2 -e) + f$; $(l_1 + f - 2e) + (l_2 -e) = (l_1 -e) + (l_2 + f -2e)$.

\end{itemize}
The five remaining classes span $N_1(X)$, and the unique relation among them is $2(l_2 - e) + f = (l_2 + f -2e) + l_2$.  Thus, we must check the following:
\begin{itemize}
    \item $l_1 + (l_2 -e) = (l_1 - e) + l_2$: This is the class of a complete intersection $C\cap \mathbb{P}^2 \subset \mathbb{P}^4$, where we impose the plane $\mathbb{P}^2$ passes through a point of $c$.  Therefore, the space of free curves of this class is irreducible.
    \item $l_i + (l_i + f - 2e) = 2(l_i -e) + f$:  we find a path from one type of curve to the other.  We may assume all relevant curves are in the same plane $H_i$.  We first move $l_i$ to meet an intersection point $c\cap H_i$.  The resulting curve has components of class $l_i -e$, $e$, and $l_i + f - 2e$.  We smooth the unique node joining the last two components, and obtain a curve of class $l_i + f -e$.  Earlier we argued this could be broken into $(l_i -e) + f$, finishing our proof. 
    \item %
    $l_1 + (l_2 + f -2e) = (l_1-e) + (l_2 -e) + f = l_2 + (l_1 + f -2e)$: we find a path from one type of curve to the other.  We break $l_1$ into a nodal curve of class $l_1 - e$ + $e$, which meets $l_2 + f -2e$ at one of the two places it intersects $c\cap H_2$.  We'd like to smooth $l_2 + f - 2e$ and $e$ into a single free curve of class $l_2 + f - e$; however, this tricky to do while preserving an intersection with $l_1 - e$.  It is necessary to simultaneously change the plane $H_2$ containing $l_2 + f - 2e$ so that its intersection with $l_1 -e$ moves off of $c \cap H_2$ (this may also be accomplished by changing the plane $H_1$ containing $l_1 - e$).  After, as argued earlier we may split $l_2 -f + e$ into two free curves of class $l_2 -e$ and $f$.
    \item %
    $(l_1 + f - 2e) + (l_2 -e) = (l_1 -e) + (l_2 + f -2e)$: contracting $E'$ realizes $X$ as a blow up of $C'$ along an elliptic curve in a moving section as well.  This description swaps $E'$ with  $E$ and $E_0$ with $E_\infty$.  This also acts nontrivially on our chosen basis for $N_1(X)$:  $e$ is sent to $f-e$, $f$ to $f$, and $l_i$ to $l_i + f -2e$.  This changes $l_1 + l_2 +f - 3e$ to $l_1 + l_2 -e$.  We have already shown this space of free curves is irreducible.
\end{itemize}
\end{proof}

\begin{lem}
For each nonzero $\alpha \in \Nef_1(X)_\mathbb{Z}$, $\overline{\free}(X,\alpha)$ is nonempty.
\end{lem}
\begin{proof}
By Lemma \ref{Gordan's Lemma} and \ref{Representability of Free Curves}, we only need to check that linear combinations $\sum a_{\alpha_i} \alpha_i$ of elements $\alpha_i \in \mathscr{C}_X$ with coefficents $0\leq a_{\alpha_i} < 1$ are all representable by free rational curves.  Let $\alpha = c_{l_1} l_1 + c_{l_2} l_2 + c_f f - c_e e = \sum a_{\alpha_i} \alpha_i$  It follows that $c_{l_i} = a_{l_i} + a_{l_i -e} + a_{l_i + f -2e} \in \{0,1,2\}$, $c_f = a_f + a_{l_1 + f -2e} + a_{l_2 + f -2e} \in \{0,1,2\}$, and $c_e = a_{l_1 -e} + a_{l_2 -e} + 2a_{l_1 + f -2e} + 2a_{l_2 + f -2e} \in \{0,1,2,3,4,5\}$.  If $c_f = 0$, then $c_{l_i} \in \{0,1\}$, and the only class that can appear is $l_1 + l_2 - e$.  %
If $c_f = 1$, then $c_e \leq 3$.  Since $\alpha. E_\infty \geq 0$ and $\alpha . E' \geq 0$, $c_e \leq c_f + c_{l_1} + c_{l_2}$ and $c_e \leq 2c_{l_1} + 2c_{l_2}$.  For all possible values of $c_e$, we may express $\alpha$ as a sum of the classes of two free curves, one of which pairs to 0 with $E$.  If $c_f = 2$, as before we may express $\alpha$ as a sum of the classes of two free rational curves.  For instance, if $c_e = 5$, $c_{l_1} + c_{l_2} \geq 3$, and the free curve pairing nontrivially with $E$ could have class $3l_i + 2f - 5e = 2(l_i + f - 2e) + (l_i -e)$ or $l_i + 2l_j + 2f -5e$.
\end{proof}

\subsection*{4.3}

\textbf{Blow-up of a curve of class $(1,1,2)$ in $(\mathbb{P}^1)^3$:} Let $f:X \rightarrow (\mathbb{P}^1)^3$ be the blow up of $(\mathbb{P}^1)^3$ along a curve $c$ of class $(1,1,2)$. For $i=1,2,3$, let $p_i: (\mathbb{P}^1)^3 \rightarrow \mathbb{P}^1$ be the projection onto the corresponding factor and let $f_i = p_i \circ f$.

\begin{thm}\label{4.3thm}
For each $\alpha \in l_3 + \Nef_1(X)_\mathbb{Z}$, $\overline{\free}(X,\alpha)$ is irreducible and nonempty.
\end{thm}

\textbf{Generators for $N^1(X)$ and $N_1(X)$:} 
Let

\begin{center}
\begin{tabular}{ll}
 $H_i$ = $f_i^{-1}(pt)$, & $l_i$ = $(f_j \times f_k)^{-1}(pt \times pt)$ where $\{i,j,k\} = \{1,2,3\}$ \\ 
 $E$ = $f^{-1}(c)$, & $e$ = an $f$-fiber of a point in $c$  \\  
\end{tabular}

\end{center}

for $i=1,2,3$ where the points are generic.

\textbf{Intersection Pairing:}
\begin{center}
\begin{tabular}{ll}
    $H_i \cdot l_j = \delta_{ij}$ &  $H_i \cdot e = 0$ \\
    $E \cdot l_j = 0$ &  $E \cdot e = -1$ \\
\end{tabular}
\end{center}

\textbf{Anticanonical Divisor:}
\begin{align*}
    -K_X = 2H_1 + 2H_2 + 2H_3 - E
\end{align*}

\textbf{Effective Divisors:}
The divisors $H_1+H_2-E$, $2H_1+H_3-E$, $2H_2+H_3-E$, $H_1$, $H_2$, $H_3$, and $E$ are all effective.

\begin{lem}\label{4.3core}
A core of free curves on $X$ is given by
\begin{align*}
    \mathscr{C}_X = \{ l_1, l_2, l_3, l_1+l_2-e, l_1+l_3-e, l_2+l_3-e, l_1+l_2-2e \}.
\end{align*}

\noindent The only separating class in $\mathscr{C}_X$ is $l_1 + l_2 - e$.
\end{lem}

\begin{proof}

\textbf{Nef Curve Classes of Anticanonical degree between $2$ and $4$:}
If $\alpha = d_1 l_1 + d_2 l_2 + d_3 l_3 -n e$ is a nef curve class of anticanonical degree between $2$ and $4$, then
\begin{align*}
    0 \leq n, d_1, d_2, d_3 \hspace{.5cm} & n \leq d_1+d_2, \hspace{.5cm} n \leq 2d_1+d_3, \hspace{.5cm} n \leq 2d_2+d_3, \\ & 2 \leq 2d_1+2d_2+2d_3-n \leq 4.
\end{align*}

\noindent Solving these inequalities, we obtain the classes shown above in $\mathscr{C}_X$, as well as $2l_1, 2l_2, 2l_3, l_1+l_2, l_1+l_3, l_2+l_3, 2 l_1+l_2 - 2e, l_1+2l_2-2e, 2l_1 + 2l_2-4e$, and $l_1+l_2+l_3-2e$.

The curves of class $2l_i$ are double covers and break freely.  The curves of class $l_i+l_j$ for $i \neq j$ freely break as a sum of two curves of class $l_i$ and $l_j$ respectively, so we need not include them in the core. We show that the curves of class $2 l_1+l_2 - 2e$ likewise break as a sum of two curves of class $l_1+l_2-2e$ and $l_1$ respectively. We first recall a well-known result which will often be used throughout this case.

\begin{lem}\label{surface lemma}
The blow-up of $\mathbb{P}^1 \times \mathbb{P}^1$ at a point $p$ is isomorphic to the blow-up of $\mathbb{P}^2$ at two points $p_1$ and $p_2$ under a map which identifies the classes $l_1 \mapsto l-e_1$, $l_2 \mapsto l-e_2$ and $e \mapsto l-e_1-e_2$ in the usual notation.
\end{lem}

Take a curve $x$ of class $2l_1+l_2-2e$ in $X$. It lies in a fiber of $f_3$ which is isomorphic to $\mathbb{P}^1 \times \mathbb{P}^1$ blown-up at two points $p$ and $q$. We may assume $x$ intersects both $p$ and $q$ with multiplicty one. Applying Lemma \ref{surface lemma} to identify this fiber with the blow-up of $\mathbb{P}^2$ at three points $p_1, p_2$ and $q$, we see that $x$ is a curve of class
$$2(l-e_1)+(l-e_2)-(l-e_1-e_2)-e_q = 2l-e_1-e_q.$$
We may then break $x$ as a sum of curves of class $l-e_1$ and $l-e_q$. Viewing the surface again as a blow-up of $\mathbb{P}^1 \times \mathbb{P}^1$, these curves are of class $l_1+l_2-2e$ and $l_1$ as desired.

Arguing similarly, the curves of class $l_1+2l_2-2e$ break freely as a sum of curves of class $l_1+l_2-e$ and $l_2$.  
Next, consider the curves of class $2l_1+2l_2-4e$. They lie in a fiber of $f_3$. Again viewing this fiber as the blow-up of $\mathbb{P}^2$ at points $p_1, p_2$ and $q$, we see that a generic such curve is of class
\begin{align*}
    2(l-e_1)+2(l-e_2)-2(l-e_1-e_2)-2e_q = 2l-2e_q
\end{align*}
on the blow-up of $\mathbb{P}^2$. Thus they are double covers, and thus break freely.

Finally, we claim that any curve $x$ of class $l_1+l_2+l_3-2e$ breaks freely as a sum of curves of class $l_1+l_2-2e$ and $l_3$. Indeed, let $p$ and $q$ be the points at which the image $f(x)$ of $x$ in $(\mathbb{P}^1)^3$ meets $c$. The fiber $p_3^{-1}(p_3(p))$ meets $c$ at $p$, as well as another point $r$. We may deform $x$, fixing the point $p$ and moving its second intersection point with $c$ along $c$ until it coincides with $r$. Generically this breaks $x$ as a sum of two curves, necessarily of classes $l_1+l_2-2e$ and $l_3$ as desired.

\textbf{Irreducible Spaces and Fibers:}  The curves of class $l_i$ are parameterized by an open subset of $\mathbb{P}^1 \times \mathbb{P}^1$ for $i=1,2,3$.  It is clear $\overline{\free}_1(X,l_i)\xrightarrow{\text{ev}}X$ has irreducible fibers over general points.   
The space of free curves of class $l_1+l_3-e$ fibers over $\mathbb{P}^1$ by composition with $f_2$. Fibers of $f_2$ are blow-ups of $\mathbb{P}^1 \times \mathbb{P}^1$ at a point, so there is a two-dimensional linear system of curves of class $l_1+l_3-e$. The same argument works for the curves of class $l_2+l_3-e$.  $\overline{\free}_1(X,l_i + l_3 -e)\xrightarrow{\text{ev}}X$ has irreducible one-dimensional fibers corresponding to pencils of curves.    

We may similarly fiber the space of free curves of class $l_1+l_2-e$ by composition with $f_3$. However, in this case, the fibers of $f_3$ are the blow-up of $\mathbb{P}^1 \times \mathbb{P}^1$ at two points, so there are two separate two-dimensional linear systems of the appropriate type. Nonetheless, we may apply Theorem \ref{Monodromy}a to the linear system $H_3$; this shows that the monodromy action is transitive. Hence $\overline{\free}(X,l_1 + l_2 -e)$ is irreducible. $\overline{\free}_1(X,l_1 + l_2 -e)\xrightarrow{\text{ev}}X$ has reducible fibers over general points corresponding to two different pencils of curves.  

Lastly, $\overline{\free}(X,l_1 + l_2 -2e)$ fibers over $\mathbb{P}^1$ via composition with $f_3$.  Each fiber is a pencil of curves.  The class $l_1+l_2-2e$ is not separating because there is a unique curve of class $l_1+l_2-2e$ through any general point.
\end{proof}

\begin{lem}\label{4.3relations}
The relations in the monoid $\mathbb{N} \mathscr{C}_X$ are generated by:

\begin{enumerate}
    \item $(l_1+l_2-e)+(l_3) = (l_1+l_3-e)+(l_2) = (l_2+l_3-e)+(l_1)$,
    \item $(l_1+l_2-2e)+(l_1)+(l_2) = 2(l_1+l_2-e)$,
    \item $(l_1+l_2-2e)+(l_1)+(l_3) = (l_1+l_2-e)+(l_1+l_3-e)$,
    \item $(l_1+l_2-2e)+(l_2)+(l_3) = (l_1+l_2-e)+(l_2+l_3-e)$,
    \item $(l_1+l_2-2e)+2(l_3) = (l_1+l_3-e)+(l_2+l_3-e)$.
\end{enumerate}
\noindent The corresponding moduli spaces of free curves of class $l_1+l_2+l_3-e, 2l_1+l_2+l_3-2e, l_1+2l_2+l_3-2e$, and $l_1+l_2+2l_3-2e$ are irreducible, while the space of free curves of class $2l_1+2l_2-2e$ has two components, one of which parameterizes double covers. 
\end{lem}

\begin{proof}
Throughout the proof, we shall consider the intersection pairings of the divisors $H_1+H_2+H_3-2E, H_1-E$, and $H_2-E$ with classes in $\mathscr{C}_X$. Explicitly, $H_1+H_2+H_3-2E$ pairs to $-1$ with $l_1+l_2+l_3-2e$, to $1$ with $l_1, l_2, l_3$, and to zero otherwise. The divisor $H_1-E$ pairs to $-1$ with $l_1+l_2+l_3-2e, l_2+l_3-e$, to $1$ with $l_1$, and to zero otherwise.

Consider a relation of the form $\sum \alpha_i = \sum \beta_j$ in $\mathbb{N} \mathscr{C}_X$. Suppose $l_1+l_2-2e$ appears in the relation; reduce to the case where it appears on the left side only. By taking the product with $H_1+H_2+H_3-2E$, we see there are at least two instances of $l_1, l_2$, and $l_3$ on the left side as well. Apply relations $(2)-(5)$ to remove $l_1+l_2-2e$ from the relation when possible. If not, there must be either two instances of $l_1$ or two instances of $l_2$. Suppose we are in the first case. By considering the pairing with $H_1-E$, we see that $l_2+l_3-e$ appears on the left side as well. Thus we may use relation $(1)$ to replace the occurrence of $l_1$ with $l_2$ and then use $(2)$ to remove $l_1+l_2-2e$ altogether. In the case where $l_2$ appears twice, we may argue similarly, considering instead the pairing with $H_2-E$.

Thus, we may assume $l_1+l_2-2e$ does not appear in the relation at all. Next, if $l_2+l_3-e$ appears, say on the left side only, we take the intersection pairing with $H_1-E$ to see that $l_1$ appears as well. Then use relation $(1)$ to remove all instances of $l_2+l_3-e$. Similarly, by looking at the pairing with $H_2-E$, we may remove $l_1+l_3-e$ from the relation by using $(1)$ without introducing any new instances of $l_2+l_3-e$.

In sum, we have reduced to the problem of looking for relations between $l_1, l_2, l_3$, and $l_1+l_2-e$. But these four classes are linearly independent, so there are none. Hence, our list of relations is indeed a generating set.

The curves of class $l_1+l_2+l_3-e$ are all graphs of a morphism $\mathbb{P}^1 \rightarrow \mathbb{P}^1 \times \mathbb{P}^1$. They are fibered over $c$ with fiber over a given point the space of all such morphisms with two linear conditions on the coefficients, corresponding to the condition that the corresponding graph passes through the given point. This is an irreducible space.

The curves of class $2l_1+2l_2-2e$ are fibered over $\mathbb{P}^1$ by the fiber of $f_3$ in which they lie. Identifying these fibers with the blow-up of $\mathbb{P}^2$ at three points, we see that all such curves which are not double covers are of class
\begin{align*}
    2(l-e_1)+2(l - e_2)-(l-e_1-e_2) - e_q = 3l - e_1 - e_2 - e_q.
\end{align*}
\noindent in the blow-up of the plane. The space of rational cubics passing through three points in the plane is irreducible.

The curves $x$ of class $2l_1+l_2+l_3-2e$ are fibered over $\Sym^2(c)$ by the pair of points $(p,q)$ on $c$ which lie on $f(x)$. The image such a curve under $f$ may be viewed as the graph of a morphism $\mathbb{P}^1 \rightarrow \mathbb{P}^1 \times \mathbb{P}^1$ of bidegree $(2,1)$. Here, we view the domain as the third factor in $(\mathbb{P}^1)^3$ and the range as the first and second factors respectively. The condition that this graph pass through both $p$ and $q$ imposes four linear conditions on the coefficients of the maps. Thus, the fibers are irreducible of dimension four. It follows that the moduli space of curves of class $2l_1+l_2+l_3-2e$ is irreducible.

Identical arguments show that the space of curves of class $l_1+2l_2+l_3-2e$ and $l_1+l_2+2l_3-2e$ are irreducible as well.
\end{proof}

\begin{lem}
For each nonzero $\alpha \in \Nef_1(X)_\mathbb{Z}$, $\overline{\free}(X,\alpha)$ is nonempty.
\end{lem}
\begin{proof}
By Lemma \ref{Gordan's Lemma} and Theorem \ref{Representability of Free Curves}, we only need to check that integral linear combinations $\sum a_{\alpha_i} \alpha_i$ of elements $\alpha_i \in \mathscr{C}_X$ with coefficents $0\leq a_{\alpha_i} < 1$ are all representable by free rational curves.  Let $\alpha = c_1 l_1 + c_2 l_2 + c_3 l_3 - c_e e = \sum a_{\alpha_i} \alpha_i$.  It follows that $c_3 \in \{0,1,2\}$, $0 \leq c_1, c_2 \leq 3$, and $0 \leq c_e \leq 4$.  Since $\alpha \in \Nef_1(X)$, we have $c_e \leq c_1 + c_2$ and $c_e \leq 2c_i + c_3$ for $i=1,2$.  By symmetry we may assume $c_2 \geq c_1$.  It follows that we may write $\alpha = \beta + a l_1 + b l_2$ with $a,b\geq 0$ and $\beta$ given by one of the following:
\begin{itemize}
    \item $c_3 l_3 + l_1 + k l_2 - re$ for $k \leq 1 + c_3$ and $r \leq k + 1$
    \item $c_3 l_3 + 2l_1 + k l_2 - re$ for $k \leq 2 + c_3$ and $r \leq k + 2$
    \item $c_3 l_3 + 3l_1 + k l_2 - re$ for $k \leq 3$ and $r \leq k + 3$
\end{itemize}
Each such integral $\beta$ is a nonnegative linear combination of $l_2 + l_3 -e$, $l_1 + l_2 -2e$, $l_1 + l_2 -e$, $l_1$, and $l_2$. 
\end{proof}

\subsection*{4.4}

\textbf{Blow-up of a Conic and Points in a Quadric Threefold:}
Let $Y \rightarrow Q$ be the blow-up of a conic where $Q \subset \mathbb{P}^4$ is a smooth quadric. Let $X$ be the blow-up of $Y$ with center two exceptional lines.  %
We apply \ref{blowup} to case $5.1$.%

\subsection*{4.5}

\textbf{Blow-up of Curves in $\mathbb{P}^1 \times \mathbb{P}^2$:} Let $f:X\rightarrow \mathbb{P}^1\times \mathbb{P}^2$ be a blow up with center two disjoint curves $c_1,c_2$ of bidegree $(2,1)$ and $(1,0)$, respectively.  Let $\pi_i : X\rightarrow \mathbb{P}^i$ be the composition of $f$ with the corresponding projection.

\begin{thm}\label{GMC4.5}
For all $\alpha \in l_1 + \Nef_1(X)_\mathbb{Z}$, $\overline{\free}(X,\alpha)$ is nonempty and irreducible.
\end{thm}

\textbf{Generators for $N^1(X)$ and $N_1(X)$:} Let $H_i = \pi_i^*\mathcal{O}_{\mathbb{P}^i}(1)$, and $E_i$ be the exceptional divisor over $c_i$. Let $l_1$ be a general fiber of $\pi_2$, $l_2$ be a general line in a fiber of $\pi_1$, and $e_i$ be a fiber of $f$ over a point in $c_i$.  %
We have
\begin{align*}
    N^1(X) = \mathbb{Z} \cdot H_1 + \mathbb{Z} \cdot H_2 + \mathbb{Z} \cdot E_1 + \mathbb{Z} \cdot E_2, \hspace{.5cm}
    N_1(X) = \mathbb{Z} \cdot l_1 + \mathbb{Z} \cdot l_2 + \mathbb{Z} \cdot e_1 + \mathbb{Z} \cdot e_2.
\end{align*}

\textbf{Intersection Pairing:} $ H_i \cdot l_j = \delta_{ij}$, $E_i \cdot e_j = -\delta_{ij}$, and all other pairings are $0$.

\textbf{Anticanonical Divisor:} $-K_X = 2H_1 + 3H_2 - E_1 - E_2.$

\textbf{Effective Divisors:} $H_1, H_2, E_1, E_2, H_2-E_1, H_2-E_2$ are effective.  So is $H_1 + 2H_2 -E_1 -2E_2$.

\begin{lem}
A core of free curves on $X$ is given by
$$\mathscr{C}_X = \{ l_1, \ l_2, \ l_2-e_1, \ l_2-e_2, \ 2l_2-2e_1-e_2, \ l_1+l_2-e_1-e_2\}.$$
The only separating class $\alpha \in \mathscr{C}_X$ is $l_2 -e_1.$
\end{lem}
\begin{proof}
\textbf{Nef Curve Classes of Anticanonical Degree Between $2$ and $4$:}
If $\alpha = a l_1 + b l_2 - c e_1 -de_2$ is nef, then $a,b,c,d \geq 0$, $c \leq b$, $d \leq b$, and $a + 2b \geq c + 2d$. Provided
$$2 \leq -K_X \cdot \alpha = 2a+3b-c-d \leq 4,$$
the only possibilities for $\alpha$ are listed below:
\begin{align*}
    & l_1, \hspace{.5cm} l_2, \hspace{.5cm} 2l_1, \hspace{.5cm} l_2-e_1, \hspace{.5cm} l_2-e_2\\
    & l_1+l_2-e_1, \hspace{.5cm} l_1+l_2-e_2, \hspace{.5cm} l_1+l_2-e_1-e_2 \\&
    2l_2-2e_1, \hspace{.5cm} 2l_2-2e_2, \hspace{.5cm} 2l_2-e_1-e_2, \hspace{.5cm} 2l_2-2e_1-e_2
\end{align*}

\textbf{Freely Breakable Classes:}  We may break free curves of class $2l_1$, $l_1 + l_2 -e_i$, $2l_2 -2e_i$, and $2l_2 -e_1 -e_2$ into a chain of two free curves.  The space $\overline{\free}(X,2l_2 -2e_1)$ has two components, one of which parameterizes double covers.  Otherwise, the corresponding space of free curves is irreducible.  It follows that each component of free curves of the indicated classes contains free chains of length 2.

\textbf{Irreducibility of $\overline{\free}(X,\alpha)$ and fibers of $\text{ev}_\alpha$ over general $x\in X$:}
Clearly, $\overline{\free}(X,l_i)$ is irreducible and $\text{ev}_{l_i}$ has irreducible fibers.  These facts are also clear for $\alpha = l_2 -e_2$ and $\alpha = 2l_2 -2e_1 -e_2$.  However, while $\overline{\free}(X,l_2-e_1)$ is irreducible, the fibers of $\text{ev}_{l_2 -e_1}$ generally have two components, corresponding to the choice of an intersection point of a fiber of $\pi_1$ with $c_1$.  Curves $f:\mathbb{P}^1\rightarrow X$ of class $l_1 + l_2 -e_1 -e_2$ are sections of $\pi_1$ mapping isomorphically under $\pi_2$ onto a line in $\mathbb{P}^2$ meeting the point $\pi_2(c_2)$.  The image $\pi_2 \circ f (\mathbb{P}^1)$ also intersects $\pi_2(c_1)$ at a point $q$, which is the image of $q' \in c_1$.  The space of such lines is irreducible, and the only constraint on the map $\pi_2 \circ f$ is that $\pi_2 \circ f ( \pi_1(q')) = q$, that is, $f(\mathbb{P}^1)\cap c_1 = q'$.  This shows the space of maps has irreducible fibers over an irreducible domain, and is hence irreducible.  Similarly, $\text{ev}_{l_1 + l_2 -e_1 -e_2}$ has irreducible fibers.
\end{proof}

\begin{lem}
Relations in the monoid $\mathbb{N}\mathscr{C}_X$ are generated by:
\begin{enumerate}
    \item $l_2 + (2l_2 -2e_1 -e_2) = 2(l_2 -e_1) + (l_2 -e_2)$,
    \item $l_2 + (l_1 + l_2 -e_1 -e_2) = l_1 + (l_2 -e_1) + (l_2 -e_2)$,
    \item $(l_2 -e_1) + (l_1 +l_2 -e_1 -e_2) = l_1 + (2l_2 -2e_1 -e_2)$.
\end{enumerate}
For each relation $\sum \alpha_i = \sum \alpha_j'$, a main component of $\prod_X \overline{\free}_2(X,\alpha_i)$ lies in the same component of free curves as a main component of $\prod_X \overline{\free}_2(X,\alpha_j')$.
\end{lem}
\begin{proof}
\textbf{Relations:}  We apply Lemma \ref{Relations}.
\begin{itemize}
    \item Consider $D=-H_2 + E_1 + E_2$: $D$ only pairs negatively with $l_2$, and positively with $2l_2 - 2e_1 -e_2$ and $l_1 + l_2 -e_1 -e_2$.  The relations we obtain are
    \begin{itemize}
        \item $l_2 + (2l_2 -2e_1 -e_2) = 2(l_2 -e_1) + (l_2 -e_2)$
        \item $l_2 + (l_1 + l_2 -e_1 -e_2) = l_1 + (l_2 -e_1) + (l_2 -e_2)$
    \end{itemize}
    \item The remaining 5 curves span $N_1(X)$, and the only relation among them is $(l_2 -e_1) + (l_1 +l_2 -e_1 -e_2) = l_1 + (2l_2 -2e_1 -e_2)$
\end{itemize}

\textbf{Main Components:}
$\overline{\free}(X, l_1 + 2l_2 - ke_1 -e_2)$ is irreducible for $k=1,2$; in both instances, the space of curves is fibered over the space of conics in $\mathbb{P}^2$ passing through $\pi_2(c_2)$.  %
The fiber over a single conic has one or two components, corresponding to a choice of two or one intersection points with $\pi_2(c_1)$.  By Theorem \ref{Monodromy}, the each component of this fiber lie in the same component of $\overline{\free}(X,l_1 + 2l_2 - ke_1 -e_2)$, proving its irreducibility.

There are two components of $\overline{\free}(X, 3l_2 -2e_1 -e_2)$, corresponding to whether or not the curve is singular along intersection with $c_1$.  To deform a free chain of type $(l_2, 2l_2 - 2e_1 -e_2)$ to a free chain of type $(l_2 -e_1, l_2 -e_1, l_2 -e_2)$, it suffices to note that there exist chains of both types in the component of $\overline{\free}(X, 3l_2 -2e_1 -e_2)$ parameterizing curves nonsingular at their intersections with $c_1$.
\end{proof}

\begin{lem}
For each $\alpha \in \Nef_1(X)_{\mathbb{Z}}$, $\overline{\free}(X,\alpha)$ is nonempty.
\end{lem}

\begin{proof}
By Lemma \ref{Gordan's Lemma} and \ref{Representability of Free Curves}, we only need to check that integral linear combinations $\sum a_{\alpha_i} \alpha_i$ of elements $\alpha_i \in \mathscr{C}_X$ with coefficents $0\leq a_{\alpha_i} < 1$ are all representable by free rational curves.  Let $\alpha = c_{l_1} l_1 + c_{l_2} l_2 - c_{e_1} e_1 - c_{e_2} e_2 = \sum a_{\alpha_i} \alpha_i$.  It follows that $c_{l_1} = a_{l_1} + a_{l_1 + l_2 -e_1 -e_2}\in \{0,1\}$.  If $c_{l_1} = 0$, then $c_{e_2} = a_{l_2 -e_2} + a_{2l_2 -2e_1 -e_2} \in \{0,1\}$.  If $c_{e_2} = 0$ as well, $\alpha$ must be $0$.  If $c_{l_1} = 0$ and $c_{e_2} = 1$, we must have $a_{l_2} = 1 - a_{l_2 -e_2}$, so that $\alpha = 3l_2 -2e_1 -e_2$ or $\alpha = 2l_2 - e_1 -e_2$.  Thus, we may assume $c_{l_1} = 1$.

When $c_{l_1} = 1$, $c_{e_2} = a_{l_2 -e_2} + a_{2l_2 -2e_1 -e_2} + a_{l_1 + l_2 -e_1 -e_2} \in \{1,2\}$.  Similarly, $c_{e_1} = a_{l_2 -e_1} + 2a_{2l_2 -2e_1 -e_2} + a_{l_1 + l_2 -e_1 -e_2} \in \{1,2,3\}$.  Note $c_{l_2} \geq \max(c_{e_1},c_{e_2})$ as well.  If $c_{e_2} = 1$, $c_{e_1} \leq 2$, and each possible curve is representable by a free curve.  If $c_{e_2} = 2$, $c_{l_2} \geq 1 + c_{e_1}$, and each possible curve is representable by a free curve.
\end{proof}

\subsection*{4.6}

\textbf{Blow-up of three lines in $\mathbb{P}^3$:}
Let $f:X \rightarrow \mathbb{P}^3$ be a blow-up with center three disjoint lines $c_1, c_2,$ and $c_3$ in $\mathbb{P}^3$.

\begin{thm}\label{4.6thm}
For each nonzero $\alpha \in \Nef_1(X)_\mathbb{Z}$, $\overline{\free}(X,\alpha)$ is nonempty and irreducible.
\end{thm}

\textbf{Generators for $N^1(X)$ and $N_1(X)$:} Let

\begin{center}
\begin{tabular}{ll}
 $H$ = hyperplane in $\mathbb{P}^3$, & $l$ = line in $\mathbb{P}^3$ \\ 
 $E_i$ = $f^{-1}(c_i)$, & $e_i$ = an $f$-fiber of a point in $c_i$  \\  
\end{tabular}
\end{center}

\textbf{Intersection Pairing:}
\begin{center}
\begin{tabular}{ll}
    $H \cdot l = -1$ &  $H \cdot e_j = 0$ \\
    $E_i \cdot l = 0$ &  $E_i \cdot e_j = -\delta_{ij}$ \\
\end{tabular}
\end{center}

\textbf{Anticanonical Divisor:}
\begin{align*}
    -K_X = 4H-E_1-E_2-E_3
\end{align*}

\textbf{Effective Divisors:}
There is a unique quadric containing any three disjoint lines in $\mathbb{P}^3$. Hence, the divisors $2H-E_1-E_2-E_3$, $H-E_1$, $H-E_2$, $H-E_3$, $E_1$, $E_2$, and $E_3$ are all effective.

\begin{lem}\label{4.6core}
A core of free curves on $X$ is given by
\begin{align*}
    \mathscr{C}_X = \{ l, l-e_i, l-e_i-e_j \}
\end{align*}

\noindent where $i \neq j$. There are no separating classes in $\mathscr{C}_X$.
\end{lem}

\begin{proof}

\textbf{Nef Curve Classes of Anticanonical degree between $2$ and $4$:}
If $\alpha = a l - b e_1 -c e_2 -d e_3$ is a nef curve class of anticanonical degree between $2$ and $4$, then
\begin{align*}
    0\leq b,c,d \leq a, \hspace{.5cm} b+c+d \leq 2a, \hspace{.5cm} 2 \leq 4a-b-c-d \leq 4.
\end{align*}

\noindent Solving these inequalities, we obtain the classes shown above in $\mathscr{C}_X$, as well as $2l-2e_i-e_j-e_k$ for $i,j,k$ distinct.

However, the curves of class $2l-2e_i-e_j-e_k$ are freely breakable. Indeed, any such curve lies in a plane $A$ containing $c_i$ and it meets the unique point in each of $A \cap e_j$ and $A \cap e_k$. It may be broken to a union of curves of class $l-e_i-e_j$ and $l-e_i-e_k$ respectively.

\textbf{Irreducible Spaces and Fibers:}  The irreducibility of $\overline{\free}(X,\alpha)$ for each $\alpha \in \mathscr{C}_X$ follows from Lemma \ref{lines_in_P3}.  Lemma \ref{lines_in_P3} also shows that for $\alpha$ such that $-K_X . \alpha \geq 3$, the fibers of $\text{ev}:\overline{\free}_1(X,\alpha)\rightarrow X$ are irreducible.  For $\alpha$ of anticanonical degree 2, this fact follows because there is a unique line of class $l-e_i-e_j$ passing through any general point $p \in \mathbb{P}^3$.
\end{proof}

\begin{lem}\label{4.6relations}
The relations in the monoid $\mathbb{N} \mathscr{C}_X$ are generated by:

\begin{enumerate}
    \item $(l-e_i-e_j)+(l) = (l-e_i)+(l-e_j)$,
    \item $(l-e_i-e_j)+(l-e_k) = (l-e_i-e_k)+(l-e_j)$.
\end{enumerate}

\noindent for $i,j,k$ distinct. The corresponding moduli spaces of curves of class $2l-e_i-e_j$ and $2l-e_i-e_j-e_k$ are irreducible.
\end{lem}

\begin{proof}

Consider a relation of the form $\sum \alpha_i = \sum \beta_j$ in $\mathbb{N} \mathscr{C}_X$. Note that, out of all curve classes in the core $\mathscr{C}_X$, the divisor $H-E_1-E_2$ pairs negatively with $l-e_1-e_2$, positively with $l$ and with $l-e_3$, and to zero otherwise. Thus, any occurrence of $l-e_1-e_2$ is accompanied by an appearance of $l$ or $l-e_3$ on the same side of the relation. Use relations $(1)$ and $(2)$ in the case $i=2, j=1, k=3$ to remove all instances of $l-e_1-e_2$ from the relation. 

Now we assume that the relation does not contain $l-e_1-e_2$. Then $H-E_1-E_3$ pairs negatively with $l-e_1-e_3$, positively with $l$ and with $l-e_2$, and to zero otherwise. Thus, we may use relations $(1)$ and $(2)$ in the case $i=3, j=1, k=2$ to remove all instances of $l-e_1-e_3$ without introducing any new curves of class $l-e_1-e_2$.

Therefore we assume the relation contains only $l$, $l-e_i$, and $l-e_2-e_3$. These five class span a four-dimensional space, so satisfy a single relation $(l-e_2-e_3)+(l)=(l-e_2)+(l-e_3)$. This proves that $(1)$ and $(2)$ generate the monoid of relations.

The free curves of class $2l-e_i-e_j$ are conics contained in an arbitrary plane $A$ meeting the two points $A \cap c_i$ and $A \cap c_j$. Thus, they are parameterized by an open subset of a $\mathbb{P}^3$-bundle over $\mathbb{G}(1,2)$.  Similarly, the curves of class $2l-e_i-e_j-e_k$ are parameterized by an open subset of a $\mathbb{P}^2$-bundle over $\mathbb{G}(1,2)$.
\end{proof}

\begin{lem}
For each nonzero $\alpha \in \Nef_1(X)_\mathbb{Z}$, $\overline{\free}(X,\alpha)$ is nonempty.
\end{lem}

\begin{proof}
By Theorem \ref{Representability of Free Curves} and Lemma \ref{Gordan's Lemma}, we only need to check that integral linear combinations $\sum a_{\alpha_i} \alpha_i$ of elements $\alpha_i \in \mathscr{C}_X$ with coefficents $0\leq a_{\alpha_i} < 1$ are all represented by free rational curves.  The possibilities for $\alpha$ are, up to symmetry, $4l - 2e_1 -2e_2 -2e_3$, $3l -2e_1 -2e_2 -e_3$ $3l -2e_1 -e_2 -e_3$, $2l-e_1 -e_2 -e_3$, and $2l -e_1 -e_2$.  All are clearly represented by free curves.
\end{proof}

\subsection*{4.7}
\textbf{Blow-up of curves of bidegree $(1,0)$ and $(1,1)$ in $\mathbb{P}^1 \times \mathbb{P}^2$:}
Let $W \subset \mathbb{P}^2 \times \mathbb{P}^2$ be a smooth divisor of bidegree $(1,1)$ and let $X \rightarrow W$ be the blow-up with center two disjoint curves $x_1$ and $x_2$ of bidegree $(0,1)$ and $(1,0)$. We proceed to find an alternate description of $W$, more suitable for our purposes.

By \cite[Table~4]{mori1981classification}, $X$ is the blow-up of $Y = \mathbb{P}^1 \times \mathbb{F}_1$ along a smooth irreducible curve $c$. We claim that $c$ has bidegree $(1,1)$ in the projection to $\mathbb{P}^1 \times \mathbb{P}^2$. By \cite{mori1981classification}, $b_3(X) = b_3(Y) = 0$ and $(-K_X)^3 = 36$, $(-K_Y)^3 = 48$.  Thus by Lemma \ref{blow-up numbers} $g(c) = 0$ and that $-K_Y . c = 5$. The only two possibilities for the class of $c$ are $l_1+l_2$ and $2l_2-e$, where $e$ is the exceptional curve of the blow-up $Y\rightarrow \mathbb{P}^1 \times \mathbb{P}^2$.  However, the blow-up of a curve of class $2l_2-e$ in $Y$ is not Fano, as such a curve would meet the locus in $Y$ swept out by anticanonical lines transversely.  

Summing up, we have obtained the following description of $X$ which we will use throughout the case: it is the blow-up $f:X \rightarrow \mathbb{P}^1 \times \mathbb{P}^2$ of $\mathbb{P}^1 \times \mathbb{P}^2$ along disjoint curves $c_1$ and $c_2$ of bidegree $(1,0)$ and $(1,1)$ respectively. Let $f_1:X \rightarrow \mathbb{P}^1$ and $f_2:X \rightarrow \mathbb{P}^2$ be $f$ composed with the corresponding projections.

\begin{thm}\label{4.7thm}
For all nonzero $\alpha \in \Nef_1(X)_\mathbb{Z}$, $\overline{\free}(X,\alpha)$ is nonempty and irreducible.
\end{thm}

\textbf{Generators for $N^1(X)$ and $N_1(X)$:}
Let

\begin{center}
\begin{tabular}{ll}
 $H_1$ = $f_1^*(pt)$, & $l_1$ = $\mathbb{P}^1 \times pt$ \\ 
 $H_2$ = $f_2^*(l)$, & $l_2$ = $pt \times l$ \\ 
 $E_i$ = $f^{-1}(c_i)$, & $e_i$ = an $f$-fiber of a point in $c_i$  \\  
\end{tabular}
\end{center}

where $l$ is the class of a line in $\mathbb{P}^2$.

\textbf{Intersection Pairing:}
\begin{center}
\begin{tabular}{ll}
    $H_i \cdot l_j = \delta_{ij}$ &  $H_i \cdot e_j = 0$ \\
    $E_i \cdot l_j = 0$ &  $E_i \cdot e_j = -\delta_{ij}$ \\
\end{tabular}
\end{center}

\textbf{Anticanonical Divisor:}
\begin{align*}
    -K_X = 2H_1+3H_2-E_1-E_2
\end{align*}

\textbf{Effective Divisors:}
The divisors $H_2-E_1$, $H_2-E_2$, $H_1+H_2-E_1-E_2$, $E_1$, and $E_2$ are effective

\begin{lem}\label{4.7core}
A core of free curves on $X$ is given by
\begin{align*}
    \mathscr{C}_X = \{ l_1, l_2, l_2-e_1, l_2-e_2, l_1+l_2-e_1-e_2 \}.
\end{align*}

\noindent There are no separating classes in $\mathscr{C}_X$.
\end{lem}

\begin{proof}

\textbf{Nef Curve Classes of Anticanonical degree between $2$ and $4$:}
If $\alpha = a l_1 + b l_2 -c e_1 -d e_2$ is a nef curve class of anticanonical degree between $2$ and $4$, then
\begin{align*}
    0\leq c \leq b, \hspace{.5cm} 0 \leq d \leq b, \hspace{.5cm} c+d \leq a+b, \hspace{.5cm} 2 \leq 2a+3b-c-d \leq 4.
\end{align*}

\noindent Solving these inequalities, we obtain the classes shown above in $\mathscr{C}_X$, as well as $2l_1$, $2l_2-e_1-e_2$, $2l_2-2e_1$, $2l_2-2e_2$, $l_1+l_2-e_1$, and $l_1+l_2-e_2$.

The curves of class $2l_1, 2l_2-2e_1$, and $2l_2-2e_2$ are all double covers. Moreover, it is clear that the curves of class $2l_2-e_1-e_2$ are freely breakable.

Pick a curve $x$ of class $l_1+l_2-e_1$. Taking the limit of the $\mathbb{C}^*$-action on $X$ fixing two fibers of $f_1$, one of which contains the intersection point of $x$ with the exceptional divisor $E_1$, we see that $x$ freely breaks as a sum of two curves of class $l_1$ and $l_2-e$. A similar argument shows that the curves of class $l_1+l_2-e_2$ are freely breakable as well.

\textbf{Irreducible Spaces and Fibers:}
The free curves of class $l_1$ are parameterized by an open subset of $\mathbb{P}^2$.  The free curves of class $l_2$ are parameterized by an open subset of $\mathbb{P}^1 \times \mathbb{G}(1,2)$.  The space of free curves of class $l_2-e_1$ and $l_2-e_2$ each fiber over $\mathbb{P}^1$ by their image under the map $f_1$, with fiber consisting of an open subset of a pencil of lines.  The above descriptions show these classes are not separating.

The free curves of class $l_1+l_2-e_1-e_2$ are precisely the graphs of degree one maps $g: \mathbb{P}^1 \rightarrow \mathbb{P}^2$ such that there exists $p \in c_1$ with $g(f_1(p)) = f_2(p)$ and $q \in c_2$ with $g(f_1(q)) = f_2(q)$. Thus, the space of such curves fibers over $c_1 \times c_2$ with fiber consisting of the space of degree one maps $\mathbb{P}^1 \rightarrow \mathbb{P}^2$ satisfying four linear conditions on the coefficients. Thus $\overline{\free}(X,l_1 + l_2 -e_1 -e_2)$ is irreducible.

Pick a general point $r \in X$. Any curve of class $l_1+l_2-e_1-e_2$ passing through $r$ projects under $f_2$ to the line $x$ in $\mathbb{P}^2$ between the points $f_2(c_1)$ and $f_2(r)$. Note that there exists a unique point $(s,t) \in c_2 \subset \mathbb{P}^1 \times \mathbb{P}^2$ for which $t \in x$. Then the free curves of class $l_1+l_2-e_1-e_2$ through $r$ may be identified with the graphs of degree one morphisms $g:\mathbb{P}^1 \rightarrow x \subset \mathbb{P}^2$ for which $g(f_1(r)) = f_2(r)$ and $g(s) = t$. This space is isomorphic to $\mathbb{P}^1$. In particular, $l_1+l_2-e_1-e_2$ is not a separating class.
\end{proof}

\begin{lem}\label{4.7relations}
The relations in the monoid $\mathbb{N} \mathscr{C}_X$ are generated by:
\begin{align*}
    (l_1+l_2-e_1-e_2)+(l_2) = (l_2-e_1)+(l_2-e_2)+(l_1).
\end{align*}
\noindent The corresponding moduli spaces of curves of class $l_1+2l_2-e_1-e_2$ is irreducible.
\end{lem}

\begin{proof}
Since the core consists of five linearly dependent classes which span a four-dimensional space, it is clear that the monoid of relations is generated by the single relation shown above.

The curves of class $l_1+2l_2-e_1-e_2$ are the graphs of degree two maps $g: \mathbb{P}^1 \rightarrow \mathbb{P}^2$ such that there exists $p \in c_1$ with $g(f_1(p)) = f_2(p)$ and $q \in c_2$ with $g(f_1(q)) = f_2(q)$. Thus, they are fibered over $c_1 \times c_2$ with fiber consisting of the space of degree two maps $\mathbb{P}^1 \rightarrow \mathbb{P}^2$ satisfying four linear conditions on the coefficients.
\end{proof}

\begin{lem}
For each nonzero $\alpha \in \Nef_1(X)_\mathbb{Z}$, $\overline{\free}(X,\alpha)$ is nonempty.
\end{lem}
\begin{proof}
By Theorem \ref{Representability of Free Curves} and Lemma \ref{Gordan's Lemma}, we only need to check that integral linear combinations $\sum a_{\alpha_i} \alpha_i$ of elements $\alpha_i \in \mathscr{C}_X$ with coefficents $0\leq a_{\alpha_i} < 1$ are all represented by free rational curves.  The only nonzero possibility for $\alpha$ is $l_1 + 2l_2 -e_1 -e_2$, which is represented by a free chain of type $(l_1 + l_2 -e_1 -e_2, l_2)$.
\end{proof}

\subsection*{4.8}

\textbf{Blow-up of a curve of tridegree $(0,1,1)$ in $(\mathbb{P}^1)^3$:}
Let $f:X \rightarrow \mathbb{P}^1 \times \mathbb{P}^1 \times \mathbb{P}^1$ be a blow-up with center a curve $c$ of tridegree $(0,1,1)$. For $i=1,2,3$, let $p_i : X \rightarrow \mathbb{P}^1$ be $f$ composed with projection onto the corresponding factor of $(\mathbb{P}^1)^3$.

\begin{thm}\label{4.8thm}
For all nonzero $\alpha \in \Nef_1(X)_{\mathbb{Z}}$, $\overline{\free}(X,\alpha)$ is nonempty and irreducible.
\end{thm}

\textbf{Generators for $N^1(X)$ and $N_1(X)$:} 
Let

\begin{center}
\begin{tabular}{ll}
 $H_i$ = $p_i^{-1}(pt)$, & $l_i$ = $(p_j \times p_k)^{-1}(pt \times pt)$ where $\{i,j,k\} = \{1,2,3\}$ \\ 
 $E$ = $f^{-1}(c)$, & $e$ = an $f$-fiber of a point in $c$  \\  
\end{tabular}

\end{center}
for $i=1,2,3$ where the points are generic.

\textbf{Intersection Pairing:}
\begin{center}
\begin{tabular}{ll}
    $H_i \cdot l_j = \delta_{ij}$ &  $H_i \cdot e = 0$ \\
    $E \cdot l_j = 0$ &  $E \cdot e = -1$ \\
\end{tabular}
\end{center}

\textbf{Anticanonical Divisor:}
\begin{align*}
    -K_X = 2H_1 + 2H_2 + 2H_3 - E
\end{align*}

\textbf{Effective Divisors:}
$H_1-E$, $H_2+H_3-E$, $H_2$, $H_3$, and $E$ are effective.

\begin{lem}\label{4.8core}
A core of free curves on $X$ is given by
\begin{align*}
    \mathscr{C}_X = \{ l_1, l_2, l_3,
    l_1+l_2-e, l_1+l_3-e \}.
\end{align*}
\noindent There are no separating classes in $\mathscr{C}_X$.
\end{lem}

\begin{proof}

\textbf{Nef Curve Classes of Anticanonical degree between $2$ and $4$:}
If $\alpha = a l_1 + b l_2 +c l_3 -d e$ is a nef curve class of anticanonical degree between $2$ and $4$, then
\begin{align*}
    b,c,d \geq 0, \hspace{.5cm} d \leq a, \hspace{.5cm} d \leq b + c, \hspace{.5cm} 2 \leq 2a+2b+2c-d \leq 4.
\end{align*}

\noindent Solving these inequalities, we obtain the classes shown above in $\mathscr{C}_X$, as well as $2l_i$ and $l_i + l_j$ for $i,j\in \{1,2,3\}$.  Curves of class $2l_i$ are all double covers, and are thus freely breakable.  Similarly, $\overline{\free}(X,l_i +l_j)$ is irreducible, as each fiber under the map induced by $p_k$ (for $k\neq i,j$) is the linear system of divisors of class $(1,1)$ in $\mathbb{P}^1 \times \mathbb{P}^1 = p_k^{-1}(pt)$.  Thus, as $\overline{\free}(X,l_i +l_j)$ contains free chains of type $(l_i,l_j)$, such curves are freely breakable.

\textbf{Irreducible Spaces and Fibers:}
The free curves of class $l_i$ are parameterized by an open subset of $\mathbb{P}^1 \times \mathbb{P}^1$.  Fibers of $\overline{\free}_1(X,l_i)\xrightarrow{\text{ev}}X$ are single points.  
The free curves of class $l_1+l_2-e$ may be parameterized by picking a point $p \in \mathbb{P}^1$ and a curve of class $(1,1)$ in $\mathbb{P}^1 \times \mathbb{P}^1 \times \{ p \}$ passing through the unique point $q = \mathbb{P}^1 \times \mathbb{P}^1 \times \{ p \} \cap c$. These curves come from the linear system corresponding to the embedding of $\mathbb{P}^1 \times \mathbb{P}^1$ in $\mathbb{P}^3$ composed with the projection to $\mathbb{P}^2$ from $q$. Hence such curve classes are parameterized by $\mathbb{P}^1 \times \mathbb{P}^2$. An identical argument works for the other class $l_1+l_3-e$.  For $i=2,3$, this argument also identifies fibers of $\overline{\free}_1(X,l_1 + l_i -e)\xrightarrow{\text{ev}}X$ as a pencil of curves given by the intersections of $\mathbb{P}^1 \times \mathbb{P}^1 \subset \mathbb{P}^3$ with planes through two fixed points.
\end{proof}

\begin{lem}\label{4.8relations}
The relations in the monoid $\mathbb{N} \mathscr{C}_X$ are generated by:
\begin{align*}
    (l_1+l_2-e)+(l_3) = (l_1+l_3-e)+(l_2).
\end{align*}
The moduli space of curves of class $l_1+l_2+l_3-e$ is irreducible.
\end{lem}

\begin{proof}
The five classes in $\mathscr{C}_X$ span a four-dimensional space, so they satisfy a single relation: $(l_1+l_2-e)+(l_3) = (l_1+l_3-e)+(l_2)$.  The curves of class $l_1+l_2+l_3-e$ are precisely the curves of class $(1,1,1)$ in $(\mathbb{P}^1)^3$ passing through $c$. We will understand the space of such curves as a bundle over $c$ with fiber over $p \in c$ being the curves of class $(1,1,1)$ in $(\mathbb{P}^1)^3$ passing through $p$. Such curves are graphs of morphisms $\mathbb{P}^1 \rightarrow \mathbb{P}^1 \times \mathbb{P}^1$ of bidegree $(1,1)$. The requirement that the graph pass through $p$ imposes two linear conditions on the coefficients, so the fiber is indeed irreducible.
\end{proof}

\begin{lem}
For each nonzero $\alpha \in \Nef_1(X)_\mathbb{Z}$, $\overline{\free}(X,\alpha)$ is nonempty.
\end{lem}
\begin{proof}
By Theorem \ref{Representability of Free Curves} and Lemma \ref{Gordan's Lemma}, we only need to check that integral linear combinations $\sum a_{\alpha_i} \alpha_i$ of elements $\alpha_i \in \mathscr{C}_X$ with coefficents $0\leq a_{\alpha_i} < 1$ are all represented by free rational curves.  The only nonzero possibility for $\alpha$ is $l_1 + l_2 + l_3 -e$, which we've shown to be the class of a free curve.
\end{proof}

\subsection*{4.9}
\textbf{Blow-up of Two Lines and an Exceptional Curve in $\mathbb{P}^3$:}
Let $Y \rightarrow \mathbb{P}^3$ be the blow-up of two disjoint lines on $\mathbb{P}^3$ and let $X$ be the blow-up of $Y$ with center an exceptional line. We apply \ref{blowup} to case $5.2$.%

\subsection*{4.10}
\textbf{$X = \mathbb{P}^1 \times S_d$ for $d = 7$:}
See Theorem \ref{products}.

\subsection*{4.11}

\textbf{Blow-up of Curves in $\mathbb{P}^1 \times \mathbb{P}^2$:}
Let $X$ be the blow-up of $\mathbb{P}^1 \times \mathbb{F}_1$ with center $t \times e$ where $t \in \mathbb{P}^1$ and $e \subset \mathbb{F}_1$ is the exceptional curve. We apply \ref{blowup} to case $5.2$.%

\subsection*{4.12}

\textbf{Blow-up of Curves in $\mathbb{P}^3$:}
Let $Y \rightarrow \mathbb{P}^3$ be the blow-up of a line on $\mathbb{P}^3$ and let $X$ be the blow-up of $Y$ with center two exceptional lines. We apply \ref{blowup} to case $5.1$.%

\subsection*{4.13}

\textbf{Blow-up of a curve of class $(1,1,3)$ in $(\mathbb{P}^1)^3$:} Let $X \rightarrow (\mathbb{P}^1)^3$ be the blow up of $(\mathbb{P}^1)^3$ along a curve $c$ of class $(1,1,3)$. For $i=1,2,3$, let $\pi_i : X \rightarrow \mathbb{P}^1$ be the blow-up composed with projection onto the corresponding factor of $\mathbb{P}^1$. 

\begin{thm}\label{GMC4.13}
For all $\alpha \in l_3 + \Nef_1(X)$, $\overline{\free}(X,\alpha)$ is irreducible and nonempty.
\end{thm}

\textbf{Generators for $N^1(X)$ and $N_1(X)$:} Let $H_i = \pi_i^*\mathcal{O}(1)$, $E$ be the exceptional divisor over $c$.  Let $l_i$ be a fiber of $\pi_j \times \pi_k$, and $e$ a fiber of $E\rightarrow c$.

\textbf{Intersection Pairing:}
\begin{center}
\begin{tabular}{ll}
    $H_i \cdot l_j = \delta_{ij}$ &  $H_i \cdot e = 0$ \\
    $E \cdot l_j = 0$ &  $E \cdot e = -1$ \\
\end{tabular}
\end{center}

\textbf{Anticanonical Divisor:} $ -K_X = 2H_1 + 2H_2 + 2H_3 - E$

\textbf{Effective Divisors:}
The divisors $H_1+H_2-E$, $3H_1+H_3-E$, $3H_2+H_3-E$, $H_1$, $H_2$, $H_3$, and $E$ are all effective.

\textbf{Effective Curves:} $\NE(X)$ is generated by $e, \ l_i -e, \ l_1 + l_2 -3e$.

\begin{rem}
Let $q \in \mathbb{P}^1$ be a general point such that $\pi^{-1}_3(q)$ contains no ramifications points of $c \xrightarrow{\pi_3} \mathbb{P}^1$.  Let $p_1, p_2, p_3 = c \cap \pi_3^{-1}(q)$.  The blow up of $\mathbb{P}^1 \times \mathbb{P}^1$ along $p_3$ is isomorphic to the blow up of $\mathbb{P}^2$ along two points $q_1,q_2$, where the exceptional curve over $q_i$ has class $l_j -e$.  The exceptional curve of class $e$ over $p_3 \in \mathbb{P}^1 \times \mathbb{P}^1$ corresponds to the strict transform of the line through $q_1,q_2$.  As $X$ is Fano, the points $q_1, q_2, p_1, p_2 \in \mathbb{P}^2$ must be linearly general.  Thus we may identify $X_q$ with $\text{Bl}_{q_1,q_2, p_1, p_2} \mathbb{P}^2$.  Note that the monodromy of $c \xrightarrow{\pi_3} \mathbb{P}^1$ identifies all three choices of $p_3 \in c \cap \pi_3^{-1}(q)$.
\end{rem}

\begin{lem}
A core set of free curves on $X$ is given by
\begin{align*}
    \mathscr{C}_X = \{ l_i ,\ l_i + l_j -e, \ l_1 + l_2 - 2e, \ l_1 + 2l_2 -3e, \ 2l_1 + l_2 -3e \},
\end{align*}
where $i\neq j$. The only separating classes $\alpha \in \mathscr{C}_X$ are $l_1 + l_2 -e$ and $l_1 + l_2 -2e$.
\end{lem}
\begin{rem}
Under the description of $\pi_3^{-1}(q)$ as $\text{Bl}_{q_1, q_2, p_1, p_2}\mathbb{P}^2$ in the above remark, we may let $h$ be the class of a line in $\mathbb{P}^2$ and $e_1,e_2,f_1,f_2$ be the exceptional curves over $q_1,q_2,p_1,p_2$, respectively.  Then curves in $\mathscr{C}_X$ contained in a fiber of $\pi_3$ correspond to the following classes:
\begin{itemize}
    \item $l_{1/2}$ corresponds to $h-e_{1/2}$,
    \item $l_1 + l_2 -e$ corresponds to either $2h - e_1 -e_2 -f_1$, $2h -e_1 -e_2 -f_2$, or $h$, depending on the choice of $p_3 \in \pi_3^{-1}(q) \cap c$,
    \item $l_1 + l_2 -2e$ corresponds to either $2h - e_1 -e_2 -f_1 -f_2$, $h -f_1$, or $h - f_2$, depending on the choice of $p_3 \in \pi_3^{-1}(q) \cap c$,
    \item $2l_{1/2} + l_{2/1} -3e$ corresponds to $2h - e_{1/2} - f_1 - f_2$.
\end{itemize}
By identifying the fibers of $\pi_1$ or $\pi_2$ with the blow up of $\mathbb{P}^2$ along two points (corresponding to $q_1,q_2$), we may similarly identify $l_3$ and $l_{1/2} + l_3 -e$ as lines or conics in $\mathbb{P}^2$ meeting certain points.
\end{rem}
\begin{proof}
\textbf{Nef Curve Classes of Anticanonical Degree Between $2$ and $4$:}
If $\alpha = d_1 l_1 + d_2 l_2 + d_3 l_3 -ne$ is nef of the appropriate degree, the effective divisors above imply 
$$d_1 + d_2 + 2d_3 \leq 4, \hspace{.05cm} 2d_1 - d_2 + d_3 \leq 4.$$
Adding these together, we obtain $3d_1 + 3d_3 \leq 8$ and so $d_i \leq 2$ for all $i$.  Moreover, if $d_3 = 2$, $\alpha = 2l_3$ corresponds to double covers and is freely breakable.  If $d_1 = 2$ or $d_2 =2$, necessarily $d_3 = 0$.  %
Thus, a complete list of possibilities for $\alpha$ is $\mathscr{C}_X \cup \{ 2l_i,\ l_i + l_j, \ 2l_{1/2} + l_{2/1} -2e,\ l_1 + l_2 + l_3 -2e, \ 2l_1 + 2l_2 -4e\}$.

\textbf{Freely Breakable Classes:}  It is obvious that each component of $\overline{\free}(X,\alpha)$ for $\alpha \in \{ 2l_i,\ l_i + l_j, \ 2l_{1/2} + l_{2/1} -2e\}$ parameterizes freely breakable curves.  If $\alpha = l_1 + l_2 + l_3 -2e,$ $\overline{\free}(X,\alpha)$ is irreducible and contains free chains of type $(l_1 + l_2 -2e, l_3)$.  By realizing fibers of $\pi_3$ as blow-ups of $\mathbb{P}^2$ along $4$ linearly general points, we may describe curves of class $2l_1 + 2l_2 -4e$ in such fibers as the linear system of conics passing through two different blown up points, or one blown-up point with multiplicity two.  %

\textbf{Irreducible Spaces and Fibers:}  Let $\alpha \in \mathscr{C}_X$.  If $\alpha \in \{l_i, l_i + l_3 -e, 2l_i + l_j -3e\}$, it is clear that $\overline{\free}(X,\alpha)$ is irreducible and $\overline{\free}_1(X,\alpha)\xrightarrow{\text{ev}} X$ has irreducible fibers over general points.  When $\alpha = 2l_i + l_j -3e$, this follows from the identification of a general fiber of $\pi_3$ with the blow-up of $\mathbb{P}^2$ along 4 linearly general points.  %
As a free rational curve of class $\alpha = 2l_i + l_j -3e$ must be smooth and meet each of $p_1, p_2, p_3$, this realizes curves of class $\alpha$ as conics meeting $q_i$, $p_1$, and $p_2$.

However, when $\alpha = l_1 + l_2 - e$ or $\alpha = l_1 + l_2 -2e$, this same description of $\pi_3^{-1}(q)$ realizes curves of class $\alpha$ as one of three classes of curves.  When $\alpha = l_1 + l_2 -e$, these classes are 1) conics passing through $q_1, q_2, p_1$, 2) conics passing through $q_1, q_2, p_2$, and 3) a general line in $\mathbb{P}^2$.  When $\alpha = l_1 + l_2 -2e$, the classes are 1) conics passing through $q_1, q_2, p_1, p_2$, 2) lines passing through $p_1$, and 3) lines passing through $p_2$.  Changing the roles of $p_1, p_2, p_3$ permutes these classes.  Thus, Theorem \ref{Monodromy} shows $\overline{\free}(X,\alpha)$ is irreducible.  However, $\overline{\free}_1(X,\alpha)\xrightarrow{\text{ev}} X$ has reducible fibers with three components over general points.  
\end{proof}

\begin{lem}
Relations in the monoid $\mathbb{N}\mathscr{C}_X$ are generated by:
\begin{enumerate}
    \item $(2l_1 + l_2 -3e) + (l_1 + 2l_2 -3e) = 3(l_1 + l_2 -2e)$
    \item $(2l_{1/2} + l_{2/1} - 3e) + l_{2/1} = (l_1 + l_2 -2e) + (l_1 + l_2 -e)$
    \item $(2l_{1/2} + l_{2/1} -3e) + l_3 = (l_1 + l_2 - 2e) + (l_{1/2} + l_3 -e)$
    \item $(2l_{1/2} + l_{2/1} -3e) + (l_1 + l_2 -e) = l_{1/2} + 2(l_1 + l_2 -2e)$
    \item $(2l_{1/2} + l_{2/1} -3e) + (l_{2/1} + l_3 -e) = l_3 + 2(l_1 + l_2 -2e)$
    \item $(l_1 + l_2 -e) + l_3 = (l_1 + l_3 -e) + l_2 = (l_2 + l_3 -e) + l_1$
    \item $(l_1 + l_2 -2e) + l_1 + l_2 = 2(l_1 + l_2 -e)$
    \item $(l_1 + l_2 -2e) + l_{1/2} + l_3 = (l_1 + l_2 -e) + (l_{1/2} + l_3 -e)$
    \item $(l_1 + l_2 -2e) + 2l_3 = (l_1 + l_3 -e) + (l_2 + l_3 -e)$
\end{enumerate}
For each relation $\sum \alpha_i = \sum \beta_j$, a main component of $\prod_X \overline{\free}_2(X,\alpha_i)$ lies in the same component of free curves as a main component of $\prod_X \overline{\free}_2(X, \beta_j)$.
\end{lem}
\begin{proof}
\textbf{Relations:}  We use Lemma \ref{Relations}.
\begin{itemize}
\item $D= 2H_2 + H_3 - E$ pairs negatively with $(2l_1 + l_2 -3e)$ and to 0 with $l_1 + l_2 -2e$, $l_1$, and $l_1 + l_3 -e$.  We obtain relations $(1)-(5)$ using the remaining curves.
\item $D = 2H_1 + H_3 -E$ pairs negatively with $(l_1 + 2l_2 -3e)$ and to 0 with $l_2$, $l_1 + l_2 -2e$, and $l_2 + l_3 -e$.  We obtain relations $(2)-(5)$ using the remaining curves.
\item The remaining relations between elements of $\mathscr{C}_X \setminus \{2l_1 + l_2 -3e, \ l_1 + 2l_2 -3e\}$ are identical to those in Lemma \ref{4.3relations}.
\end{itemize}

\textbf{Main Components:} We address each relation in order below:
\begin{enumerate}
    \item There are multiple components of $\overline{\free}(X, 3l_1 + 3l_2 -6e)$, but only one which contains free chains of type $(2l_1 + l_2 -3e, l_1 + 2l_2 -3e)$.  Under the description of $\pi_3^{-1}(q)$ as $\text{Bl}_{q_1, q_2, p_1, p_2}\mathbb{P}^2$, these curves have class $4h - e_1 -e_2 -2f_1 -2f_2$.  By Theorem \ref{del pezzo curves thm}, we may deform such a curve to a free chain of type $(l_1 + l_2 -2e, l_1 + l_2 -2e, l_1 + l_2 -2e)$ with components corresponding to classes $2h -e_1 -e_2 -f_1 -f_2$, $h - f_1$, and $h - f_2$.
    \item There are multiple components of $\overline{\free}(X, 2l_1 + 2l_2 -3e)$, but only one which contains free chains of type $(2l_{1/2} + l_{2/1} -3e, l_{2/1})$.  Under the description of $\pi_3^{-1}(q)$ as $\text{Bl}_{q_1, q_2, p_1, p_2}\mathbb{P}^2$, these curves have class $3h - e_1 -e_2 -f_1 -f_2$.  By Theorem \ref{del pezzo curves thm}, we may deform such a curve to a free chain of type $(l_1 + l_2 -2e, l_1 + l_2 -e)$ with components corresponding to classes $2h -e_1 -e_2 -f_1 -f_2$ and $h$.
    \item The space $\overline{\free}(X, 2l_{1/2} + l_{2/1} + l_3 -3e)$ may be identified with the space of sections of $\pi_3$ meeting three points of $c$ and mapping to a curve of class $(2,1)$ or $(1,2)$ under $\pi_1 \times \pi_2$.  Since $c$ maps to a curve of class $(1,1)$ under $\pi_1 \times \pi_2$, for each curve of class $(2,1)$ or $(1,2)$, there is a unique choice of three intersection points with $\pi_1 \times \pi_2 (c)$.  Thus $\overline{\free}(X, 2l_{1/2} + l_{2/1} + l_3 -3e)$ is equivalent to the space of divisors of class $(2,1)$ or $(1,2)$ in $\mathbb{P}^1 \times \mathbb{P}^1$.
    \item There are multiple components of $\overline{\free}(X, 3l_{1/2} + 2l_{2/1} -4e)$, but only one which contains free chains of type $(2l_{1/2} + l_{2/1} -3e, l_1 + l_2 -e)$.  Under the description of $\pi_3^{-1}(q)$ as $\text{Bl}_{q_1, q_2, p_1, p_2}\mathbb{P}^2$, we may suppose these curves have class $3h - e_{1/2} -f_1 -f_2$.  By Theorem \ref{del pezzo curves thm}, we may deform such a curve to a free chain of type $(l_{1/2}, l_1 + l_2 -2e, l_1 + l_2 -2e)$ with components corresponding to classes $h - e_{1/2}$, $h - f_1$, and $h - f_2$.
    \item We may degenerate a free chain of type $(2l_{1/2} + l_{2/1} -3e, l_{2/1} + l_3 -e)$ to a chain of curves of type $(2l_{1/2} + l_{2/1} -3e, l_{2/1} -e, l_3)$ by degenerating the component of class $(l_{2/1} + l_3 -e)$ to a nodal curve.  Since the normal bundle of any curve of class $(l_{2/1} -e)$ is $\mathcal{O} \oplus \mathcal{O}(-1)$, this degeneration is a smooth point of $\overline{\mathcal{M}}_{0,0}(X)$.  We may smooth the subchain of type $(2l_{1/2} + l_{2/1} -3e, l_{2/1} -e)$ into a free rational curve of class $(2l_1 + 2l_2 -4e)$, and break this into two free curves of class $l_1 + l_2 -2e$ to obtain a chain of free curves of type $(l_1 + l_2 -2e, l_1 + l_2 -2e, l_3)$.
    \item The space $\overline{\free}(X, l_1 + l_2 + l_3 -e)$ may be identified with the space of sections of $\pi_3$ meeting a point of $c$ and mapping to a curve of class $(1,1)$ under $\pi_1 \times \pi_2$.  Since $c$ maps to a curve of class $(1,1)$ under $\pi_1 \times \pi_2$, for each curve of class $(1,1)$, there are two intersection points with $\pi_1 \times \pi_2(c)$.  Thus, we may fiber $\overline{\free}(X, l_1 + l_2 + l_3 -e)$ over the space of curves of class $(1,1)$ in $\mathbb{P}^1 \times \mathbb{P}^1$.  Each fiber contains two irreducible components of dimension 2.  A standard monodromy argument proves $\overline{\free}(X, l_1 + l_2 + l_3 -e)$ is irreducible.
    \item There are multiple components of $\overline{\free}(X, 2l_1 + 2l_2 -2e)$, but only one which contains free chains of type $(l_1 + l_2 -2e, l_1, l_2)$.  Under the description of $\pi_3^{-1}(q)$ as $\text{Bl}_{q_1, q_2, p_1, p_2}\mathbb{P}^2$, we may suppose these curves have class $3h - e_1 -e_2 -f_1$.  By Theorem \ref{del pezzo curves thm}, we may deform such a curve to a free chain of type $(l_1 + l_2 -e, l_1 + l_2 -e)$ with components corresponding to classes $h$ and $2h -e_1 -e_2 - f_1$.
     \item The space $\overline{\free}(X, 2l_{1/2} + l_{2/1} + l_3 -2e)$ may be identified with the space of sections of $\pi_3$ meeting two points of $c$ and mapping to a curve of class $(2,1)$ or $(1,2)$ under $\pi_1 \times \pi_2$.  Since $c$ maps to a curve of class $(1,1)$ under $\pi_1 \times \pi_2$, for each curve of class $(2,1)$ or $(1,2)$, there are three choices of two intersection points with $\pi_1 \times \pi_2(c)$.
     Thus, we may fiber $\overline{\free}(X, 2l_{1/2} + l_{2/1} + l_3 -2e)$ over the space of curves of class $(2,1)$ or $(1,2)$ in $\mathbb{P}^1 \times \mathbb{P}^1$.  Each fiber contains three irreducible components of dimension 1.  A standard monodromy argument proves $\overline{\free}(X, 2l_{1/2} + l_{2/1} + l_3 -2e)$ is irreducible.
     \item The space $\overline{\free}(X, l_1 + l_2 + 2l_3 -2e)$ may be identified with the space of sections of $\pi_1$ meeting two points of $c$ and mapping to a curve of class $(2,1)$ under $\pi_2 \times \pi_3$.  Since $c$ maps to a curve of class $(3,1)$ under $\pi_2 \times \pi_3$, for each curve of class $(2,1)$, there are ${ 5 \choose 2}$ choices of two intersection points with $\pi_2 \times \pi_3(c)$.  Thus, we may fiber $\overline{\free}(X, l_1 + l_2 + 2l_3 -2e)$ over the space of curves of class $(2,1)$ in $\mathbb{P}^1 \times \mathbb{P}^1$.  Each fiber contains ${ 5 \choose 2}$ irreducible components of dimension 1.  A standard monodromy argument proves $\overline{\free}(X, l_1 + l_2 + 2l_3 -2e)$ is irreducible.
\end{enumerate}
\end{proof}

\begin{lem}
For each $\alpha \in \Nef_1(X)_\mathbb{Z}$, $\overline{\free}(X,\alpha)$ is nonempty.
\end{lem}

\begin{proof}
By Lemma \ref{Gordan's Lemma} and \ref{Representability of Free Curves}, we only need to check that integral linear combinations $\sum a_{\alpha_i} \alpha_i$ of elements $\alpha_i \in \mathscr{C}_X$ with coefficents $0\leq a_{\alpha_i} < 1$ are all representable by free rational curves.  Let $\alpha = c_1 l_1 + c_2 l_2 + c_3 l_3 - c_e e = \sum a_{\alpha_i} \alpha_i$.  It follows that $c_3 \in \{0,1,2\}$, $0 \leq c_1, c_2 \leq 6$, and $0 \leq c_e \leq 10$.  Since $\alpha \in \Nef_1(X)$, we have $c_e \leq c_1 + c_2$ and $c_e \leq 3c_i + c_3$ for $i=1,2$.  By symmetry we may assume $c_2 \geq c_1$.  It follows that we may write $\alpha = \beta + a l_1 + b l_2$ with $a,b\geq 0$ and $\beta$ given by one of the following:
\begin{itemize}
    \item $c_3 l_3 + l_1 + k l_2 - re$ for $k \leq 2 + c_3$ and $r \leq k + 1$
    \item $c_3 l_3 + 2l_1 + k l_2 - re$ for $k \leq 4 + c_3$ and $r \leq k + 2$
    \item $c_3 l_3 + 3l_1 + k l_2 - re$ for $k \leq 6$ and $r \leq k + 3$
    \item $c_3 l_3 + 4l_1 + k l_2 - re$ for $k \leq 6$ and $r \leq k + 4$
\end{itemize}
Each such integral $\beta$ is a nonnegative linear combination of $l_2 + l_3 -e$, $l_1 + 2l_2 -3e$, $l_1 + l_2 -2e$, $l_1 + l_2 -e$, $l_1$, and $l_2$. 
\end{proof}

\section{Fano 3-Folds with $E5$ Divisors}\label{E5 cases}

By \cite[Theorem~2.4]{beheshti2020moduli}, the only Fano threefolds admitting an E5 divisorial contraction are the six appearing below.

\subsection*{3.9}

\textbf{Blow-up of $\mathbb{P}_{\mathbb{P}^2}(\mathcal{O} \oplus \mathcal{O}(2))$ along quartic curve:}  This case is covered in \cite[Section~8]{beheshti2020moduli}.  We summarize their results below.  Let $\phi : X \rightarrow \mathbb{P}_{\mathbb{P}^2}(\mathcal{O} \oplus \mathcal{O}(2))$ be the blow-up of a smooth quartic curve $c$ in a minimal moving section $D$ of the projective bundle.  Let $\pi : X \rightarrow \mathbb{P}^2$ be the composition of $\phi$ with the natural map $\mathbb{P}_{\mathbb{P}^2}(\mathcal{O} \oplus \mathcal{O}(2)) \rightarrow \mathbb{P}^2$.

\begin{thm}[\cite{beheshti2020moduli}]
For all $\alpha \in \Nef_1(X)_\mathbb{Z}$, $\overline{\free}(X,\alpha)$ is irreducible.
\end{thm}

\textbf{Generators for $N^1(X)$ and $N_1(X)$:} 

\begin{tabular}{ll}
 $H = \pi^*\mathcal{O}(1)$ & $l$ = a general line in a minimal moving section \\ 
 $E$ = the exceptional divisor $\phi^{-1}(c)$ & $e$ = an $\phi$-fiber over a point on $c$  \\  
 $F$ = the strict transform of the rigid section & $f$ = a fiber of $\pi$.
\end{tabular}

\textbf{Intersection Pairing:}
\begin{center}
\begin{tabular}{lll}
    $H \cdot l = 1$ &  $H \cdot e = 0$ & $H \cdot f = 0$, \\
    $E \cdot l = 0$ &  $E \cdot e = -1$ & $E \cdot f = 0$ \\
    $F \cdot l = 0$ & $F \cdot e = 0$ & $F \cdot f = 1$
\end{tabular}
\end{center}

\textbf{Anticanonical Divisor:} $-K_X = 5H + 2F -E$.

\textbf{Effective Divisors:} $E$, $F$, $4H - E$, and $D = 2H + F - E$ generate $\text{Eff}(X)$.  $F$ and $D$ are the $E5$ divisors on $X$.

\textbf{Effective Curves:} $e$, $l-2f$, $f-e$, and $l-4e$ generate $\NE(X)$.  There is pseudosymmetry which swaps $e$ with $f-e$ and $l-4e$ with $l-2f$.

\begin{lem}[\cite{beheshti2020moduli}]
A core of free curves on $X$ is given by 
\begin{align*}
    \mathscr{C}_X = \{ f, l-2e, l-e, l + f -3e, l, l + 2f -4e \}.
\end{align*}
There are no separating classes in $\mathscr{C}_X$.
\end{lem}

\begin{lem}[\cite{beheshti2020moduli}]
The relations in the monoid $\mathbb{N} \mathscr{C}_X$ are generated by: 
\begin{enumerate}
    \item $(l-2e) + l = 2(l-e)$,
    \item $(l-2e) + (l + 2f -4e) = 2(l + f - 3e)$,
    \item $f + 2(l-2e) = (l-e) + (l + f -3e)$,
    \item $f + (l-2e) + (l-e) = l + (l + f - 3e)$,
    \item $f + (l-2e) + (l + f -3e) = (l + 2f -4e) + (l-e)$,
    \item $f + (l-e) + (l + f + 3e) = l + (l + 2f -4e)$.
\end{enumerate}
For each such relation $\sum \alpha_i = \sum \beta_j$, a main component of $\prod_X \overline{\free}_2(X,\alpha_{i})$ lies in the same component of $\overline{\free}_2(X,\sum \alpha_i)$ as a main component of $\prod_X \overline{\free}_2(X,\beta_{j})$.
\end{lem}

\subsection*{3.14}

\textbf{Blow-up of a cubic and a point:}
Let $f:X \rightarrow \mathbb{P}^3$ be the blow-up of $\mathbb{P}^3$ with center a union of a cubic $c$ in a plane $S$ and a point $p$ not on $S$.

\begin{thm}\label{3.14thm}
For all $\alpha \in \Nef_1(X)_\mathbb{Z}$, $\overline{\free}(X,\alpha)$ is irreducible.
\end{thm}

\textbf{Generators for $N^1(X)$ and $N_1(X)$:}
\begin{center}
\begin{tabular}{ll}
 $H$ = the class of a hyperplane & $l$ = a line in $\mathbb{P}^3$ \\ 
 $E$ = the exceptional divisor $f^{-1}(c)$ & $e$ = an $f$-fiber over a point on $c$  \\  
 $F$ = the exceptional divisor $f^{-1}(p)$ & $f$ = a line in $F$
\end{tabular}
\end{center}

\textbf{Intersection Pairing:}
\begin{center}
\begin{tabular}{lll}
    $H \cdot l = 1$ &  $H \cdot e = 0$ & $H \cdot f = 0$, \\
    $E \cdot l = 0$ &  $E \cdot e = -1$ & $E \cdot f = 0$ \\
    $F \cdot l = 0$ & $F \cdot e = 0$ & $F \cdot f = -1$
\end{tabular}
\end{center}

\textbf{Anticanonical Divisor:}
\begin{align*}
    -K_X = 4H-E-2F
\end{align*}

\textbf{Effective Divisors:} The divisors $H-E$, $H-F$, $E$, and $F$ are effective, as well as $3H-3F-E$.

\begin{lem}\label{3.14core}
A core of free curves on $X$ is given by 
\begin{align*}
    \mathscr{C}_X = \{ l, l-e, l-f, 2l-2e-f, 3l-3e-2f \}.
\end{align*}
There are no separating classes in $\mathscr{C}_X$.
\end{lem}

\begin{proof}
\textbf{Nef Curve Classes of Anticanonical degree between $2$ and $5$:}

If $\alpha = a l - b e -c f$ is a nef curve class of anticanonical between between $2$ and $5$,
\begin{align*}
    0 \leq b,c \leq a, \hspace{.5cm} 2 \leq 4a-b-2c \leq 5.
\end{align*}

Moreover, since we are only considering components of curves mapped birationally onto their image, a nonlinear curve intersects $p$ less often than its degree. In other words, $c < a$ whenever $a > 1$.  The classes satisfying these constraints are those in $\mathscr{C}_X$, as well as $2l-e-f$, which is freely breakable.

\textbf{Irreducible Spaces and Fibers:}
The curves of class $l$, $l-e$, and $l-f$ are clearly parameterized by an irreducible family and they have irreducible fibers. 
For the curves of class $2l-2e-f$, apply Theorem \ref{Monodromy} to the linear system of planes through $p$ on $\mathbb{P}^3$. The monodromy action is the full symmetric group so, in particular, the space of pairs $(A,r,s)$ where $A$ is a plane containing $p$ and $r,s \in c \cap A$. But the space of curves of class $2l-2e-f$ form a $\mathbb{P}^2$ bundle over this space, so they are irreducible. To show the fibers are irreducible, we apply the same argument to the linear system of planes through $p$ and $q$ where $q \in \mathbb{P}^3$ is general.

The curves of class $3l-3e-2f$ are all of degree three with a double point at $p$, so are planar. Thus, the space of such curves comprises a bundle over the space of planes through $p$. The fiber over such a plane $A$ is the space of cubics in $A$ through the three collinear points $q_1, q_2, q_3 \in A \cap c$ with double point at $p$. Up to isomorphism, we may identify $A$ with $\mathbb{P}^2$, $p$ with $[0,0,1]$, and $q_1, q_2$, and $q_3$ with $[1,0,0]$, $[0,1,0]$, and $[s,t,0]$ for some $s,t$. One sees that the cubics with these prescribed conditions are solutions sets of equations
\begin{align*}
    a x^2z + bxyz+cy^2z + (t d) x^2 y  - (s d) xy^2 = 0,
\end{align*}

\noindent so they are parameterized by $\mathbb{P}^3$. Being a bunde over $\mathbb{P}^2$ with fiber $\mathbb{P}^3$, the space of curves of class $3l-3e-2f$ is irreducible. Moreover, using the same method above, we may realize the space of curves of class $3l-3e-2f$ through a generic point $q \in \mathbb{P}^3$ as a bundle over $\mathbb{P}^1$ with fiber $\mathbb{P}^2$.
\end{proof}

\begin{lem}\label{3.14relations}
The relations in the monoid $\mathbb{N} \mathscr{C}_X$ are generated by: 
\begin{enumerate}
    \item $(2l-2e-f)+(l)=2(l-e)+(l-f)$
    \item $(3l-3e-2f)+(l)=(2l-2e-f)+(l-e)+(l-f)$
    \item $(3l-3e-2f)+(l-e)=2(2l-2e-f)$.
\end{enumerate}

For each such relation $\sum \alpha_i = \sum \beta_j$, a main component of $\prod_X \overline{\free}_2(X,\alpha_{i})$ lies in the same component of $\overline{\free}_2(X,\sum \alpha_i)$ as a main component of $\prod_X \overline{\free}_2(X,\beta_{j})$.
\end{lem}

\begin{proof}
Consider a relation of the form $\sum \alpha_i = \sum \beta_j$. Note that the divisor $E+F-H$ pairs negatively with $l$, positively with $2l-2e-f$ and $3l-3e-2f$, and to zero otherwise. Thus, we may use relations $(1)$ and $(2)$ to remove $l$ from the relation. Then, out of the remaining classes, $2H-E-2F$ pairs negatively with $3l-3e-2f$, positively with $l-e$, and to zero otherwise. Use relation $(3)$ to remove all instances to $3l-3e-2f$. The remaining three class $l-e$, $l-f$, and $2l-2e-f$ are linearly independent and consequently admit no relations.

We show that a free chain of curves of type $2l-2e-f$ and $l$ may be deformed to a free chain of curves of type $l-e$, $l-e$, and $l-f$. So start with a union of a conic $x$ and a line $y$. Move $y$ along $x$ until meets $p$ and push the exceptional divisor on the line onto $x$. In sum, there is a chain of type $(2l-2e)+(l-f)$ in the same component. Finally, break the conic, fixing its intersection point $q$ with the line, as a union of two lines, each of type $l-e$.

We next show that a chain of type $(3l-3e-2f)+(l-e)$ may be deformed to a chain of type $2(2l-2e-f)$. So consider a union of a cubic curve $x$ and a line $y$ meeting at a point $r$. Since $x$ has a double point at $p$, it lies in a plane $A$; let $q_1, q_2, q_3$ be the points of $A \cap c$. We know the space of curves of class $(3l-3e-2f)$ in $A$ is three-dimensional, so we may deform $x$, fixing the points double point at $p$, as well as the points $r,q_1,q_2,q_3$ to contain a fourth point on the line from $p$ to say $q_1$. This breaks $x$ as a free chain of type $(l-e-f)+(2l-2e-f)$. After moving $y$ along this chain, we may assume it intersects the curve of class $l-e-f$. Smooth these curves out, maintaining their intersection point with the conic, to obtain another curve of class $2l-2e-f$. Thus, there is indeed a free chain of type $2(2l-2e-f)$ in the same component of the Kontsevich space.

Finally, consider a free chain of type $(3l-3e-2f)+(l)$. As in the preceding paragraph, the chain may be deformed to another chain of $(2l-2e-f)+(l-e-f)+(l)$. But then, moving the line of class $l$ until it intersects $p$ and pushing over the exceptional divisor, we see that there is a chain of type $(2l-2e-f)+(l-e)+(l-f)$ in the same component as desired.
\end{proof}

\subsection*{3.22}

\textbf{Blow-up of a curve in $\mathbb{P}^1 \times \mathbb{P}^2$:}
Let $f:X \rightarrow \mathbb{P}^1 \times \mathbb{P}^2$ be the blow-up of a conic $c$ in $t \times \mathbb{P}^2$ where $t \in \mathbb{P}^1$. Let $p_1: \mathbb{P}^1 \times \mathbb{P}^2 \rightarrow \mathbb{P}^1$ and $p_2: \mathbb{P}^1 \times \mathbb{P}^2 \rightarrow \mathbb{P}^2$ be the projection maps and set $f_1 = p_1 \circ f$ and $f_2 = p_2 \circ f$.

\begin{thm}\label{3.22thm}
For all $\alpha \in \Nef_1(X)_\mathbb{Z}$, $\overline{\free}(X,\alpha)$ is irreducible.
\end{thm}

\textbf{Generators for $N^1(X)$ and $N_1(X)$:} 
Let

\begin{center}
\begin{tabular}{ll}
 $H_1 = f_1^{-1}(pt) = pt \times \mathbb{P}^2$ & $l_1 = H_2^2 = \mathbb{P}^1 \times pt$ \\
 $H_2 = f_2^{-1}(l) = \mathbb{P}^1 \times l$ & $l_2 = H_1 \cdot H_2 = pt \times l$ \\ 
 $E$ = the exceptional divisor $f^{-1}(c)$ & $e$ = the $e$-fiber over a point on $c$   
\end{tabular}
\end{center}

where $l$ is the class of a line in $\mathbb{P}^2$.

\textbf{Intersection Pairing:}
\begin{center}
\begin{tabular}{lll}
    $H_1 \cdot l_1 = 1$ &  $H_1 \cdot l_2 = 0$ & $H_1 \cdot e = 0$, \\
    $H_2 \cdot l_1 = 0$ &  $H_2 \cdot l_2 = 1$ & $H_2 \cdot e = 0$ \\
    $E \cdot l_1 = 0$ & $E \cdot l_2 = 0$ & $E \cdot e = -1$
\end{tabular}
\end{center}

\textbf{Anticanonical Divisor:}
\begin{align*}
    -K_X = 3H_2+2H_1-E
\end{align*}

\textbf{Effective Divisors:} 
The divisors $H_1-E$, $2H_2-E$, and $E$ are effective.

\begin{lem}\label{3.22core}
A core of free curves on $X$ is given by 
\begin{align*}
    \mathscr{C}_X = \{ l_1, l_2, l_1+l_2-e, 2l_1+l_2 - 2e \}. 
\end{align*}

There are no separating curves in $\mathscr{C}_X$.
\end{lem}

\begin{proof}
\textbf{Nef Curve Classes of Anticanonical degree between $2$ and $5$:}  
If $\alpha = a l - b e - cf$ is a nef curve class of anticanonical degree between $2$ and $5$, then
\begin{align*}
    0 \leq c \leq a, 2b, \hspace{.5cm} 2 \leq 3b+2a-c \leq 5.
\end{align*}
\noindent Solving the inequalities, we obtain the classes shown above in $\mathscr{C}_X$, as well as $2l_1$ and $l_1+l_2$. The curves of class $2l_1$ are all double covers, so they contain no component of birational curves, and the curves of class $l_1+l_2$ are freely breakable.

\textbf{Irreducible Spaces and Fibers:} 
The curves of class $l_1$ and $l_2$ are clearly irreducible and have irreducible fibers.  
The curves of type $l_1+l_2-e$ are the the vanishing sets of polynomials of bidgree $(1,1)$ whose fiber over $t \in \mathbb{P}^1$ is in the curve $c$. This is a linear condition on the coefficients of the polynomial, so such polynomials form a $4$-dimensional linear system. Moreover, to require that the curve pass through a general point in $\mathbb{P}^1 \times \mathbb{P}^2$ imposes two additional linear conditions on the coefficients, so curves of class $l_1+l_2-e$ have irreducible fibers.

Note that the curves of type $2l_1+l_2-2e$ meet $\{ t \} \times \mathbb{P}^2$ in two points, both of which are required to lie in $c$. Therefore, they are fibered over $\Sym^2 c$. Over an unordered pair $(p,q) \in \Sym^2 c$, if $x$ denotes the line spanned by $f_2(p)$ and $f_2(q)$ in $\mathbb{P}^2$, and we identify $x$ with the space $\mathbb{P}^1$, the fiber is parameterized by the space of degree two morphisms $\phi: \mathbb{P}^1 \rightarrow \mathbb{P}^1$ for which $\phi(f_2(p)) = f_1(p) = t$ and $\phi(f_2(q)) = f_1(p) = t$. Additionally, such curves passing through a general point $r \in \mathbb{P}^1 \times \mathbb{P}^2$ are fibered over $\Sym^2 c$ with fiber over $(p,q)$ the space of morphisms morphisms $\phi:\mathbb{P}^1 \rightarrow \mathbb{P}^1$ such that $\phi(f_2(p)) = t$, $\phi(f_2(q))=t$, and $\phi(f_2(r)) = f_1(r)$. This moduli space is irreducible.
\end{proof}

\begin{lem}\label{3.22relations}
The relations in the monoid $\mathbb{N} \mathscr{C}_X$ are generated by $(2l_1+l_2-2e)+(l_2) = 2(l_1+l_2-e)$. The corresponding moduli space of curves of class $2l_1+2l_2-2e$ is irreducible.
\end{lem}

\begin{proof}
The four classes in $\mathscr{C}_X$ span a three dimensional vector space, so satisfy a single relation $(2l_1+l_2)+(l_2) = 2(l_1+l_2-e)$.  

Note that the variety $X$ admits a $\mathbb{C}^*$-action on the fibers of $f_1$ which fixes $f_1^{-1}(t)$ and $f_1^{-1}(s)$ for a fixed $s \in \mathbb{P}^1$ distinct from $t$. Taking the limit of this action as we approach the divisor $f_1^{-1}(s)$, we find that a curve of class $2l_1+2l_2-2e$ breaks as $2(l_1-e)+(2l_2)$. We must prove that such chains of curves correspond to smooth points in the Kontsevich space and that the family of such curves is irreducible.

First, the irreducible curves of class $2l_2$ are clearly free. Consider now a curve $y$ of class $l_1-e$. It lies in the divisor $A =p_2^{-1}(p_2(c))$ of class $2H_2-E$. Consider the short exact sequence
\begin{align*}
    0 \rightarrow N_{y / A} \rightarrow N_{y / X} \rightarrow \left.N_{A / X}\right|_{y} \rightarrow 0.
\end{align*}

Since $y^2 = 0$ in $A$, we see $N_{y / A} = \mathcal{O}_{\mathbb{P}^1}$. Also, $\left.N_{A / X}\right|_{y} = \left.\mathcal{O}(2H_2 - E)\right|_y = \mathcal{O}_{\mathbb{P}^1}(-1)$ because $(2H_2-E) \cdot (l_1-e) = -1$. Hence $N_{y / X} = \mathcal{O}_{\mathbb{P}^1} \oplus \mathcal{O}_{\mathbb{P}^1}(-1)$ has no higher cohomology.

Next, we see that the space of curves of class $2(l_1-e)+(2l_2)$ is generically fibered over $\Sym^2 c$ with fiber over $(p,q)$ the conics in $\{s\} \times \mathbb{P}^2$ passing through the points $(s, f_2(p))$ and $(s, f_2(q))$. This is a $3$-dimensional linear system, so it is irreducible.
\end{proof}

\subsection*{3.29}
\textbf{Blow-up of a line in an exceptional divisor:}
Let $f:V_7 \rightarrow \mathbb{P}^3$ be the blow-up of a point $p \in \mathbb{P}^3$. Then $g:X \rightarrow V_7$ is the blow-up of a line $c$ in the exceptional divisor of $f$.

\begin{thm}\label{3.29thm}
For all $\alpha \in \Nef_1(X)_\mathbb{Z}$, $\overline{\free}(X,\alpha)$ is irreducible.
\end{thm}

\textbf{Generators for $N^1(X)$ and $N_1(X)$:}

\begin{center}
\begin{tabular}{ll}
 $H$ = a general hyperplane & $l$ = a general line \\ 
 $E$ = $(g \circ f)^{-1}(p)$ & $e$ = the $g$-preimage of a general line in $f^{-1}(p)$ \\
 $F$ = $g^{-1}(c)$ & $f$ = a fiber of $g|_F$%
\end{tabular}
\end{center}

\textbf{Intersection Pairing:} 
To compute the intersection pairing, note first that $H \cdot l = 1$, while $H \cdot e$, $H \cdot f$, $E \cdot l$, and $F \cdot l$ are all zero. From $(H-E) \cdot e = 1$, we obtain the identity $E \cdot e = -1$ and from $(H-E) \cdot f = 0$, we get $E \cdot f = 0$. Then $(E-F) \cdot f = 1$ implies $F \cdot f = -1$, while finally $F \cdot (e-f) = 1$ implies $F \cdot e =0$. Summarizing, we have the following relations:

\begin{center}
\begin{tabular}{lll}
    $H \cdot l = 1$ &  $H \cdot e = 0$ & $H \cdot f = 0$, \\
    $E \cdot l = 0$ &  $E \cdot e = -1$ & $E \cdot f = 0$ \\
    $F \cdot l = 0$ & $F \cdot e = 0$ & $F \cdot f = -1$
\end{tabular}
\end{center}

\textbf{Anticanonical Divisor:}
\begin{align*}
    -K_X = 4H-2E-F
\end{align*}

\textbf{Effective Divisors:} 
The divisors $H-E-F$ is effective, as it is the $(g \circ f)$-strict transform of the hyperplane in $\mathbb{P}^3$ whose $f$-strict transform intersects the exceptional divisor in the curve $c$. The divisor $E-F$ are effective as well.

\begin{lem}\label{3.29core}
A core of free curves on $X$ is given by 
\begin{align*}
    \mathscr{C}_X = \{ l, l-e, 2l-e-f \}.
\end{align*}

There are no separating curves in $\mathscr{C}_X$.
\end{lem}

\begin{proof}
\textbf{Nef Curve Classes of Anticanonical degree between $2$ and $5$:}  If $\alpha = a l - b e - cf$ is a nef curve class of anticanonical degree between $2$ and $5$, then
\begin{align*}
    0 \leq c \leq b, \hspace{.5cm} b+c \leq a, \hspace{.5cm} 2 \leq 4a-2b-c \leq 5.
\end{align*}

\noindent Solving the inequalities, we obtain the classes shown above in $\mathscr{C}_X$, as well is $2l-2e$. However, the curves of class $2l-2e$ are all double covers, so they do not contain a component of free birational curves.

\textbf{Irreducible Spaces and Fibers:} 
The curves of type $l$ are parametrized by an open subset of $\mathbb{G}(1,3)$ and those of type $l-e$ by an open subset of the planes in $\mathbb{P}^3$ passing through $p$. Both have irreducible fibers.

The curves of type $2l-e-f$ are, in particular, conics $x$ lying on a plane $A$ in $\mathbb{P}^3$ passing through $p$. For the conic $x$ in a general such plane $A$ to be of class $2l-e-f$, it must pass through $p \in A$ in the direction specified by the unique point in $c$ intersected with the $f$-strict transform of $A$. This determines a $3$-dimensional linear system of conics on $A$. The corresponding component of the Kontsevich space is therefore parameterized by a bundle with base the planes in $\mathbb{P}^3$ passing through $p$ and with a $3$-dimensional fiber; it is irreducible. This arguments does not apply to the unique plane $A$ through $p$ whose $f$-strict transform contains $c$. In this situation, any conic in $A$ passing through $p$ is of class $2l-e-f$. But the anticanonical degree of $2l-e-f$ is $5$, so this $4$-dimensional subvariety must lie in the previously described component of free curves.

Fix a general point $q \in \mathbb{P}^3$. The curves of class $2l-e-f$ containing $q$ are fibered over the space of planes $A$ in $\mathbb{P}^3$ containing $p$ and $q$. Over such a plane $A$, the fiber is a $2$-dimensional linear system of conics. It follows that the curves of class $2l-e-f$ have irreducible fibers as well.
\end{proof}

\begin{lem}\label{3.29relations}
There are no relations in the monoid $\mathbb{N} \mathscr{C}_X$.
\end{lem}

\begin{proof}
This is clear because $l$, $l-e$, and $2l-e-f$ are linearly independent curve classes.
\end{proof}

\subsection*{2.28}

\textbf{Blow-up of a planar cubic in $\mathbb{P}^3$:}
Let $f:X \rightarrow \mathbb{P}^3$ be the blow-up of a cubic $c$ in a plane $S \subset \mathbb{P}^3$. We apply \ref{blowup} to case $3.14$.%

\subsection*{2.36}

\textbf{Blow-up of the Veronese cone at its vertex:}
Let $X$ be the blow-up of the Veronese cone $W_4 \subset \mathbb{P}^6$ with center the vertex. In other words, $X$ is the projective bundle $\mathbb{P}(\mathcal{O} \oplus \mathcal{O}(2))$ over $\mathbb{P}^2$. We apply \ref{blowup} to case $3.9$.%

\section{Picard Rank 3}\label{Picard Rank 3 Cases}

\subsection*{3.1}
\textbf{Double Cover of $(\mathbb{P}^1)^3$:}
Let $\pi: X \rightarrow \mathbb{P}^1 \times \mathbb{P}^1 \times \mathbb{P}^1$ be a double cover whose branch locus is a divisor $D$ of tridegree $(2,2,2)$.  Let $\pi_i : X \rightarrow \mathbb{P}^1$ be the composition of $\pi$ with the $i^{th}$ projection.

\textbf{Generators for $N^1(X)$ and $N_1(X)$:}  Let $H_i = \pi_i^*\mathcal{O}(1)$ and $c_i$ be the class of an irreducible component of a reducible fiber of $\pi_j \times \pi_k$, where $\{i,j,k\} = \{1,2,3\}$.

\begin{thm}\label{3.1thm}
For all $\alpha \in (c_1+c_2+c_3) + \Nef_1(X)_\mathbb{Z}$, $\overline{\free}(X,\alpha)$ is irreducible and nonempty.
\end{thm}

\textbf{Intersection Pairing:} 
The intersection pairing is diagonal: $H_i \cdot c_j = \delta_{ij}$.

\textbf{Anticanonical Divisor:} 
$-K_X = - \pi^*(K_{(\mathbb{P}^1)^3}) - \frac{1}{2} \pi^*D = H_1+H_2+H_3$

\textbf{Effective Divisors:} 
The divisors $H_1$, $H_2$, and $H_3$ span the effective cone.

\textbf{Mori Structure:} The contraction of $c_i$ is $\pi_j \times \pi_k$.  For general branch loci $D$, the discriminant locus $\Delta_i$ of $\pi_j \times \pi_k$ is a smooth curve of bidegree $(4,4)$ on $\mathbb{P}^1 \times \mathbb{P}^1$ with simple ramification under either projection.  %
Thus for general $X$, $\overline{\mathcal{M}}_{0,0}(X,c_i)$ is irreducible and smooth for each $i$.  The fiber $X_p$ of $\pi_i : X\rightarrow \mathbb{P}^1$ over a general point $p\in \mathbb{P}^1$ is a smooth del Pezzo surface of degree four.  Let $i = 1$.  We may consider $X_p$ as a double cover of $p \times \mathbb{P}^1 \times \mathbb{P}^1 \subset (\mathbb{P}^1)^3$, branched over $D_p$, or as a five-point blow-up of $\mathbb{P}^2$.  Considered as a branched cover, the sixteen anticanonical lines in $X_p$ are its intersection with the preimages of $\Delta_2$ and $\Delta_3$.  Identified as a blow-up $X \rightarrow \mathbb{P}^2$ with exceptional divisors $e_1, \ldots e_5$ instead, without loss of generality these lines are the irreducible components of the reducible fibers of conic fibrations on $X_p$ defined by $|l-e_1|$ and $|2l - e_2 -e_3 -e_4 -e_5|$, where $l$ is the pullback of a line from $\mathbb{P}^2$.  By Theorem \ref{conic monodromy}, the monodromy of $\pi_i$ acts as the full stabilizer of $l-e_1$ in $W(D_5)$.  In other words, the monodromy group $G$ of $\pi_i$ fits into a split exact sequence
\begin{align*}
    0 \rightarrow (\mathbb{Z}/2\mathbb{Z})^3 \rightarrow G \rightarrow S_4 \rightarrow 0.
\end{align*}
\noindent with a splitting map whose image acts as permutation of $e_2,\ldots e_5$ (fixing $e_1$ and $l$).  The kernel $(\mathbb{Z}/2\mathbb{Z})^3$ is generated by elements which act by reflection about $l-e_1 -e_j -e_k$ and permutation of $e_m$ and $e_n$, where $\{j,k,m,n\}=\{2,3,4,5\}$.

\begin{lem}\label{3.1core}
Let $X$ be general in moduli.  A core of free curves on $X$ is given by
\begin{align*}
   \mathscr{C}_X = &\{ 2c_i, c_i+c_j, 2c_i+c_j, c_1+c_2+c_3  \} 
\end{align*}
for $1 \leq i,j \leq 3$ distinct. The separating classes in $\mathscr{C}_X$ are $c_i+c_j$ and $2c_i+c_j$.  In addition, $\overline{\free}(X,\alpha)$ is irreducible for all $\alpha \in \Nef_1(X)_{\mathbb{Z}}\setminus \partial\NE(X)$ with $-K_X. \alpha \leq 4$.
\end{lem}

We refer to Proposition \ref{low degree curves degeneration} for the appropriate generalization to arbitrary $X$.

\begin{proof}
\textbf{Nef Curve Classes of Anticanonical Degree Between $2$ and $4$:}
The nef curve classes of anticanonical degree between $2$ and $4$ are those shown in the core, in addition to $3c_i$, $4c_i$, $3c_i+c_j$, $2c_i+2c_j$, and $2c_i+c_j+c_k$. Irreducible curves of class $3c_i$ clearly do not dominate $X$ and those of class $4c_i$ break freely. We show that the remaining classes break freely as well.

Indeed, the curves of class $3c_i+c_j$ and $2c_i+2c_j$ all lie in a fiber of the del Pezzo fibration $f_k: X \rightarrow \mathbb{P}^1$. As curves of degree four in a del Pezzo surface, they are freely breakable to chains of type ($2c_i, c_i+c_j)$ and either $(2c_i, 2c_j)$ or $(c_i + c_j, c_i + c_j)$.

Free curve of class $2c_i+c_j+c_k$ map to curves of bidegree $(1,1)$ in $\mathbb{P}^1 \times \mathbb{P}^1$ under $\pi_j \times \pi_k$.  Corollary \ref{cor: conic monodromy cor} thus shows $\overline{\free}(X, 2c_i + c_j + c_k)$ is irreducible.

\textbf{Irreducible Spaces and Fibers:}
Free curves of class $2c_i$ are smooth fibers of the map $\pi_j \times \pi_k :X \rightarrow \mathbb{P}^1 \times \mathbb{P}^1$, so they are paramterized by an irreducible family and there is a unique such curve through every point of $X$.

Curves of class $c_i+c_j$ lie in fibers of the map $\pi_k: X \rightarrow \mathbb{P}^1$ and, under the identification of this fiber as a blow-up of $\mathbb{P}^2$, are of class $l-e_2$, $l-e_3$, $l-e_4$, $l-e_5$, $2l-e_1-e_3-e_4-e_5$, $2l-e_1-e_2-e_3-e_4$, $2l-e_1-e_2-e_3-e_5$, or $2l-e_1-e_2-e_3-e_4$. These classes are all identified under the monodromy action of $G$, so the moduli space of curves of class $c_i+c_j$ is irreducible. However, there are eight such curves through a general point of $X$, so this is a separating curve class.

Consider next the curves of class $2c_i+c_j$. They likewise lie in a fiber of $\pi_k: X \rightarrow \mathbb{P}^1$. We may identify the fiber with as above such that $2c_i$ corresponds to $l-e_1$.  Then  $2c_i + c_j$ corresponds to $l$, $3l -2e_1 - e_2 -e_3 -e_4 -e_5$, or $2l-e_1-e_p-e_q$ in $S$ for $2 \leq p < q \leq 5$. These classes are all identified under the monodromy action. However, there are eight components of curves passing through a general point of $X$, so this is another separating curve class.

Finally, consider free curves of class $c_1+c_2+c_3$.  For any choice of $i,j$, the map $\pi_i \times \pi_j : X \to \mathbb{P}^1 \times \mathbb{P}^1$ sends curves of class $c_1 + c_2 + c_3$ to curves of bidegree $(1,1)$ in $\mathbb{P}^1 \times \mathbb{P}^1$.  Thus Corollary \ref{cor: conic monodromy cor} shows $\overline{\free}(X, c_1 + c_2 + c_3)$ is irreducible.  Note that by \ref{reducible fibers: 4 author result}, the class $c_1 + c_2 + c_3$ cannot be separating.
\end{proof}

\begin{lem}\label{3.1relations1}
Relations in the monoid $\mathbb{N}\mathscr{C}_X$ are generated by:
\begin{enumerate}
    \item $2(c_i+c_j) = (2c_i)+(2c_j)$,
    \item $3(c_i+c_j) = (2c_i+c_j)+(2c_j+c_i)$,
    \item $(2c_i+c_j)+(2c_j) = (2c_j+c_i)+(c_i+c_j)$,
    \item $2(c_i+c_j)+(2c_i) = 2(2c_i+c_j)$,
    \item $2(c_1+c_2+c_3) = (2c_i)+2(c_j+c_k) = (2c_i+c_k)+(2c_j+c_k) = (c_1+c_2)+(c_1+c_3)+(c_2+c_3)$,
    \item $3(c_1+c_2+c_3) = (2c_i+c_j)+(2c_j+c_k)+(2c_k+c_i)$,
    \item $(2c_i)+(c_j+c_k) = (c_i+c_j)+(c_i+c_k)$
    \item $(2c_i)+(c_i+c_j)+(c_i+c_k) = (2c_i+c_j)+(2c_i+c_k)$
    \item $(c_1+c_2+c_3)+(c_i+c_j) = (2c_i+c_j)+(c_j+c_k) = (2c_j+c_k)+(2c_i)$,
    \item $(c_1+c_2+c_3)+(2c_i) = (2c_i+c_j)+(c_i+c_k)$,
    \item $(c_1+c_2+c_3)+(2c_i+c_j) = (2c_i+c_k)+(2c_j+c_i) = (c_i+c_k)+2(c_i+c_j)$,
    \item $(2c_i+c_j)+(2c_i+c_k) = (2c_i)+(c_i+c_j)+(c_i+c_k)$,
\end{enumerate}
for $1 \leq i,j,k \leq 3$ distinct.  For each relation $\sum \alpha_i = \sum \beta_j$, a main component of $\prod_X \overline{\free}_2(X,\alpha_i)$ lies in the same component of free curves as a main component of $\prod_X \overline{\free}_2(X, \beta_j)$.
\end{lem}

\begin{proof}
Given an arbitrary relation, we may assume no member of the core appears on both sides. The idea of the proof is then to use the intersection pairing to show that the above generators may be used to remove any instance of $2c_1$ from the relation. Next, we remove $2c_2$, $2c_3$, $2c_1+c_2$, and finally $2c_i+c_j$ for all other combinations of $i$ and $j$. The remaining four classes $c_i+c_j$ and $c_1+c_2+c_3$ satisfy the unique relation $(5)$.

Firstly, consider the divisor $D = 2H_2+2H_3-H_1$. It pairs to $-2$ with $2c_1$, to zero with $2c_1+c_2$ and $2c_1+c_3$, to $1$ with $c_1+c_2$ and $c_1+c_3$, and to larger integers with all other classes $\beta \in \mathscr{C}_X$. Thus, to remove $2c_1$ from consideration, it suffices to find, for each class $\beta \in \mathscr{C}_X$ with $D . \beta \geq 2$, a relation involving only $2c_1$ and $\beta$ on one side of the equation, as well as a relation involving $2c_1$ and $2(c_1+c_2)$, one involving $2c_1$ and $2(c_1+c_3)$, and one involving $2c_1$ and $(c_1+c_2)+(c_1+c_3)$. These are found in relations $(1)$, $(1)$, $(7)$, $(3)$, $(9)$, $(3)$, $(9)$, $(10)$, $(4)$, $(4)$, and $(8)$ respectively.

In removing $2c_1$, we did not alter the occurrences of $2c_2$ or $2c_3$. Thus, we may likewise remove these two classes from consideration without adding any new instances of $2c_1$.

We have reduced the problem to considering relations between $c_i+c_j$, $2c_i+c_j$, and $c_1+c_2+c_3$. Now note that $2H_1+H_2-H_3$ pairs to $-1$ with $2c_3+c_2$, to zero with $2c_3+c_1$ and $c_2+c_3$, and positively with all other classes. As before, to remove $2c_3+c_2$ from consideration, it suffices to find, for each of these positive classes, a relation involving only $2c_3+c_2$ and that class on one side of the equation. These may be found in relations $(6)$, $(2)$, $(11)$, $(11)$, $(9)$, $(5)$, and $(11)$.

We may not immediately remove $2c_3+c_1$ in the same manner, for, in so doing, we would reintroduce instances of $2c_3+c_2$. However, note that the divisor $H_1+H_2-H_3$ pairs to $-1$ with $2c_3+c_1$, to zero with $c_1+c_3$ and $c_2+c_3$ and positively with all other classes. Then $(2)$, $(11)$, $(11)$, $(9)$, $(11)$, and $(5)$ exhibit relations involving only $2c_3+c_1$ on each of the positive classes on one side of the equation, and no other classes of the form $2c_i+c_j$ on the other. Hence, we may remove $2c_3+c_1$ from consideration in this manner and likewise remove the other classes of the form $2c_i+c_j$.

Finally, as remarked above, the remaining four classes $c_i+c_j$ and $c_1+c_2+c_3$ satisfy the unique relation $(5)$.  We prove main components of free chains of both types appear in some component of free curves below.  Given a general fiber $S$ of a del Pezzo fibration on $X$, we may identify $S$ with a blow-up of $\mathbb{P}^2$ as above.
\begin{enumerate}
    \item General free chains of type $(2c_i, 2c_j)$ smooth to geometrically rational anticanonical curves in smooth fibers of $\pi_k$.  There is a unique component $M$ of $\overline{\free}(X, 2c_i + 2c_j)$ containing such curves.  By smoothing a free chain of type $(c_i + c_j, c_i + c_j)$, where, under the identification of the fiber of $\pi_k$ as a blow-up of $\mathbb{P}^2$, the components are pushforwards of curves of class $l-e_2$ and $2l-e_1 -e_3 -e_4 -e_5$, we see $M$ also contains a main component of $\overline{\free}_1(X, c_i + c_j) \times_X \overline{\free}_1(X, c_i + c_j)$.
    \item Chains of type $(c_i + c_j, c_i + c_j, c_i + c_j)$ corresponding to pushforwards of chains in a fiber $S$ of $\pi_k$ of type $(l-e_2, l-e_3, l-e_4)$ lie in the same component of free curves as chains of type $(2c_i + c_j, c_i + 2c_j)$ corresponding to pushforwards of chains of type $(l, 2l -e_2 -e_3 -e_4)$.
    \item Free chains of type $(2c_i + c_j, 2c_j)$ given by pushforwards of chains of type $(l, 2l-e_2 -e_3 -e_4 -e_5)$ lie in the same component of free curves as free chains of type $(2c_j + c_i, c_i + c_j)$ given by pushforwards of chains of type $(2l-e_2 -e_3 -e_4, l-e_5)$.
    \item Free chains of type $(c_i + c_j, c_i + c_j, 2c_i)$ given by pushforwards of chains of type $(l-e_2, l-e_3, l-e_1)$ lie in the same component of free curves as free chains of type $(2c_i + c_j, 2c_i + c_j)$ given by pushforwards of chains of type $(l, 2l- e_1 -e_2 -e_3)$.
    \item  Smoothing immersed chains of type $(c_i + c_j, c_k, c_k, c_i + c_j)$ or of type $(c_1 + c_2, c_3, c_1, c_2 + c_3)$ proves the unique component $M$ of $\overline{\free}(X, 2c_1  + 2c_2 + 2c_3)$ containing free chains of type $(c_1 + c_2 + c_3, c_1 + c_2 + c_3)$ also contains free chains of types $(c_j + c_k, 2c_i, c_j + c_k)$ and of type $(c_1  +c_2, c_1 + c_3, c_2 + c_3)$.  Then, note that as there is a unique main component of $\overline{\free}_1(X, 2c_i) \times_X \overline{\free}_1(X, 2c_j)$, which by (1) lies in the same component of free curves as free chains of type $(c_i + c_j, c_i + c_j)$, smoothing an immersed chain of type $(2c_i, c_k ,c_k, 2c_j)$ proves our claim using \ref{Basic Properties Chains}.
    \item Similarly, as there is a unique main component of $\overline{\free}_1(X, c_1 + c_2 + c_3) \times_X \overline{\free}_1(X, c_1  +c_2 + c_3)$, as well as a unique main component component of free curves containing smoothings of chains of type $(2c_i, c_j + c_k, c_j + c_k)$, we reduce to proving a component of free curves contains free chains of type $(2c_i, c_j + c_k, c_1 + c_2 + c_3, c_j + c_k)$ and of type $(2c_i + c_j, 2c_k + c_i, 2c_j + c_k)$.  This is accomplished by smoothing an immersed chain of type $(2c_i, c_j, c_k, c_i + c_k, c_j, c_j + c_k)$ in two ways.
    \item As shown in the proof of Lemma \ref{3.1core}, $\overline{\free}(X, 2c_i + c_j + c_k)$ is irreducible.
    \item Smooth an immersed chain of type $(c_i + c_j, c_i, c_i, c_i + c_k)$ in two ways.
    \item Smooth immersed chains of type $(c_j + c_k, c_i, c_i + c_j)$ or of type $(2c_i, c_j, c_j, c_k)$ in two ways.
    \item Smooth an immersed chain of type $(c_i + c_k, c_j, 2c_i)$ in two ways.
    \item Smooth an immersed chain of type $(c_i + c_k, c_j, c_i, c_i + c_j)$ in two ways to prove a component of free curves contains free chains of type $(c_1 + c_2 + c_3, 2c_i + c_j)$ and $(c_i + c_k, c_i + c_j, c_i + c_j)$.  Then smooth an immersed chain of type $(c_i + c_k, c_i, c_j, c_i + c_j)$ in two ways.
    \item Smooth an immersed chain of type $(2c_i, c_j, c_i, c_i, c_k)$ in two ways.
\end{enumerate}
\end{proof}

\begin{lem}
For each $\alpha \in (c_1 + c_2 + c_3) + \Nef_1(X)_\mathbb{Z}$, $\overline{\free}(X,\alpha)$ is nonempty.
\end{lem}
\begin{proof}
The monoid of integer points in $\NE(X) = \Nef_1(X)$ is generated by $c_1, c_2, c_3$.  An inductive argument may be used to show the only curve classes in $\NE(X)$ not representable by free curves are odd multiples of $c_i$.
\end{proof}

\subsection*{3.2}

\textbf{Description of Variety $X$:} $X$ is a member of $|L^{\otimes 2} \otimes_{\mathcal{O}_{\mathbb{P}^1 \times \mathbb{P}^1}} \mathcal{O}(2,3)|$ on $\mathbb{P}_{\mathbb{P}^1 \times \mathbb{P}^1}(\mathcal{O} \oplus \mathcal{O}(-1,-1)^{\oplus 2})$ such that $X \cap Y$ is irreducible, where $L$ is the tautological line bundle and $Y$ is a member of $|L|$.  We keep the notation from \cite{mori1981classification}, which for a vector bundle $V$ identifies $\mathbb{P}(V)$ with $\text{Proj}(\text{Sym}^*V)$, the space of 1 dimensional quotients of $V$.  Thus, $Y$ is the unique member of $|L|$.  Let $\pi : X\rightarrow \mathbb{P}^1_1 \times \mathbb{P}^1_2$ be the natural projection, and $\pi_i : X \rightarrow \mathbb{P}^1_i$ be the projection of $X$ to the $i^{th}$ factor.

\textbf{Generators for $N^1(X)$ and $N_1(X)$:} Let $E = Y \cap X$, $H_1 = \pi_1^* \mathcal{O}(1)$, and $H_2 = \pi_2^* \mathcal{O}(1)$.  Note that $E \cong \mathbb{P}^1 \times \mathbb{P}^1$, as $E = X \cap \mathbb{P}_{\mathbb{P}^1 \times \mathbb{P}^1}(\mathcal{O}(-1,-1)^{\oplus 2}) \subset \mathbb{P}_{\mathbb{P}^1 \times \mathbb{P}^1}(\mathcal{O} \oplus \mathcal{O}(-1,-1)^{\oplus 2})$.  %
The restriction of $\pi : X \rightarrow \mathbb{P}^1 \times \mathbb{P}^1$ to $E$ is a double cover branched over a reducible divisor of class $\mathcal{O}(0,2)$.  Let $f$ be an irreducible component of a reducible fiber of $\pi : X\rightarrow \mathbb{P}^1 \times \mathbb{P}^1$, $e_1 = H_1 \cap E$ be a ruling of $E$, and $e_2 = \frac{1}{2} H_2 \cap E$ be the other ruling.  We define $l = \frac{1}{2}(f + e_1)$, and let $D$ be the divisorial image of $\overline{M}_{0,1}(X,l)$.

\begin{thm}\label{3.2thm}
For all $\alpha \in (f + e_2 + l) +  \Nef_1(X)_\mathbb{Z}$, $\overline{\free}(X,\alpha)$ is nonempty and irreducible.
\end{thm}

\textbf{Intersection Pairing:} $H_1 . e_2 = 1$, $H_2 . e_1 = 2$, $E \cdot e_i = -1$, $E.f =1$, and all other pairings are $0$.

\textbf{Anticanonical Divisor:} $-K_X = E  + 2H_1 + H_2$

\textbf{Effective Divisors:} $E$ is the only non-moving extreme divisor on $X$.

\textbf{Effective Curves:}  The Mori cone is generated by $f$, $e_1$, and $e_2$.  The class $l = \frac{1}{2}(f + e_1)$ is a nef anticanonical line.

\textbf{Mori Structure on $X$:} The contraction of $f$ is the morphism $\pi : X \rightarrow \mathbb{P}^1 \times \mathbb {P}^1$.  It is a conic fibration with reducible fibers and discriminant locus $\Delta_\pi$, a member of $|\mathcal{O}_{\mathbb{P}^1 \times \mathbb{P}^1}(2,5)|$.  The contraction of $e_i$ is a morphism $X\rightarrow Y_i$ onto a smooth weak fano threefold, as $E$ is of $E1$ type and $\mathcal{O}_E(E) \cong \mathcal{O}_E(-1,-1)$.  $Y_1$ and $Y_2$ are related by the flips of their small extreme rays.  For each $i$, the contraction of the other extreme ray (spanned by the pushforward of $f \in N_1(X)$) yields a del Pezzo fibration $Y_i \rightarrow \mathbb{P}^1$.  Generic fibers of $Y_1 \rightarrow \mathbb{P}^1$ are del Pezzo surfaces of degree 4, while generic fibers of $Y_2 \rightarrow \mathbb{P}^1$ smooth quadrics.  The composition $X \rightarrow Y_i \rightarrow \mathbb{P}^1$ is $\pi_i$, and thus factors as $X\rightarrow \mathbb{P}^1_1 \times \mathbb{P}^1_2 \rightarrow \mathbb{P}^1_i$.  Each fiber of $X \subset \mathbb{P}_{\mathbb{P}^1 \times \mathbb{P}^1}(\mathcal{O} \oplus \mathcal{O}(-1,-1)^{\oplus 2}) \xrightarrow{\pi_i} \mathbb{P}^1_i$ is the strict transform of a hypersurface in the blow-up $\text{Bl}_{\ell} \mathbb{P}^3$ of $\mathbb{P}^3$ along a line $\ell$.  For $\pi_1$, a general fiber is the strict transform of a smooth cubic containing $\ell$.  For $\pi_2$, a general fiber is the strict transform of a smooth quadric meeting $\ell$ transversely.  On each fiber $\pi_1^{-1}(p)$ of $\pi_1$, the morphism $X_p \rightarrow Y_{1,p}$ is the contraction of the preimage of $\ell \subset \mathbb{P}^3$ in $X_p \subset \text{Bl}_{\ell} \mathbb{P}^3$.  On each fiber $\pi_2^{-1}(p)$ of $\pi_2$, the morphism $X_p \rightarrow Y_{2,p}$ is the restriction of contraction $\text{Bl}_{\ell} \mathbb{P}^3 \rightarrow \mathbb{P}^3$.

\textbf{Anticanonical lines on $X$:} For numeric reasons, the only anticanonical lines on $X$ are of class $f$, $e_i$, or $l = \frac{1}{2}(f + e_1)$.  We describe these lines in terms of the Mori structure on $X$.  

Each fiber $X_p$ of $\pi_1$ contains a unique curve of class $e_1$.  This curve is the exceptional locus of $X_p \rightarrow Y_{1,p}$.  For general $p \in \mathbb{P}^1_1$, $X_p \rightarrow Y_{1,p}$ is the contraction of a $(-1)$ curve on a smooth del Pezzo surface of degree 3.  There are 10 $(-1)$ curves in $X$ which meet the exceptional locus of $X_p \rightarrow Y_{1,p}$ transversely.  These 10 lines are the irreducible components of the intersection of $X_p$ with the preimage $\pi^{-1}(\Delta_\pi)$ of the discriminant locus of $\pi$, and thus have class $f \in N_1(X)$.  Observe that in $Y_{1,p}$, the pushforward of these curves are the 10 distinct anticanonical conics in $Y_{1,p}$.  The 16 other lines in $X_p$ each have class $l \in N_1(X)$, and are in bijection with the 16 lines in $Y_{1,p}$.  We note that other fibers of $X_p \rightarrow p \times \mathbb{P}^1_2$ are the strict transforms of geometrically rational members of $|-K_{Y_{1,p}}|$ which are singular along the image of the exceptional locus of $X_p \rightarrow Y_{1,p}$.

For general $p \in \mathbb{P}^1_2$, the fiber $X_p$ of $\pi_2$ contains two curves of class $e_2$.  The map $X_p \rightarrow Y_{2,p} \cong \mathbb{P}^1 \times \mathbb{P}^1$ contracting these curves is the blow-up of two general points in $Y_{2,p}$.  The strict transforms of the four lines in $Y_{2,p}$ which meet the exceptional locus of $X_p \rightarrow Y_{2,p}$ are the irreducible components of $\pi^{-1}(\Delta_\pi) \cap X_p$; thus, each has class $f \in N_1(X)$.  We note that other fibers of $X_p \rightarrow \mathbb{P}^1_1 \times p$ are strict transforms of conics in $Y_{2,p} \cong \mathbb{P}^1 \times \mathbb{P}^1$ through the two blown-up points.

\begin{lem}\label{lem: 3.2 divisor of lines of class l}
Let $D$ be the divisor swept out by all curves of class $l$.  $D$ has class $8H_1 + 4H_2 + 8E$.  For $X$ general in moduli, $D$ is irreducible and intersects fibers $X_p$ of $\pi_2$ along preimages of curves of class $(8,8)$ in $Y_{2,p} \cong \mathbb{P}^1 \times \mathbb{P}^1$ that are disjoint from the exceptional locus of $X_p \rightarrow Y_{2,p}$.  
\end{lem}

\begin{proof}
From our description above it is clear that $D.e_1 = 0$ and $D.f = 8$, as a general curve of class $2f$ meets 16 curves of class $l$.  To see that $D.e_2 =0$, note that $D$ is the reduced image of $\overline{M}_{0,1}(X,l)$ under the evaluation map.  Suppose $D. e_2 > 0$.  Then $D \cap E$ must be a curve of class $(D.e_2) e_1$.  This yields a contradiction, as the preimage of $E$ in $\overline{M}_{0,1}(X,l)$ would dominate some component of $\overline{M}_{0,0}(X,l)$ under projection, while each fiber of $\pi_1$ has at most one $A_1$-singularity.  %
The following lemma proves our remaining claims.
\end{proof}

\begin{lem}\label{3.2monodromy}
Assume $X$ is general in moduli.  Then the discriminant locus $\Delta_\pi$ of $\pi$ is a smooth member of $|\mathcal{O}_{\mathbb{P}^1 \times \mathbb{P}^1}(2,5)|$.  The monodromy group of $\pi_1 : X\rightarrow \mathbb{P}^1_1$ acts trivially on the kernel of $N_1(X_p) \rightarrow N_1(Y_{1,p})$ and as the full Weyl group $W(D_5)$ on $N^1(Y_{1,p})$.  The monodromy group of $\pi_2 : X \rightarrow \mathbb{P}^1_2$ acts as $S_2 \times S_2$ on $N_1(X_p)$, interchanging the exceptional curves of $X_p \rightarrow Y_{2,p}$ and the rulings of $Y_{2,p} \cong \mathbb{P}^1 \times \mathbb{P}^1$.
\end{lem}
\begin{proof}
A general threefold $X$ corresponds to a general section of $\text{Sym}^2(\mathcal{O}(1,1) \oplus \mathcal{O}^{\oplus 2}) \otimes \mathcal{O}(0,1)$.  This is a globally generated vector bundle whose 6 coordinates correspond to the coefficients of monomials in the homogeneous, quadratic-defining equation of $X$ in a given fiber.  The determinant of the corresponding symmetric bilinear form vanishes precisely where the fiber of $X$ is singular.  This characterizes $\Delta_\pi$, %
which by \cite[Section~11]{noheighttwo16} is a simple elementary cover of both copies of $\mathbb{P}^1$ under projection.

General fibers of $\pi_1$ contain 5 reducible fibers of $\pi$.  By our choice of $\Delta_\pi$, the monodromy group $G_{mon}$ of $\pi_1 : X \rightarrow \mathbb{P}^1_1$ acts as the full symmetric group on these fibers.  By \cite{mori1983classification}, $\pi^{-1}(\Delta_\pi)$ is irreducible.  The Hilbert scheme of curves of class $f$ is therefore a degree 2 etale cover of $\Delta_\pi$, which shows $G_{mon}$ is a nontrivial extension of $S_5$.  For a given general fiber $X_p$ of $\pi_1$ we previously identified its intersection with $\pi^{-1}(\Delta_\pi)$ as curves which map to 10 distinct anticanonical conics in $Y_{1,p}$.  As $G_{mon}$ acts as $S_5$ on pairs of these conics summing to $-K_{Y_{1,p}}$, the restriction of the natural map $W(D_5) \rightarrow S_5$ to $G_{mon} \subset W(D_5)$ is surjective.  Since $W(D_5)$ is a semidirect product with abelian kernel
$$1 \rightarrow (\mathbb{Z}/2)^4 \rightarrow W(D_5) \rightarrow S_5 \rightarrow 1$$
wherein the conjugacy class of any nontrivial element in $(\mathbb{Z}/2)^4$ generates the entire subgroup, $G_{mon}$ must be the full Weyl group.

Similarly, general fibers of $\pi_2$ contain 2 reducible fibers of $\pi$.  By our choice of $\Delta_\pi$, the monodromy group $G_{mon}$ of $\pi_2 : X \rightarrow \mathbb{P}^1_2$ sends one of these four irreducible components to any other.  By our previous description of $Y_{2,p}$ this proves our claim.
\end{proof}

\begin{rem}
Below, we prove Theorem \ref{3.2thm} when $X$ is general in moduli.  The proof for arbitrary $X$ then follows from Proposition \ref{low degree curves degeneration} and Corollary \ref{connected fibers fixall}.
\end{rem}

\begin{lem}\label{lem: 3.2 weak core}
Assume $X$ is general in moduli.  Every component of $\overline{\free}(X)$ contains free chains with components of class in $\mathscr{C}_X$, where
\begin{align*}
   \mathscr{C}_X = &\{ 2f, \ f+e_1, \ f+e_2, \ f+l, \ 2f + e_1, \ 2f + e_2, \ 2f + l, f + e_1 + l, f + e_2 + l \}%
\end{align*}
For each $\alpha \in \mathscr{C}_X\setminus\{2f + e_1\}$, $\overline{\free}(X,\alpha)$ is irreducible, while $\overline{\free}(X, 2f + e_1)$ has 2 irreducible components.  For $\alpha \in \{2f, f + e_2 + l\}$, general fibers of $\overline{\free}_1(X,\alpha)\rightarrow X$ are irreducible.  Lastly, $\overline{\free}(X,\alpha)$ is irreducible for all $\alpha \in \Nef_1(X)_{\mathbb{Z}}\setminus \partial\NE(X)$ with $-K_X. \alpha \leq 4$.%
\end{lem}
\begin{proof}
\textbf{Nef Curve Classes of Anticanonical Degree Between $2$ and $4$:}
Nef curve classes of the appropriate degree are either those in $\mathscr{C}_X$ or $3f, 4f, 3f + e_i, 3f + l, 2f + 2e_i, 2f + e_1 + e_2, 2f + e_i + l$.  We remark that there are no irreducible free curves of class $3f$.  The classes $4f$, $3f + e_i$, $3f + l$, $2f + 2e_i$, and $2f + e_1 + l$ are freely breakable by Mori's bend-and-break, as they are free quartic curves contracted by del Pezzo fibrations.  $\overline{\free}(X, 2f)$ is irreducible because it parameterizes fibers of a conic bundle.  For all other $\alpha \in \mathscr{C}_X\setminus \{ 2f + e_1, f + e_2 + l\}$, irreducibility of $\overline{\free}(X,\alpha)$ follows from Lemma \ref{3.2monodromy}.  This also demonstrates $\overline{\free}(X, 2f + e_1)$ has two components, one of which corresponds to geometrically rational anticanonical curves in general fibers of $\pi_1 : X \rightarrow \mathbb{P}^1$; the other component generically parameterizes strict transforms of quartics in fibers of $Y_1 \rightarrow \mathbb{P}^1$ that meet the exceptional locus of $X\rightarrow Y_1$.  %
In the following paragraphs, we will prove irreducibility of $\overline{\free}(X, f  + e_2 + l)$, $\overline{\free}(X, 2f + e_1 + e_2)$, and $\overline{\free}(X, 2f + e_2 + l)$.  

First, let $\alpha \in \{ f  + e_2 + l, \ 2f  + e_2 + l\}$. The contraction $\pi : X \rightarrow \mathbb{P}^1 \times \mathbb{P}^1$ induces a morphism  
$$\pi_* : \overline{\free}(X, \alpha) \rightarrow \overline{\free}(\mathbb{P}^1 \times \mathbb{P}^1, \pi_*\alpha) \cong |\mathcal{O}(1,1)|$$
which maps each component of $\overline{\free}(X, \alpha)$ dominantly onto $|\mathcal{O}(1,1)|$.  By Corollary \ref{cor: conic monodromy cor} this shows $\overline{\free}(X, \alpha)$ is in fact irreducible.  From this and Theorem \ref{reducible fibers: 4 author result}, it follows that general fibers of $\overline{\free}_1(X, f + e_2 + l) \to X$ are irreducible.

Lastly, suppose $\alpha = 2f + e_1 + e_2$ and let $M \subset \overline{\free}(X,\alpha)$ be any component.  In this case, $\pi$ induces a generically finite map $\pi_* : M \rightarrow \overline{\free}(\mathbb{P}^1\times \mathbb{P}^1, \pi_*(e_1 + e_2)) \cong |\mathcal{O}(2,1)|$.  Consider the locus $T \subset M$ parameterizing maps passing through a general point $q \in X$ and which meets a general complete intersection curve $c \subset X$ that is disjoint from $E$.  Since $\pi_*$ is generically finite and $T$ is one-dimensional, its image in $|\mathcal{O}(2,1)|$ meets the divisor $|\mathcal{O}(1,1)| \times |\mathcal{O}(1,0)| \subset |\mathcal{O}(2,1)|$.  

Let $g : C \to X$ be a stable map parameterized by a point $[g] \in T$ such that $\pi_*[g] \in \mathcal{O}(1,1)| \times |\mathcal{O}(1,0)|$.  The domain of $C$ of $g$ may be expressed as a union $C = C_1 \cup C_2$ of connected curves such that $(\pi \circ g)_*C_1$ has bidegree $(1,1)$ and $(\pi \circ g)_*C_2$ has bidegree $(1,0)$ in $\mathbb{P}^1 \times \mathbb{P}^1$.  We must have $g_*C_1 = af + e_2 + l$ and $g_*C_2 = bf + l$ with $\{a,b\} = \{0,1\}$.  Since $C_1$ is connected, we find that $(a,b) = (1,0)$.  

As $[g] \in T$, the incidence of $g(C)$ with the general point $q \in X$ and complete intersection curve $c \subset X$ forces $C_1$ to be irreducible and $g|_{C_1}$ to be free.  In fact, by varying $q$ and $c$ we may assume $g$ is general in a component of $\overline{\free}_{1}(X, g_* C_1) \times_X \overline{M}_{0,1}(X, g_* C_2)$ that dominates an irreducible component of both $\overline{\free}(X, g_*C_1)$ and $\overline{M}_{0,0}(X, g_* C_2)$ under projection.  It follows that $g( C_1 \cap C_2)$ is a general point in the divisor $D \subset X$ swept out by curves of class $l$.

By generality of the restrictions $g|_{C_i}$, we may assume $g(C_1)$ meets $D$ transversely and $[g]$ is a smooth point of $M$.  Note that $\overline{\free}_1(X, f+ e_2 + l) \rightarrow X$ is equidimensional away from a codimension 2 locus.  We will prove that it is smooth over an open subset of $D$.  This implies there is a unique component of $\overline{\free}_1(X, f+ e_2 + l) \times_X \overline{M}_{0,1}(X, l)$ satisfying our constraints.  Since it meets the smooth locus of any component of $\overline{\free}(X, 2f + e_1 + e_2)$, irreducibility of the entire space follows.

Consider the possible singularities of a fiber $F_p$ of $\overline{\free}_1(X, f+ e_2 + l) \rightarrow X$ over a general point $p \in D$.  As shown in the proof of Theorem \ref{main result 4}, since $X$ is general the only rational curves in $D$ are those of class $l$.    The locus of maps in $\overline{\free}(X, f+ e_2 + l) \setminus \free(X, f + e_2 +l)$ with irreducible domain sweeps out a proper subscheme $V_1 \subset X$.  Since none of these curves can lie in $D$, the divisor $V_1$ cannot contain $D$ as a component.  Hence for general $p\in D$, singular points of $F_p$ must lie on pointed maps with reducible image.  The locus of maps in $\overline{\free}(X, f+ e_2 + l) \setminus \free(X, f + e_2 +l)$ with reducible image parameterizes chains of type $(e_2, f, l)$, $(f + e_2, l)$, and $(f + l, e_2)$.  Chains of type $(e_2, f, l)$ only meet special points in $D$.  If a curve of class $f + e_2$ meets a general point in $D$, it is free and transverse to $D$.  Thus the chain $(f + e_2, l)$ would be a smooth point of the fiber $F_p$.  If a chain of type $(f + l, e_2)$ meets a general point in $D$, the component of class $f + l$ must be free, as we may assume it lies in a smooth fiber of $\pi_1$.  We conclude that for general $p \in D$, the fiber $F_p$ is smooth.  Since $f + e_2 + l$ is nonseparating, $F_p$ is connected, which proves irreduciblity of $\overline{\free}(X, 2f + e_1 + e_2)$.
\end{proof}

\begin{lem}
Relations in the monoid $\mathbb{N}\mathscr{C}_X$ are generated by:
\begin{enumerate}
    \item $2f + (f+e_1) = 2(f + l)$,
    \item $2f + (2f + e_1) = (f + l) + (2f + l)$,
    \item $2f + (f + e_2 + l) = (f + e_2) + (2f + l) = (f+l) + (2f + e_2)$,
    \item $2f + (f + e_1 + l) = (f + e_1) + (2f + l) = (f+l) + (2f + e_1)$,
    \item $2f +2(f+ l) = 2(2f + l)$,
    \item $2f + 2(f + e_2) = 2(2f + e_2)$,
    \item $2f + (f + l) + (f +e_2) = (2f + l) + (2f + e_2)$,
    \item $(2f + l) + (2f + e_1) = 3(f + l)$,
    \item $(2f + l) + (f + e_1 + l) = 2(2f + e_1) = 2(f + l) + (f + e_1)$,
    \item $(2f + l) + (f + e_2 + l) = (2f + e_1) + (2f + e_2) = (f + e_2) + 2(f + l)$,
    \item $(2f + e_2) + (f + e_1) = (2f + e_1) + (f + e_2) = (f + l) + (f + e_2 + l)$,
    \item $(2f + e_2) + (f + e_1 + l) = (f + l) + (f + e_1) + (f +e_2) = (2f +e_1) + (f + e_2 + l)$,
    \item $(2f + e_2) + (f + e_2 + l) = (f + l) + 2(f + e_2)$,
    \item $(f + l) + (f + e_1 + l) = (2f + e_1) + (f + e_1)$,
    \item $(f + e_2) + (f + e_1 + l) = (f + e_1) + (f + e_2 + l)$,
    \item $2(f + e_2) + (f + e_1) = 2(f + e_2 + l)$,
    \item $2(f + e_1 + l) = 3(f + e_1)$.
\end{enumerate}
For each relation $\sum \alpha_i = \sum \beta_j$, a main component of $\prod_X \overline{\free}_2(X,\alpha_i)$ lies in the same component of free curves as a main component of $\prod_X \overline{\free}_2(X, \beta_j)$.
\end{lem}
\begin{proof}
\textbf{Relations:} We apply Lemma \ref{Relations}.
\begin{itemize}
    \item $D = -E + 2H_1 + 2H_2$ satisfies $D.2f = -2$, $D.(2f + l) = D.(2f + e_2) = 0$, and $D. \alpha > 0$ for all other $\alpha \in \mathscr{C}_X$.  We find relations (1) through (7).
    \item $D = -E + H_1 + H_2$ pairs negatively with $2f + l$, trivially with $f + l$ and $2f + e_2$, and positively with all other $\alpha \in \mathscr{C}_X \setminus \{2f\}$.  We obtain relations (3),(4), and (8)-(10).
    \item $D = -E + H_2$ pairs negatively with $2f + e_2$, trivially with $f + e_2$ and $f +l$, and positively with all other $\alpha \in \mathscr{C}_X \setminus\{ 2f , 2f + l\}$.  We obtain relations (10)-(13).
    \item $D = -2E + H_2$ pairs negatively with $f + l$, trivially with $f + e_2$ and $2f + e_1$, and positively with each other $\alpha \in \mathscr{C}_X \setminus \{2f ,2f + l, 2f + e_2\}$.  We obtain relations (9), (11), (12), and (14).
    \item $D = -E -2H_1 + 2H_2$ pairs negatively with $f + e_2$, trivially with $2f + e_1$ and $f + e_2 + l$, and positively with each other $\alpha \in \mathscr{C}_X \setminus \{2f, 2f + l, 2f + e_2, f + l\}$.  We obtain relations (15) and (16).
    \item The remaining four classes in $\mathscr{C}_X$ span $N_1(X)$ and satisfy relation (17).
\end{itemize}

\textbf{Main Components:} We address each relation below.  For convenience, we let $X_p$ be a general fiber of $\pi_1 : X \rightarrow \mathbb{P}^1$, and write $N_1(X_p) \cong \mathbb{Z} h \oplus_{i = 1}^6 \mathbb{Z} \epsilon_i$, where $\epsilon_1$ is the $(-1)$ curve contracted by $X_p \rightarrow Y_{1,p}$, $h$ is the pullback of a hyperplane class under some contraction $X_p \rightarrow Y_{1,p} \rightarrow \mathbb{P}^2$, and $\epsilon_i$ are exceptional divisors of $Y_{1,p}\rightarrow \mathbb{P}^2$ for $i > 1$.  Note that under the inclusion $X_p \rightarrow X$, $\epsilon_1 \rightarrow e_1$, $\epsilon_i \rightarrow l$ for $i > 1$, and $h \rightarrow f + e_1 + l$.  Similarly, we let $X_q$ be a general fiber of $\pi_2 : X \rightarrow \mathbb{P}^1$ and write $N_1(X_q) \cong \oplus_{i=1}^2\mathbb{Z}h_i \oplus_{j = 1}^2 \mathbb{Z} \sigma_j$, where each $\sigma_j$ is a $(-1)$ curve contracted by $X_q \rightarrow Y_{2,q} \cong \mathbb{P}^1 \times \mathbb{P}^1$, and each $h_i$ is the pullback of a ruling of $Y_{2,q}$.  Note that under the inclusion $X_q \rightarrow X$, $\sigma_j \rightarrow e_2$ for each $j$, and $h_i \rightarrow f + e_2$ for each $i$.
\begin{enumerate}
    \item A component of $\overline{\free}(X, 3f + e_1)$ contains curves in $X_p$ of class $4h - 2\epsilon_1 -2\epsilon_2 - \sum_{i \geq 3}^6 \epsilon_i$.  This component contains free chains of type $(2f, f + e_1)$ and of type $(f + l, f + l)$.
    \item A component of $\overline{\free}(X, 4f + e_1)$ contains curves in $X_p$ of class $5h - 3\epsilon_1 -2\epsilon_2 -2\epsilon_3 - \epsilon_4 - \epsilon_5 - \epsilon_6$.  This component contains free chains of type $(2f, 2f + e_1)$ and $(f + l, 2f + l)$.
    \item A component of $\overline{\free}(X, 3f + e_2 + l)$ contains immersed chains of type $(f, e_2, f, f, l)$.  There is a one parameter family of such chains, and a generic member must be unobstructed.  Therefore, there is a component of $\overline{\free}(X, 3f + e_2 + l)$ containing free chains of type $(2f + e_2, f + l)$ and $(2f + l, f + e_2)$.  Another component of $\overline{\free}(X, 3f + e_2 + l)$ contains immersed chains of type $(f,f,e_2, f, l)$.  As before, this component contains free chains of type $(2f + e_2, f + l)$ and of type $(2f, f + e_2 +l)$.
    \item A component of $\overline{\free}(X, 3f + e_1 + l)$ contains curves in $X_p$ of class $4h - 2\epsilon_1 - \sum_{i \geq 2}^6 \epsilon_i$.  This component contains free chains of type $(2f, f + e_1 + l)$, $(f + e_1, 2f + l)$, and $(f + l, 2f + e_1)$.
    \item A component of $\overline{\free}(X, 5f + e_1)$ contains curves in $X_p$ of class $6h -4 \epsilon_1 - 2\epsilon_2 -2\epsilon_3 -2\epsilon_4 - \epsilon_5 -\epsilon_6$.  This component contains free chains of type $(2f, f + l, f + l)$ and of type $(2f + l, 2f + l)$.  
    \item A component of $\overline{\free}(X, 4f + 2e_2)$ contains curves in a general fiber $X_q$ of $\pi_2$ of class $2h_1 + 2h_2 - \sigma_1 - \sigma_2$.  This component contains free chains of type $(2f , f + e_2, f + e_2)$ and of type $(2f + e_2, 2f + e_2)$.
    \item $\overline{\free}(X, 4f + e_2 + l)$ contains smooth points that parameterize immersed chains of type $(f + e_2, f, f, f + l)$.  Components containing these chains contain free chains of type $(f + e_2, 2f, f + l)$ and $(2f + e_2, 2f + l)$.
    \item A component of $\overline{\free}(X, 4f + e_1 + l)$ contains curves in $X_p$ of class $5h - 3\epsilon_1 - 2\epsilon_2 - \sum_{i \geq 3}^6 \epsilon_i$.  This component contains free chains of type $(2f + l, 2f + e_1)$ and $(f + l, f + l, f + l)$.  
    \item A component of $\overline{\free}(X, 4f + 2e_1)$ contains curves in $X_p$ of class $4h -2\epsilon_1 - \sum_{i \geq 2}^5 \epsilon_i$.  This component contains free chains of type $(2f + l, f + e_1 + l)$, $(2f + e_1, 2f + e_1)$, and $(f + l, f + e_1, f + l)$.
    \item $\overline{\free}(X, 4f + e_1 + e_2)$ contains smooth points that parameterize immersed chains of type $(f + e_2, f, l, f + l)$ and of type $(f + e_2, l, f, f + l)$.  Components containing the first type of chain contain free chains of type $(2f  +e_2, 2f + e_1)$ and $(f + e_2, f + l, f +l)$, while components containing the second type of chain contain free chains of type $(2f + l, f + e_2 + l)$ and $(f + e_2,f + l, f + l)$.
    \item $\overline{\free}(X, 3f + e_1 + e_2)$ contains smooth points that parameterize immersed chains of type $(f + e_2, f, f + e_1)$ and $(f + l,  l, f +e_2)$.  Components containing the first type of chain contain free chains of type $(2f + e_2, f +e_1)$ and $(f + e_2, 2f + e_1)$, while components containing chains of the second type contain free chains of type $(2f + e_1, f  + e_2)$ and $(f + l, f + e_2 + l)$.
    \item $\overline{\free}(X, 3f + e_1 + e_2 + l)$ contains smooth points parameterizing immersed chains of type $(l, f , e_1, f , f + e_2)$ and immersed chains of type $(f, l, l, l, f + e_2)$.  Components containing the first type of chain contain free chains of type $(f + l, f + e_1, f + e_2)$ and $(f + e_1 + l, 2f + e_2)$.  The second type of chain smooths to free chains of type $(f + l, f + e_1, f + e_2)$ and $(2f + e_1, f + e_2 + l)$.
    \item There are immersed chains of type $(f + e_2, f, e_2, f + l)$ parameterized by smooth points of $\overline{\free}(X, 3f + 2e_2 + l)$.  Components containing these chains contain free chains of type $(2f + e_2, f + e_2 + l)$ and $(f + e_2, f + e_2, f + l)$.
    \item A component of $\overline{\free}(X, 3f + 2e_1)$ contains curves in $X_p$ of class $4h - \epsilon_2 - \sum_{i = 1}^6 \epsilon_i$.  This component contains free chains of class $(2f + e_1, f + e_1)$ and of class $(f + l, f + e_1 + l)$.
    \item There are immersed chains of type $(f + e_2, l, f + e_1)$ parameterized by smooth points of $\overline{\free}(X, 2f + e_1 + e_2 + l)$.  Components containing these chains contain free chains of type $(f + e_2, f + e_1 + l)$ and $(f + e_2 + l, f + e_1)$.
    \item There are immersed chains of type $(e_2, f, l, l, f, e_2)$ contained in the smooth locus of $\overline{\free}(X, 3f + e_1 + 2e_2)$.  The component containing these chains contains free chains of type $(f + e_2, f + e_1, f + e_2)$ and $(f + e_2 + l, f + e_2 + l)$.
    \item A component of $\overline{\free}(X, 3f + 3e_1)$ contains curves in $X_p$ of class $3h - \epsilon_2 - \epsilon_3 - \epsilon_4$.  This component contains free chains of type $(f + e_1 + l, f+ e_1 + l)$ and of type $(f + e_1, f + e_1, f + e_1)$.
\end{enumerate}
\end{proof}

\begin{lem}
For each $\alpha \in (f + e_2 + l) + \Nef_1(X)_\mathbb{Z}$, $\overline{\free}(X,\alpha)$ is nonempty.
\end{lem}

\begin{proof}
The lattice of integer points in $\NE(X)$ is generated by $e_1, e_2, f$, and $l$.  Lemma \ref{Gordan's Lemma} and \ref{Representability of Free Curves} show that the class $f + e_2 \in \mathscr{C}_X$, together with $f$ and $l$, generate the lattice of integer in $\Nef_1(X)$.  We may write $\alpha = (l + f + e_2) +  a l + b f + \sum_{\alpha_i \in \mathscr{C}_X} c_i \alpha_i$, with $0 \leq a,b \leq 1$ and $c_i \geq 0$.  Our claim follows immediately.
\end{proof}

\subsection*{3.3}

\textbf{Description of Variety $X$:} $X$ is a divisor on $\mathbb{P}^1\times \mathbb{P}^1\times \mathbb{P}^2$ of tridegree $(1,1,2)$.  Both projections to $\mathbb{P}^1\times \mathbb{P}^2$ realize $X$ as a blow up of $\mathbb{P}^1\times \mathbb{P}^2$ along the complete intersection $c$ of two divisors of class $(1,2)$.  We find $c$ is a curve of degree $(4,4)$ and genus $g(c)=3$ which embeds under the second projection $\pi_2 :\mathbb{P}^1\times \mathbb{P}^2 \rightarrow \mathbb{P}^2$.  Let $\pi_1 :X \rightarrow \mathbb{P}^1$ and $\pi_2:X\rightarrow \mathbb{P}^2$ be the compositions of the blow-up with each projection.

\textbf{Generators for $N^1(X)$ and $N_1(X)$:} Let $H_1 = \pi_1^* \mathcal{O}(1)$, $H_2 = \pi_2^* \mathcal{O}(1)$, and $E$ be the exceptional divisor of the blow-up.  Let $l_1$ be the class of a fiber of $\pi_2$, $l_2$ be the class of a general line in each fiber of $\pi_1$, and $e$ be a fiber of $E\rightarrow c$.  Note that there is pseudosymmetry, explained below.%

\begin{thm}\label{3.3thm}
For all $\alpha \in l_1 + \Nef_1(X)_\mathbb{Z}$, $\overline{\free}(X,\alpha)$ is nonempty and irreducible.
\end{thm}

\textbf{Intersection Pairing:} $ H_i \cdot l_j = \delta_{ij}$, $E \cdot e = -1$, and all other pairings are $0$.

\textbf{Anticanonical Divisor:} $-K_X = 2H_1 + 3H_2 - E$

\textbf{Effective Divisors:} $4H_2 -E$, $H_1 + 2H_2 - E$, $E$ are effective.

\textbf{Effective Curves:}  The Mori cone is generated by $l_2 -2e$, $e$, and $l_1-e$.  Realizing $X$ as a blow up of $\mathbb{P}^1\times \mathbb{P}^2$ via the two different projections amounts to permuting $e$ and $l_1-e$ while fixing $l_2-2e$. %

\textbf{Pseudosymmetry:} The above permutation amounts to the linear action $e\rightarrow l_1 -e$, $l_1 \rightarrow l_1$, $l_2 \rightarrow 2l_1 + l_2 -4e$. 

\begin{lem}
A core of free curves on $X$ is given by
\begin{align*}
   \mathscr{C}_X = &\{ l_1, \ l_2, \ l_2-e, \ l_1 + l_2-2e, \ 2l_2-3e, \ 2l_2 -4e,\\
   & 2l_1 + l_2 -4e, \ l_1 + l_2 -3e, \ l_1 + 2l_2 -5e  \} 
\end{align*}
The only separating classes $\alpha \in \mathscr{C}_X$ are $l_2 -e$, $2l_2 -3e$, $l_1 + l_2 -3e$, and $l_1 + 2l_2 -5e$.
\end{lem}
\begin{proof}
\textbf{Nef Curve Classes of Anticanonical Degree Between $2$ and $4$:}
If $\alpha = d_1 l_1 + d_2 l_2 -ne$ is nef and of appropriate degree, we find $n\leq \text{min}(4d_2, d_1 + 2d_2)$, $d_1 + d_2 \leq 4$ and $2d_1 -d_2 \leq 4$.  %
This gives the following curve classes.
\begin{itemize}
    \item $l_1, l_2, l_2 -e, l_1 + l_2 -e, l_1 + l_2 -2e, 2l_2 -3e, 2l_2 -4e, 3l_2 -5e, 3l_2 -6e, 4l_2 -8e, 2l_1, 2l_2 -2e, l_1 + 2l_2 -4e$,
    \item $l_1 + l_2 -3e, l_1 + 2l_2 -5e, l_1 + 3l_2 -7e, 2l_1 + l_2 -3e, 2l_1 + l_2 -4e, 2l_1 + 2l_2 -6e$
\end{itemize}
Pseudosymmetry reduces the number of classes we need to consider.  More specifically, it interchanges $l_2$ with $2l_1 + l_2 -4e$, swaps $l_2 -e$ with $l_1 + l_2 -3e$, swaps $l_1 + l_2 -e$ with $2l_1 + l_2 -3e$, swaps $2l_2 -3e$ with $l_1 + 2l_2 -5e$, swaps $3l_2 -5e$ with $l_1 + 3l_2 -7e$, and swaps $2l_2 -2e$ with $2l_1 + 2l_2 -6e$.  In particular, it suffices to study classes appearing in the first bullet.

The class $3l_2 -6e$ is not representable by an irreducible free curve.  A general fiber of $\pi_1$ meets $c$ transversely at $4$ points.  Any map $f:\mathbb{P}^1\rightarrow X$ with $f_*[\mathbb{P}^1] = 3l_2 -6e$ is contained in such a fiber.  However, the only cubics meeting 4 points with total multiplicity 6 are triple covers of lines.

\textbf{Freely Breakable Classes:}  By standard means, we see the classes $l_1 + l_2 -e$, $2l_1$, and $2l_2 -2e$ are freely breakable.  Similarly, $4l_2 -8e$ is only representable by free maps $f:\mathbb{P}^1\rightarrow X$ which are double covers, which we may break.  Likewise, the space $\free(X,3l_2 -5e)$ is irreducible, and its closure contains maps from nodal curves with free components of class $l_2-e$ and $2l_2-4e$.  Lastly, Lemma \ref{curves2 in P1 x P2} shows that $\free(X, l_1 + 2l_2 -4e)$ is irreducible, as there is a unique map $f:\mathbb{P}^1\rightarrow \mathbb{P}^2$ passing through 4 prescribed linearly general points at any 4 chosen points of $\mathbb{P}^1$.  Therefore, as both $l_1$ and $2l_2 -4e$ are free curves, this component parameterizes freely breakble curves.  Together with the symmetry mentioned above, that leaves us with considering the following classes: $l_1, l_2, l_2 -e, l_1 + l_2 -2e, 2l_2 -3e, 2l_2 -4e$.

\textbf{Irreducible Spaces and Fibers:} 
It is clear from Lemmas \ref{curves1 in P1 x P2}, \ref{curves2 in P1 x P2}, and \ref{curves3 in P1 x P2} that all the curve classes above have irreducible spaces of free curves.  We find that fibers of $\text{ev}_{2l_2 -3e} : \overline{\free}_1(X,2l_2 -3e)\rightarrow X$ generically have four irreducible components, corresponding to the choice of 3 out of four points in a general fiber of $\pi_1 : X \rightarrow \mathbb{P}^1$.  Similarly, fibers of $\text{ev}_{l_2 - e} : \overline{\free}_1(X,l_2 -e)\rightarrow X$ also have four irreducible components.  However, fibers of $\text{ev}_{l_1} : \overline{\free}_1(X,l_1)\rightarrow X$, $\text{ev}_{l_2} : \overline{\free}_1(X,l_2)\rightarrow X$, and $\text{ev}_{2l_2 - 4e} : \overline{\free}_1(X,2l_2 -4e)\rightarrow X$ are all irreducible.  Showing that fibers of $\text{ev}_{l_1 + l_2 - 2e} : \overline{\free}_1(X,l_1 + l_2 -2e)\rightarrow X$ are irreducible takes care.  Fixing a point $p \in X$, we find the fiber $\text{ev}_{l_1 + l_2 -2e}^{-1}(p)$ is a 6-fold cover of the space of lines in $\mathbb{P}^2$ through $\pi_2(p)$.  The 6 different sheets correspond to the 6 choices of 2 out of 4 points at which each line meets $\pi_2(c)$.  It suffices to show the monodromy of these points as we vary lines through $\pi_2(p)$ is two transitive.  Provided that $\pi_2(p)$ is general, the ramification of the map $\pi_2(c)\rightarrow \mathbb{P}^1$ induced by projection from $\pi_2(p)$ satisfies \ref{Monodromy}a.  Therefore, the monodromy is the full symmetric group, and $\text{ev}_{l_1 + l_2 -2e}^{-1}(p)$ is irreducible.
\end{proof}

\begin{lem}
Up to pseudosymmetry, relations in the monoid $\mathbb{N}\mathscr{C}_X$ are generated by:
\begin{enumerate}
     \item $l_2 + (l_1 + l_2 - 2e) = l_1 + 2(l_2 -e)$
        \item $l_2 + (2l_2 -3e) = 3(l_2 -e)$
        \item $l_2 + (2l_2 -4e) = (l_2 -e) + (2l_2 -3e)$
        \item $l_2 + (2l_1 + l_2 - 4e) = (l_2 -e) + l_1 + (l_1 + l_2 -3e) = 2(l_1 + l_2 -2e)$
        \item $l_2 + (l_1 + l_2 -3e) = (l_2 -e) + (l_1 + l_2 -2e)= l_1 + (2l_2 -3e)$
        \item $l_2 + (l_1 + 2l_2 -5e) = 2(l_2 -e) + (l_1 + l_2 -3e)= (l_1 + l_2 -2e) + (2l_2 -3e)$
        \item $l_1 + (2l_2 - 4e) = (l_2 -e) + (l_1 + l_2 -3e)$
        \item $(l_1 + l_2 -2e) + (2l_2 -4e) = (l_1 + l_2 -3e) + (2l_2 -3e) $%
        \item $(2l_2 -4e) + 2(l_2 -e) = 2(2l_2 -3e)$
        \item $(2l_2 -4e) + (l_2 -e) + (l_1 + l_2 -3e) = (2l_2 -3e) + (l_1 + 2l_2 -5e)$
\end{enumerate}
For each relation $\sum \alpha_i = \sum \beta_j$, a main component of $\prod_X \overline{\free}_2(X,\alpha_i)$ lies in the same component of free curves as a main component of $\prod_X \overline{\free}_2(X, \beta_j)$.
\end{lem}
\begin{proof}

\textbf{Relations:}
Following Lemma \ref{Relations}, we order $\mathscr{C}_X$ as $l_2$, $2l_1 + l_2 -4e$, $l_1$, $l_1 +l_2 -2e$, $2l_2 -4e, \ldots$ and proceed.
\begin{itemize}
    \item $D=-H_2 +E$: $D$ pairs negatively with $l_2$ and positively with $l_1 + l_2 -2e$, $2l_2 -3e$, $2l_2 -4e$, $2l_1 + l_2 -4e$, $l_1 + l_2 -3e$, and $l_1 + 2l_2 -5e$.  We find the following relations: %
    \begin{itemize}
        \item $l_2 + (l_1 + l_2 - 2e) = l_1 + 2(l_2 -e)$
        \item $l_2 + (2l_2 -3e) = 3(l_2 -e)$
        \item $l_2 + (2l_2 -4e) = (l_2 -e) + (2l_2 -3e)$
        \item $l_2 + (2l_1 + l_2 - 4e) = (l_2 -e) + l_1 + (l_1 + l_2 -3e)$
        \item $l_2 + (l_1 + l_2 -3e) = (l_2 -e) + (l_1 + l_2 -2e)$
        \item $l_2 + (l_1 + 2l_2 -5e) = 2(l_2 -e) + (l_1 + l_2 -3e)$
    \end{itemize}
    \item Relations involving $2l_1 + l_2 -4e$: By symmetry, these relations are the same as relations with $l_2$.
    \item $D= -H_1 -H_2 + E$: $D$ pairs negatively with $l_1$ and positively with $2l_2 -3e$, $2l_2 -4e$, $l_1 + l_2 -3e$, $l_1 + 2l_2 -5e$.  $D$ pairs to $0$ with $l_2 -e$ and $l_1 + l_2 -2e$.  We find the following relations:
    \begin{itemize}
        \item $l_1 + (2l_2 -3e) = (l_2 -e) + (l_1 + l_2 -2e)$
        \item $l_1 + (2l_2 - 4e) = (l_2 -e) + (l_1 + l_2 -3e)$
        \item $l_1 + (l_1 + 2l_2 -5e) = (l_1 + l_2 -2e) + (l_1 + l_2 -3e)$ (symmetric to the first relation)
        \item There is no relation $l_1 + (l_1 + l_2 - 3e) = \ldots$.  Any relation $a l_1 + b(l_1 + l_2 -3e) + c(l_1 + l_2 -2e) = \ldots$ must have either $l_1$ or $2l_1 + l_2 -4e)$ on the Right Hand Side, as $H_2 -H_1$ pairs nonnegatively with each other nef curve.  Relations involving $2l_1 + l_2 -4e$ were addressed above. Thus, the only minimal relation is $l_1 + (l_1 + l_2 -3e) + (l_2 -e) = 2(l_1 + l_2 -2e)$.
    \end{itemize}
    \item $D=E-2H_1 -H_2:$ $D$ pairs negatively with $l_1 + l_2 -2e$ and positively with $2l_2 -3e$, $2l_2-4e$, and $l_1 + 2l_2 -5e$.  We find the following relations:
    \begin{itemize}
        \item $(l_1 + l_2 -2e) + (2l_2 -3e) = (l_1 + l_2 -3e) + 2(l_2 -e)$
        \item $(l_1 + l_2 -2e) + (2l_2 -4e) = (l_1 + l_2 -3e) + (2l_2 -3e)$
        \item $(l_1 + l_2 -2e) + (l_1 + 2l_2 -5e) = 2(l_1 + l_2 -3e) + (l_2 -e)$ (symmetric to the first relation)
    \end{itemize}
    \item $D= 4H_1 + 3H_2 -2E$: $D$ pairs negatively with $2l_2 -4e$ and positively with $l_2-e$ and $l_1 + l_2 -3e$.  %
    As $D.(2l_2 -4e) = -2$ and $D.(l_2 -e) = D.(l_1 + l_2 -3e) = 1$, any relation involves at least 2 copies of $(l_2 -e)$ and $(l_1 + l_2 -3e)$.  We find the following relations:
    \begin{itemize}
        \item $(2l_2 -4e) + 2(l_2 -e) = 2(2l_2 -3e)$
        \item $(2l_2 -4e) + 2(l_1 + l_2 -3e) = 2(l_1 + 2l_2 -5e)$ (symmetric to the first relation)
        \item $(2l_2 -4e) + (l_2 -e) + (l_1 + l_2 -3e) = (2l_2 -3e) + (l_1 + 2l_2 -5e)$
    \end{itemize}
    \item The remaining 4 curve classes span $N_1(X)$, and the unique relation among them is $(l_1 + l_2 -3e) + (2l_2 -3e) = (l_2 -e) + (l_1 + 2l_2 -5e)$.  However, as $(l_1 + l_2 -2e) + (2l_2 -4e) = (l_1 + l_2 -3e) + (2l_2 -3e)$ is pseudosymmetric to $(l_1 + l_2 -2e) + (2l_2 -4e) = (l_2 -e) + (l_1 + 2l_2 -5e)$, this relation is redundant.
\end{itemize}

\textbf{Main Components:} We address each relation below.
\begin{enumerate}
    \item $\overline{\free}(X, l_1 + 2l_2 -2e)$ is irreducible by Lemma \ref{curves2 in P1 x P2}
    \item $\overline{\free}(X, 3l_2 -3e)$ has at least three components.  One component parameterizes triple covers of curves of class $l_2 -e$. At least one component parameterizes planar cubics that are singular at one of their two intersection points with the blown-up curve $c$.  The only other component parameterizes planar cubics meeting $c$ at three points in their smooth locus.  This last component contains free chains of type $(l_2, 2l_2 -3e)$ and $(l_2 -e, l_2 -e, l_2 -e)$.
    \item $\overline{\free}(X,3l_2 -4e)$ has multiple components, but only one component that parameterizes planar cubics meeting $c$ at four nonsingular points.  This component contains free chains of type $(l_2, 2l_2 -4e)$ and $(l_2 -e, 2l_2 -3e)$.
    \item $\overline{\free}(X,2l_1 + 2l_2 -4e)$ is irreducible.  Using $\pi_2: X\rightarrow \mathbb{P}^2$, we may fiber an open locus of any component over the space of conics in $\mathbb{P}^2$.  A general fiber consists of curves of class $(2,1)$ in $\mathbb{P}^1 \times \mathbb{P}^1$ meeting 4 of 8 intersection points with $c$.  Generally these 8 points lie in separate fibers of each projection of $\mathbb{P}^1\times \mathbb{P}^1$ to $\mathbb{P}^1$.  Hence, there is a one dimensional family of twisted cubics through any 4 of them. As the monodromy of intersections of $\pi_2(c)$ with conics in $\mathbb{P}^2$ is the full symmetric group, this proves our claim.
    \item $\overline{\free}(X,l_1 + 2l_2 -3e)$ is irreducible by Lemma \ref{curves2 in P1 x P2}.
    \item $\overline{\free}(X,l_1 + 3l_2 -5e)$ is irreducible by Lemma \ref{curves2 in P1 x P2}.
    \item $\overline{\free}(X,l_1 + 2l_2 -4e)$ is irreducible by Lemma \ref{curves2 in P1 x P2}.
    \item We must show that a component of $\overline{\free}(X,l_1 + 3l_2 - 6e)$ contains free chains of type $(l_1 + l_2 -2e, 2l_2 -4e)$ and $(l_1 + l_2 -3e, 2l_2 -3e)$.  Since $\overline{\free}(X, l_1 + l_2 -2e)$ is irreducible, it contains maps $g: C_1 \cup C_2 \rightarrow X$ of type $(l_1 + l_2 -3e, e)$ in its smooth locus.  We may attach a third component $C_3$ to this nodal curve to create another map $g' : C_1 \cup C_2 \cup C_3 \rightarrow X$ of type $(l_1 + l_2 -3e, e, 2l_2 -4e)$, with $[g'] \in \overline{\free}(X,l_1 + 3l_2 - 6e)$ contained in the smooth locus.  In fact, $\mathcal{N}_{g'}$ is globally generated on each component of the domain of $g'$.  Studying $\mathcal{N}_{g'_2} := \mathcal{N}_{g'|_{C_2}} \cong \mathcal{O}(1)$, we see there are one dimensional families of deformations smoothing either node of $g'$, while maintaining the other node.  This proves our claim.  
    \item Let $F\subset \mathbb{P}^1\times \mathbb{P}^2$ be a general fiber of $\pi_1$.  Let $\{p_1, p_2, p_3, p_4\} = c \cap F$.  By Lemma \ref{curves1 in P1 x P2}, there is a bijection bewteen components of $\overline{\free}(X,4l_2 -6e)$ which generically parameterize geometrically rational planar quartics meeting $c$ at four points with multiplicities $(2,2,1,1)$ and orbits of the monodromy group of $\pi_1|_c$ on $(c\cap F) \times (c \cap F)$.  Therefore free chains of type $(2l_2 -3e, 2l_2 -3e)$ which are strict transforms conics meeting $\{p_1, p_2, p_3\}$ and $\{p_2, p_3, p_4\}$ are contained in the same component of $\overline{\free}(X,4l_2 -6e)$ as free chains of type $(l_2 -e, 2l_2 -4e, l_2 -e)$ whose components meet $p_2$, $\{p_1, p_2, p_3, p_4\}$, and $p_3$, respectively.
    \item We must show a component of $\overline{\free}(X,l_1 + 4l_2 -8e)$ contains free chains of type $(2l_2 -4e, l_2 -e, l_1 + l_2 -3e)$ and $(2l_2 -3e, l_1 + 2l_2 -5e)$.  %
    We may smooth a free chain of the first type to a free chain $f:C_1 \cup C_2 \rightarrow X$ of type $(2l_2 -4e, l_1 + 2l_2 -4e)$.  By deforming $f|_{C_1}$, we may move the image of $C_1 \cap C_2$ until it lies in any fiber of $\pi_1$. Since $\overline{\free}(X, l_1 + 2l_2 -4e)$ is irreducible, by deforming $f|_{C_2}$ to a map from a nodal curve of type $(e, l_1 + 2l_2 -5e)$, we may attain a map $g: C_1' \cup C_2' \cup C_3' \rightarrow X$ of type $(2l_2 -4e, e, l_1 + 2l_2 -5e)$.  Since $[g]\in \overline{\free}(X, l_1 + 4l_2 -8e)$ is a smooth point, after deforming $g$ to a free chain of type $(2l_2 -3e, l_1 + 2l_2 -5e)$, our claim follows. 
\end{enumerate}
\end{proof}

\begin{lem}
For each $\alpha \in l_1 + \Nef_1(X)_\mathbb{Z}$, $\overline{\free}(X,\alpha)$ is nonempty.
\end{lem}

\begin{proof}
$\NE(X)$ is generated by $l_1 -e$, $l_2 -2e$, and $e$.  The generators of $\Nef_1(X)$ are $l_1 = (l_1 -e) + e$, $l_2 = (l_2 -2e) + 2e$, $l_2 -2e$, and $2l_1 + l_2 -4e = 2(l_1 -e) + (l_2 -2e)$.  Lemma \ref{Gordan's Lemma}  and \ref{Representability of Free Curves} shows that the curves in $\mathscr{C}_X$, together with $l_2 -2e$, generate the monoid of integer points in $\Nef_1(X)$.  %
Since $\alpha \in l_1 + \Nef_1(X)$, there exists an expression $\alpha = l_1 + c(l_2 -2e) + \sum \alpha_i$ with $\alpha_i \in \mathscr{C}_X$ and $c\in \{0,1\}$.  This proves our claim, as $l_1 + l_2 -2e$ is represented by a free curve.
\end{proof}

\subsection*{3.4}
\textbf{Blow-up of a Double Cover of $\mathbb{P}^1 \times \mathbb{P}^2$:} 
Let $Y \rightarrow \mathbb{P}^1 \times \mathbb{P}^2$ branched over a smooth divisor $B$ of class $(2,2)$.  Let $\phi : X \rightarrow Y$ be the blow-up of a smooth fiber of the the projection $Y\rightarrow \mathbb{P}^2$, and $\pi_i : X \rightarrow \mathbb{P}^i$ be the natural map.  The variety $X$ is also obtained from $\mathbb{P}^1 \times \mathbb{F}_1$, where $\mathbb{F}_1 \rightarrow \mathbb{P}^2$ is the blow-up of a point $p$, by taking the double cover branched along the strict transform of $B$.  We let $\pi_3 : X \rightarrow \mathbb{F}_1$ be the natural map and $\psi : X \rightarrow \mathbb{P}^1$ be $\pi_3$ composed with the resolution $\mathbb{F}_1 \rightarrow \mathbb{P}^1$ of the projection from $p$.

\textbf{Generators for $N^1(X)$ and $N_1(X)$:}  The classes of generators are given by
\begin{tabular}{ll}
 $E$ = the exceptional divisor of $\phi$ & $e$ = a fiber of $\phi|_E$; \\
 $H_1 = \pi_1^* \mathcal{O}(1)$ & $c_1$ = a component of a reducible fiber of $\pi_2$; \\
 $H_2 = \pi_2^* \mathcal{O}(1)$ & $c_2$ = a curve contracted by $\pi_1$ that is also \\
  & \hspace{20pt} embedded as a line by $\pi_2$.
\end{tabular}

\begin{thm}\label{3.4thm}
For all $\alpha \in (c_1 + c_2) + \Nef_1(X)_\mathbb{Z}$, $\overline{\free}(X,\alpha)$ is irreducible and nonempty.
\end{thm}

As in the proof of Theorem \ref{3.2thm}, we will prove Theorem \ref{3.4thm} below assuming $X$ is general in moduli.  The proof for arbitrary $X$ then follows from Proposition \ref{low degree curves degeneration}, Corollary \ref{connected fibers fixall}, and Lemma \ref{3.4relations1}.

\textbf{Intersection Pairing:}
\begin{center}
\begin{tabular}{ll}
    $H_i \cdot c_j = \delta_{ij}$ &  $H_i \cdot e = 0$, \\
    $E \cdot c_j = 0$ &  $E \cdot e = -1$, \\
\end{tabular}
\end{center}

\textbf{Anticanonical Divisor:}
\begin{align*}
    -K_X = H_1 + 2H_2 -E.
\end{align*}

\textbf{Effective Divisors:}
The divisors $H_1$, $H_2-E$, and $E$ are effective.

\textbf{Effective Curves and Mori Structure:}  $\NE(X)$ is generated by the extreme rays $e$, $c_1$, and $c_2 -e$.  The contraction of $e$ is $\phi$, that of $c_1$ is $\pi_3$, and the contraction of $c_2 -e$ is $\pi_1 \times \psi$.  There are no other classes of effective lines in $X$, and for $X$ general, each of $\overline{M}_{0,0}(X, c_1)$ and $\overline{M}_{0,0}(X, c_2 -e)$ are irreducible.

The discriminant locus $\Delta$ of $\pi_3$ is the strict transform of a quartic curve in $\mathbb{P}^2$ which does not contain $p$.  For $Y$ general in moduli, $\Delta$ is a smooth curve of degree 4.  For $X$ general in moduli ($p \in \mathbb{P}^2$ general), the ramification of $\mathbb{F}_1 \rightarrow \mathbb{P}^1$ restricted to $\Delta$ is simple.  General fibers of $\psi : X \rightarrow \mathbb{P}^1$ are smooth degree four del Pezzo surfaces $S_4$, on which the monodromy of $\psi$ acts as the stabilizer of a smooth fiber of $\pi_3$.

Similarly, for $X$ general in moduli the discriminant locus $\Delta'$ of $\pi_1 \times \psi$ is a smooth curve of bidegree $(4,2)$ in $\mathbb{P}^1 \times \mathbb{P}^1$, with simple ramification under either projection.  A smooth fiber $\pi_1^{-1}(q)$ may be considered as the blow-up of $\mathbb{P}^1 \times \mathbb{P}^1 \subset Y$ along two linearly general points.  The monodromy on smooth fibers of $\pi_1$ is the natural action of $S_2 \times S_2$ permuting the two blown-up points and rulings of $\mathbb{P}^1 \times \mathbb{P}^1$.

\begin{lem}\label{3.4core}
Let $X$ be general in moduli.  A core of free curves on $X$ is given by
\begin{align*}
   \mathscr{C}_X = &\{ 2c_1, c_2, 2c_2-2e, c_1+c_2-e, c_1+c_2, c_1 + 2c_2 -2e, 2c_2-e,  2c_1+c_2-e  \}.
\end{align*}
The only nonseparating classes in $\mathscr{C}_X$ are $2c_1$, $2c_2 -2e$, and $c_1 + c_2$.  In addition, $\overline{\free}(X,\alpha)$ is irreducible for all $\alpha \in \Nef_1(X)_{\mathbb{Z}}\setminus \partial\NE(X)$ with $-K_X. \alpha \leq 4$.
\end{lem}

\begin{proof}
\textbf{Nef Curve Classes of Anticanonical Degree Between $2$ and $4$:}
The nef curve classes of anticanonical degree $2$ and $4$ are those shown in the core, in addition to $3c_1, 4c_1, 2c_1+c_2, 2c_2, c_1+2c_2-e, 3c_1+c_2-e, 3c_2-3e, 2c_1+2c_2-2e, c_1+3c_2-3e, 3c_2-2e$, and  $4c_2-4e$.  Irreducible curves of class $3c_1$ and $3c_2 -3e$ are never free; rather, they are triple covers of reducible fibers of conic fibrations.  Free curves of class $4c_1$, and $4c_2-4e$ are all multiple covers of conics.

Any component of free curves parameterizing quartics lying in fibers of $\pi_1$ or fibers of $\psi : X \rightarrow \mathbb{P}^1$ (recall $\psi$ factors through $\pi_3 : X \rightarrow \mathbb{F}^1$) contain chains of conics.  Indeed, general fibers are smooth del Pezzo surfaces, so there are one parameter families of curves passing through 2 general points on a general fiber.  This proves free curves of class $3c_1 + c_2 -e$, $c_1 + 3c_2 -3e$, $2c_1 + 2c_2 -2e$, $2c_2$, and $3c_2 -2e$ degenerate to chains of free conics.

The only classes $\alpha \in \Nef_1(X)_{\mathbb{Z}}\setminus \partial\NE(X)$ with $-K_X. \alpha \leq 4$ and $\alpha \notin \mathscr{C}_X$ are $2c_1 + c_2$ and $c_1 + 2c_2 - e$.  General free curves of class $2c_1 + c_2$ lie in the $\pi_2$-preimage $X_\ell \cong S_4$ of general lines $\ell \subset \mathbb{P}^2\setminus\{p\}$.  Similarly, curves of class $c_1 + 2c_2 - e$ map to curves of bidegree $(1,1)$ under $\pi_1 \times \psi : X \to \mathbb{P}^1 \times \mathbb{P}^1$.  From these descriptions, Corollary \ref{cor: conic monodromy cor} proves $\overline{\free}(X, 2c_1 + c_2)$ and $\overline{\free}(X, c_1 + 2c_2 -e)$ are irreducible.

\textbf{Irreducible Spaces and Fibers:} Free curves of class $2c_1$ are precisely smooth fibers of $\pi_2$ and, as such, there is a unique such curve passing through any point.  Free curves of class $c_2$ are the irreducible components the $\pi_1 \times \pi_2$-preimage of lines of the type $\{ q \} \times \ell \subset \mathbb{P}^1 \times \mathbb{P}^2$ which are tangent to the restriction $B_q$ of $B$ to $\{q\} \times \mathbb{P}^2$. This moduli space is fibered over $\mathbb{P}^1$ with two components in general fibers, interchanged under the monodromy action. Hence it is irreducible. However, there are two such curves through any point.  Free curves of class $2c_2-2e$ are the $\pi_1 \times \pi_2$-preimage of lines of the form $\{ q \} \times \ell \subset \mathbb{P}^1 \times \mathbb{P}^2$ where $p \in \ell$. This moduli space is irreducible and there is a unique such curve through any point.

Free curves of class $c_1+c_2-e$ are the components of reducible $\pi_1 \times \pi_2$-preimages of curves of class $(1,1)$ in $\mathbb{P}^1 \times \mathbb{P}^2$ whose projection in $\mathbb{P}^2$ contains the point $p$.  In other words, these curves are conics in fibers $F$ of $\psi$ given which are sections of $\pi_3|_F$.  There are eight classes of such conics in $F$, identified under $\psi$-monodromy. This shows the moduli space of such curves is irreducible. However, there are eight components of such curves passing through any point.  Similarly, free curves of class $2c_1 + c_2 -e$ or $c_1 + 2c_2 -2e$ correspond to cubics in general fibers $F$ of $\psi$.  Respectively, these are sections and bisections of $\pi_3 |_F $.  There are eight classes of cubics in $F$ corresponding to each class in $X$, which the monodromy of $\psi$ identifies.

Free curves of class $c_1+c_2$ are components of reducible $\pi_1 \times \pi_2$-preimages of curves of class $(1,1)$ in $\mathbb{P}^1 \times \mathbb{P}^2$. This moduli space is fibered over the points $\ell \in (\mathbb{P}^2)^*$ with general fiber consisting of eight pencils of conics in the degree four del pezzo surface $\pi_2^{-1}(\ell)$.  A Lefschetz pencil argument proves irreducibility of $\overline{\free}(X, c_1 + c_2)$ in an identical manner to irreducibility of $\overline{\free}(X, c_1  +c_2 -e)$.  This time, however, the fiber of $\overline{\free}_1(X, c_1 + c_2) \rightarrow X$ is irreducible, as fixing a general point in $X$ corresponds to choosing a general pencil of lines in $\mathbb{P}^2$.

Free curves of class $2c_2-e$ are contracted by $\pi_1$.  Under our description of a general fiber $X_q = \pi_1^{-1}(q)$ as the blow-up of $\mathbb{P}^1 \times \mathbb{P}^1 \subset Y$ along two general points, curves of class $2c_2 -e$ correspond to curves of bidegree $(1,1)$ meeting one blown-up point.  A general fiber of $\overline{\free}_1(X, 2c_2 -e) \rightarrow X$ contains two components, interchanged by the monodromy of $\pi_1$.
\end{proof}

\begin{lem}\label{3.4relations1}
Relations in the monoid $\mathbb{N}\mathscr{C}_X$ are generated by:
\begin{enumerate}
    \item $(2c_2-2e)+(2c_1)=2(c_1+c_2-e)$
    \item $(2c_2-2e)+(c_1+c_2)=(c_1+2c_2-2e)+(c_2) =(c_1+c_2-e)+(2c_2-e)$
    \item $(2c_2-2e)+(2c_1+c_2-e)=(c_1+2c_2-2e)+(c_1+c_2-e)$
    \item $(2c_2-2e)+2(c_2)=2(2c_2-e)$
    \item $(2c_2-2e)+2(c_1+c_2-e)=2(c_1+2c_2-2e)$
    \item $(2c_2-2e)+(c_1+c_2-e)+(c_2)=(c_1+2c_2-2e)+(2c_2-e)$
    \item $(c_1+2c_2-2e)+(2c_1+c_2-e)=3(c_1+c_2-e)$
    \item $(c_1+2c_2-2e)+(2c_1)=(2c_1+c_2-e)+(c_1+c_2-e)$
    \item $(c_1+2c_2-2e)+(c_1+c_2)=(2c_1+c_2-e)+(2c_2-e)=2(c_1+c_2-e)+(c_2)$
    \item $(2c_1)+2(c_2)=2(c_1+c_2)$
    \item $(2c_1)+(c_2)+(c_1+c_2-e)=(c_1+c_2)+(2c_1+c_2-e)$
    \item $(2c_1)+2(c_1+c_2-e)=2(2c_1+c_2-e)$
    \item $(2c_1)+(2c_2-e)=(c_1+c_2)+(c_1+c_2-e) = (2c_1+c_2-e)+(c_2)$
    \item $(2c_2-e)+(c_1+c_2)=(c_1+c_2-e)+2(c_2)$.
\end{enumerate}
For each relation $\sum \alpha_i = \sum \beta_j$, a main component of $\prod_X \overline{\free}_2(X,\alpha_i)$ lies in the same component of free curves as a main component of $\prod_X \overline{\free}_2(X, \beta_j)$.
\end{lem}

\begin{proof}
Consider the intersection pairing of core classes with the divisor $2H_1+H_2-2E$. We see $2c_2-2e$ pairs to $-2$, that $2c_2-e$ and $c_1+2c_2-2e$ pair to $0$, that $c_1+c_2-e$ and $c_2$ pair to $1$, and that the other classes pair to greater numbers. Thus, anytime $2c_2-2e$ appears in a relation, we may use relations $(1)-(6)$ to remove it. Hence we reduce to considering relations between core classes other than $2c_2-2e$.

Next, consider the intersection pairing of $H_1+H_2-2E$ with the remaining core classes. It pairs to $-1$ with $c_1+2c_2-2e$, to $0$ with $c_1+c_2-e$ and $2c_2-e$, and to positive values with the other classes. Hence, relations $(2)$ and $(7)-(9)$ remove $c_1+2c_2-2e$ from consideration without re-introducing any curves of class $2c_2-2e$.

The divisor $-H_1+H_2+E$ pairs to $-2$ with $2c_1$, to $0$ with $c_1+c_2$ and $2c_1+c_2-e$, to $1$ with $c_2$ and $c_1+c_2-e$, and to $3$ with $2c_2-e$. As before, relations $(10)-(13)$ remove all occurrences of $2c_1$.

Finally, the divisor $H_1-E$ pairs to $-1$ with $2c_2-e$, to $0$ with $c_2$ and $c_1+c_2-e$, and to $1$ with $c_1+c_2$ and $2c_1+c_2-e$. Use relations $(9)$ and $(14)$ to remove the class $2c_2-e$. The remaining four core class $2c_1+c_2-e$, $c_2$, $c_1+c_2-e$, and $c_1+c_2$ span a three-dimensional span, so they satisfy a unique relation shown in $(13)$.

\textbf{Main Components:} For each relation whose class is annihilated by $\psi_*$ or $\pi_{1*}$, our claim about main components may be deduced by studying a smooth fiber of $\psi$ or $\pi_1$.  These relations are (1), (3)-(5), (7), (8), and (12).  For each relation other than (1), there are immersed chains of curves smoothing to free chains of both types.   We use this strategy frequently in the cases below.
\begin{enumerate}
    \item Smoothings of free chains of type $(2c_2 -2e, 2c_1)$ are geometrically rational anticanonical curves on smooth fibers of $\psi$.  This component of free curves also contains smoothings of free chains of type $(c_1 + c_2 -e, c_1 + c_2 -e)$.
    \item There is a unique component of free chains of type $(2c_2 -2e, c_1 + c_2)$ contained in a unique component $M \subset \overline{\free}(X, c_1 + 3c_2 -2e)$.  Smoothing immersed chains of type $(c_2, c_1, 2c_2 -2e)$ and $(c_1, c_2 -e, e, 2c_2 -2e)$ demonstrates that $M$ also contains chains of type $(c_2, c_1 + 2c_2 -2e)$ and $(c_1 + c_2 -e, 2c_2 -e)$.
    \item Smoothing an immersed chain of type $(2c_2 -2e, c_1, c_1 + c_2 -e)$ in two ways proves our claim.
    \item Smoothing an immersed chain of type $(c_2, c_2 -e, c_2 -e, c_2)$ in two ways proves our claim.
    \item Smoothing an immersed chain of type $(c_1 + c_2 -e, c_2 -e, c_2 -e, c_1 + c_2 -e)$ in two ways proves our claim.
    \item There is a unique main component of $\overline{\free}_1(X,c_1 + 2c_2 -2e) \times_X \overline{\free}_1(X, 2c_2 -e)$, since monodromy over the attachment to a general free curve of class $c_1 + 2c_2 -2e$ acts transitively on fibers of $\overline{\free}_1(X, 2c_2 -e) \rightarrow X$.  Smoothing an immersed chain of type $(c_1 + c_2 -e, c_2 -e, c_2 -e, c_2)$ in two ways proves our claim.
    \item Smoothing an immersed chain of type $(c_1, c_2 -e, c_1, c_2 -e, c_1, c_2 -e)$ in two ways proves our claim.
    \item Smoothing an immersed chain of type $(2c_1, c_2 -e, c_1 + c_2 -e)$ in two ways proves our claim.
    \item There is a unique component of free chains of type $(c_1 + 2c_2 -2e, c_1 + c_2)$ contained in a unique component $M \subset \overline{\free}(X, 2c_1 + 3c_2 -2e)$.  Smoothing immersed chains of type $(c_1 + c_2 -e, c_2 -e, c_1, c_2)$ in two ways proves $M$ contains free chains of type $(c_1 + c_2 -e, c_1 + c_2 -e, c_2)$.  Then, smoothing immersed chains of type $(c_1 + c_2 -e, c_1, c_2 -e, c_2)$ proves there is a component of $\overline{\free}(X, 2c_1 + 3c_2 -2e)$ that contains free chains of type $(c_1 + c_2 -e, c_1 + c_2 -e, c_2)$ and $(2c_1 + c_2 -e, 2c_2 -e)$.
    \item Smoothing an immersed chain of type $(c_2, c_1, c_1, c_2)$ in two ways proves our claim.
    \item Smoothing an immersed chain of type $(c_2, c_1, c_2 -e, 2c_1)$ in two ways proves our claim.
    \item Smoothing an immersed chain of type $(c_1 + c_2 -e, c_1, c_1, c_1 + c_2 -e)$ in two ways proves our claim.
    \item There is a unique component of free chains of type $(c_1 + c_2 -e, c_1 + c_2)$ contained in a unique component $M \subset \overline{\free}(X, 2c_1 + 2c_2 -e)$.  Smoothing immersed chains of type $(c_1 + c_2 -e, c_1, c_2)$ in two ways proves $M$ contains free chains of type $(2c_1 + c_2 -e, c_2)$.  Then, smoothing immersed chains of type $(2c_1, c_2 -e, c_2)$ in two ways proves there is a component of $\overline{\free}(X, 2c_1  +2c_2 -e)$ that contains free chains of type $(2c_1 + c_2 -e, c_2)$ and $(2c_1, 2c_2 -e)$.
    \item Smoothing an immersed chain of type $(c_2, c_1, c_2 -e, c_2)$ in two ways proves our claim.
\end{enumerate}
\end{proof}

\begin{lem}
For each $\alpha \in c_1 + c_2 + \Nef_1(X)_\mathbb{Z}$, $\overline{\free}(X,\alpha)$ is nonempty.
\end{lem}

\begin{proof}
Generators of $\Nef_1(X)_\mathbb{Z}$ as a monoid of integer points are $c_1$, $c_2$, and $c_2 -e$.  Looking at $\mathscr{C}_X$, a nef class is represented by a free curve iff it is not an odd multiple of $c_2 -e$ or $c_1$.
\end{proof}

\subsection*{3.5}

\textbf{Description of Variety $X$:} $X$ is the blow up of $\mathbb{P}^1\times \mathbb{P}^2$ along a curve $c$ of bidegree $(5,2)$ embedded under projection to $\mathbb{P}^2$.

\textbf{Generators for $N^1(X)$ and $N_1(X)$:} Let $H_1 = \pi_1^* \mathcal{O}(1)$, $H_2 = \pi_2^* \mathcal{O}(1)$, and $E$ be the exceptional divisor of the blow-up.  Let $l_1$ be the class of a fiber of $\pi_2$, $l_2$ be the class of a general line in each fiber of $\pi_1$, and $e$ be a fiber of $E\rightarrow c$.

\begin{thm}\label{3.5thm}
For all $\alpha \in l_1 + \Nef_1(X)_\mathbb{Z}$, $\overline{\free}(X,\alpha)$ is nonempty and irreducible.
\end{thm}

\textbf{Intersection Pairing:} $ H_i \cdot l_j = \delta_{ij}$, $E \cdot e = -1$, and all other pairings are $0$.

\textbf{Anticanonical Divisor:} $-K_X = 2H_1 + 3H_2 - E$

\textbf{Effective Divisors:} $2H_2 -E$, $H_1$, and $E$ are effective.

\textbf{Effective Curves:}  The Mori cone is generated by $2l_2 -5e$, $e$, and $l_1-e$.

\begin{lem}\label{lem: 3.5 core}
A weak core of free curves on $X$ is given by 
\begin{align*}
   \mathscr{C}_X = &\{ l_1,\ l_2,\ l_2 -e,\ 2l_2 -3e,\ 2l_2 -4e,\ 3l_2 -6e, \ l_1 + l_2 -2e  \} 
\end{align*}
The only separating classes $\alpha \in \mathscr{C}_X$ are $l_2 -e$, $2l_2 -3e$, $2l_2 -4e$, and $3l_2 -6e$.  When $X$ is general in moduli, $\mathscr{C}_X$ is a core of free curves.  For arbitrary $X$, $\overline{\free}(X,\alpha)$ is irreducible for $\alpha \in \mathscr{C}_X\setminus \{2l_2 - 3e\}$ and for $\alpha \in \Nef_1(X)_{\mathbb{Z}}\setminus \partial\NE(X)$ with $-K_X. \alpha \leq 4$.  
\end{lem}
\begin{proof}
\textbf{Nef Curve Classes of Anticanonical Degree Between $2$ and $4$:}
If $\alpha = d_1 l_1 + d_2 l_2 -ne$ is nef of appropriate anticanonical degree, then $n\leq 2d_2$ and $2 \leq 2d_1 + 3d_2 -n \leq 4$.  From this we see $2d_1 + d_2 \leq 4$ and find the following curves:
\begin{itemize}
    \item $l_1, l_2, l_2 -e, 2l_1, 2l_2 -2e, 2l_2 -3e, 2l_2 -4e, 3l_2 -5e, 3l_2 - 6e, 4l_2-8e, l_1 + l_2 -e, l_1 + l_2 -2e, l_1 + 2l_2 -4e$
\end{itemize}

\textbf{Freely Breakable Classes:}  $2l_1, 2l_2 -2e, l_1 + l_2 -e$, and $l_1 + 2l_2 -4e$ are all breakable.  Similarly, $3l_2 -5e$ breaks into $l_2-e$ and $2l_2 -4e$, and $4l_2 -8e$ breaks into two curves of class $2l_2-4e$.  $4l_2 -8e$ must be either a double cover of a curve of class $2l_2-4e$, or a nodal rational curve with three double points along $c$.  In the latter case, we may impose the curve pass through 3 additional, general points contained in a conic containing 4 of the 5 intersection points of $c$ with a fiber of $\pi_1$.

\textbf{Irreducibility of Spaces and Fibers:}
Lemmas \ref{curves1 in P1 x P2}, \ref{curves2 in P1 x P2}, \ref{curves3 in P1 x P2} prove that for all $\alpha \in \mathscr{C}_X \setminus \{2l_2 -3e\}$, $\overline{\free}(X,\alpha)$ is irreducible; Lemmas \ref{curves2 in P1 x P2}, \ref{curves3 in P1 x P2} also prove $\overline{\free}(X,\alpha)$ is irreducible for $\alpha \in \Nef_1(X)_{\mathbb{Z}}\setminus \partial\NE(X)$ with $-K_X. \alpha \leq 4$.  Irreducibility of $\overline{\free}(X,2l_2 -3e)$ is equivalent to the monodromy of $\pi_1|_c : c\rightarrow \mathbb{P}^1$ being transitive on unordered pairs of points a general fiber.  For general $X$, this follows from Theorem \ref{monodromy del Pezzo main result}.  %
Clearly, $l_1$, $l_2$, and $l_1 + l_2 -2e$ are not separating classes ($c$ projects to a conic in $\mathbb{P}^2$), while the other classes are separating.
\end{proof}

\begin{rem}
Whether or not $\overline{\free}(X, 2l_2-3e)$ is irreducible has no impact on the number of components of $\overline{\free}(X, \alpha)$ for $\alpha \not\in \partial \NE(X)$.  Indeed, both $\overline{\free}(X, l_1 + 2l_2 -3e)$ and $\overline{\free}(X, l_1 + 3l_2 -5e)$ are irreducible by Lemma \ref{curves2 in P1 x P2}. 
\end{rem}

\begin{lem}
Relations in the monoid $\mathbb{N}\mathscr{C}_X$ are generated by:
\begin{enumerate}
        \item $l_1 + (2l_2 -3e) = (l_1 + l_2 -2e) + (l_2 -e)$
        \item $l_1 + (3l_2 -6e) = (l_1 + l_2 -2e) + (2l_2 -4e)$
        \item $l_1 + 2(l_2 -e) = (l_1 + l_2 -2e) + l_2$
        \item $l_1 + (l_2 -e) + (2l_2 -4e) = (l_1 + l_2 -2e) + (2l_2 -3e)$
        \item $2l_1 + (2l_2 -4e) = 2(l_1 + l_2 -2e)$
        \item $l_1 + 2(2l_2 -4e) = (l_1 + l_2 -2e) + (3l_2 -6e)$
        \item $l_2 + (2l_2 -3e) = 3(l_2 -e)$
        \item $l_2 + (2l_2 -4e) = (l_2 -e) + (2l_2 -3e)$
        \item $l_2 + (3l_2 -6e) = 2(2l_2 -3e) = 2(l_2 -e) + (2l_2 -4e)$
        \item $(l_2 -e) + (3l_2 -6e) = (2l_2 -3e) + (2l_2 -4e)$
        \item $(l_2 -e) + 2(2l_2 -4e) = (2l_2 -3e) + (3l_2 -6e)$
        \item $3(2l_2 -4e) = 2(3l_2 -6e)$.
\end{enumerate}
For each relation $\sum \alpha_i = \sum \beta_j$, a main component of $\prod_X \overline{\free}_2(X,\alpha_i)$ lies in the same component of free curves as a main component of $\prod_X \overline{\free}_2(X, \beta_j)$.
\end{lem}
\begin{proof}
\textbf{Relations:}
We follow Lemma \ref{Relations}.
\begin{itemize}
    \item $D=-2H_1 + E$: We have $D.l_1=-2$, $D.l_2 = 0$, $D.l_1 + l_2 -2e = 0$, $D.l_2 -e = 1$.  This implies any relation involves at least two copies of $l_2-e$, or another curve that pairs positively with $D$.  We note that since $l_1$ and $l_1 + l_2 -2e$ are the only divisors that do not pair to $0$ with $H_1$, both cannot appear on the same side of a relation.  We find the following relations:
    \begin{itemize}
        \item $l_1 + (2l_2 -3e) = (l_1 + l_2 -2e) + (l_2 -e)$
        \item $l_1 + (3l_2 -6e) = (l_1 + l_2 -2e) + (2l_2 -4e)$
        \item $l_1 + 2(l_2 -e) = (l_1 + l_2 -2e) + l_2$
        \item $l_1 + (l_2 -e) + (2l_2 -4e) = (l_1 + l_2 -2e) + (2l_2 -3e)$
        \item $2l_1 + (2l_2 -4e) = 2(l_1 + l_2 -2e)$
        \item $l_1 + 2(2l_2 -4e) = (l_1 + l_2 -2e) + (3l_2 -6e)$
        \item $l_1 + (2l_2 -4e) + l_2 = (l_1 + l_2 -2e) + 2(l_2-e)$ (generated by other relations)
    \end{itemize}
    No other relations may involve $l_1 + l_2 -2e$.
    \item $D=-H_1 - H_2 + E$: $D$ pairs negatively with $l_2$.  We find the following relations.
    \begin{itemize}
        \item $l_2 + (2l_2 -3e) = 3(l_2 -e)$
        \item $l_2 + (2l_2 -4e) = (l_2 -e) + (2l_2 -3e)$
        \item $l_2 + (3l_2 -6e) = 2(2l_2 -3e)$
    \end{itemize}
    \item $D=-3H_2 -H_1 -2E$: $D$ pairs negatively with $l_2 -e$.  We find the following relations:
    \begin{itemize}
        \item $(l_2 -e) + (3l_2 -6e) = (2l_2 -3e) + (2l_2 -4e)$
        \item $2(l_2 -e) + (2l_2 -4e) = 2(2l_2 -3e)$
        \item $(l_2 -e) + 2(2l_2 -4e) = (2l_2 -3e) + (3l_2 -6e)$
    \end{itemize}
    The remaining 4 curves span $N_1(X)$ and the sole relation between them is $3(2l_2 -4e) = 2(3l_2 -6e)$.
\end{itemize}

\textbf{Main Components:} We address each relation below.
\begin{enumerate}
    \item $\overline{\free}(X, l_1 + 2l_2 -3e)$ is irreducible by Lemma \ref{curves2 in P1 x P2}.
    \item A general map $g: C_1 \cup C_2 \cup C_3 \rightarrow X$ of type $(l_1 + l_2 -2e, e, 2l_2 -5e)$ is a smooth point of $\overline{\mathcal{M}}_{0,0}(X)$ with globally generated normal bundle $\mathcal{N}_g$.  We may deform $g$ to a a free chain of type $(l_1 + l_2 -2e, e, 2l_2 -4e)$.  Alternatively, we may deform to a map $g' : C_1' \cup C_2' \rightarrow X$ of type $(l_1 + l_2 -e, 2l_2 -5e)$ by smoothing the node $C_1 \cap C_2$.  Since $\overline{\free}(X,l_1 + l_2 -e)$ is irreducible, we may deform $g'$ to a chain of type $(l_1, l_2 -e, 2l_2 -5e)$.  This smooths to the desired free chain.
    \item $\overline{\free}(X, l_1 + 2l_2 -2e)$ is irreducible by Lemma \ref{curves2 in P1 x P2}.
    \item $\overline{\free}(X, l_1 + 3l_2 -5e)$ is irreducible by Lemma \ref{curves2 in P1 x P2}.
    \item $\overline{\free}(X, 2l_1 + 2l_2 -4e)$ is irreducible by Lemma \ref{curves3 in P1 x P2}.
    \item A general map $g : C_1\cup C_2 \cup C_3 \cup C_4 \rightarrow X$ of type $(l_1 + l_2 -2e, e, 2l_2 -5e, l_2 -2e)$ is a smooth point of $\overline{\mathcal{M}}_{0,0}(X)$.  Clearly, there are deformations of $g$ which are free chains of type $(l_1 + l_2 -2e, 3l_2 - 6e)$.  If Instead we deform $g$ to $g' : C_1' \cup C_2' \cup C_3' \rightarrow X$ of type $(l_1 + l_2 -e, 2l_2 -5e, l_2 -2e)$, again $[g']$ is a smooth point of $\mathcal{M}_{0,0}(X)$.  As before, we deform $g'$ to a chain of type $(l_1, l_2 -e, 2l_2 -5e, l_2 -2e)$, and smooth all but the first component to attain the desired free chain.
    \item By Lemma \ref{curves1 in P1 x P2} and our analysis of the monodromy of $\pi_1|_c$, $\overline{\free}(X, 3l_2 -3e)$ contains exactly one component which parameterizes cubics that are smooth at every intersection point with $c$.  This component contains free chains of type $(l_2, 2l_2 -3e)$ and $(l_2 -e, l_2 -e, l_2 -e)$.
    \item By Lemma \ref{curves1 in P1 x P2}, $\overline{\free}(X, 3l_2 -4e)$ contains exactly one component which parameterizes cubics that are smooth at every intersection point with $c$.  This component contains free chains of type $(l_2, 2l_2 -4e)$ and $(l_2 -e, 2l_2 -3e)$.
    \item By Lemma \ref{curves1 in P1 x P2}, $\overline{\free}(X, 4l_2 -6e)$ contains exactly one component which parameterizes quartics that are double at exactly one intersection point with $c$.  This component contains free chains of each indicated type.
    \item By Lemma \ref{curves1 in P1 x P2} and our analysis of the monodromy of $\pi_1|_c$, $\overline{\free}(X, 4l_2 -7e)$ contains exactly one component which parameterizes quartics that are double at exactly two intersection points with $c$.  This component contains free chains of type $(l_2 -e, 3l_2 -6e)$ and $(2l_2 -3e, 2l_2 -4e)$.
    \item By Lemma \ref{curves1 in P1 x P2}, $\overline{\free}(X, 5l_2 -9e)$ contains exactly one component which parameterizes quintics that are double at all but one intersection point with $c$.  This component contains free chains of each indicated type.
    \item By Lemma \ref{curves1 in P1 x P2} and our analysis of the monodromy of $\pi_1|_c$, $\overline{\free}(X, 6l_2 -12e)$ contains exactly one component which parameterizes sextics that are triple at two intersection points with $c$ and double at the rest.  This component contains free chains of each indicated type.
\end{enumerate}
\end{proof}

\begin{lem}
For each $\alpha \in l_1 + \Nef_1(X)_\mathbb{Z}$, $\overline{\free}(X,\alpha)$ is nonempty.
\end{lem}

\begin{proof}
The generators of $\Nef_1(X)$ are $l_1 = (l_1 -e) + e$, $l_2 = (l_2 -2e) + 2e$, and $l_2 -2e = \frac{1}{2}(2l_2 -5e + e)$.  Lemma \ref{Gordan's Lemma}  and \ref{Representability of Free Curves} shows that the curves in $\mathscr{C}_X$, together with $l_2 -2e$, generate the monoid of integer points in $\Nef_1(X)$.  %
Since $\alpha \in l_1 + \Nef_1(X)$, there exists an expression $\alpha = l_1 + c(l_2 -2e) + \sum \alpha_i$ with $\alpha_i \in \mathscr{C}_X$ and $c\in \{0,1\}$.  This proves our claim, as $l_1 + l_2 -2e$ is represented by a free curve.
\end{proof}

\subsection*{3.6}

\textbf{Description of Variety $X$:}
Let $X$ be the blow up of $\mathbb{P}^3$ along the disjoint union of a line $\ell$ and an elliptic curve $c$ of degree 4.  %

\textbf{Generators for $N^1(X)$ and $N_1(X)$:} Let $\pi:X\rightarrow \mathbb{P}^3$ be the blow-up, $H$ be the pullback of a hyperplane section, and $E_\ell$, $E_c$ be the exceptional divisors lying over $\ell$ and $c$.  Let $l$ be the strict transform of a general line, and $e_\ell, e_c$ be the class of a fiber of $\pi$ over a point in $\ell$ and $c$ respectively.

\begin{thm}\label{3.6thm}
For all $\alpha \in l + \Nef_1(X)_\mathbb{Z}$, $\overline{\free}(X,\alpha)$ is nonempty and irreducible.
\end{thm}

\textbf{Intersection Pairing:} $H.l = 1$, $E_\ell . e_\ell = -1$, and $E_c . e_c = -1$.  All other pairings are $0$.

\textbf{Anticanonical Divisor:} $4H - E_\ell -E_c = E_{\ell} + 2(H - E_{\ell}) + (2H - E_c)$

\textbf{Effective Divisors:} $H - E_\ell$, $2H - E_c$, $E_\ell$, $E_c$.

\textbf{Effective Curves:} $e_\ell$, $e_c$, $l - e_\ell - 2e_c$.

\begin{lem}
A core of free curves on $X$ is given by
$$\mathscr{C}_X = \{ l, l-e_{\ell}, l-e_c, l -e_{\ell} - e_{c}, l-2e_c, 2l - e_\ell - 4e_c, 2l - 2e_\ell - 3e_c, 2l - 2e_\ell - 4e_c \} $$
The separating classes are $l-e_\ell -e_c$, $l -2e_c$, $2l - 2e_\ell -3e_c$, and $2l -e_\ell - 4e_c$.  
\end{lem}

\begin{proof}
\textbf{Nef Curve classes of Anticanonical Degree Between 2 and 4:} If $\alpha= d l - m e_{\ell} - n e_{c}$ is a nef class and $2 \leq \alpha.(-K_X)\leq 4$, then $d\leq 4$, $m\leq d$, $n\leq 2d$, and $2(d-m) + m \leq 4 \Rightarrow m\geq 2d-4$.  Thus $\alpha$ must be one of the following classes:
\begin{itemize}
     \item $l, l-e_{\ell}, l-e_c, l -e_{\ell} - e_{c}, l-2e_c$
    \item $2l - 4e_c$, $2l - e_{\ell} - 3e_c$, $2l - e_\ell - 4e_c$, $2l - 2e_\ell - 2e_c$, $2l - 2e_\ell - 3e_c$, $2l - 2e_\ell - 4e_c$,
    \item $3l - 2 e_\ell - 6 e_c$, $3l - 3 e_\ell - 5 e_c$, $3l - 3 e_\ell - 6 e_c$, and $4l - 4e_\ell - 8 e_c$.
\end{itemize}
\textbf{Freely Breakable Classes:} $2l -4e_c$, $2l - e_\ell -3e_c$, and $2l -2e_\ell -2e_c $ are all freely breakable by Lemma \ref{lines_in_P3}.  Moreover, each cubic or quartic curve class above is either freely breakable or not represented by free curves.  Indeed, any irreducible cubic curve that meets $c$ with multiplicity $6$ must be a twisted cubic lying in a quadric that contains $c$.  Similarly, any irreducible cubic curve meeting $\ell$ with multiplicity $3$ must lie in a plane containing $\ell$.  Thus, $3l - 3 e_\ell - 6 e_c$ cannot be represented by a free rational curve.  Moreover, we may break planar representatives of $3l - 3 e_\ell - 5 e_c$ into free chains of type $(l - e_\ell -e_c, 2l - 2e_\ell -4e_c)$ by deforming it within the plane.  Similarly, by deforming representatives of $3l -2e_\ell - 6e_c$ within a quadric containing $c$, we may break them to free chains of type $(l-2e_c, 2l - 2e_\ell -4e_c)$.  Lastly, similar arguments show that each free curve of class $4l - 4e_\ell - 8e_c$ must be a double cover of an anticanonical conic, and therefore breakable.

\textbf{Irreducible Spaces and Fibers:} Lemma \ref{lines_in_P3} proves that for $\alpha$ in the first row, each $\overline{\free}(X,\alpha)$ is nonempty and irreducible.  Moreover, when $\alpha \not\in \{ l -e_\ell -e_c, l -2e_c\}$, Lemma \ref{lines_in_P3} proves the fibers of $\overline{\free}_1(X, \alpha) \xrightarrow{\text{ev}} X$ are generally irreducible.  However, when $\alpha \in \{ l -e_\ell -e_c, l -2e_c\}$, the fibers of $\overline{\free}_1(X, \alpha) \xrightarrow{\text{ev}} X$ generally contain four or two isolated points.

Since each free curve of class $\alpha = 2l - 2e_\ell - ne_c$ with $n\in \{3,4\}$ is contained a plane that contains $\ell$, we may parameterize such free curves via an open subset of a $\mathbb{P}^{5-n}$-bundle over the space of such planes.  This shows that fibers of $\overline{\free}_1(X, 2l - 2e_\ell - 4e_c) \xrightarrow{\text{ev}} X$ are irreducible, while fibers of $\overline{\free}_1(X, 2l - 2e_\ell -3e_c) \xrightarrow{\text{ev}} X$ generally have four components.  Lastly, free curves of class $2l - e_\ell - 4e_c$ are strict transforms of the complete intersection of a quadric $Q$ that contains $c$ and a plane in $\mathbb{P}^3$ containing one of the two points in $\ell \cap Q$.  Therefore, while $\overline{\free}(X, 2l - e_\ell - 4e_c)$ is irreducible, fibers of $\overline{\free}_1(X, 2l - e_\ell - 4e_c) \xrightarrow{\text{ev}} X$ generally contain two components, corresponding to each point in $\ell \cap Q$.
\end{proof}

\begin{lem}
Relations in the monoid $\mathbb{N}\mathscr{C}_X$ are generated by:
\begin{enumerate}
        \item $l + (l - e_\ell - e_c) = (l -e_\ell) + (l- e_c)$
        \item $l + (l - 2e_c) = 2(l- e_c)$
        \item $l + (2l - e_\ell -4e_c) = (l -e_\ell) + 2(l -2e_c)$
        \item $l + (2l - 2e_\ell -3e_c) = (l -e_c) + 2(l -e_\ell -e_c)$
        \item $l + (2l - 2e_\ell -4e_c) = (2l -e_\ell -4e_c) + (l-e_\ell) = (2l -2e_\ell -3e_c) + (l-e_c) =  2(l-e_\ell -e_c) + (l-2e_c)$ %
        \item $(2l -2e_\ell -4e_c) + (l-e_c) = (2l -e_\ell -4e_c) + (l-e_\ell -e_c) = (2l-2e_\ell -3e_c) + (l-2e_c)$
        \item $(2l -2e_\ell -4e_c) + (l-e_\ell) = (2l -2e_\ell -3e_c) + (l-e_\ell -e_c)$
        \item $(2l -2e_\ell -4e_c) + 2(l-e_\ell - e_c) = 2(2l -2 e_\ell - 3e_c)$
        \item $(2l -2e_\ell -4e_c) + 2(l - 2e_c) = 2(2l - e_\ell - 4e_c)$
        \item $(2l -2e_\ell -4e_c) + (l-e_\ell - e_c) + (l-2e_c) = (2l -2 e_\ell - 3e_c) + (2l - e_\ell - 4e_c)$
        \item $(2l-e_\ell -4e_c) + (l-e_c) = (l-e_\ell -e_c) + 2(l-2e_c)$
        \item $(2l -2e_\ell -3e_c) + (l-e_\ell) = 3(l-e_\ell -e_c)$
        \item $(l-e_\ell) + (l -2e_c) = (l-e_c) + (l-e_\ell -e_c)$
\end{enumerate}
For each relation $\sum \alpha_i = \sum \beta_j$, a main component of $\prod_X \overline{\free}_2(X,\alpha_i)$ lies in the same component of free curves as a main component of $\prod_X \overline{\free}_2(X, \beta_j)$.
\end{lem}
\begin{proof}
\textbf{Relations:} 
We proceed using Lemma \ref{Relations}.  Note that $D = -H + E_\ell + E_c$ pairs negatively with $l$ and nonnegatively with each other element of $\mathscr{C}_X$.  In particular, since $D.(l - e_\ell) = D.(l - e_c) = 0$, we find relations $(1)$ through $(5)$ above.  Next, as $2l - 2e_\ell - 4e_c$ is the only curve that pairs negatively with $D = 5H - 2E_\ell - 2E_c$, while $D.(2l -2e_\ell -3e_c) = D.(2l -e_\ell -4e_c) = 0$.  We obtain relations (6) through (10).  There are no expressions of $(2l -2e_\ell -4e_c) + (l-e_\ell - e_c)$ or $(2l -2e_\ell -4e_c) + (l- 2e_c)$ as a sum of other curve classes, because these still pair to $-1$ with $D$.

Of the remaining curves, only $2l-e_\ell -4e_c$ pairs negatively with $D=H - \frac{1}{2} E_c - \frac{1}{4} E_\ell$, and only $l-e_\ell$, $l-e_c$, and $l-e_\ell-e_c$ pair positively with $D$.  We find relations (5), (6), and (11).  Of classes in $\mathscr{C}_X\setminus \{l, 2l - 2e_\ell - 4e_c, 2l -e_\ell -4e_c\}$, only $2l-2e_\ell -3e_c$ pairs negatively with $D=H - \frac{1}{2} E_c - \frac{1}{2} E_\ell$, while $l-e_\ell$ and $l-e_c$ pair positively with $D$.  We obtain relations (5) and (12).  Lastly, the only relation between the remaining curves $\{l-e_\ell, l-2e_c, l-e_c, l-e_\ell -e_c\}$ is (13).

\textbf{Main Components:} We address each relation below.
\begin{enumerate}
    \item Irreducibility of $\overline{\free}(X, 2l -e_\ell -e_c)$ follows from Lemma \ref{lines_in_P3}.
    \item Irreducibility of $\overline{\free}(X, 2l -2e_c)$ follows from Lemma \ref{lines_in_P3}.
    \item We claim $\overline{\free}(X, 3l -e_\ell -4e_c)$ is irreducible.  Indeed, the locus of free, integral, planar curves of class $3l -e_\ell -4e_c$ is $3 + \text{dim } \mathbb{G}(2,3) = 6$, while each component of $\overline{\free}(X, 3l -e_\ell -4e_c)$ has dimension 7.  Therefore, a general point in any component of $\overline{\free}(X, 3l -e_\ell -4e_c)$ corresponds to the strict transform of a twisted cubic.  It follows that for any such representative $f: \mathbb{P}^1 \rightarrow X$, the five points at which $f(\mathbb{P}^1)$ meets the exceptional locus of $X\rightarrow \mathbb{P}^3$ have linearly general images in $\mathbb{P}^3$.  An open locus $U \subset \ell \times \text{Sym}^4(c)$ parameterizes such a collection of points.  To each point of $U$, there corresponds a 2 dimensional family (birational to $U \times \overline{M}_{0,5}$) of twisted cubics with strict transforms of class $3l -e_\ell -4e_c$.  Therefore $\overline{\free}(X, 3l -e_\ell -4e_c)$ is irreducible.
    \item A nearly identical argument proves $\overline{\free}(X, 3l -2e_\ell -3e_c)$ is irreducible.  Merely let $U$ be an open subset of $\text{Sym}^2(\ell) \times \text{Sym}^3(c)$.
    \item Again, a generic curve parameterized by any component of $\overline{\free}(X, 3l -2e_\ell -4e_c)$ is a twisted cubic.  This time, an open subset $U\subset \text{Sym}^2(\ell) \times \text{Sym}^4(c)$ is birational to $\overline{\free}(X, 3l -2e_\ell -4e_c)$.
    \item We claim $\overline{\free}(X, 3l -2e_\ell -5e_c)$ is irreducible.  As no planar curve has strict transform of class $3l -2e_\ell -5e_c$, each component of $\overline{\free}(X, 3l -2e_\ell -5e_c)$ generically parameterizes strict transforms of twisted cubics.  Recall that the union of a twisted cubic $C$ and any line $L$ spanned by two points of $C$ may be expressed as the complete intersection of two quadrics.  If $C$ has strict transform of class $3l -2e_\ell -4e_c$, we may let $L = \ell$.  The space of quadrics containing $\ell$, $|2H -E_\ell|$, is a linear system of dimension $6$.  Through any five linearly general points on $c$ there passes a one dimension subsystem of such quadrics.  Therefore, an open subset of $\text{Sym}^5(c)$ dominates $\overline{\free}(X, 3l -2e_\ell -5e_c)$.
    \item We claim the only main component of free chains of type $(2l -2e_\ell -4e_c, l-e_\ell)$ lies in a component of $\overline{\free}(X, 3l -3e_\ell -4e_c)$ that contains free chains of type $(2l -2e_\ell -3e_c, l-e_\ell -e_c)$.  Any integral curve $C$ of class $3l - 3e_\ell -4e_c$ is the strict transform of a planar cubic.  The plane $P$ containing such a cubic $C$ must contain $\ell$ as well (i.e. $P\in |H -E_\ell|$), and we may assume $P$ meets $c$ transversely at 4 points.  Thus, either $C$ is singular at one of three intersection points with $c$, or $C$ meets $c$ at four nonsingular points.  The unique component of $\overline{\free}(X, 3l -3e_\ell -4e_c)$ that generically parameterizes cubics meeting $c$ at four nonsingular points contains free chains of type $(2l -2e_\ell -4e_c, l-e_\ell)$ and $(2l -2e_\ell -3e_c, l-e_\ell -e_c)$.
    \item Each component of $\overline{\free}(X, 4l -4e_\ell -6e_c)$ must generically parameterize strict transforms of planar quartics.  Such quartics must intersect $c$ at four distinct points, and may either have a triple point or two double points among these intersections.  Any component of $\overline{\free}(X, 4l -4e_\ell -6e_c)$ parameterizing quartics with two double points along $c$ must contain free chains of type $(2l -2e_\ell -4e_c, l-e_\ell -e_c, l-e_\ell -e_c)$ and $(2l -2e_\ell -3e_c, 2l -2e_\ell -3e_c)$.
    \item Each component of $\overline{\free}(X,4l -2e_\ell -8e_c)$ must generically parameterize strict transforms of nondegenerate quartic curves in $\mathbb{P}^3$.  Every geometrically rational, nondegenerate quartic curve $C$ in $\mathbb{P}^3$ is either smooth or contains precisely one double point.  If the strict transform of $C$ has class $4l -2e_\ell -8e_c$, exactly one quadric surface $S$ contains both $c$ and $C$.  %
    For a general curve parameterized by any component of $\overline{\free}(X,4l -2e_\ell -8e_c)$, $S$ must be nonsingular.  Smooth rational quartics in $S$ have class $(3,1)$ or $(1,3)$, while nondegenerate, singular quartic curves have class $(2,2)$, and are the complete intersection of two quadrics.  For each nonsingular quadric $S$ containing $c$, the space of geometrically rational curves of class $(2,2)$ containing $S\cap \ell$ is irreducible \ref{del pezzo curves thm}. This proves that exactly one component of $\overline{\free}(X,4l -2e_\ell -8e_c)$ generically parameterizes maps $f:\mathbb{P}^1\rightarrow X$ which are not embeddings.  The existence of free chains of type $(2l - e_\ell -4e_c, 2l -e_\ell -4e_c)$ (pick two curves in different components of a fiber of $\overline{\free}_1(X,2l-e_\ell-4e_c)\rightarrow X$) and of type $(2l -2e_\ell -4e_c, l-2e_c, l-2e_c)$ in this component proves our claim.
    \item We claim $\overline{\free}(X, 4l -3e_\ell -7e_c)$ is irreducible.  Each component of $\overline{\free}(X, 4l -3e_\ell -7e_c)$ must generically parameterize nondegenerate quartic curves.  Moreover, general quartics $f:\mathbb{P}^1\rightarrow X$ parameterized by any component of $\overline{\free}(X, 4l -3e_\ell -7e_c)$ must be the strict transforms of smooth curves in $\mathbb{P}^3$ under $\pi: X\rightarrow \mathbb{P}^3$.  This may be seen from \ref{reducible fibers: 4 author result} and \cite[Proposition~2.9]{beheshti2020moduli}, or by noticing that otherwise $\pi \circ f(\mathbb{P}^1)$ would be contained in a quadric containing $c$.  It follows that for any such $f$, $\pi \circ f(\mathbb{P}^1)$ is contained in a unique quadric $Q$, which must be smooth.  Since $\ell$ must meet $Q$ at 3 or more points, $\ell \subset Q$, i.e. $Q \in |2H - E_\ell|$.  We may identify $\ell$ as a divisor of class $(1,0)$ in $Q$, so that $\pi \circ f(\mathbb{P}^1)$ is a divisor of class $(1,3)$ in $Q$.  %
    If $Y\rightarrow \mathbb{P}^3$ denotes the blow-up of $\ell$ and $\pi' :X \rightarrow Y$ factors $\pi$, it follows that $\mathcal{N}_{\pi' \circ f} \cong \mathcal{O}(5) \oplus \mathcal{O}(6)$.  %
    Therefore through a general collection of 6 points of $c$ there exists a finite (nonzero) number of free curves parameterized by each component of $\overline{\free}(X, 4l -3e_\ell -7e_c)$.  For $[f] \in \overline{\free}(X, 4l -3e_\ell -7e_c)$, the associated quadric $Q$ must be the only $Q\in |2H -E_\ell|$ that contains any 6 of the 7 points $\pi\circ f(\mathbb{P}^1)\cap c$. It follows from \cite[Theorem~1.2]{beheshti2020moduli} that for $[f]$ general in each component, all 7 points of $\pi\circ f(\mathbb{P}^1)\cap c$ must impose independent conditions on the space of divisors of class $(1,3)$ in $Q$.  A dimension count thus shows that a general $[f]$ in any component of $\overline{\free}(X, 4l -3e_\ell -7e_c)$ lies in a general $Q\in |2H -E_\ell|$.  Thus, we may assume $Q$ is general.  In particular, we may suppose $Q$ meets $c$ transversely and that all 8 points of $Q \cap c$ impose independent conditions on divisors of class $(1,3)$.  Thus the association $[f] \rightarrow Q$ corresponds to a generically $8-1$ map $\overline{\free}(X, 4l -3e_\ell -7e_c) \dashrightarrow |2H -E_\ell|$ which dominant on each component of $\overline{\free}(X, 4l -3e_\ell -7e_c)$.  As the monodromy of intersections of $c$ with $Q \in |2H -E_\ell|$ is the full symmetric group, this proves irreducibility of $\overline{\free}(X, 4l -3e_\ell -7e_c)$.
    \item We claim $\overline{\free}(X, 3l -e_\ell -5e_c)$ is irreducible.  The argument is similar to (3) and (5).  
    \item While $\overline{\free}(X, 3l -3e_\ell -3e_c)$ is not irreducible, there is a unique component $M$ that contains free chains of type $(2l -2e_\ell -3e_c, l-e_\ell)$.  $M$ is the only component which generically parameterizes planar cubics meeting $c$ transversely at 3 points.  We may obtain a free chain of type $(l-e_\ell -e_c, l-e_\ell -e_c, l-e_\ell -e_c)$ in $M$ by gluing three curves from separate fibers of $\overline{\free}_1(X,l-e_\ell -e_c) \rightarrow X$.
    \item Irreducibility of $\overline{\free}(X, 2l -e_\ell -2e_c)$ follows from Lemma \ref{lines_in_P3}.
\end{enumerate}
\end{proof}

\begin{lem}
For each $\alpha \in l + \Nef_1(X)_\mathbb{Z}$, $\overline{\free}(X,\alpha)$ is nonempty.
\end{lem}

\begin{proof}
The extreme rays of $\Nef_1(X)$ are spanned by $l, l-e_\ell, l-2e_c$, and $l -e_\ell -2e_c$.  By Gordan's Lemma, these classes, together with elements in $\mathscr{C}_X$, generate the monoid of integer points in $\Nef_1(X)$.  Thus we may write $\alpha \in l + \Nef_1(X)_\mathbb{Z}$ as a sum $\alpha = l + b(l-e_\ell -2e_c) + \sum_i a_i c_i$ for $c_i \in \mathscr{C}_X$, and $b \in \{0,1\}$.  If $b =1$, substituting $l + (l-e_\ell -2e_c) = (l-e_\ell) + (l-2e_c)$ expresses $\alpha$ as a sum of free curve classes.
\end{proof}

\subsection*{3.7}
\textbf{Description of Variety $X$:} $X$ is given in \cite{mori1981classification} as a blow up $X\rightarrow \mathbb{P}(T_{\mathbb{P}^2})$ of the complete intersection of two members of $|-\frac{1}{2} K_{\mathbb{P}(T_{\mathbb{P}^2})}|$.  Alternatively, we may realize $X$ as a blow-up of $\mathbb{P}^1 \times \mathbb{P}^2$ along a curve $c$ of bidegree $(3,3)$ using \ref{blow-up numbers}.  We let $\pi: X \rightarrow \mathbb{P}^1 \times \mathbb{P}^2$ denote this blow-up, and $\pi_1 :X \rightarrow \mathbb{P}^1$ and $\pi_2:X\rightarrow \mathbb{P}^2$ be the compositions of the blow-up with each projection.

\textbf{Generators for $N^1(X)$ and $N_1(X)$:} Let $H_1 = \pi_1^* \mathcal{O}(1)$, $H_2 = \pi_2^* \mathcal{O}(1)$, and $E$ be the exceptional divisor of the blow-up $\pi$.  Let $l_1$ be the class of a fiber of $\pi_2$, $l_2$ be the class of a general line in each fiber of $\pi_1$, and $e$ be a fiber of $E\rightarrow c$.  Note that there is pseudosymmetry, explained after Effective Curves.

\begin{thm}\label{3.7thm}
For all $\alpha \in l_1 + \Nef_1(X)_\mathbb{Z}$, $\overline{\free}(X,\alpha)$ is nonempty and irreducible.
\end{thm}

\textbf{Intersection Pairing:} $ H_i \cdot l_j = \delta_{ij}$, $E \cdot e = -1$, and all other pairings are $0$.

\textbf{Anticanonical Divisor:} $-K_X = 2H_1 + 3H_2 - E$

\textbf{Effective Divisors:} $E$, $3H_1 + 3H_2 -2E$, $3H_2 -E$, and $H_1$ are all effective.

\textbf{Effective Curves:} The extreme rays of $\NE(X)$ are $e$, $l_1-e$, and $l_2 -2e$.

\textbf{Pseudosymmetry:}  The contraction of $l_2 -2e$ also realizes $X$ as a blow-up of $\mathbb{P}^1 \times \mathbb{P}^2$ along a curve of bidegree $(3,3)$ by \cite{matsuki1995weyl}.  Thus, we obtain the following pseudoaction on $N_1(X)$: $e\rightarrow l_2 -2e$, $l_1 \rightarrow l_1 + l_2 -3e$, $l_2 \rightarrow 2l_2 - 3e$.

\begin{lem}
A core of free curves on $X$ is given by
\begin{align*}
   \mathscr{C}_X = &\{ l_1, \ l_2, \ l_2 -e, \ 2l_2 -3e,\ l_1 + l_2 -2e,\ l_1 + l_2 -3e \} 
\end{align*}
The only separating class in $\mathscr{C}_X$ is $l_2 -e$.
\end{lem}
\begin{proof}
\textbf{Nef Curve Classes of Anticanonical Degree Between $2$ and $4$:}
If $\alpha = d_1 l_1 + d_2 l_2 -ne$ is nef and of appropriate degree, we find $2n\leq \text{min}(6d_2, 3d_1 + 3d_2)$, $d_1,n \geq 0$, and $2 \leq 2d_1 + 3d_2 -n \leq 4$.  The integral solutions to this system of inequalities are the classes in $\mathscr{C}_X$ and $2l_1, 2l_2 -2e, l_1 + l_2 -e, l_1 + 2l_2 -4e, 2l_1 + l_2 -3e$, and $2l_1 + 2l_2 -6e$.
As pseudosymmetry swaps $l_1$ with $l_1 + l_2 -3e$ and swaps $l_2$ with $2l_2 -3e$, this reduces the number of classes we need to consider. Lemmas \ref{curves1 in P1 x P2}, \ref{curves2 in P1 x P2}, and \ref{curves3 in P1 x P2} finish our proof.
\end{proof}

\begin{lem}
Up to pseudosymmetry, relations in the monoid $\mathbb{N}\mathscr{C}_X$ are generated by:
\begin{enumerate}
\item $l_2 + (2l_2 -3e) = 3(l_2 -e)$
\item $l_2 + (l_1 + l_2 -2e) = l_1 + 2(l_2 -e)$
\item $l_2 + (l_1 + l_2 -3e) = (l_2 -e) + (l_1 + l_2 -2e)$
\item $l_1 + (l_2 -e) + (l_1 + l_2 -3e) = 2(l_1 + l_2 -2e)$
\end{enumerate}
For each relation $\sum \alpha_i = \sum \beta_j$, a main component of $\prod_X \overline{\free}_2(X,\alpha_i)$ lies in the same component of free curves as a main component of $\prod_X \overline{\free}_2(X, \beta_j)$.
\end{lem}
\begin{proof}
\textbf{Relations:}
We follow Lemma \ref{Relations}.
\begin{itemize}
    \item $D=-H_2 +E$: $D$ pairs negatively with $l_2$ and positively with $2l_2 -3e$, $l_1 + l_2 -2e$, and $l_1 + l_2 -3e$.  We find relations (1) through (3).
    \item Relations involving $2l_2 -3e$ are pseudosymmetric to relations involving $l_2$.
    \item The remaining four classes span $N_1(X)$ and satisfy relation (4).
\end{itemize}

\textbf{Main Components:} We address each relation below.
\begin{enumerate}
    \item $\overline{\free}(X, 3l_2 -3e)$ has three components.  One component parameterizes triple covers of curves of class $l_2 -e$. One component parameterizes planar cubics that are singular at one of their two intersection points with the blown-up curve $c$.  The only other component parameterizes planar cubics meeting $c$ at three points in their smooth locus.  This last component contains free chains of type $(l_2, 2l_2 -3e)$ and $(l_2 -e, l_2 -e, l_2 -e)$.
    \item $\overline{\free}(X, l_1 + 2l_2 -2e)$ is irreducible by Lemma \ref{curves2 in P1 x P2}
    \item $\overline{\free}(X,l_1 + 2l_2 -3e)$ is irreducible by Lemma \ref{curves2 in P1 x P2}.
    \item $\overline{\free}(X,2l_1 + 2l_2 -4e)$ is irreducible by Lemma \ref{curves3 in P1 x P2}.
\end{enumerate}
\end{proof}

\begin{lem}
For each nonzero $\alpha \in \Nef_1(X)_\mathbb{Z}$, $\overline{\free}(X,\alpha)$ is nonempty.
\end{lem}

\begin{proof}
The generators of $\Nef_1(X)$ are $l_1, l_2, 2l_2 -3e, l_1 + l_2 -3e$ by \ref{Representability of Free Curves}.  It follows from \ref{Gordan's Lemma} that classes in $\mathscr{C}_X$ generate the monoid of integer points in $\Nef_1(X)$.
\end{proof}

\subsection*{3.8}
\textbf{Description of Variety $X$:} $X$ may be realized as a blow-up $\pi:X\rightarrow \mathbb{P}^1 \times \mathbb{P}^2$ along a curve $c$ of bidegree $(4,2)$ using \ref{blow-up numbers}.  We let $\pi_1 :X \rightarrow \mathbb{P}^1$ and $\pi_2:X\rightarrow \mathbb{P}^2$ be the compositions of the blow-up with each projection.

\textbf{Generators for $N^1(X)$ and $N_1(X)$:} Let $H_1 = \pi_1^* \mathcal{O}(1)$, $H_2 = \pi_2^* \mathcal{O}(1)$, and $E$ be the exceptional divisor of the blow-up $\pi$.  Let $l_1$ be the class of a fiber of $\pi_2$, $l_2$ be the class of a general line in each fiber of $\pi_1$, and $e$ be a fiber of $E\rightarrow c$.

\begin{thm}\label{3.8thm}
For all $\alpha \in l_1 + \Nef_1(X)_\mathbb{Z}$, $\overline{\free}(X,\alpha)$ is nonempty and irreducible.
\end{thm}

\textbf{Intersection Pairing:} $ H_i \cdot l_j = \delta_{ij}$, $E \cdot e = -1$, and all other pairings are $0$.

\textbf{Anticanonical Divisor:} $-K_X = 2H_1 + 3H_2 - E$

\textbf{Effective Divisors:} $E$, $H_1$, and $2H_2 -E$ are all effective.

\textbf{Effective Curves:} The extreme rays of $\NE(X)$ are $e$, $l_1-e$, and $l_2 -2e$.

\begin{lem}
A core of free curves on $X$ is given by
\begin{align*}
   \mathscr{C}_X = &\{ l_1, \ l_2, \ l_2 -e, \ 2l_2 -3e, \ 2l_2 -4e, \ l_1 + l_2 -2e \} 
\end{align*}
The only separating classes in $\mathscr{C}_X$ are $l_2 -e$ and $2l_2 -3e$.
\end{lem}
\begin{proof}
\textbf{Nef Curve Classes of Anticanonical Degree Between $2$ and $4$:}
If $\alpha = d_1 l_1 + d_2 l_2 -ne$ is nef and of appropriate degree, we find $n\leq 2d_2$, $d_1,n \geq 0$, and $2 \leq 2d_1 + 3d_2 -n \leq 4$.  The integral solutions to this system of inequalities are the classes in $\mathscr{C}_X$ and $2l_1, 2l_2 -2e, l_1 + l_2 -e, l_1 + 2l_2 -4e, 3l_2 -6e,$ and $4l_2 -8e$.  Aside from $3l_2 -6e$, each of these classes is freely breakable (see Lemmas \ref{curves1 in P1 x P2}, \ref{curves2 in P1 x P2}, and \ref{curves3 in P1 x P2}).  $3l_2 -6e$ is never representable by a free rational curve.  Lemmas \ref{curves1 in P1 x P2}, \ref{curves2 in P1 x P2}, and \ref{curves3 in P1 x P2} finish our proof.
\end{proof}

\begin{lem}
Relations in the monoid $\mathbb{N}\mathscr{C}_X$ are generated by:
\begin{enumerate}
\item $l_2 + (2l_2 -3e) = 3(l_2 -e)$
\item $l_2 + (2l_2 -4e) = (l_2 -e) + (2l_2 -3e)$
\item $l_2 + (l_1 + l_2 -2e) = l_1 + 2(l_2 -e)$
\item $l_1 + (2l_2 -3e) = (l_1 + l_2 -2e) + (l_2 -e)$
\item $2l_1 + (2l_2 -4e) = 2(l_1 + l_2 -2e)$
\item $2(2l_2 -3e) = 2(l_2 -e) + (2l_2 -4e)$
\end{enumerate}
For each relation $\sum \alpha_i = \sum \beta_j$, a main component of $\prod_X \overline{\free}_2(X,\alpha_i)$ lies in the same component of free curves as a main component of $\prod_X \overline{\free}_2(X, \beta_j)$.
\end{lem}
\begin{proof}
\textbf{Relations:}
We follow Lemma \ref{Relations}.
\begin{itemize}
    \item $D=-H_2 +E$: $D$ pairs negatively with $l_2$ and positively with $2l_2 -3e$, $2l_2 -4e$, and $l_1 + l_2 -2e$.  We find relations (1) through (3).
    \item $D = H_1 + 3H_2 -2E$: $D$ pairs negatively with $2l_2 -4e$ and positively with $l_1$ and $l_2 -e$.  We obtain relations (5) and (6).
    \item The remaining four classes span $N_1(X)$ and satisfy relation (4).
\end{itemize}

\textbf{Main Components:} We address each relation below.
\begin{enumerate}
    \item $\overline{\free}(X, 3l_2 -3e)$ has at least three components by \ref{curves1 in P1 x P2}.  The component that parameterizes planar cubics meeting $c$ at three points in their smooth locus contains free chains of type $(l_2, 2l_2 -3e)$ and $(l_2 -e, l_2 -e, l_2 -e)$.
    \item By Lemma \ref{curves1 in P1 x P2}, a unique component in $\overline{\free}(X, 3l_2 -4e)$ parameterizes strict transforms of planar cubics meeting $c$ transversely at four points.  This component contains free chains of type $(l_2, 2l_2 -4e)$ and $(l_2 -e, 2l_2 -3e)$.
    \item $\overline{\free}(X, l_1 + 2l_2 -2e)$ is irreducible by Lemma \ref{curves2 in P1 x P2}
    \item $\overline{\free}(X,l_1 + 2l_2 -3e)$ is irreducible by Lemma \ref{curves2 in P1 x P2}.
    \item $\overline{\free}(X,2l_1 + 2l_2 -4e)$ is irreducible by Lemma \ref{curves3 in P1 x P2}.
    \item Let $F\subset \mathbb{P}^1\times \mathbb{P}^2$ be a general fiber of $\pi_1$.  Let $\{p_1, p_2, p_3, p_4\} = c \cap F$.  By Lemma \ref{curves1 in P1 x P2}, there is a bijection bewteen components of $\overline{\free}(X,4l_2 -6e)$ which generically parameterize geometrically rational planar quartics meeting $c$ at four points with multiplicities $(2,2,1,1)$ and orbits of the monodromy group of $\pi_1|_c$ on $(c\cap F) \times (c \cap F)$.  Therefore free chains of type $(2l_2 -3e, 2l_2 -3e)$ which are strict transforms conics meeting $\{p_1, p_2, p_3\}$ and $\{p_2, p_3, p_4\}$ are contained in the same component of $\overline{\free}(X,4l_2 -6e)$ as free chains of type $(l_2 -e, 2l_2 -4e, l_2 -e)$ whose components meet $p_2$, $\{p_1, p_2, p_3, p_4\}$, and $p_3$, respectively.  
\end{enumerate}
\end{proof}

\begin{lem}
For each $\alpha \in l_1 + \Nef_1(X)_\mathbb{Z}$, $\overline{\free}(X,\alpha)$ is nonempty.
\end{lem}

\begin{proof}
The generators of $\Nef_1(X)$ are $l_1, l_2$, and $2l_2 -4e$ by \ref{Representability of Free Curves}.  It follows from \ref{Gordan's Lemma} that classes in $\mathscr{C}_X$ together with $l_2 -2e$ generate the monoid of integer points in $\Nef_1(X)$.  Thus we may express $\alpha$ as a sum $\alpha = l_1 + c(l_2 -2e) + \sum \alpha_i$ with $\alpha_i \in \mathscr{C}_X$ and $c \in \{0,1\}$.
\end{proof}

\subsection*{3.9}
See Section \ref{E5 cases}.

\subsection*{3.10}
\textbf{Blow-up of a Quadric along two conics:}
Let $f:X \rightarrow Q$ be the blow-up of a smooth quadric $Q \subset \mathbb{P}^4$ along a disjoint union of two conics $c_1$ and $c_2$ which lie on it.  Note that the planes containing these conics must intersect at a single point.

\begin{thm}\label{3.10thm}
For each $\alpha \in l + \Nef_1(X)_\mathbb{Z}$, $\overline{\free}(X,\alpha)$ is irreducible and nonempty.
\end{thm}

\textbf{Generators for $N^1(X)$ and $N_1(X)$:}
\begin{center}
\begin{tabular}{ll}
 $H$ = the class of a hyperplane & $l$ = a line in $\mathbb{P}^4$ \\ 
 $E_1$ = the exceptional divisor $f^{-1}(c_1)$ & $e_1$ = the $f$-fiber over a point on $c_1$  \\ 
 $E_2$ = the exceptional divisor $f^{-1}(c_2)$ & $e_2$ = the $f$-fiber over a point on $c_2$   
\end{tabular}
\end{center}

\textbf{Intersection Pairing:}
\begin{center}
\begin{tabular}{lll}
    $H \cdot l = 1$ &  $H \cdot e_1 = 0$ & $H \cdot e_2 = 0$, \\
    $E_1 \cdot l = 0$ &  $E_1 \cdot e_1 = -1$ & $E_1 \cdot e_2 = 0$ \\
    $E_2 \cdot l = 0$ & $E_2 \cdot e_1 = 0$ & $E_2 \cdot e_2 = -1$
\end{tabular}
\end{center}

\textbf{Anticanonical Divisor:}
\begin{align*}
    -K_X = 3H - E_1 - E_2.
\end{align*}

\textbf{Effective Divisors:}
The divisors $H-E_1, H-E_2, E_1$, and $E_2$ are effective.
\textbf{Effective Curves:} $\NE(X)$ is generated by $e_1, e_2, l-e_1 -e_2$.  The contraction associated to $l-e_1 -e_2$ is a conic fibration onto $\mathbb{P}^1 \times \mathbb{P}^1$.  Pseudosymmetry swapping $c_1$ with $c_2$ interchanges $e_1$ with $e_2$.

\begin{lem}\label{3.10core}
A core of free curves on $X$ is given by
\begin{align*}
    \mathscr{C}_X = \{ l, l-e_i, 2l-2e_i-e_j, 2l -2e_1 -2e_2 \}
\end{align*}

\noindent for $i\neq j$.  %
The separating classes in $\mathscr{C}_X$ are $l -e_i$ and $2l -2e_i -e_j$.
\end{lem}

\begin{proof}
\textbf{Nef Curve Classes of Anticanonical degree between $2$ and $4$:}
Such nef curve classes $\alpha = al - b e_1 - c e_2$ satisfy the equations
\begin{align*}
    0 \leq b,c \leq a, \hspace{.5cm} 2 \leq 3a-b-c \leq 4.
\end{align*}
\noindent The solutions are those classes in $\mathscr{C}_X$, as well as $2l-2e_i$, $2l-e_1-e_2$, $3l -3e_i -2e_j$, $3l-3e_1-3e_2$, and $4l-4e_1-4e_2$.  These additional classes are all freely breakable or not represented by free curves.  Indeed, any irreducible curve of class $2l -2e_i$ meets $c_i$ twice.  Unless the curve is a double cover, we may parameterize a family of such curves by $\text{Sym}^2(c_i) \times Q$ by associating the plane spanned by three points to its intersection with $Q$.  The map associated to this family has image which dominates $\overline{\free}^{bir}(X,2l -2e_i)$.  As the family contains free chains of type $(l-e_i, l-e_i)$, this proves $2l -2e_i$ is freely breakable.  A similar technique shows $2l-e_1 -e_2$ is freely breakable.  Since the contraction associated to $l-e_1 -e_2$ is a conic fibration, there are no free curves of class $3l -3e_1 -3e_2$.  Similarly, each free curve of class $4l -4e_1 -4e_2$ is a double cover.

Each free curve $C$ of class $3l -3e_1 -2e_2$ must be the strict transform of a twisted cubic.  If $P \cong \mathbb{P}^3$ is the hyperplane containing $C$, then for general $C$ in $P\cap Q \cong \mathbb{P}^1 \times \mathbb{P}^1$.  Clearly $c_1 \subset P \cap Q$ is a divisor of class $(1,1)$ and $C$ is a divisor of class $(1,2)$ or $(2,1)$.  Note that $P\cap c_2$ consists of two points which are not contained in a line in $P \cap Q$.  Each twisted cubic $C \subset P \cap Q$ with strict transform of class $3l -3e_1 -2e_2$ must contain $P \cap c_2$.  For a fixed hyperplane $P$ containing $c_2$, there are two components of such curves.  However, there is a one dimensional family of such hyperplanes.  The resulting monodromy action on $N^1(P \cap Q)$ is nontrivial.  This shows $\overline{\free}(X,3l -3e_i -2e_j)$ is irreducible.  Since $2l-2e_1 -2e_2$ and $l-e_i$ are classes of free curves, this proves $3l -3e_i -2e_j$ is freely breakable.

\textbf{Irreducible Spaces and Fibers:}  It is well known that the space of lines in $Q$ is irreducible.  Moreover, through any point on $Q$, there exists a one parameter family of lines spanning a quadric cone.  Thus, $\overline{\free}(X, l-e_i)$ is irreducible for each $i$; however, this also shows that the general fiber of $\text{ev}: \overline{\free}_1(X,l-e_i)\rightarrow X$ consists of two points.  Comparably, $\overline{\free}(X, 2l -2e_i -e_j)$ is irreducible and the general fiber of $\text{ev}: \overline{\free}_1(X,l-e_i)\rightarrow X$ consists of two irreducible curves.  Indeed, each choice of two points on $c_i$ and one point on $c_j$ yields a plane whose intersection with $Q$ generally has strict transform of class $2l -2e_i -e_j$.  Thus, $\overline{\free}(X, 2l -2e_i -e_j)$ is dominated by an open subset of $c_i \times c_i \times c_j$.  The plane corresponding to $p_1, p_2 \in c_i$ and $p_3 \in c_j$ contains a general point $p \in Q$ iff the secant line to $c_i$ through $p_1$ and $p_2$ meets the line through $p$ and $p_3$.  Since the plane $P_i \subset \mathbb{P}^4$ spanned by $c_i$ intersects the cone over $c_j$ with vertex $p$ at two points, there are two, one dimensional families of such planes.  Lastly, since $2l-2e_1 -2e_2$ is the class of a general fiber of the contraction morphism associated to $l-e_1 -e_2$, $\overline{\free}(X, 2l -2e_1 -2e_2)$ is irreducible and $2l -2e_1 -2e_2$ is nonseparating.
\end{proof}

\begin{lem}
Relations in the monoid $\mathbb{N}\mathscr{C}_X$ are generated by:
\begin{enumerate}
 \item $l + (2l -2e_i -e_j) = 2(l-e_i) + (l-e_j)$
        \item $l + (2l -2e_1 -2e_2) = (l-e_j) + (2l -2e_i -e_j)$
        \item $(2l -2e_1 -2e_2) + 2(l-e_i) = 2(2l -2e_i -e_j)$
        \item $(2l -2e_1 -2e_2) + (l-e_1) + (l-e_2) = (2l -2e_1 -e_2) + (2l -e_1 -2e_2)$
\end{enumerate}
For each relation $\sum \alpha_i = \sum \beta_j$, %
$\overline{\free}(X, \sum \alpha_i)$ is irreducible.
\end{lem}
\begin{proof}
\textbf{Relations:}  We proceed as in Lemma \ref{Relations}.
\begin{itemize}
    \item $D= -H + E_1 + E_2$ pairs negatively with $l$, trivially with $l-e_i$, and positively with each other class in $\mathscr{C}_X$.  We find the following relations.
    \begin{itemize}
        \item $l + (2l -2e_i -e_j) = 2(l-e_i) + (l-e_j)$
        \item $l + (2l -2e_1 -2e_2) = (l-e_j) + (2l -2e_i -e_j)$
    \end{itemize}
    \item $D = 3H -2E_1 -2E_2$ pairs negatively with $2l -2e_1 -2e_2$, trivially with $2l -2e_i -e_j$, and positively with $l-e_i$.  Since $D.(2l -2e_1 -2e_2) = -2$ while $D.(l-e_i) = 1$, we obtain the following relations:
    \begin{itemize}
        \item $(2l -2e_1 -2e_2) + 2(l-e_i) = 2(2l -2e_i -e_j)$
        \item $(2l -2e_1 -2e_2) + (l-e_1) + (l-e_2) = (2l -2e_1 -e_2) + (2l -e_1 -2e_2)$
    \end{itemize}
    The remaining classes span $N_1(X)$ and satisfy a unique relation, $(2l -2e_1 -e_2) + (l-e_2) = (2l -e_1 -2e_2) + (l-e_1)$.
\end{itemize}

\textbf{Main Components:}
\begin{enumerate}
    \item We claim $\overline{\free}(X, 3l-2e_i -e_j)$ is irreducible.  %
    For a general hyperplane $P\subset \mathbb{P}^4$, $P\cap Q \cong \mathbb{P}^1 \times \mathbb{P}^1$.  Irreducible curves in $P \cap Q$ whose strict transforms have class $3l -2e_i -e_j$ must be twisted cubics (divisors of class $(2,1)$ or $(1,2)$) that meet both points of $c_i \cap P$ and one point in $c_j \cap P$.  Changing $P$ shows the space of all such curves is irreducible.
    \item We claim $\overline{\free}(X, 3l-2e_i -2e_j)$ is irreducible.  The argument is nearly identical to (1).
    \item  We claim $\overline{\free}(X, 4l-4e_i -2e_j)$ is irreducible.  The argument follows our proof that $\overline{\free}(X,3l-3e_i -2e_j)$ is irreducible.
    \item  We claim $\overline{\free}(X, 4l-3e_i -3e_j)$ is irreducible.  First, we show that there is one main component of free chains of type $(2l -2e_1 -2e_2, l-e_1, l-e_2)$ and of type $(2l -2e_1 -e_2, 2l -e_1 -2e_2)$.  Previously we showed that for $\alpha \in \{l-e_i, 2l-2e_i -e_j, 3l - 3e_i -2e_j\}$, general fibers of $\overline{\free}_1(X,\alpha)\xrightarrow{\text{ev}} X$ contained exactly two components.  Let $Y\rightarrow X$ be the Stein factorization of $\text{ev}$.  If $\alpha \in \{l-e_1, 2l -2e_1 -e_2 3l - 3e_1 -2e_2\}$ and $\beta \in \{l-e_2, 2l-e_1 -2e_2, 3l - 2e_1 -3e_2\}$, each curve of class $\beta$ must meet the ramification locus of $Y\rightarrow X$.  This shows the monodromy action on components of a general fiber of the projection $\overline{\free}_1(X,\alpha) \times_X \overline{\free}_1(X,\beta) \rightarrow \overline{\free}_1(X,\beta)$ must be nontrivial.  Hence, there is only one main component of $\overline{\free}_1(X,\alpha) \times_X \overline{\free}_1(X,\beta)$.  This proves our claim.
    
    Next, we show that the unique main components of free chains of type $(2l -2e_1 -2e_2, l-e_1, l-e_2)$ and of type $(2l -2e_1 -e_2, 2l -e_1 -2e_2)$ lie in the same component of $\overline{\free}(X, 4l-3e_i -3e_j)$.  It suffices to deform a free chain of type $(2l -2e_1 -2e_2, 2l -e_1 -e_2)$ to one of type $(2l -2e_1 -e_2, 2l -e_1 -2e_2)$.  Generalization of a stable map $g: C_1 \cup C_2 \cup C_3 \rightarrow X$ of type $(2l -2e_1 -2e_2, e_1,  2l -2e_1 -e_2)$ yield chains of both types, which proves our claim.
\end{enumerate}
\end{proof}

\begin{lem}
For each $\alpha \in l + \Nef_1(X)_\mathbb{Z}$, $\overline{\free}(X,\alpha)$ is nonempty.
\end{lem}

\begin{proof}
As the monoid of integer points in $\NE(X)$ is generated by $e_1, e_2, l-e_1 -e_2$ by Lemma \ref{Gordan's Lemma}, \ref{Representability of Free Curves} show that curves in $\mathscr{C}_X$, together with $l -e_1 -e_2$, generate the monoid of integer points in $\Nef_1(X)$.  %
Since $\alpha \in l + \Nef_1(X)$, there exists an expression $\alpha = l + c(l -e_1 -e_2) + \sum \alpha_i$ with $\alpha_i \in \mathscr{C}_X$ and $c\in \{0,1\}$.  This proves our claim, as $2l -e_1 -e_2$ is represented by a free curve.
\end{proof}

\subsection*{3.11}

\textbf{Description of Variety $X$:}
Let $g:V_7 \rightarrow \mathbb{P}^3$ be the blow-up of a point $p \in \mathbb{P}^3$. Let $f:X \rightarrow V_7$ be the blow-up of the strict transform of an elliptic curve $c$ of degree four passing through the exceptional divisor.

\begin{thm}\label{3.11thm}
For each $\alpha \in \Nef_1(X)_\mathbb{Z}$, $\overline{\free}(X,\alpha)$ is irreducible.
\end{thm}

\textbf{Generators for $N^1(X)$ and $N_1(X)$:}
\begin{center}
\begin{tabular}{ll}
 $H$ = the class of a hyperplane & $l$ = a line in $\mathbb{P}^3$ \\ 
 $E$ = the exceptional divisor $(g \circ f)^{-1}(p)$ & $e$ = a line in $E$ disjoint from $F$  \\  
 $F$ = the exceptional divisor $f^{-1}(c)$ & $f$ = the $f$-fiber over a point on $c$   
\end{tabular}
\end{center}

\textbf{Intersection Pairing:}
\begin{center}
\begin{tabular}{lll}
    $H \cdot l = 1$ &  $H \cdot e = 0$ & $H \cdot f = 0$, \\
    $E \cdot l = 0$ &  $E \cdot e = -1$ & $E \cdot f = 0$ \\
    $F \cdot l = 0$ & $F \cdot e = 0$ & $F \cdot f = -1$
\end{tabular}
\end{center}

\textbf{Anticanonical Divisor:}
\begin{align*}
    -K_X = 4H - 2E - F.
\end{align*}

\textbf{Effective Divisors:} By \cite{hartshorne2013algebraic} Exercise $4.3.6$, $c$ is a complete intersection of two quadric surfaces in $\mathbb{P}^3$. Therefore, the divisors $2H-E-F$, $E$, and $F$ are effective. In addition, projection from $p \in \mathbb{P}^3$ maps $c$ to a planar cubic. The cone over this cubic is an effective divisor of class $3H-3E-F$.

\begin{lem}\label{3.11core}
A core of free curves on $X$ is given by
\begin{align*}
    \mathscr{C}_X = \{ l, l-e, l-f, l-2f, 2l-e-3f \}.
\end{align*}

\noindent The only separating classes in $\mathscr{C}_X$ is $l-2f$.
\end{lem}

\begin{proof}
\textbf{Nef Curve Classes of Anticanonical degree between $2$ and $4$:}

Such nef curve classes $\alpha = al - b e - c f$ satisfy the equations
\begin{align*}
    0 \leq b,c, \hspace{.5cm} b+c \leq 2a, \hspace{.5cm} 3b+c \leq 3a, \hspace{.5cm} 2 \leq 4a-2b-c \leq 4.
\end{align*}

\noindent Moreover, any curve of degree $a>1$ meeting $p$ with multiplicity at least $a$ is reducible. Hence, we may assume $b<a$ whenever $a>1$. The solutions are those classes in $\mathscr{C}_X$, as well as $2l-e-2f$, and $2l-4f$.  The curves of class $2l-e-2f$ each lie in a plane, so may be freely broken into a sum of two curves of class $l-e$ and $l-2f$ respectively.  Similarly, the curves of class $2l-4f$ may be freely broken into a sum of two curves, each of class $l-2f$.

\textbf{Irreducible Spaces and Fibers:}  The curves of class $l$ are parameterized by an open subset of $\mathbb{G}(1,3)$, curves of class $l-e$ by an open subset of $\mathbb{P}^2$, and curves of class $l - f$ by an open subset of a $\mathbb{P}^2$-bundle over $c$. They all have irreducible fibers.

The curves of class $l - 2f$ are parameterized by an open subset of $\Sym^2 c$. There are, however, a separating curve class. Indeed, the number of such curves through a generic point $q \in \mathbb{P}^3$ is equal to the number of double points on the image of $c$ under projection from $p$; this is $\binom{4-1}{2} - 1 = 2$ because $c$ is a degree $4$ elliptic curve. 

The curves of class $2l-e-3f$ lie in a plane $A$ necessarily containing $p$ and they intersect the remaining three points in $A \cap c$. Thus, they are parameterized by a $\mathbb{P}^1$-bundle over the planes in $\mathbb{P}^3$ through the point $p$. They clearly do not form a separating curve class.
\end{proof}

\begin{lem}\label{3.11relations}
Relations in the monoid $\mathbb{N}\mathscr{C}_X$ are generated by:
\begin{enumerate}
    \item $(l)+(l-2f)=2(l-f)$,
    \item $2(l)+(2l-e-3f) = 3(l-f)+(l-e)$,
    \item $2(l-2f)+(l-e)=(2l-e-3f)+(l-f)$.
\end{enumerate}
For each relation $\sum \alpha_i = \sum \alpha_j'$, a main component of $\prod_X \overline{\free}_2(X,\alpha_i)$ lies in the same component of free curves as a main component of $\prod_X \overline{\free}_2(X,\alpha_j')$.
\end{lem}

\begin{proof}
Note that the divisor $H-E-F$ pairs to $1$ with $l$, to $-1$ with $l-2f$, to $-2$ with $2l-e-3f$, and to zero otherwise amongst the classes in $\mathscr{C}_X$. Therefore, we may use relations $(1)$ and $(2)$ to remove $l$ from consideration. The remaining four classes consequently satisfy the single relation $(3)$.

Every curve of class $2l-2f$ is, in particular, a planar conic. It may be degenerated as either a sum of two curves of classes $l-2f$ and $l$ or of classes $l-f$ and $l-f$.

For the curves of class $4l-e-3f$, it suffices to show a curve $w$ of type $(2l-e-3f)+2(l)$ may be deformed to one of type $3(l-f)+(l-e)$. So let $w$ consist of a curve $x$ of class $2l-e-3f$ lying in a plane $A$ meeting $c$ at the points $q, q_1, q_2$, and $q_3$, together with curves $y$ and $z$ of class $l$ glued onto $x$ at points $r$ and $s$ respectively. Now deform $y$ and $z$ along $x$ until they meet $x$ at the points $q_1$ and $q_2$ respectively. Thus $w$ becomes a curve of class $(2l-e-f)+2(l-f)$. Then degenerate $x$, maintaining its intersection points $p, q_1, q_2, q_3$ with $c$ to contain another point collinear to $p$ and $q_1$. This breaks $x$ into a union of two lines. The curve $w$ has therefore been degenerated to a curve of type $3(l-f)+(l-e)$ as needed.

For the curves of class $3l-e-4f$, it suffices to show a curve of type $(2l-e-3f)+(l-f)$ may be deformed to one of type $2(l-2f)+(l-e)$. So let $x$ be a curve of class $2l-e-3f$ lying in a plane $A$ meeting $c$ at the points $p, q_1, q_2$, and $q_3$ and let $y$ be a curve of class $l-f$ meeting $x$ at a point $r$. Deform $y$ along $x$, maintaining its intersection point with $c$ until it meets $x$ at the $q_1$. Then deform $x$ in the plane $A$ in such a way that $p,q_1,q_2,q_3$ remain on $x$ until $x$ contains another point $s$ collinear to $q_2$ and $q_3$. This breaks $x$ into a union of two lines. The curve is therefore of type $2(l-2f)+(l-e)$ as desired.
\end{proof}

\subsection*{3.12}
\textbf{Description of Variety $X$:} $X$ may be realized as a blow-up $\pi:X\rightarrow \mathbb{P}^1 \times \mathbb{P}^2$ along a curve $c$ of bidegree $(3,2)$ using \ref{blow-up numbers}.  We let $\pi_1 :X \rightarrow \mathbb{P}^1$ and $\pi_2:X\rightarrow \mathbb{P}^2$ be the compositions of the blow-up with each projection.

\textbf{Generators for $N^1(X)$ and $N_1(X)$:} Let $H_1 = \pi_1^* \mathcal{O}(1)$, $H_2 = \pi_2^* \mathcal{O}(1)$, and $E$ be the exceptional divisor of the blow-up $\pi$.  Let $l_1$ be the class of a fiber of $\pi_2$, $l_2$ be the class of a general line in each fiber of $\pi_1$, and $e$ be a fiber of $E\rightarrow c$.

\begin{thm}\label{3.12thm}
For all $\alpha \in l_1 + \Nef_1(X)_\mathbb{Z}$, $\overline{\free}(X,\alpha)$ is nonempty and irreducible.
\end{thm}

\textbf{Intersection Pairing:} $ H_i \cdot l_j = \delta_{ij}$, $E \cdot e = -1$, and all other pairings are $0$.

\textbf{Anticanonical Divisor:} $-K_X = 2H_1 + 3H_2 - E$

\textbf{Effective Divisors:} $E$, $H_1$, $2H_1 + 3H_2 -2E$, and $2H_2 -E$ are all effective.

\textbf{Effective Curves:} The extreme rays of $\NE(X)$ are $e$, $l_1-e$, and $l_2 -2e$.

\begin{rem}
$X$ may also be realized as the blow-up of $\mathbb{P}^3$ along a line $\ell$ and twisted cubic $c$.  If $l$ is the class of a line in $\mathbb{P}^3$ and $e_\ell,e_c$ are the classes of fibers in $X$ over points in $\ell$ and $c$, respectively, $e = l-e_\ell -2e_c$, $l_1 -e = e_\ell$, and $l_2 -2e = e_c$.  The divisor swept out by curves of class $l-e_\ell -2e_c$ has class $4H -2E_c -E_\ell$.  Theorem \ref{3.12thm} also holds for all $\alpha \in l + \Nef_1(X)_\mathbb{Z}$.
\end{rem}

\begin{lem}
A core of free curves on $X$ is given by
\begin{align*}
   \mathscr{C}_X = &\{ l_1, \ l_2, \ l_2 -e, \ 2l_2 -3e, \ l_1 + l_2 -2e, \ l_1 + 2l_2 -4e \} 
\end{align*}
The only separating class in $\mathscr{C}_X$ is $l_2 -e$.
\end{lem}
\begin{proof}
\textbf{Nef Curve Classes of Anticanonical Degree Between $2$ and $4$:}
If $\alpha = d_1 l_1 + d_2 l_2 -ne$ is nef and of appropriate degree, we find $2n\leq \min(4d_2,2d_1 + 3d_2)$, $d_1,n \geq 0$, and $2 \leq 2d_1 + 3d_2 -n \leq 4$.  The integral solutions to this system of inequalities are the classes in $\mathscr{C}_X$ and $2l_1, 2l_2 -2e, l_1 + l_2 -e$.  Each of these additional classes is freely breakable (see Lemmas \ref{curves1 in P1 x P2}, \ref{curves2 in P1 x P2}, and \ref{curves3 in P1 x P2}).  Lemmas \ref{curves1 in P1 x P2}, \ref{curves2 in P1 x P2}, and \ref{curves3 in P1 x P2} finish our proof.
\end{proof}

\begin{lem}
Relations in the monoid $\mathbb{N}\mathscr{C}_X$ are generated by:
\begin{enumerate}
    \item $(l_1 + 2l_2 -4e) + (l_2 -e) = (l_1 + l_2 -2e) + (2l_2 -3e)$,
    \item $(l_1 + 2l_2 -4e) + l_1 = 2(l_1 + l_2 -2e)$,
    \item $(l_1 + 2l_2 -4e) + l_2 = (l_1 + l_2 -2e) + 2(l_2 -e)$,
    \item $l_2 + (2l_2 -3e) = 3(l_2 -e)$,
    \item $l_2 + (l_1 + l_2 -2e) = 2(l_2 -e) + l_1$,
    \item $(l_2 -e) + (l_1 + l_2 -2e) = l_1 + (2l_2 -3e)$.
\end{enumerate}
For each relation $\sum \alpha_i = \sum \beta_j$, a main component of $\prod_X \overline{\free}_2(X,\alpha_i)$ lies in the same component of free curves as a main component of $\prod_X \overline{\free}_2(X, \beta_j)$.
\end{lem}
\begin{proof}
\textbf{Relations:}
We follow Lemma \ref{Relations}.
\begin{itemize}
    \item $D= H_1 + 3H_2 -2E$: $D$ pairs negatively with $l_1 + 2l_2 -4e$ and positively with $l_2 -e$, $l_1$, and $l_2$.  We find relations (1) through (3).
    \item $D= -H_2 + E$: $D$ pairs negatively with $l_2$ and positively with $2l_2 -3e$ and $l_1 + l_2 -2e$.  We obtain relations (4) and (5).
    \item The remaining four classes span $N_1(X)$ and satisfy relation (6).
\end{itemize}

\textbf{Main Components:} We address each relation below.
\begin{enumerate}
    \item $\overline{\free}(X,l_1 + 3l_2 -5e)$ is irreducible by Lemma \ref{curves2 in P1 x P2}.
    \item $\overline{\free}(X,2l_1 + 2l_2 -4e)$ is irreducible by Lemma \ref{curves3 in P1 x P2}.
    \item $\overline{\free}(X,l_1 + 3l_2 -4e)$ is irreducible by Lemma \ref{curves2 in P1 x P2}.
    \item $\overline{\free}(X, 3l_2 -3e)$ has at least three components by \ref{curves1 in P1 x P2}.  The component that parameterizes planar cubics meeting $c$ at three points in their smooth locus contains free chains of type $(l_2, 2l_2 -3e)$ and $(l_2 -e, l_2 -e, l_2 -e)$.
    \item $\overline{\free}(X, l_1 + 2l_2 -2e)$ is irreducible by Lemma \ref{curves2 in P1 x P2}.
    \item $\overline{\free}(X,l_1 + 2l_2 -3e)$ is irreducible by Lemma \ref{curves2 in P1 x P2}.
\end{enumerate}
\end{proof}

\begin{lem}
For each nonzero $\alpha \in \Nef_1(X)_\mathbb{Z}$, $\overline{\free}(X,\alpha)$ is nonempty.
\end{lem}

\begin{proof}
The generators of $\Nef_1(X)$ are $l_1, l_2$, $2l_2 -3e$, and $l_1 + 2l_2 -4e$ by \ref{Representability of Free Curves}.  It follows from \ref{Gordan's Lemma} that classes in $\mathscr{C}_X$ generate the monoid of integer points in $\Nef_1(X)$.
\end{proof}

\subsection*{3.13}

\textbf{Blow-up of a Divisor in $\mathbb{P}^2 \times \mathbb{P}^2$ along a curve:}
Let $W \subset \mathbb{P}^2 \times \mathbb{P}^2$ be a smooth divisor of bidegree $(1,1)$. Let $f:X \rightarrow W$ be the blow-up of a curve $c \subset W$ of bidegree $(2,2)$ such that the projection of $c$ onto both factors of $\mathbb{P}^2$ is an embedding. Let $f_i: X \rightarrow \mathbb{P}^2$ be $f$ composed with the projection $\pi_i$ onto the corresponding factor of $\mathbb{P}^2$ for $i=1,2$.

Since $W$ is the blow-down of a Fano threefold of deformation type 4.7, Theorem \ref{4.7thm} and Lemma \ref{blowup} prove Theorem \ref{Main Result} for $W$.  We will use this fact below.

\begin{thm}\label{3.13thm}
For each $\alpha \in \Nef_1(X)_\mathbb{Z}$, $\overline{\free}(X,\alpha)$ is irreducible.
\end{thm}

\textbf{Generators for $N^1(X)$ and $N_1(X)$:}
\begin{center}
\begin{tabular}{ll}
 $H_1$ = $f_1^*(l)$ & $l_1$ = $f_2^*(pt)$ \\ 
 $H_2$ = $f_2^*(l)$ & $l_2$ = $f_1^*(pt)$ \\
 $E$ = the exceptional divisor $f^{-1}(c)$ & $e$ = the $f$-fiber over a point in $c$
\end{tabular}
\end{center}

\textbf{Intersection Pairing:}
\begin{center}
\begin{tabular}{lll}
    $H_1 \cdot l_1 = 1$ &  $H_1 \cdot l_2 = 0$ & $H_1 \cdot e = 0$, \\
    $H_2 \cdot l_1 = 0$ &  $H_2 \cdot l_2 = 1$ & $H_2 \cdot e = 0$ \\
    $E \cdot l_1 = 0$ & $E \cdot l_2 = 0$ & $E \cdot e = -1$
\end{tabular}
\end{center}

\textbf{Anticanonical Divisor:}
\begin{align*}
    -K_X = 2H_1 + 2H_2 - E.
\end{align*}

\textbf{Effective Divisors:} The divisors $2H_1-E$, $2H_2-E$, and $E$ are effective.

\textbf{Pseudosymmetry:}
The contraction of $X$ by $l_1-e$ realizes $X$ as the blow-up of $W' \subset \mathbb{P}^2 \times \mathbb{P}^2$ along a curve $c'$ for which both maps $c' \rightarrow \mathbb{P}^2$ are embeddings by \cite{matsuki1995weyl} and \cite{mori1981classification}. This corresponds to the pseudosymmetry
\begin{align*}
    l_1-e \mapsto e, \hspace{.5cm} e \mapsto l_1-e, \hspace{.5cm} l_2-e \mapsto l_2-e.
\end{align*}

There is another pseudosymmetry simply swapping the classes $l_1$ and $l_2$.

\begin{lem}\label{3.13core}
A core of free curves on $X$ is given by
\begin{align*}
    \mathscr{C}_X = \{ l_1, l_2, l_1+l_2-e, l_1+l_2-2e \}.
\end{align*}

\noindent There are no separating classes in $\mathscr{C}_X$.
\end{lem}

\begin{proof}
\textbf{Nef Curve Classes of Anticanonical degree between $2$ and $4$:}

Such nef curve classes $\alpha = a l_1 + b l_2 - c e$ satisfy the equations

\begin{align*}
    0 \leq c \leq 2a, 2b, \hspace{.5cm} 2 \leq 2a+2b-c \leq 4.
\end{align*}

\noindent The solutions are those classes in $\mathscr{C}_X$, as well as $2l_1$, $2l_2$, $l_1+l_2$, $l_1+2l_2-2e$, $2l_1+l_2-2e$, and $2l_1+2l_2-4e$.

The curves of class $2l_1$, $2l_2$, and $l_1+l_2$ may be freely broken because they each constitute a class of curves on $W$ and we already know Theorem \ref{Main Result} for $W$. Under pseudosymmetry, the curves of class $l_1+2l_2-2e$ and $2l_1+l_2-2e$ may both be transformed to class $l_1+l_2$, which we just said are freely breakable. Similarly, the curves of class $2l_1+2l_2-4e$ are transformed to class $2l_2$.

\textbf{Irreducible Spaces and Fibers:} 
The curves of class $l_1$ and $l_2$ are parameterized by an irreducible space whose universal family has irreducible fibers because they correspond to classes in $W$ for which this is already known. The same holds for curves of class $l_1+l_2-2e$ because they are are transformed to $l_2$ under pseudosymmetry.

A curve of class $l_1+l_2-e$ is determined by the choice of a line $x \subset \mathbb{P}^2$ and a degree one morphism $x \rightarrow \mathbb{P}^2$ whose graph lies in $W$ and intersects the curve $c$. Clearly, the space of lines $x$ in $\mathbb{P}^2$ is irreducible. The condition that the graph of $x \rightarrow \mathbb{P}^1$ lies in $W$ is linear in the coefficients of the morphism. Moreover, the projection $\pi_1(c)$ of $c$ in $\mathbb{P}^2$ intersects a fixed line $x$ in two points $\pi_1(p)$ and $\pi_1(q)$. In order for its graph to intersect $c$, the morphism $x \rightarrow \mathbb{P}^2$ must map either $\pi_1(p) \mapsto \pi_2(p)$ or $\pi_1(q) \mapsto \pi_2(q)$; both these conditions are likewise linear in the coefficients of the morphism. The monodromy action of the linear system $|H_1|$ is transitive by \ref{Monodromy}, so it follows that the moduli space of curves of class $l_1+l_2-e$ is indeed irreducible.

The space of such curves passing through a general point $p \in W$ is also irreducible because it admits the same description, except that now we need only consider lines $x$ in $\mathbb{P}^2$ passing through $\pi_1(p)$ and morphisms $x \rightarrow \mathbb{P}^2$ for which $\pi_1(p) \mapsto \pi_2(p)$.
\end{proof}

\begin{lem}\label{3.13relations}
There is a single relation $(l_1+l_2-2e)+(l_1)+(l_2)=2(l_1+l_2-e)$ in the monoid $\mathbb{N} \mathscr{C}_X$. The corresponding moduli space of free curves of class $2l_1+2l_2-2e$ is irreducible.
\end{lem}

\begin{proof}
The four classes in the core $\mathscr{C}_X$ span a three-dimensional space, so satisfy a unique relation $(l_1+l_2-2e)+(l_1)+(l_2)=2(l_1+l_2-e)$.

A curve of class $2l_1+2l_2-2e$ is determined by the choice of a conic $x \subset \mathbb{P}^2$ and a degree two morphism $x \cong \mathbb{P}^1 \rightarrow \mathbb{P}^2$ whose graph lies in $W$ and intersects the curve $c$ in two points. The space of conics in $\mathbb{P}^2$ is five-dimensional and irreducible. A degree two morphism $x \cong \mathbb{P}^1 \rightarrow \mathbb{P}^2$ is determined by nine homogeneous coordinates, so the family of such morphisms is eight-dimensional. Moreover, since $W$ is of bidegree $(1,1)$, the condition that it lie in $W$ imposes five linear conditions on the coefficients. Indeed, substituting the coordinate expression of the morphism into the equation defining $W$, this condition translates into the vanishing of a degree four polynomial on $x \cong \mathbb{P}^1$. The projection $\pi_1(c)$ of $c$ in $\mathbb{P}^2$ intersects a fixed conic $x$ in four points $\pi_1(p_1) \ldots \pi_1(p_4)$. In order for its graph to intersect $c$, a morphism $x \rightarrow \mathbb{P}^2$ must map two of $\pi_1(p_1) \mapsto \pi_2(p_1) \ldots \pi_1(p_4) \mapsto \pi_2(p_4)$; each of these is a single linear condition on the coefficients because the graph has already been forced to lie in $W$. The monodromy action of the linear system $|2H_1|$ is $2$-transitive by \ref{Monodromy}, so it follows that the moduli space of curves of class $2l_1+2l_2-2e$ is indeed irreducible.
\end{proof}

\subsection*{3.14}
See Section \ref{E5 cases}.

\subsection*{3.15}

\textbf{Description of Variety $X$:} $X$ may be realized as a blow-up $\pi:X\rightarrow \mathbb{P}^1 \times \mathbb{P}^2$ along a curve $c$ of bidegree $(2,2)$ using \ref{blow-up numbers}.  We let $\pi_1 :X \rightarrow \mathbb{P}^1$ and $\pi_2:X\rightarrow \mathbb{P}^2$ be the compositions of the blow-up with each projection.

\textbf{Generators for $N^1(X)$ and $N_1(X)$:} Let $H_1 = \pi_1^* \mathcal{O}(1)$, $H_2 = \pi_2^* \mathcal{O}(1)$, and $E$ be the exceptional divisor of the blow-up $\pi$.  Let $l_1$ be the class of a fiber of $\pi_2$, $l_2$ be the class of a general line in each fiber of $\pi_1$, and $e$ be a fiber of $E\rightarrow c$.

\begin{thm}\label{3.15thm}
For all $\alpha \in l_1 + \Nef_1(X)_\mathbb{Z}$, $\overline{\free}(X,\alpha)$ is nonempty and irreducible.
\end{thm}

\textbf{Intersection Pairing:} $ H_i \cdot l_j = \delta_{ij}$, $E \cdot e = -1$, and all other pairings are $0$.

\textbf{Anticanonical Divisor:} $-K_X = 2H_1 + 3H_2 - E$

\textbf{Effective Divisors:} $E$, $H_1$, $H_1 +H_2 -E$, and $2H_2 -E$ are all effective.

\textbf{Effective Curves:} The extreme rays of $\NE(X)$ are $e$, $l_1-e$, and $l_2 -2e$.

\begin{rem}
$X$ may also be realized as the blow-up of $Q\subset \mathbb{P}^4$ along a line $\ell$ and disjoint conic $c$.  If $l$ is the class of a line in $\mathbb{P}^4$ and $e_\ell,e_c$ are the classes of fibers in $X$ over points in $\ell$ and $c$, respectively, then $e = l-e_\ell -e_c$, $l_2 -2e = e_\ell$, and $l_1 -e = e_c$.  The divisor swept out by curves of class $l-e_\ell -e_c$ has class $2H -E_c -2E_\ell$.  Theorem \ref{3.15thm} also holds for all $\alpha \in l + \Nef_1(X)_\mathbb{Z}$.
\end{rem}

\begin{lem}
A core of free curves on $X$ is given by
\begin{align*}
   \mathscr{C}_X = &\{ l_1, \ l_2, \ l_2 -e, \ l_1 + l_2 -2e \} 
\end{align*}
The only separating class in $\mathscr{C}_X$ is $l_2 -e$.
\end{lem}
\begin{proof}
\textbf{Nef Curve Classes of Anticanonical Degree Between $2$ and $4$:}
If $\alpha = d_1 l_1 + d_2 l_2 -ne$ is nef and of appropriate degree, we find $n\leq \min(2d_2,d_1 + d_2)$, $d_1,n \geq 0$, and $2 \leq 2d_1 + 3d_2 -n \leq 4$.  The integral solutions to this system of inequalities are the classes in $\mathscr{C}_X$ and $2l_1, 2l_2 -2e, l_1 + l_2 -e$.  Each of these additional classes is freely breakable (see Lemmas \ref{curves1 in P1 x P2}, \ref{curves2 in P1 x P2}, and \ref{curves3 in P1 x P2}).  Lemmas \ref{curves1 in P1 x P2}, \ref{curves2 in P1 x P2}, and \ref{curves3 in P1 x P2} finish our proof.
\end{proof}

\begin{lem}
Relations in the monoid $\mathbb{N}\mathscr{C}_X$ are generated by:
$$l_1 + 2(l_2 -e) = l_2 + (l_1 + l_2 -2e).$$
The corresponding space of free curves, $\overline{\free}(X, l_1 + 2l_2 -2e)$ is irreducible.  
\end{lem}
\begin{proof}
The four classes in $\mathscr{C}_X$ span $N_1(X)$ and satisfy the above relation.  Irreducibility of $\overline{\free}(X, l_1 + 2l_2 -2e)$ follows from Lemma \ref{curves2 in P1 x P2}.
\end{proof}

\begin{lem}
For each nonzero $\alpha \in \Nef_1(X)_\mathbb{Z}$, $\overline{\free}(X,\alpha)$ is nonempty.
\end{lem}

\begin{proof}
The generators of $\Nef_1(X)$ are elements in $\mathscr{C}_X$ by \ref{Representability of Free Curves}.  It follows from \ref{Gordan's Lemma} that classes in $\mathscr{C}_X$ generate the monoid of integer points in $\Nef_1(X)$.
\end{proof}

\subsection*{3.16}

\textbf{Blow-up of a Point and a Twisted Cubic in $\mathbb{P}^3$:}
Let $g:V_7 \rightarrow \mathbb{P}^3$ be the blow-up of a point $p \in \mathbb{P}^3$. Let $f:X \rightarrow V_7$ be the blow-up of the strict transform of a twisted cubic $c$ passing through the exceptional divisor.

\begin{thm}\label{3.16thm}
For each $\alpha \in \Nef_1(X)_\mathbb{Z}$, $\overline{\free}(X,\alpha)$ is irreducible.
\end{thm}

\textbf{Generators for $N^1(X)$ and $N_1(X)$:}
\begin{center}
\begin{tabular}{ll}
 $H$ = the class of a hyperplane & $l$ = a line in $\mathbb{P}^3$ \\ 
 $E$ = the exceptional divisor $(g \circ f)^{-1}(p)$ & $e$ = a line in $E$ disjoint from $F$  \\  
 $F$ = the exceptional divisor $f^{-1}(c)$ & $f$ = the $f$-fiber over a point on $c$   
\end{tabular}
\end{center}

\textbf{Intersection Pairing:}
\begin{center}
\begin{tabular}{lll}
    $H \cdot l = 1$ &  $H \cdot e = 0$ & $H \cdot f = 0$, \\
    $E \cdot l = 0$ &  $E \cdot e = -1$ & $E \cdot f = 0$ \\
    $F \cdot l = 0$ & $F \cdot e = 0$ & $F \cdot f = -1$
\end{tabular}
\end{center}

\textbf{Anticanonical Divisor:}
\begin{align*}
    -K_X = 4H - 2E - F.
\end{align*}

\textbf{Effective Divisors:} The divisors $H-E$, $2H-2E-F$, $E$, and $F$ are all effective. Indeed, the projection from $p \in \mathbb{P}^3$ maps $c$ to a planar conic and the cone over this conic is of class $2H-2E-F$.

\begin{lem}\label{3.16core}
A core of free curves on $X$ is given by
\begin{align*}
    \mathscr{C}_X = \{ l, l-e, l-f, l-2f \}.
\end{align*}

\noindent There are no separating classes in $\mathscr{C}_X$.
\end{lem}

\begin{proof}
\textbf{Nef Curve Classes of Anticanonical degree between $2$ and $4$:}

Such nef curve classes $\alpha = a l - b e - c f$ satisfy the equations
\begin{align*}
    0 \leq b,c, \hspace{.5cm} b \leq a, \hspace{.5cm} 2b+c \leq 2a, \hspace{.5cm} 2 \leq 4a-2b-c \leq 4.
\end{align*}

\noindent Moreover, any curve of degree $a>1$ meeting $p$ with multiplicity at least $a$ is reducible. Hence, we may assume $b<a$ whenever $a>1$. The solutions are those classes in $\mathscr{C}_X$, as well as $2l-e-2f$ and $2l-4f$.

The curves of class $2l-e-2f$ each lie in a plane, so may be freely broken into a sum of two curves of class $l-e$ and $l-2f$ respectively. The curves of class $2l - 4f$ are all double covers.

\textbf{Irreducible Spaces and Fibers:}  The curves of class $l$ are parameterized by an open subset of $\mathbb{G}(1,3)$, curves of class $l-e$ by an open subset of $\mathbb{P}^2$, and curves of class $l - f$ by an open subset of a $\mathbb{P}^2$-bundle over $c$. They all have irreducible fibers.

The curves of class $l - 2f$ are parameterized by an open subset of $\Sym^2 c$. They have irreducible fibers because there is a unique secant line to $c$ passing through any point $q \in \mathbb{P}^3 - c$.
\end{proof}

\begin{lem}\label{3.16relations}
There is a single relation $(l-2f)+(l)=2(l-f)$ in the monoid $\mathbb{N} \mathscr{C}_X$. The corresponding moduli space of free curves of class $2l-2f$ is irreducible.
\end{lem}

\begin{proof}
The four classes in the core $\mathscr{C}_X$ span a three-dimensional space, so satisfy a unique relation $(l-2f)+(l)=2(l-f)$. A curve of class $2l-2f$ is determined by a choice of plane $A$ in $\mathbb{P}^3$, two points in $A \cap c$, and a conic in $A$ passing through those two points. By \ref{Monodromy}, the monodromy action on these three points is transitive, so the moduli space of such curves is irreducible.
\end{proof}

\subsection*{3.17}

\textbf{Divisor of Tridegree $(1,1,1)$ on $\mathbb{P}^1 \times \mathbb{P}^1 \times \mathbb{P}^2$:}
Let $X$ be a smooth divisor on $\mathbb{P}^1 \times \mathbb{P}^1 \times \mathbb{P}^2$ of tridegree $(1,1,1)$. We apply \ref{blowup} to case $4.3$.%

\subsection*{3.18}

\textbf{Blow-up of a line and a conic in $\mathbb{P}^3$:}
Let $f:X \rightarrow \mathbb{P}^3$ be the blow-up of $\mathbb{P}^3$ with center a disjoint union of a line $c$ and a conic $c'$. We apply \ref{blowup} to case $4.4$.%

\subsection*{3.19}

\textbf{Blow-up of a Quadric Threefold:}
Let $X$ be the blow-up of a smooth quadric $Q \subset \mathbb{P}^4$ with center two points $p$ and $q$ on it which are not colinear. We apply \ref{blowup} to case $4.4$.%

\subsection*{3.20}

\textbf{Description of Variety $X$:}
Let $f:X \rightarrow Q$ be the blow-up of a smooth quadric $Q \subset \mathbb{P}^4$ with center a union of two disjoint lines $c_1$ and $c_2$.

\begin{thm}\label{3.20thm}
For each $\alpha \in \Nef_1(X)_\mathbb{Z}$, $\overline{\free}(X,\alpha)$ is irreducible.
\end{thm}

\textbf{Generators for $N^1(X)$ and $N_1(X)$:}

\begin{center}
\begin{tabular}{ll}
 $H$ = a hyperplane in $\mathbb{P}^4$ & $l$ = a line in $\mathbb{P}^4$ \\ 
 $E_1$ = the exceptional divisor $f^{-1}(c_1)$ & $e_1$ = an $f$-fiber over a point on $c_1$  \\
 $E_2$ = the exceptional divisor $f^{-1}(c_2)$ & $e_2$ = an $f$-fiber over a point on $c_2$
\end{tabular}
\end{center}

\textbf{Intersection Pairing:}
\begin{center}
\begin{tabular}{lll}
    $H \cdot l = 1$ &  $H \cdot e_1 = 0$ & $H \cdot e_2 = 0$, \\
    $E_1 \cdot l = 0$ &  $E_1 \cdot e_1 = -1$ & $E_1 \cdot e_2 = 0$ \\
    $E_2 \cdot l = 0$ & $E_2 \cdot e_1 = 0$ & $E_2 \cdot e_1 = -1$
\end{tabular}
\end{center}

\textbf{Anticanonical Divisor:}
\begin{align*}
    -K_X = 3H-E_1-E_2
\end{align*}

\textbf{Effective Divisors:} The divisors $H-E_1-E_2$, $E_1$, and $E_2$ are effective.

\begin{lem}\label{3.20core}
A core of free curves on $X$ is given by 
\begin{align*}
    \mathscr{C}_X = \{ l, l-e_1, l-e_2 \}.
\end{align*}
There are no separating curve classes in $\mathscr{C}_X$.
\end{lem}

\begin{proof}
\textbf{Nef Curve Classes of Anticanonical degree between $2$ and $4$:}

Such nef curve classes $\alpha = a l - b e_1 -c e_2$ satisfy the equations
\begin{align*}
    0 \leq b,c, \hspace{.5cm} b+c \leq a, \hspace{.5cm} 2 \leq 3a-b-c \leq 4.
\end{align*}

\noindent The solutions are those classes in $\mathscr{C}_X$, as well as $2l-2e_1, 2l-e_2$, and $2l-e_1-e_2$.

Any irreducible curve $c$ of class $2l-2e_i$ lies in a plane $A$ containing $c_i$. But this implies $A \cap Q$ contains a curve of degree $3$; since $\deg(Q) = 2 < 3$, it follows that $A \subset Q$. But this contradicts \cite[Corollary~6.26]{eisenbud20163264}. Therefore all such curve classes are double covers.

The curves of class $2l-e_1-e_2$ are precisely the intersection with $Q$ of those planes in $\mathbb{P}^4$ which intersect both $c_1$ and $c_2$. Any such curve may be freely broken into a sum of curves of class $l-e_1$ and $l-e_2$ respectively by degenerating the plane so that its intersection with $Q$ is reducible.

\textbf{Irreducible Spaces and Fibers:}  The curves of class $l$ are parameterized by an open subset of the lines in $Q$. By \cite[Corollary~6.33]{eisenbud20163264}, this space is smooth.  For any point $p \in Q$, every line in $Q$ passing through $p$ lies in the projective tangent space of $Q$ at $p$, $\mathbb{P}\mathcal{T}_p Q \subset \mathbb{P}^4$.  %
The intersection $\mathbb{P}\mathcal{T}_p Q \cap Q$ is the cone over a smooth conic with vertex $p$.  Hence, the space of lines in $Q$ through $p$ is irreducible.  %

The curves of class $l-e_i$ may be parameterized by an open subset of a bundle over $c_i$ whose fiber over a point $p \in c_i$ is the space of lines through $p \in \mathbb{P}^4$ lying in $Q$. As noted above, this fiber is generally a smooth conic. If there were multiple curves of class $l-e_i$ passing through a point $p \in Q - c_1 - c_2$, then the plane $A$ that they span would intersect $Q$ in three lines. But then $A \subset Q$, which is impossible. Therefore the curves of class $l-e_i$ have irreducible fibers.
\end{proof}

As the classes of $\mathscr{C}_X$ are linearly independent, we obtain the following.

\begin{lem}\label{3.20relations}
There are no relations in the monoid $\mathbb{N}\mathscr{C}_X$.
\end{lem}

\subsection*{3.21}

\textbf{Blow-up of $\mathbb{P}^1 \times \mathbb{P}^2$:}
Let $X$ be the blow-up of $\mathbb{P}^1 \times \mathbb{P}^2$ with center a curve $c$ of bidegree $(2,1)$. We apply \ref{blowup} to case $4.5$.%

\subsection*{3.22}  See Section \ref{E5 cases}.

\subsection*{3.23}

\textbf{The Blow-up of a Point and the Strict Transform of a Curve:}  Let $g:V_7 \rightarrow \mathbb{P}^3$ be the blow-up of a point $p \in \mathbb{P}^3$. Let $f:X \rightarrow V_7$ be the blow-up of the strict transform of a conic $c$ passing through the exceptional divisor.

\begin{thm}\label{3.23thm}
For each $\alpha \in \Nef_1(X)_\mathbb{Z}$, $\overline{\free}(X,\alpha)$ is irreducible.
\end{thm}

\textbf{Generators for $N^1(X)$ and $N_1(X)$:} 
\begin{center} 
\begin{tabular}{ll}
 $H$ = the class of a hyperplane & $l$ = a line in $\mathbb{P}^3$ \\ 
 $E$ = the exceptional divisor $(g \circ f)^{-1}(p)$ & $e$ = a line in $E$ disjoint from $F$  \\
 $F$ = the exceptional divisor $f^{-1}(c)$ & $f$ = the $f$-fiber over a point on $c$   
\end{tabular}
\end{center}

\textbf{Intersection Pairing:}
\begin{center}
\begin{tabular}{lll}
    $H \cdot l = 1$ &  $H \cdot e = 0$ & $H \cdot f = 0$, \\
    $E \cdot l = 0$ &  $E \cdot e = -1$ & $E \cdot f = 0$ \\
    $F \cdot l = 0$ & $F \cdot e = 0$ & $F \cdot f = -1$
\end{tabular}
\end{center}

\textbf{Anticanonical Divisor:}
\begin{align*}
    -K_X = 4H - 2E - F.
\end{align*}

\textbf{Effective Divisors:} The divisors $H-E-F$, $E$, and $F$ are effective.

\begin{lem}\label{3.23core}
A core of free curves on $X$ is given by
\begin{align*}
    \mathscr{C}_X = \{ l, l-e, l-f \}.
\end{align*}
\noindent There are no separating curve classes in $\mathscr{C}_X$
\end{lem}

\begin{proof}
\textbf{Nef Curve Classes of Anticanonical degree between $2$ and $4$:}

Such nef curve classes $\alpha = a l - b e - c f$ satisfy the equations
\begin{align*}
    0 \leq b,c, \hspace{.5cm} b+c \leq a, \hspace{.5cm} 2 \leq 4a-2b-c \leq 4.
\end{align*}

\noindent The solutions are those classes in $\mathscr{C}_X$, as well as $2l-2e$. The curves of class $2l-2e$ lie in a plane, so it is clear that they may be freely broken into a sum of two curves, each of class $l-e$.

\textbf{Irreducible Spaces and Fibers:} 
The curves of class $l$ are parameterized by an open subset of $\mathbb{G}(1,3)$ and have irreducible fibers. The curves of class $l-e$ are parameterized by an open subset of the space of lines in $\mathbb{P}^3$ through $p$, which is isomoprhic to $\mathbb{P}^2$. Moreover, there is a unique such curve through a general point. Finally, the curves of class $l-f$ are parameterized by an open subset of a $\mathbb{P}^2$-bundle over $c$, and also have irreducible fibers.
\end{proof}

As the classes of $\mathscr{C}_X$ are linearly independent, we obtain the following.

\begin{lem}\label{3.23relations}
There are no relations in the monoid $\mathbb{N}\mathscr{C}_X$.
\end{lem}

\subsection*{3.24}

\textbf{Blow-up of a $\mathbb{P}^1$-bundle over $\mathbb{P}^2$:}
Let $X$ be a fiber product $W \times_{\mathbb{P}^2} \mathbb{F}_1$ where $W \rightarrow \mathbb{P}^2$ is any $\mathbb{P}^1$-bundle and where $\mathbb{F}_1 \rightarrow \mathbb{P}^2$ is the blow-up at a point. We apply \ref{blowup} to case $4.7$.%

\subsection*{3.25}

\textbf{Blow-up of $\mathbb{P}^3$ along two lines:}
Let $f:X \rightarrow \mathbb{P}^3$ be the blow-up of $\mathbb{P}^3$ at two disjoint lines. We may apply \ref{blowup} to case $4.6$.%

\subsection*{3.26}

\textbf{Blow-up of $\mathbb{P}^3$ along a point and a line:}
Let $f:X \rightarrow \mathbb{P}^3$ be the blow-up of $\mathbb{P}^3$ with center a disjoint union of a point $p$ and a line $c$.  We apply \ref{blowup} to case $4.9$.%

\subsection*{3.27}

\textbf{$(\mathbb{P}^1)^3$:}
Let $X = \mathbb{P}^1 \times \mathbb{P}^1 \times \mathbb{P}^1$. We may apply \ref{blowup} to case $4.10$.%

\subsection*{3.28}

\textbf{$\mathbb{P}^1 \times \mathbb{F}_1$:}
Let $X = \mathbb{P}^1 \times \mathbb{F}^1$ where $\mathbb{F}^1$ is the blow-up of $\mathbb{P}^2$ at a point. We apply \ref{blowup} to case $4.10$.%

\subsection*{3.29}  See Section \ref{E5 cases}.

\subsection*{3.30}
\textbf{Blow-up of $\mathbb{P}^3$:}
Let $X$ be the blow up of $\mathbb{P}^3$ along a point $p$ and the strict transform of a line containing $p$.  We apply \ref{blowup} to case $4.12$.%

\subsection*{3.31}

\textbf{Blow-up of the Cone over a Quadric Surface:}
Let $X$ be the blow-up of the cone $C \subset \mathbb{P}^4$ over a smooth quadric surface $S \subset \mathbb{P}^3$ along its vertex $p$. Alternatively, we may view $X$ as the $\mathbb{P}^1$-bundle $\mathbb{P}(\mathcal{O} \oplus \mathcal{O}(1,1))$ over $\mathbb{P}^1 \times \mathbb{P}^1$. We apply \ref{blowup} to case $4.2$.%

\section{Picard Rank 2}\label{Picard Rank 2 Cases}

\subsection{Conic and Del Pezzo Fibrations on Blow-Ups of Del Pezzo Threefolds and $Q$}
We study Fano threefolds of deformation types 2.1, 2.3-2.5, 2.7, 2.9-2.11, 2.13, 2.14, 2.16, and 2.20 using their conic and del Pezzo fiber structures.  Each of the aforementioned threefolds is a blow-up $X \rightarrow V_i$ of some del Pezzo threefold $V_i$ of Picard rank 1 along a smooth curve, or a blow-up $X \rightarrow Q \subset \mathbb{P}^4$.  We use the following notation throughout these cases, as well as the following result.

\begin{lem}[Theorem 7.6 \cite{2019}]
Let $X$ be a smooth Fano threefold of Picard rank $1$ and index $r \geq 2$.  Furthermore when $(-K_X)^3 \leq 16$, assume $X$ is general in moduli.  Let $\alpha \in \Nef_1(X)$ and $d = -K_X . \alpha$.  Suppose $W$ is a component of $\overline{M}_{0,0}(X,\alpha)$ and let $W_p$ denote the sublocus parameterizing curves through the point $p\in X$.  There is a finite union of points $S \subset X$ such that:
\begin{itemize}
    \item $W_p$ has the expected dimension $d-2$ for points $p \not\in S$;
    \item $W_p$ has dimension at most $d-1$ for points $p\in S$.
\end{itemize}
Furthermore for $p \notin S$ the general curve parameterized by $W_p$ is irreducible.
\end{lem}

\textbf{Blow-up of a curve $c \subset V_i$ or $c \subset Q$:} %
Let $Y$ be a smooth quadric $Q \subset \mathbb{P}^4$ or a del Pezzo threefold $V_i$ of Picard rank 1.  Let $\phi : X \rightarrow Y$ be the blow-up of the smooth curve $c\subset Y$.  Write $r$ for the index of $Y$. It follows from \cite{mori1981classification} that $c$ is scheme-theoretically the intersection of a linear system $L \subset |-\frac{r-1}{r}K_{Y}|$.  Let $\pi : X \rightarrow \mathbb{P}(L)$ be the contraction of the other extreme ray of $\NE(X)$. 

\textbf{Generators for $N^1(X)$ and $N_1(X)$:} Let $H$ be the pullback of the ample generator of $\text{Pic}(Y)$, i.e. $H = \frac{-1}{r}K_{Y}$.  Let $E$ be the exceptional divisor over $c$.  Let $l$ be the pullback of an $H$-line from $Y$, and $e$ be the fiber over a point in $c$.

\textbf{Intersection Pairing:}
\begin{center}
\begin{tabular}{ll}
    $H \cdot l = 1$ &  $H \cdot e = 0$, \\
    $E \cdot l = 0$ &  $E \cdot e = -1$, \\
\end{tabular}
\end{center}

\textbf{Anticanonical Divisor:}
\begin{align*}
    -K_X = rH-E.
\end{align*}

\textbf{Generators of $\NE(X)$ and $\text{Eff}(X)$:}  The extreme rays of $\NE(X)$ are spanned by $e$ and $l - (r-1) e$.  These are the only classes of effective $-K_X$ lines.  The extreme rays of $\text{Eff}(X)$ are spanned by $E$ and $(r-1) H -E$.  

\subsubsection*{Del Pezzo Fibrations}
When $c \subset Y$ is a complete intersection of a pencil $L \subset |-\frac{r-1}{r}K_{Y}|$, for general $X$ the monodromy group of the resulting del Pezzo fibration $\pi : X \rightarrow \mathbb{P}(L)$ acts as the full Weyl group on smooth fibers by Theorem \ref{monodromy del Pezzo main result}.  Moreover, for general $X$ each fiber of $\pi$ contains only finitely many $-K_X$ lines.  Therefore, unless $Y = V_1$, $\overline{M}_{0,0}(X, l-(r-1)e)$ is irreducible and generically parameterizes smooth embedded curves with normal bundles $\mathcal{O} \oplus \mathcal{O}(-1)$.  When $Y = V_1$, \cite{Tihomirov_1982} shows for general $X$ there is one additional component of $\overline{M}_{0,0}(X, l-e)$ parametrizing singular anticanonical curves in fibers of $\pi$.

\begin{thm}\label{del Pezzo fibration GMC}
Let $X$ be the blow-up of a Fano threefold of index $r > 1$ along a complete intersection of ample generators.  %
Then for all $\alpha \in l + \Nef_1(X)_\mathbb{Z}$, $\overline{\free}^{bir}(X,\alpha)$ is irreducible and nonempty.  $\overline{\free}(X,\alpha)$ is irreducible unless $\alpha = n(l-(r-2)e)$ for some $n > 1$, in which case there are exactly two components.

\end{thm}

We first prove results about low-degree curves on $X$.  For $X$ general in moduli, we enumerate components of $\overline{\free}(X, \alpha)$ with $-K_X . \alpha \leq 4$.  Proposition \ref{low degree curves degeneration} shows this description extends to arbitrary $X$ when $\alpha \not\in \partial \NE(X)$.  The proof of Theorem \ref{del Pezzo fibration GMC} for arbitrary $X$ follows from Corollary \ref{connected fibers fixall}, Lemma \ref{del Pezzo relations}, and Lemma \ref{lem: del pezzo gordon}.

\begin{lem}\label{del Pezzo fibration Core}
Suppose $X$ is general in moduli.  Every component of $\overline{\free}(X)$ contains free chains with components of class in $\mathscr{C}_X$, where if
\begin{itemize}
    \item $r = 2$, $\mathscr{C}_X = \{ 2l -2e, l , 3l-3e, 2l-e\}$,
    \item $r = 3$, $\mathscr{C}_X = \{ 2l -4e, l -e, 3l-6e, 2l -3e, l\}$
    \item $r = 4$, $\mathscr{C}_X = \{ 2l -6e, l -2e, 3l - 9e, 2l -5e, l -e, l \}$
\end{itemize}
For each $\alpha \in \mathscr{C}_X$ interior to $\NE(X)$, either $-K_X . \alpha = 2$ and $\overline{\free}^{bir}(X, 2\alpha)$ is irreducible, or $-K_X . \alpha \geq 3$ and $\overline{\free}_1(X, \alpha) \rightarrow X$ has irreducible general fiber.  

For $\alpha \in \mathscr{C}_X$, $\overline{\free}(X, \alpha)$ is irreducible unless 1) $r = 4$ and $\alpha = 3l - 9e$, 2) $X$ is a blow-up of $V_3$ and $\alpha = 3l -3e$, or 3) $X$ is a blow-up of $V_1$ or $V_2$ and $\alpha \in \{2l -2e, 3l-3e\}$.  In these cases, if $X$ is not a blow-up of $V_1$ then $\overline{\free}(X,\alpha)$ has two components; otherwise, $\overline{\free}(X,2l -2e)$ and $\overline{\free}(X,3l -3e)$ have three and five components, respectively.
\end{lem}
\begin{proof}
$\overline{\free}(X,l)$ is shown to be irreducible in \cite{2019} and \cite{lastDelPezzoThreefold}.  Let $n = r-1$.  Since for general $X$ monodromy of the del Pezzo fibration is maximal, Lemma \ref{testa theorem} may be used to identify the number of components of $\overline{\free}(X, 2(l-ne))$ and $\overline{\free}(X, 3(l-ne))$ parameterizing free curves.  Mori's bend and break shows any component of $\overline{\free}(X, 4(l-ne))$ contains a chain of free conics.  When $r > 2$, $-K_X$ is very ample, and our claim about the irreduciblity of fibers of $\overline{\free}_1(X, 2(l-ne) + e) \rightarrow X$ follows from irreducibility of the space itself, which may be seen using \ref{Monodromy} (when $r = 4$) or by associating a plane in $\mathbb{P}^4$ to three points through the blown-up curve $c\subset Q$ (when $r = 3$).

Assume $r = 2$.  Then $X \rightarrow V_i$ is the blow-up of the baselocus of a pencil in $|H|$.  Irreducibility of $\overline{\free}(X, 2l-e)$ follows in each case from the analysis of \cite{2019} and \cite{lastDelPezzoThreefold}.  They prove that for general $V_i$, there is a unique component of $\overline{\free}^{bir}(V_i, 2l)$, and it parameterizes very free curves.  Consequently, general fibers of $\overline{\free}_1^{bir}(V_i,2l) \rightarrow V_i$ are irreducible.  This proves that for general $X$, $\overline{\free}(X, 2l-e)$ is irreducible.

Fibers of $\overline{\free}_1(X, 2l-e)\rightarrow X$ over general $p\in X$ are birational to fibers of $\overline{\free}_2(V_i, 2l) \xrightarrow{\text{ev}_1 \times \text{ev}_2} V_i \times V_i$ over $c \times p$, where $c \subset V_i$ is the blown-up curve.  The restriction of $|H|$ to a general (normal) $H_0 \in |H|$ is ample.  Thus, we may construct $(\text{ev}_1 \times \text{ev}_2)^{-1}(c \times p)$ as a subvariety of $(\text{ev}_1 \times \text{ev}_2)^{-1}(H_0 \times p)$ given by the preimage of the support of an ample divisor.  A lemma of Enriques-Severi-Zariski \cite[Exercise~III.11.3]{hartshorne2013algebraic} proves the preimage $(\text{ev}_1 \times \text{ev}_2)^{-1}(c \times p)$ is connected.  If $i > 1$, then $|H|$ has no base points.  Since $c \subset V_i$ is a general complete intersection of members of $|H|$, in this case Bertini's theorem implies general fibers of $\overline{\free}_1(X, 2l-e) \rightarrow X$ are irreducible.  If instead $i = 1$, $|H|$ has a single base point $q$.  One may still apply Bertini's theorem to conclude any two components of $(\text{ev}_1 \times \text{ev}_2)^{-1}(c \times p)$ must intersect along a point in $(\text{ev}_1 \times \text{ev}_2)^{-1}(q \times p)$.  Irreducible curves parameterized by $(\text{ev}_1 \times \text{ev}_2)^{-1}(q \times p)$ are smooth points of the fiber $(\text{ev}_1 \times \text{ev}_2)^{-1}(c \times p)$, as their strict transforms are free curves in $X$.  Similarly, any reducible curve parameterized by $(\text{ev}_1 \times \text{ev}_2)^{-1}(q \times p)$ has globally generated normal bundle, as it must be the union of a general free conic and a line with balanced normal bundle.  As such, $(\text{ev}_1 \times \text{ev}_2)^{-1}(c \times p)$ is smooth along $(\text{ev}_1 \times \text{ev}_2)^{-1}(q \times p)$.  This proves general fibers of $\overline{\free}_1(X, 2l-e) \rightarrow X$ are irreducible.

When $r = 3$, $\overline{\free}(X, l-e)$ is birational to the space of lines in a quadric $Q$ meeting the blown-up curve $c\subset Q$.  Therefore, $\overline{\free}(X, l-e)$ is dominated by an open subset of a conic bundle over $c$.  The fibers of $\overline{\free}(X, l-e)$ are reducible, but $\overline{\free}^{bir}(X, 2l-2e)$ is irreducible as it is dominated by an open subset of a $\mathbb{P}^2$-bundle over $c\times c$.  When $r = 4$, irreducibility of $\overline{\free}(X, l-2e)$ and $\overline{\free}(X, l-e)$ follows from Lemma \ref{lines_in_P3}, as does irreducibility of $\overline{\free}^{bir}(X, 2l-4e)$.  For the aforementioned classes $\alpha$ of anticanonical degree at least 3, irreducibility of the general fiber of $\overline{\free}_1^{bir}(X, \alpha) \rightarrow X$ follows from Theorem \ref{reducible fibers: 4 author result}, as when $r \geq 3$ the anticanonical divisor is very ample.  The following lemma proves each free quartic decomposes into two free conics, aside from the class of a line when $X$ is a blow-up of $\mathbb{P}^3$.  This completes our proof.
\end{proof}

\begin{lem}\label{lem: del Pezzo low degree curves}
    Suppose $\phi : X \rightarrow Y$ is the blow-up of a general complete intersection curve $c$ and $Y$ is general in moduli.  If $\alpha \in \Nef_1(X)_\mathbb{Z}$ satisfies $-K_X.\alpha = 4$ and $\alpha \notin\partial\NE(X)$, then $\overline{\free}^{bir}(X, \alpha)$ is irreducible and generically parameterizes very free curves.
\end{lem}
\begin{proof}
    Let $r$ be the index of $Y$ and set $n = r-1$. %
    We may assume $\alpha = 3(l-ne) + e$, as all other cases were addressed in the previous lemma.  To address this case, we consider $X$ as a divisor in $Y \times \mathbb{P}^1$ given by the vanishing of a general global section of $ \mathcal{L} = \pi_1^*\mathcal{O}_Y(nH) \otimes \pi_2^*\mathcal{O}_{\mathbb{P}^1}(1)$.  Let $B = |\mathcal{L}|$ and consider the universal divisor
    \[\begin{tikzcd}
	{\mathcal{X} \subset (Y \times \mathbb{P}^1) \times B} \\
	B
	\arrow["\pi", from=1-1, to=2-1]
\end{tikzcd}\]
    We may identify the relative cone of curves $\NE(\pi) \cong \NE(Y \times \mathbb{P}^1)$ with the Mori cone $\NE(X)$ using the isomorphism $\NE(\mathcal{X}) \cong \NE(Y \times \mathbb{P}^1 \times B)$.  Under this description,
    curves of class $\alpha$ on $X$ are curves of bidegree $(H.\alpha , (nH - E). \alpha) = (3,1)$ on $Y \times \mathbb{P}^1$.  The space of all such curves on $Y \times \mathbb{P}^1$ is irreducible, as $\overline{\free}^{bir}(Y, 3l)$ is itself irreducible.  We claim that $\overline{\free}(\mathcal{X}, \alpha)$ is irreducible as well.

    To see this, note that every irreducible curve of bidegree $(3,1)$ on $Y\times \mathbb{P}^1$ is embedded.  Moreover, any curve $C \subset \mathcal{X}$ corresponding to a general point of $\overline{\free}(\mathcal{X}, \alpha)$ must avoid the baselocus of $|H|$, when it is nonempty.  %
    Let $C \subset Y \times \mathbb{P}^1$ be any such curve.  It is sufficient to show the map $H^0(Y \times \mathbb{P}^1, \mathcal{L}) \rightarrow H^0(C, \mathcal{L}|_C) \cong H^0(\mathbb{P}^1, \mathcal{O}_{\mathbb{P}^1}(3n + 1))$ is surjective.  %
    This follows from a direct calculation.  Let $s,t$ be the coordinates on $C \cong \mathbb{P}^1$.  The map $H^0(Y \times \mathbb{P}^1, \pi_2^*\mathcal{O}_{\mathbb{P}^1}(1)) \rightarrow H^0(C, \mathcal{O}_C(1))$ is always surjective.  Since the image of $C$ in $Y$ is a reduced cubic curve, the cokernel of $\theta : H^0(Y \times \mathbb{P}^1, \pi_1^*\mathcal{O}_{Y}(H)) \rightarrow H^0(C, \mathcal{O}_C(H))$ has length at most one.  When $\theta$ is not surjective, by precomposing with an automorphism of $C \cong \mathbb{P}^1$, we may suppose the images of $\theta$ is spanned by either $\{s^3 + t^3, s^2t, st^2\}$ or $\{s^3, s^2t, t^3\}$.  It follows that the image of $H^0(Y \times \mathbb{P}^1, \pi_1^*\mathcal{O}_{Y}(nH)) \rightarrow H^0(C, \mathcal{O}_C(nH))$ contains the span of either $\{s^{3n} + t^{3n}, s^{3n-1}t, \ldots , st^{3n-1}\}$ or $\{s^{3n}, s^{3n-1}t, \ldots ,  s^2t^{3n-2}, t^{3n}\}$.  In both cases, $H^0(Y \times \mathbb{P}^1, \mathcal{L}) \rightarrow H^0(C, \mathcal{L}|_C) \cong H^0(\mathbb{P}^1, \mathcal{O}_{\mathbb{P}^1}(3n + 1))$ is surjective.  We conclude that $\overline{\free}(\mathcal{X}, \alpha)$ is irreducible, since a dense open subset of it has irreducible, equidimensional fibers over an irreducible variety.

    By Lemma \ref{gluing del Pezzo}, there is a unique main component of $\overline{\free}_1(X, l - ne + e) \times_X \overline{\free}_1(X, 2l-2ne)$ wherein the latter component parameterizes embedded curves.  This proves irreducibility of $\overline{\free}(X, \alpha)$ from irreducibility of $\overline{\free}(\mathcal{X}, \alpha)$ using Lemma \ref{lem: boundary stratum irreducible}.  It follows that $\overline{\free}(X, \alpha)$ contains the image of a component of $\overline{\free}_1(X, 3l-3ne) \times_X \overline{M}_{0,1}(X, e)$.  Lemma \ref{irreducible fibers del Pezzo} and Proposition \ref{very free curves} thus imply $\overline{\free}(X, \alpha)$ generically parameterizes very free curves.
\end{proof}

We will use the following lemma when studying relations between classes in $\mathscr{C}_X$.  This allows us to construct stable maps $f : C \rightarrow X$ wherein every irreducible component of $C$ is a $(-K_X)$-line and the normal bundle $\mathcal{N}_f$ is globally generated.

\begin{lem}\label{lem: transverse lines Del Pezzo fibrations}
When $r \leq 3$, there exists a component $M \subset \overline{M}_{0,1}(X, l - (r-1)e)$ such that the normal bundle of the general curve parameterized by $M$ is $\mathcal{O}\oplus \mathcal{O}(-1)$ and the preimage of $E\subset X$ in $M$ under the evaluation map is generically reduced. %
\end{lem}
\begin{proof}
When $r = 2$, every anticanonical line parameterized by $M$ meets $E$ transversely.  Our claim is trivial in these cases.  Otherwise, when $r = 3$, recall that $X$ is the blow-up of a quadric $Y \subset \mathbb{P}^4$ along a smooth complete intersection curve $c \subset Y$.  Note that $c \rightarrow \mathbb{P}^4$ is the canonical embedding of a genus 5 curve.  %
It is sufficient to show that there exists a bisecant line to $c$, contained in $Y$, which passes through a general point $p \in c$ and is nowhere tangent to $c$.  Because $X$ is Fano, every line in $Y$ meets $c$ with multiplicity at most two.  
Suppose every line in $Y$ meeting $c$ with multiplicity two is tangent to $c$.  Then the tangent hyperplane to $Y$ at $p \in c$ must only meet $c$ along $p$.  Hence, $K_c = 8p$ for any $p \in c$.  This implies an embedding of $c$ in its  Jacobian is contained in the torsion subgroup, a contradiction.
\end{proof}

As a replacement for Lemma \ref{lem: transverse lines Del Pezzo fibrations} when $r = 4$, we note that there must exist reducible conics $f : C \rightarrow X$ of class $l-2e$ whose normal bundles $\mathcal{N}_f$ satisfy $h^1(C, \mathcal{N}_f) = 0$.  By Proposition \ref{properties nonfree conics}, these reducible conics smooth to free conics.

\begin{lem}\label{del Pezzo relations}
Relations in the monoid $\mathbb{N}\mathscr{C}_X$ are generated by:
\begin{itemize}
    \item if $r = 2$:
    \begin{enumerate}
        \item $2l + (2l-2e) = 2(2l-e)$,
        \item $l + 2(2l-2e) = (3l-3e) + (2l-e)$,
        \item $l + (3l-3e) = (2l-e) + (2l-2e)$,
        \item $3(2l-2e) = 2(3l-3e)$.
    \end{enumerate}
    \item If $r = 3$:
    \begin{enumerate}
        \item $l + (2l-3e) = 3(l-e)$,
        \item $l + (2l-4e) = (l-e) + (2l-3e)$,
        \item $l + (3l-6e) = 2(l-e) + (2l-4e) = 2(2l-3e)$,
        \item $(l-e) + 2(2l-4e) = (2l-3e) + (3l-6e)$,
        \item $(l-e) + (3l-6e) = (2l -3e) + (2l-4e)$,
        \item $3(2l-4e) = 2(3l-6e)$.
    \end{enumerate}
    \item If $r=4$:
    \begin{enumerate}
        \item $l + (l-2e) = 2(l-e)$,
        \item $l + (2l-6e) = (l-e) + (2l-5e) = 3(l-2e)$,
        \item $l + (2l-5e) = (l-e) + 2(l-2e)$,
        \item $l + (3l -9e) = 2(l-2e) + (2l-5e)$,
        \item $(l-e) + (2l-6e) = (l-2e) + (2l-5e)$,
        \item $(l-e) + (3l-9e) = 2(l-2e) + (2l-6e) = 2(2l-5e)$,
        \item $(l-2e) + 2(2l-6e) = (2l-5e) + (3l-9e)$,
        \item $(l-2e) + (3l-9e) = (2l-5e) + (2l-6e)$,
        \item $3(2l-6e) = 2(3l-9e)$
    \end{enumerate}
\end{itemize}
For each relation $\sum \alpha_i = \sum \beta_j$, a main component of $\prod_X \overline{\free}_2(X,\alpha_i)$ lies in the same component of free curves as a main component of $\prod_X \overline{\free}_2(X, \beta_j)$.
\end{lem}
\begin{proof}
We use Lemma \ref{Relations}.  %
The analysis of each case may be whittled down to relations between four generators.  No relation may therefore involve more than two distinct curve classes on either side.  Note that the last four relations for each index are of the same form.  The last relation involves curves contracted by the del Pezzo fibration $\pi : X \rightarrow \mathbb{P}^1$.  Our claim concerning main components follows from \ref{del pezzo curves thm}.  The second (resp. third) to last relation involves a free chain $(\alpha, \beta)$ where $\alpha$ is a section of $\pi$ and $\beta$ is a cubic (resp. quartic) in its fiber.  Such relations were handled in \cite{sectionsDelPezzo}.  The fourth to last relation involves a free chain $(\alpha, \beta, \alpha)$ where $\alpha$ is a section of $\pi$ and $\beta$ is a conic in its fiber.  We may begin with a degenerate conic $\beta$ and smooth to obtain a free chain of type $(\alpha + \frac{\beta}{2}, \alpha + \frac{\beta}{2})$.  Alternatively, if $r \geq 3$, we may begin with an immersed chain of type $(\alpha, e, l - (r-1)e, \beta)$ and smooth to free chains of type $(\alpha, \alpha, \beta)$ and $(\alpha + e, \frac{3\beta}{2})$.  The fifth to last relation when $r = 4$ and $r=3$ are also analogous: in each case, the right-hand-side is a sum of two sections of $\pi$ of $(-K_X)$-degree $2$ and $3$.  By beginning with a degenerate section of degree 3, expressed as an exceptional curve $e$ and free conic contracted by $\pi$, we may smooth to free chains of either desired type.  %
Lastly, when $r = 3$, by Lemma \ref{lem: transverse lines Del Pezzo fibrations} there are immersed chain with 6 irreducible components, whose classes alternate between exceptional curves $e$ and $-K_{X/\mathbb{P}^1}$ lines $l - 2e$.  Such chains smooth to free chains of either type in the first relation, proving our claim.  

Suppose $r = 4$.  We will consider relations (1) through (4).  Relation (1) follows from Lemma \ref{lines_in_P3}.  For relation (2), we claim that $\free^{bir}(X, 3l-6e)$ is irreducible.  Indeed, the general map $f : C \rightarrow X$ parameterized by any component of $\free^{bir}(X, 3l-6e)$ must be the strict transform of a twisted cubic by Theorem \ref{reducible fibers: 4 author result}, Proposition \ref{very free curves}, and Proposition \ref{properties very free curves}.  Irreducibility of $\free^{bir}(X, 3l-6e)$ then follows from the existence of a unique twisted cubic through six general points in the complete intersection curve $c \subset \mathbb{P}^3$.  
Next, by smoothing a chain of type $(l-e, e, \alpha)$ in two different ways, we prove our claim for relation (3).  Similarly, irreducibility of $\overline{\free}^{bir}(X, 3l-6e)$ proves a component of $\overline{\free}_1^{bir}(X, 3l-6e) \times_X \overline{M}_{0,1}(X, l-3e)$ contains chains of type $(l, 2l-6e, l-3e)$ and $(l-2e, l-2e, l-2e, l-3e)$ in its smooth locus.  Smoothing each chain implies our claim for relation (4).
\end{proof}

\begin{lem}\label{lem: del pezzo gordon}
For each nonzero $\alpha \in \Nef_1(X)_\mathbb{Z}$, $\overline{\free}(X,\alpha)$ is nonempty unless $\alpha = (l-(r-1)e)$.
\end{lem}

\begin{proof}
The generators of $\Nef_1(X)$ are elements in $\mathscr{C}_X$ by \ref{Representability of Free Curves}.  It follows from \ref{Gordan's Lemma} that classes in $\mathscr{C}_X \cup \{l-(r-1)e\}$ generate the monoid of integer points in $\Nef_1(X)$.
\end{proof}

\subsubsection*{Conic Fibrations}
Suppose $\phi : X \rightarrow Y$ is the blow-up of the baselocus of a net $L$ of members of $|(r-1)H|$.  One may see the proof of Lemma \ref{lines in C1 fibrations} for an explicit description of such varieties.  In this case, $\pi : X \rightarrow \mathbb{P}(L) = \mathbb{P}^2$ is a conic fibration.  %

\begin{thm}\label{conic fibration GMC}
Let $X$ be the blow-up of a Fano threefold of index $r > 1$ along the baselocus of a net $L$ of members of $|(r-1)H|$.  Then for all nonzero $\alpha \in \Nef_1(X)_\mathbb{Z}$, $\overline{\free}^{bir}(X,\alpha)$ is irreducible.  $\overline{\free}^{bir}(X,\alpha)$ is nonempty if $\alpha \neq n(l-(r-1)e)$.  $\overline{\free}(X, \alpha)$ is irreducible unless $\alpha = n(l-(r-2)e)$ for some $n > 1$, in which case there are exactly two components.
\end{thm}

\begin{lem}\label{conic fibration Core}
Let $X$ be the blow-up of a Fano threefold of index $r > 1$ along the baselocus of a net $L$ of members of $|(r-1)H|$.  A core of free curves on $X$ is given by %
\begin{itemize}
    \item $r = 2$, $\mathscr{C}_X = \{ 2l -2e, l, 2l-e\}$,
    \item $r = 3$, $\mathscr{C}_X = \{ 2l -4e, l -e, 2l -3e, l\}$
    \item $r = 4$, $\mathscr{C}_X = \{ 2l -6e, l -2e, 2l -5e, l -e, l \}$.
\end{itemize}
The only separating class for each $X$ is the unique conic $\alpha \in \mathscr{C}_X$ interior to $\NE(X)$.
\end{lem}
\begin{proof}
$\overline{\free}(X,l)$ is shown to be irreducible in \cite{2019}.  Let $n = r-1$. Since $\pi : X \rightarrow \mathbb{P}^2$ is a conic fibration contracting $l-ne$, $\overline{\free}(X, 2(l-ne))$ is irreducible with irreducible fibers and $\overline{\free}(X, 3(l-ne))$ is empty.  Similarly, it is clear that $\overline{\free}(X, 4(l-ne))$ is irreducible and generically parameterizes double covers.  General curves of class $3(l-ne) + e$ map to general lines in $\mathbb{P}(L)$ under $\pi$.  The preimage of such a curve in $X$ is a smooth del Pezzo surface $S$ of degree $(r-1)H^3$.  Curves of class $3(l-ne) + e$ correspond to rational cubic curves (they must be very free, and therefore embedded) which are sections of the conic fibration $\pi|_S$.  These cubics always deform to the union of a smooth fiber of $\pi|_S$ and a $-K_S$-line.  The $-K_S$-line must be a section of $\pi|_S$ of class $(l-ne) + e$, which is generically free.  This decomposes the curve of class $3(l-ne) +e$ into a free chain of type $(2(l-ne), (l-ne) + e)$.  Since $-K_X$ is very ample, irreducibility of general fibers of $\overline{\free}_1(X, \alpha) \rightarrow X$ for all $\alpha \in \mathscr{C}_X\setminus \{(l-ne) + e\}$ follows from irreducibility of the space itself.  We prove the remaining cases below.

Suppose $r = 4$.  $\overline{\free}(X, l-2e)$ is birational to $\text{Sym}^2(c)$, where $c\subset \mathbb{P}^3$ is the blown-up curve.  Similarly, $\overline{\free}(X, l-e)$ may be fibered over $c$.  Theorem \ref{Monodromy} proves $\overline{\free}(X, 2l-5e)$ is irreducible.  The only other nef class of anticanonical degree at most 4 is $2l-4e$.  Lemma \ref{lines_in_P3} shows $\overline{\free}^{bir}(X, 2l-4e)$ has a unique component.

Suppose $r = 3$.  Then $\overline{\free}(X, l-e)$ is irreducible, as there is a quadric cone of lines thorough any point of $Q \subset \mathbb{P}^4$.  $\overline{\free}(X, 2l-3e)$ is birational to $\text{Sym}^3(c)$, where $c \subset Q$ is the blown-up curve.  Similarly, $\overline{\free}(X, 2l-2e)$ has two components: one parameterizes double covers of curves of class $l-e$; the other parameterizes embedded curves, and is birational to a $\mathbb{P}^2$-bundle over $\text{Sym}^2(c)$.

Lastly, suppose $r = 2$.  Irreducibility of $\overline{\free}^{bir}(X, 2l)$ follows from \cite{2019}.  Irreducibility of $\overline{\free}(X, 2l-e)$ follows casewise.  If $X$ is a blow-up of a line $c \subset V_3$, curves of class $2l-e$ are strict transforms of conics meeting $c$ once.  The space of such conics is birational to $c \times \overline{\free}(X, l)$.  %
If $X$ is a blow-up of $V_4$ (resp. $V_5$) along a conic (resp. twisted cubic) $c$, then by \cite[Theorem~7.2]{2019} there are at most finitely many points on $c \subset V_i$ for which the fibers of $\overline{\free}_1^{bir}(V_i, 2l) \rightarrow V_i$ are of dimension $3$.  Similarly, there are finitely many points on $c$ that meet cones of non-free lines in $V_i$, and by extension finitely many points of $c$ meet lines which pass through cones of non-free lines at their vertices.  Away from this subset of $c$, the fibers of $\overline{\free}_1^{bir}(V_i, 2l) \rightarrow V_i$ are of dimension 2 over $c$, connected by \cite{2019}, and nonsingular in codimension 1.  It follows that each fiber is irreducible.  This constructs $\overline{\free}(X, 2l-e)$ as the collection of such fibers.  This completes our proof.
\end{proof}

\begin{lem}
Let $X$ be the blow-up of a Fano threefold of index $r > 1$ along the baselocus of a net $L$ of members of $|(r-1)H|$.  Relations in the monoid $\mathbb{N}\mathscr{C}_X$ are generated by:
\begin{itemize}
    \item if $r = 2$:
    \begin{enumerate}
        \item $2l + (2l-2e) + 2(2l-e)$.
    \end{enumerate}
    \item If $r = 3$:
    \begin{enumerate}
        \item $l + (2l-3e) = 3(l-e)$,
        \item $l + (2l-4e) = (l-e) + (2l-3e)$,
        \item $2(l-e) + (2l-4e) = 2(2l-3e)$.
    \end{enumerate}
    \item If $r=4$:
    \begin{enumerate}
        \item $l + (l-2e) = 2(l-e)$,
        \item $l + (2l-6e) = (l-e) + (2l-5e) = 3(l-2e)$,
        \item $l + (2l-5e) = (l-e) + 2(l-2e)$,
        \item $(l-e) + (2l-6e) = (l-2e) + (2l-5e)$,
        \item $2(l-2e) + (2l-6e) = 2(2l-5e)$.
    \end{enumerate}
\end{itemize}
For each relation $\sum \alpha_i = \sum \beta_j$, a main component of $\prod_X \overline{\free}_2(X,\alpha_i)$ lies in the same component of free curves as a main component of $\prod_X \overline{\free}_2(X, \beta_j)$.
\end{lem}
\begin{proof}
We use Lemma \ref{Relations}.  Identifying a generating set of relations follows numerically from Lemma \ref{del Pezzo relations}.  Let $n = r-1$.  By smoothing an immersed chain of type $(l-ne + e, l-ne, l-ne, l-ne + e)$ we may obtain free chains of type $(l-ne + e, 2(l-ne), l-ne + e)$ and of type $(2(l-ne) + e, 2(l-ne) + e)$.  Similarly, if $r \geq 3$, we may smooth immersed chains of type $(l-ne + e, e, 2(l-ne))$ to obtain free chains of type $(l-ne + 2e, 2(l-ne))$ and $(l-ne + e, 2(l-ne) + e)$.  Similarly, we may smooth immersed chains of type $(l-ne + e, e, l-ne + e, l-ne)$ into free chains of type $(l-ne + 2e, 2(l-ne) + e)$ and $(l-ne + e, 2(l-ne + e))$, the latter of which which break into chains of type $(l-ne + e, l-ne + e, l-ne + e)$.  This addresses each relation aside from, when $r = 4$, (1), the first equality in (2), and (3).  Relation (1) follows from Lemma \ref{lines_in_P3}.  By smoothing immersed chains of type $(l-e, e, 2l-6e)$, we obtain free chains of type $(l, 2l-6e)$ and $(l-e, 2l-5e)$.  Similarly, by smoothing immersed chains of type $(l-e, e, 2l-5e)$, we obtain chains of type $(l, 2l-5e)$ and $(l-e, 2l-4e)$.  Breaking the latter of these into a chain of free curves with three components using \ref{Basic Properties Chains} concludes our proof.
\end{proof}

\begin{lem}
Let $X$ be the blow-up of a Fano threefold of index $r > 1$ along the baselocus of a net $L$ of members of $|(r-1)H|$.  For each nonzero $\alpha \in \Nef_1(X)_\mathbb{Z}$, $\overline{\free}(X,\alpha)$ is nonempty unless $\alpha = m(l-(r-1)e)$ for odd $m > 1$.
\end{lem}

\begin{proof}
The generators of $\Nef_1(X)$ are elements in $\mathscr{C}_X$ by \ref{Representability of Free Curves}.  It follows from \ref{Gordan's Lemma} that classes in $\mathscr{C}_X \cup \{l-(r-1)e\}$ generate the monoid of integer points in $\Nef_1(X)$.
\end{proof}

\subsection*{2.2}
\textbf{A double cover of $\mathbb{P}^1 \times \mathbb{P}^2$ branched over a divisor of bidegree $(2,4)$:} Let $\pi : X \rightarrow \mathbb{P}^1 \times \mathbb{P}^2$ be a double cover branched over a smooth divisor $B$ of bidegree $(2,4)$.  Let $\pi_i : X \rightarrow \mathbb{P}^i$ be the composition of $\pi$ with the $i^{th}$ projection.

\textbf{Generators for $N^1(X)$ and $N_1(X)$:} Let $H_i = \pi_i^*\mathcal{O}(1)$ and $l_i$ be the class of a $(-K_X)$-line contracted by $\pi_j$.

\begin{thm}\label{2.2thm}
For all $\alpha \in (l_1 + l_2) +  \Nef_1(X)_\mathbb{Z}$, $\overline{\free}^{bir}(X,\alpha)$ is nonempty and irreducible.
\end{thm}

\textbf{Intersection Pairing:} The intersection pairing is diagonal: $H_i \cdot l_j = \delta_{ij}$.

\textbf{Anticanonical Divisor:} $-K_X = H_1 + H_2$.

\textbf{Generators of $\NE(X)$ and $\text{Eff}(X)$:}  The extreme rays of $\NE(X)$ are spanned by $l_1$ and $l_2$.  These are the only classes of effective $-K_X$ lines.  The extreme rays of $\text{Eff}(X)$ are spanned by $H_1$ and $H_2$.  The contraction of $l_2$ is a $D1$ fibration $\pi_1 : X\rightarrow \mathbb{P}^1$.  By \cite{V2monodromy} and \cite{kollarLefschetz}, for general $X$ the monodromy of $\pi_1$ is the full Weyl group $W(E_7)$.  As well, $\overline{M}_{0,0}(X, l_2)$ is irreducible for general $X$.  The contraction of $l_1$ is the conic bundle $\pi_2 : X \rightarrow \mathbb{P}^2$ with discriminant locus $\Delta \subset \mathbb{P}^2$ of degree 8.  For general $X$, $\Delta$ is smooth and the preimage $\pi_2^{-1}(\Delta)$ is tangent to finitely many lines of type $l_2$.

\begin{lem}\label{lem: 2.2 low degree curves}
    Suppose $X$ is general in moduli.  If $\alpha \in \Nef_1(X)_\mathbb{Z}$ satisfies $\alpha \not\in\partial\NE(X)$ and $2 \leq -K_X . \alpha \leq 4$, $\overline{\free}^{bir}(X, \alpha)$ is irreducible.  When $-K_X . \alpha = 4$, $\overline{\free}^{bir}(X, \alpha)$ generically parameterizes very free curves.
\end{lem}
\begin{proof}
    We may write $\alpha = al_1 + bl_2$ for $a,b > 0$ with $2 \leq a + b \leq 4$.  When $b = 1$, Corollary \ref{cor: conic monodromy cor} shows $\overline{\free}(X, al_1 + l_2)$ is irreducible.  %
    Suppose $b \geq 2$ instead.  We will show $\overline{\free}^{bir}(X, al_1 + bl_2)$ is irreducible using Lemma \ref{lem: boundary stratum irreducible}.  To apply Lemma \ref{lem: boundary stratum irreducible}, consider the embedding $\mathbb{P}^1 \times \mathbb{P}^2 \rightarrow \mathbb{P}^{N}$ induced by the complete linear system $|\mathcal{O}_{\mathbb{P}^1\times\mathbb{P}^2}(1,2)|$, and let $Y \subset \mathbb{P}^{N+1}$ be the cone over its image.  The intersection of $Y$ with a general quadric hypersurface in $\mathbb{P}^{N+1}$ is a double cover of $\mathbb{P}^1 \times \mathbb{P}^2$ branched over a general divisor of bidegree $(2,4)$.  Let $B \subset |\mathcal{O}_Y(2)|$ be the open locus of quadric hypersurfaces avoiding the singular point of $Y$.  We consider $X$ as a general fiber of the universal quadric hypersurface $\mathcal{X}$: %
    \[\begin{tikzcd}
	{\mathcal{X} \subset Y \times B} \\
	B
	\arrow["\phi", from=1-1, to=2-1]
\end{tikzcd}\]
    Note that the relative Mori cone $\NE(\phi)$ of the contraction $\phi : \mathcal{X} \rightarrow B$ is canonically isomorphic to both $\NE(X)$ and $\NE(\mathcal{X})$ as $B$ is affine.  We will show $\overline{\free}^{bir}(\mathcal{X}, al_1 + bl_2)$ is irreducible.  Given this, we study the boundary stratum of $\overline{\free}^{bir}(\mathcal{X}, al_1 + bl_2)$ obtained by gluing an embedded free curve of class $2l_2$ to a general curve of class $\alpha - 2l_2 = al_1 + (b-2)l_2$.  By generality of $X$, a unique component $M_1 \subset \overline{\free}_1(X, 2l_2)$ generically parameterizes embedded curves.  Since $0 \leq b - 2 \leq 1$ and by generality of $X$, exactly one of $\overline{\free}_1(X, \alpha - 2l_2)$ and $\overline{M}_{0,1}(X, \alpha - 2l_2)$ is irreducible and nonempty; we let $M_2$ be this irreducible variety.  It follows from Lemma \ref{gluing del Pezzo} and Remark \ref{big divisor remark} that $M_1 \times_X M_2$ has a unique component.  As this component meets the smooth locus of $\overline{\free}(X, \alpha)$, by Lemma \ref{lem: boundary stratum irreducible} we conclude $\overline{\free}^{bir}(X, \alpha)$ is irreducible. %

    It remains to show $\overline{\free}^{bir}(\mathcal{X}, \alpha)$ is irreducible.  Let $f : C \rightarrow X \subset Y$ be a map corresponding to a general point in some component of $\free^{bir}(\mathcal{X}, \alpha)$.  We will show the composition $\pi \circ f$ of $f$ with the double cover map $\pi : X \rightarrow \mathbb{P}^1 \times \mathbb{P}^2$ is an embedding and that the natural map 
    $$\psi : H^0(\mathbb{P}^1 \times \mathbb{P}^2, \mathcal{O}(1,2)) \rightarrow H^0(C, (\pi \circ f)^*\mathcal{O}(1,2))$$ 
    is a surjection.  This implies the natural map $H^0(Y, \mathcal{O}_Y(2)) \rightarrow H^0(C, f^*\mathcal{O}_Y(2))$ is surjective as well.  Since the locus of curves in $Y^{sm}$ mapping under projection from the vertex of $Y$ to embedded curves of the appropriate class in $\mathbb{P}^1 \times \mathbb{P}^2$ form an irreducible family, this will prove irreducibility of $\overline{\free}^{bir}(\mathcal{X}, \alpha)$.

    Recall that $\alpha = al_1 + bl_2 \in \{l_1 + 2l_2, l_1 + 3l_2, 2l_1 + 2l_2\}$.  When $a = 1$, $\pi \circ f : C \rightarrow \mathbb{P}^1 \times \mathbb{P}^2$ is necessarily an embedding by irreducibility of $C$.  When $-K_X . \alpha = a + b = 4$, we claim that the composition $\pi_2 \circ f : C \rightarrow \mathbb{P}^2$ is also birational onto its image.  Otherwise, by Lemma \ref{3.1 lemma} the image of $\pi_2 \circ f$ would be a general line $\ell \subset \mathbb{P}^2$.  Deformations of $f$ contained in the divisor $S = \pi_2^{-1}(\ell) \subset X$ would form a covering family of dimension 2 by generality of $f$.  However, since $S$ is a smooth divisor with anticanonical class $-K_S = H_1|_S$, this is larger than the expected dimension $-K_S . [f(C)] - 1 = a - 1$, a contradiction.  Thus $\pi_2 \circ f$ is birational onto its image.

    Let $s,t$ be coordinates on $C \cong \mathbb{P}^1$.  We will demonstrate surjectivity of $\psi$ via direct calculation.  Without loss of generality, we may suppose $\pi \circ f : C \rightarrow \mathbb{P}^1 \times \mathbb{P}^2$ is given in coordinates by $[s: t] \rightarrow [s^a : t^a] \times [f_{0} : f_{1} : f_{2}]$.   The image of $\psi$ is the linear span of generators of the product ideal $I = (s^a, t^a)(f_0, f_1, f_2)^2$.
    
    When $\alpha = l_1 + 2l_2$, either $\pi_2 \circ f : C \rightarrow \mathbb{P}^2$ is an embedding or we may suppose $f_0 = s^2, f_1 = t^2$, and $f_2 = 0$.  In both cases, $\psi$ is evidently surjective.  A dimension count shows that for a general map $f : C \rightarrow X$ parameterized by $\overline{\free}(\mathcal{X}, l_1 + 2l_2)$, $\pi_2 \circ f$ is an embedding.  We will later use this fact.

    Suppose $\alpha = 2l_1 + 2l_2$ instead.  Since $\pi_2 \circ f$ is birational onto its image, we may suppose $[f_{0} : f_{1} : f_{2}] = [s^2 : st : t^2]$.  As $a = 2$, the product ideal $I$ contains all monomials of the form $s^n t^{6-n}$.  Similarly, if $\alpha = l_1 + 3l_2$, we may suppose either $[f_{0} : f_{1} : f_{2}] = [s^3 + t^3 : s^2t : st^2]$ or $[f_{0} : f_{1} : f_{2}] = [s^3 : s^2t : t^3]$.  In both cases, the corresponding product ideal $I$ would contain all monomials of the form $s^nt^{7-n}$.  We conclude $\psi$ is surjective, so that $\overline{\free}^{bir}(\mathcal{X}, \alpha)$ is irreducible. %

    Having shown that $\overline{\free}^{bir}(X, \alpha)$ is irreducible, we prove it generically parameterizes very free curves when $-K_X . \alpha = 4$.  Consider embedded nodal curves $C = C_1 \cup C_2$ with two irreducible rational components $C_1, C_2$ of class $l_1$ and $\alpha - l_1$, respectively. We previously demonstrated that $\overline{\free}_1(X, \alpha - l_1)$ has a component $M$ which generically parameterizes embeddings $f$ such that $\pi_2 \circ f : C_2 \rightarrow \mathbb{P}^2$ is birational onto its image.  Let $C$ be a general curve parameterized by $\overline{M}_{0,1}(X, l_1) \times_X M$.  We claim $C$ generalizes to a very free curve.  To see this, let $D_l$ be the divisor swept out by curves of class $l_1$.  There is a one-parameter family of curves $M_p \subset M$ containing a general point $p \in X$.  The restriction of $\pi_2$ to the divisor $D_p$ covered by maps in $M_p$ is dominant, as $\pi_2 \circ f$ is birational onto its image for the general map $f$ parameterized by $M_p$.  Hence, $D_p . l_1 > 0$ and $D_l . [C_2] > 0$.  Since $C_2$ is a general free curve, $C_2 \cap D_l$ is reduced.  Therefore, $D_p \cap D_l$ is generically reduced, and $N_{C/X}(-p)$ is globally generated. 
\end{proof}

\begin{lem}
Every component of $\overline{\free}(X)$ contains free chains with components of class in $\mathscr{C}_X$, where
\begin{align*}
   \mathscr{C}_X = &\{ 2l_1, 2l_2, l_1 + l_2, 3l_2, 2l_1 + l_2, l_1 + 2l_2\}
\end{align*}
For $\alpha \in \mathscr{C}_X$ such that $\alpha . H_1 > 0$, $\overline{\free}(X, \alpha)$ is irreducible.  Moreover, $\overline{\free}^{bir}(X, 2l_1 + 2l_2)$ is irreducible.  When $X$ is general in moduli, $\overline{\free}(X,\alpha)$ has two components if $\alpha\in \{2l_2, 3l_2\}$. %
\end{lem}
\begin{proof}
\textbf{Nef Curve Classes of Anticanonical Degree Between $2$ and $4$:}
The nef classes of anticanonical degree between $2$ and $4$ are those in $\mathscr{C}_X$ along with $3l_1$, $4l_i$, $3l_i + l_j$, and $2l_1 + 2l_2$.  Free curves of class $4l_i$ may be broken in smooth fibers of $\pi_j$, while there are no free curves of class $3l_1$.  $\overline{\free}(X, 2l_1)$ parameterizes fibers of $\pi_2 : X \rightarrow \mathbb{P}^2$, and as such is irreducible.    When $X$ is general in moduli, our claims concerning $\alpha = 2l_2$ and $\alpha = 3l_2$ follow from the $W(E_7)$-monodromy of $\pi_1$.  Similarly, Lemma \ref{lem: 2.2 low degree curves} proves irreducibility of $\overline{\free}^{bir}(X, \alpha)$ for all other integral classes $\alpha$ with $2 \leq -K_X . \alpha \leq 4$.  This specializes to arbitrary $X$ by Proposition \ref{low degree curves degeneration} and Theorem \ref{thm: conics}.  In particular, this shows when $-K_X . \alpha = 4$ that $\overline{\free}^{bir}(X, \alpha)$ contains chains of free conics.
\end{proof}

\begin{lem}\label{2.2rels}
Relations in the monoid $\mathbb{N}\mathscr{C}_X$ are generated by:
\begin{enumerate}
    \item $2(3l_2) = 3(2l_2)$,
    \item $3l_2 + (l_1 + l_2) = 2l_2 + (l_1 + 2l_2)$,
    \item $3l_2 + (2l_1 + l_2) = 2l_1 + 2(2l_2)$, %
    \item $3l_2 + 2l_1 = 2l_2 + (2l_1 + l_2)$, %
    \item $3l_2 + (l_1 + 2l_2) = 2(2l_2) + (l_1 + l_2)$,
    \item $2l_1 + 2l_2 = 2(l_1 + l_2)$,
    \item $2l_i + (l_i + 2l_j) = (l_1 + l_2) + (2l_i + l_j)$,
    \item $2l_i + 2(l_1 + l_2) = 2(2l_i + l_j)$,
    \item $3(l_1 + l_2) = (2l_1 + l_2) + (l_1 + 2l_2)$,
\end{enumerate}
for $i \neq j$.  For each relation $\sum \alpha_i = \sum \beta_j$, a main component of $\prod_X \overline{\free}_2(X,\alpha_i)$ lies in the same component of free curves as a main component of $\prod_X \overline{\free}_2(X, \beta_j)$.
\end{lem}
\begin{proof}
\textbf{Relations:}  Note that each relation involving $3l_2$ may be simplified using (1)-(5) to a relation without $3l_2$.  We then follow Lemma \ref{Relations}.  The divisor $-H_1 + 2H_2$ pairs to $-2$ with $2l_1$, to $0$ with $2l_1 + l_2$, to $1$ with $l_1 + l_2$, and positively with all other curves.  This gives relations (6)-(8).  We may remove $2l_2$ from any relation using (6)-(8) as well.  The remaining three elements of $\mathscr{C}_X$ satisfy (9).

\textbf{Main Components:} We address each relation below.
\begin{enumerate}
    \item Our claim follows from smoothing an immersed chain of type $(2l_2, l_2, l_2, 2l_2)$ in two separate ways.
    \item Smooth an immersed chain of type $(2l_2, l_2, l_1 + l_2)$ in two separate ways.
    \item Smooth an immersed chain of type $(2l_2, l_2, l_2, 2l_1)$ in two separate ways.
    \item Smooth an immersed chain of type $(2l_2, l_2, 2l_1)$ in two separate ways.
    \item Smooth an immersed chain of type $(2l_2, l_2, l_2, l_1 + l_2)$ in two separate ways.
    \item This follows from the irreducibility of $\overline{\free}^{bir}(X, 2l_1 + 2l_2)$, proven above.
    \item This follows from smoothing an immersed chain of type $(2l_i, l_j, l_1 + l_2)$ in two separate ways.
    \item Our claim follows from smoothing an immersed chain of type $(l_1 + l_2, l_i, l_i, l_1 + l_2)$ in two separate ways.
    \item We may smooth an immersed chain of type $(l_1 + l_2, l_1 + l_2, l_1 + l_2)$ to a chain of type $(2l_1 + 2l_2, l_1 + l_2)$.  Since $\overline{\free}_1^{bir}(X, 2l_1 + 2l_2) \times_X \overline{\free}_1(X, l_1 + l_2)$ has a unique main component $M$, the smooth locus of $M$ contains an immersed chain of type $(2l_1 + l_2, l_2, l_1 + l_2)$.  Smoothing this chain implies our claim.
\end{enumerate}
\end{proof}

\begin{lem}
For each $\alpha \in (l_1 + l_2) + \Nef_1(X)_\mathbb{Z}$, $\overline{\free}(X,\alpha)$ is nonempty.
\end{lem}

\begin{proof}
The generators of $\Nef_1(X)$ are $l_1$ and $l_2$.  Our claim follows immediately.
\end{proof}

\subsection*{2.6}
\textbf{A Verra Threefold:} Let $X$ be a smooth divisor of bidegree $(2,2)$ in $\mathbb{P}^2 \times \mathbb{P}^2$, or double cover of $\mathbb{P}_{\mathbb{P}^2}(\mathcal{T}_{\mathbb{P}^2}) \subset \mathbb{P}^2 \times \mathbb{P}^2$ branched over a smooth anticanonical divisor.  The latter type is a specialization of the former.  Let $\pi_i : X \rightarrow \mathbb{P}^2$ be the restriction of the $i^{th}$ projection.

\textbf{Generators for $N^1(X)$ and $N_1(X)$:} Let $H_i = \pi_i^*\mathcal{O}(1)$ and $l_i$ be the class of an irreducible component of a reducible fiber of $\pi_j$.

\begin{thm}\label{2.6thm}
For all $\alpha \in (l_1 + l_2) +  \Nef_1(X)_\mathbb{Z}$, $\overline{\free}^{bir}(X,\alpha)$ is nonempty and irreducible.
\end{thm}

\textbf{Intersection Pairing:} The intersection pairing is diagonal: $H_i \cdot l_j = \delta_{ij}$.

\textbf{Anticanonical Divisor:} $-K_X = H_1 + H_2$.

\textbf{Generators of $\NE(X)$ and $\text{Eff}(X)$:}  The extreme rays of $\NE(X)$ are spanned by $l_1$ and $l_2$.  These are the only classes of effective $-K_X$ lines.  The extreme rays of $\text{Eff}(X)$ are spanned by $H_1$ and $H_2$.  The contraction of $l_i$ is the conic bundle $\pi_j : X \rightarrow \mathbb{P}^2$ with discriminant locus $\Delta_j \subset \mathbb{P}^2$ of degree 6.  For general $X$, $\Delta_j$ is smooth.

\begin{lem}
A core of free curves on $X$ is %
\begin{align*}
   \mathscr{C}_X = &\{ 2l_1, 2l_2, l_1 + l_2, 2l_1 + l_2, l_1 + 2l_2\}
\end{align*}
The only separating class $\alpha \in \mathscr{C}_X$ is $l_1 + l_2$.  Moreover, $\overline{\free}^{bir}(X, 2l_1 + 2l_2)$ is irreducible.
\end{lem}
\begin{proof}
\textbf{Nef Curve Classes of Anticanonical Degree Between $2$ and $4$:}
The nef classes of anticanonical degree between $2$ and $4$ are those in $\mathscr{C}_X$ along with $3l_i$, $4l_i$, $3l_i + l_j$, and $2l_1 + 2l_2$.  Irreducibility of $\overline{\free}(X, \alpha)$ for $\alpha \in \mathscr{C}_X$ follows from Corollary \ref{cor: conic monodromy cor}, Proposition \ref{low degree curves degeneration}, and Theorem \ref{thm: conics}.  Free curves of class $4l_i$ are always double covers of smooth fibers of $\pi_j$, while there are no free curves of class $3l_i$ for similar reasons.  Clearly, multiple covers of free curves of class $l_1 + l_2$ are freely breakable.

The conic fibrations $\pi_1$ and $\pi_2$ could be used to prove every other component of free curves of degree four contain chains of free conics; however, there is a cleaner argument which also proves irreducibility of these spaces.  For $\alpha \in \{3l_i + l_j, 2l_1 + 2l_2\}$, every component of $\overline{\free}^{bir}(X, \alpha)$ must generically parameterize very free curves.  Hence, a general map $f: \mathbb{P}^1 \rightarrow C \subset X$ parameterized by any component of $\overline{\free}^{bir}(X, \alpha)$ is an embedding which maps to an at-worst nodal curve under both $\pi_1$ and $\pi_2$.  There is a unique main component of both $\overline{\free}_1(X, 2l_1) \times_X \overline{\free}_1(X, 2l_2)$ and $\overline{\free}_1(X, 2l_i) \times_X \overline{\free}_1(X, l_1 + l_2)$.  When $X$ is general in moduli, by Lemma \ref{lem: boundary stratum irreducible} it suffices to show for the image $C \subset \mathbb{P}^2 \times \mathbb{P}^2$ of the aforementioned general map, $H^0(\mathbb{P}^2 \times \mathbb{P}^2, \mathcal{O}(2,2)) \rightarrow H^0(C, \mathcal{O}_{\mathbb{P}^2 \times \mathbb{P}^2}(2,2)|_C)$, is surjective.  This claim is clear when $\pi_i(C)$ is a smooth conic for each $i$; when $C$ projects to a nodal cubic in one factor, one may observe this by fixing the node as the image of $0$ and $\infty$.  By Proposition \ref{low degree curves degeneration}, irreducibility of $\overline{\free}^{bir}(X, \alpha)$ holds for arbitrary $X$ as well.
\end{proof}

\begin{lem}
Relations in the monoid $\mathbb{N}\mathscr{C}_X$ are generated by:
\begin{enumerate}
    \item $2l_1 + 2l_2 = 2(l_1 + l_2)$,
    \item $2l_i + (l_i + 2l_j) = (l_1 + l_2) + (2l_i + l_j)$,
    \item $2l_i + 2(l_1 + l_2) = 2(2l_i + l_j)$,
    \item $3(l_1 + l_2) = (2l_1 + l_2) + (l_1 + 2l_2)$,
\end{enumerate}
for $i \neq j$.  For each relation $\sum \alpha_i = \sum \beta_j$, a main component of $\prod_X \overline{\free}_2(X,\alpha_i)$ lies in the same component of free curves as a main component of $\prod_X \overline{\free}_2(X, \beta_j)$.
\end{lem}
\begin{proof}
\textbf{Relations:}
We follow Lemma \ref{Relations}.  The divisor $-H_1 + 2H_2$ pairs to $-2$ with $2l_1$, to $0$ with $2l_1 + l_2$, to $1$ with $l_1 + l_2$, and positively with all other curves.  This gives relations (1)-(3).  We may remove $2l_2$ from any relation using (1)-(3) as well.  The remaining three elements of $\mathscr{C}_X$ satisfy (4).

\textbf{Main Components:} We address each relation below.
\begin{enumerate}
    \item This follows from the irreducibility of $\overline{\free}^{bir}(X, 2l_1 + 2l_2)$, proven above.
    \item This follows from smoothing an immersed chain of type $(2l_i, l_j, l_1 + l_2)$ in two separate ways.
    \item Our claim follows from smoothing an immersed chain of type $(l_1 + l_2, l_i, l_i, l_1 + l_2)$ in two separate ways.
    \item We may smooth an immersed chain of type $(l_1 + l_2, l_1 + l_2, l_1 + l_2)$ to a chain of type $(2l_1 + 2l_2, l_1 + l_2)$.  Since $\overline{\free}_1^{bir}(X, 2l_1 + 2l_2) \times_X \overline{\free}_1(X, l_1 + l_2)$ has a unique main component $M$, the smooth locus of $M$ contains an immersed chain of type $(2l_1 + l_2, l_2, l_1 + l_2)$.  Smoothing this chain implies our claim.
\end{enumerate}
\end{proof}

\begin{lem}
For each $\alpha \in (l_1 + l_2) + \Nef_1(X)_\mathbb{Z}$, $\overline{\free}(X,\alpha)$ is nonempty.
\end{lem}

\begin{proof}
The generators of $\Nef_1(X)$ are $l_1$ and $l_2$.  Our claim follows immediately.
\end{proof}

\subsection*{2.8}

\textbf{A double cover of $V_7$:} Let $\phi : X \rightarrow V_7$ be the double cover of $V_7$ branched over a smooth member $B$ of $|-K_{V_7}|$.  For simplicity, we assume $B$ intersects the exceptional locus of $V_7 \rightarrow \mathbb{P}^3$ in a smooth curve.  However, our claims do not depend on this assumption.

\textbf{Generators for $N^1(X)$ and $N_1(X)$:}  Let $H$ be the pullback a hyperplane class under $X \rightarrow V_7 \rightarrow \mathbb{P}^3$ and $E$ be the preimage of the exceptional divisor of $V_7 \rightarrow \mathbb{P}^3$.  Let $l$ be the class of a component of the reducible preimage of an $H$-line, and $e$ be the class of a ruling of $E\cong \mathbb{P}^1 \times \mathbb{P}^1$.

\begin{thm}\label{2.8thm}
For all $\alpha \in l + \Nef_1(X)_\mathbb{Z}$, $\overline{\free}^{bir}(X,\alpha)$ is nonempty and irreducible.
\end{thm}

\textbf{Intersection Pairing:} $H.l = 1$, $E.e = -1$, $H.e = 0 = E.l$.

\textbf{Anticanonical Divisor:} $ -K_X = 2H-E$.

\textbf{Generators of $\NE(X)$ and $\text{Eff}(X)$:}  The extreme rays of $\NE(X)$ are spanned by $e$ and $l - e$.  These are the only classes of effective $-K_X$ lines.  The extreme rays of $\text{Eff}(X)$ are spanned by $E$ and $H -E$.  

\textbf{Mori Structure:} The contraction of $l -e$ is a conic bundle $\pi : X \rightarrow \mathbb{P}^2$ with discriminant locus of degree 6.  We assume $X$ is general, so that the discriminant locus $\Delta$ is irreducible.  It follows that the $\pi$-preimage $X_\ell$ of a general line $\ell \subset \mathbb{P}^2$ is a smooth weak del Pezzo surface of degree 2, with a unique $-2$ curve coming from its intersection with $E$.  By describing $X_\ell$ as the blow-up of $\mathbb{P}^2$ along 7 points, 6 of which lie in a conic, we obtain the following lemma.

\begin{lem}
A core of free curves on $X$ is %
\begin{align*}
   \mathscr{C}_X = &\{ 2l - 2e, l, 2l -e\}
\end{align*}
The only separating class $\alpha \in \mathscr{C}_X$ is $l$.  Moreover, $\overline{\free}^{bir}(X, 2l)$ is irreducible.
\end{lem}
\begin{proof}
\textbf{Nef Curve Classes of Anticanonical Degree Between $2$ and $4$:} The nef classes of appropriate anticanonical degree are those in $\mathscr{C}_X$, $3l-3e$, $4l-4e$, $3l-2e$, and $2l$.  Our claim for arbitrary $X$ follows from Proposition \ref{low degree curves degeneration}, Theorem \ref{thm: conics}, and the proof for general $X$.  Hence, we may suppose $X$ is general in moduli.  

Corollary \ref{cor: conic monodromy cor} proves $\overline{\free}(X, l)$ and $\overline{\free}(X, 2l-e)$ are irreducible.  Free curves of class $2l-2e$ are fibers of a conic fibration.  Since every free curve of class $3l-2e$ lies over a line in $\mathbb{P}^2$, we find such curves must break into free chains of type $(2l-2e, l)$.  There are no free curves of class $3l-3e$, and every free curve of class $4l-4e$ is a double cover.  

A general immersed free curve of class $2l$ must be a component of the reducible preimage of a conic four times tangent to $B$ in $\mathbb{P}^3$, contained in a plane which avoids the locus where $V_7 \rightarrow \mathbb{P}^3$ is not an isomorphism.  Using an identical argument to \cite[Proposition~7.4]{2019}, \cite{V2monodromy} proves monodromy over such planes acts as the full Weyl group $W(E_7)$ on their double covers in $X$.  Since curves of class $2l$ correspond to conics in these surfaces, there is one component of $\overline{\free}^{bir}(X,2l)$.  Note that anticanonical curves in double covers of general planes cannot form a full component of $\overline{\free}^{bir}(X,2l)$, as there is only a three parameter family of such curves (the tangent lines to the branch locus).  Therefore, $\overline{\free}(X,2l)$ has two components: one parameterizing double covers, and one parameterizing very free curves, and both of which contain chains of free conics.
\end{proof}

\begin{lem}
Relations in the monoid $\mathbb{N}\mathscr{C}_X$ are generated by
$$ 2l + (2l-2e) = 2(2l-e).$$
$\overline{\free}(X, 4l-2e)$ is irreducible.
\end{lem}
\begin{proof}
To see $\overline{\free}(X, 4l-2e)$ is irreducible, it suffices to prove some component contains free chains of type $(l, 2l-2e, l)$ and $(2l-e, 2l-e)$.  This follows from smoothing an immersed chain of type $(l, l-e, l-e, l)$ contained in the $\pi$-preimage of a line in $\mathbb{P}^2$ in two separate ways.
\end{proof}

\begin{lem}
For each $\alpha \in l + \Nef_1(X)_\mathbb{Z}$, $\overline{\free}(X,\alpha)$ is nonempty.
\end{lem}

\begin{proof}
The generators of $\Nef_1(X)$ are $l$ and $l-e$.  Our claim follows immediately.
\end{proof}

\subsection{Blow-Ups of $\mathbb{P}^3$ Along a Smooth Curve}

Many of the Picard rank $2$ cases may be realized as the blow-up of a smooth curve $c$ in $\mathbb{P}^3$. We proceed to describe this general situation, specializing later on.

\textbf{Blow-up of a curve $c$ in $\mathbb{P}^3$:} Let $f:X \rightarrow \mathbb{P}^3$ be the blow-up of a smooth curve $c$ in $\mathbb{P}^3$.

\textbf{Generators for $N^1(X)$ and $N_1(X)$:}
\begin{center}
\begin{tabular}{ll}
 $H$ = the class of a hyperplane & $l$ = the class of a line \\ 
 $E$ = the exceptional divisor over $c$ & $e$ = the fiber over a point in $c$.  \\  
\end{tabular}
\end{center}

\textbf{Intersection Pairing:}
\begin{center}
\begin{tabular}{ll}
    $H \cdot l = 1$ &  $H \cdot e = 0$, \\
    $E \cdot l = 0$ &  $E \cdot e = -1$, \\
\end{tabular}
\end{center}

\textbf{Anticanonical Divisor:}
\begin{align*}
    -K_X = 4H-E.
\end{align*}

\textbf{Effective Divisors and Nef Curve Classes:}

\begin{prop}\label{Quadric Curve}
Suppose $c$ is nonplanar and contained in a quadric. The cone of effective divisors is generated by $2H-E$ and $E$,  while the cone of nef curve classes is generated by $l$ and $l-2e$. The nef curve classes of anticanonical degree between $2$ and $4$ are
\begin{align*}
    l, \hspace{.5cm} l-e, \hspace{.5cm} l-2e, \hspace{.5cm} 2l-4e,
\end{align*}

\noindent and the curves of class $2l-4e$ are freely breakable. The only relation between the remaining three classes is $(l)+(l-2e) = 2(l-e)$. Any chain of type $(l)+(l-2e)$ may be freely deformed to one of type $2(l-e)$.
\end{prop}

\begin{proof}

Note that $2H-E$ and $E$ are effective divisors, while $l$ is a movable curve class, hence nef. If $l-2e$ were not movable, the secant variety of $X$ would be a surface $S$ with each point lying on a pencil of lines. Hence $S = \mathbb{P}^2$ and $c$ is planar, contrary to hypothesis. It follows that $l-2e$ is movable and consequently nef. The cones generated by $2H-E$ and $E$ and by $l$ and $l-2e$ in $N^1(X)$ and $N_1(X)$ respectively are dual via the intersection pairing, so these classes are indeed generators.

The condition that $\alpha = al-be$ have anticanonical degree between $2$ and $4$ means that $2 \leq 4a-b \leq 4$. One can check that $l$, $l-e$, $l-2e$, and $2l-4e$ are the only nef solutions. Any curve of class $2l-4e$ lies in a plane, so is clearly freely breakable. There is a single relation $(l)+(l-2e) = 2(l-e)$ between the remaining classes because they span a two-dimensional vector space. The last condition is clear; we can take such deformation to occur in a plane.
\end{proof}

\begin{prop}\label{2l-4e}
Suppose $c$ is nonplanar and contained in a quadric.  A core of free curves consists of $\mathscr{C}_X = \{l, l-e, l-2e\}$.  The only possibly separating class is $\alpha = l-2e$. %
$\overline{\free}^{bir}(X,2l-4e)$ is irreducible in all cases.  $\overline{\free}(X,\alpha)$ is nonempty for all nonzero $\alpha \in \Nef_1(X)_\mathbb{Z}$.
\end{prop}

\begin{proof}
The curves of class $l$ are clearly parameterized by an open subset $\mathbb{G}(1,3)$. The curves of class $l-e$ are seen to be parameterized by an open subset $\mathbb{P}^2$-bundle over $c$. Neither of these are separating classes. The curves of class $l-2e$ may be parameterized by an open subset of $c \times c$.  This class is separating if and only if $c$ has trisecants, as then it is a conic interior to $\NE(X)$.  The curve $c$ has trisecants whenver it is not a twisted cubic or smooth degree 4 elliptic curve.  In every case, Theorem \ref{Monodromy} proves irreducibility of $\overline{\free}^{bir}(X, 2l-4e)$, and \ref{Gordan's Lemma} proves the existence of free curves of class $\alpha$ for all nonzero $\alpha \in \Nef_1(X)_\mathbb{Z}$.
\end{proof}

\subsection*{2.12}
\textbf{Blow-up of a Curve of Degree Six and Genus Three:} Let $f: X\rightarrow \mathbb{P}^3$ be the blow-up of a curve $c$ of degree 6 and genus 3 which is an intersection of cubics.

\begin{thm}\label{2.12thm}
For nonzero $\alpha \in \Nef_1(X)_\mathbb{Z}$, $\overline{\free}^{bir}(X,\alpha)$ is nonempty and irreducible.
\end{thm}

\textbf{Effective Divisors:} $E$ and $8H -3E$ span the effective cone of divisors on $X$.

\textbf{Effective Curves:} The extreme rays of $\NE(X)$ are $e$ and $l-3e$.

\textbf{Pseudosymmetry:} The contraction of $l-3e$ realizes $X$ as another blow-up $X\rightarrow \mathbb{P}^3$ along a curve of degree 6 and genus 3.  Thus we obtain a pseudoaction on $X$ by permuting $e$ and $l-3e$.  The action on $N_1(X)$ is given by $e \rightarrow l-3e$, $l \rightarrow 3l -8e$.

\begin{lem}
A core of free curves on $X$ is given by
\begin{align*}
   \mathscr{C}_X = &\{ l-2e, \ l -e, \ 2l -5e, \ l, \ 3l -8e  \} 
\end{align*}
The only separating class in $\mathscr{C}_X$ is $l -2e$.  Moreover, $\overline{\free}^{bir}(X, 2l-4e)$ is irreducible.
\end{lem}
\begin{proof}
\textbf{Nef Curve Classes of Anticanonical Degree Between $2$ and $4$:}
The only nef curve classes of appropriate degree are those in $\mathscr{C}_X$ and $2l -4e$.  However, $2l-4e$ is freely breakable, as free curves of this class are either double covers of bisecant lines to $c$ or conics through four distinct points on $c$.  The latter type of curve is parameterized by $\free^{bir}(X, 2l - 4e)$, which is irreducible by Theorem \ref{Monodromy}.  Pseudosymmetry swaps $l$ with $3l -8e$ and $l-e$ with $2l-5e$.  It is clear that $\overline{\free}(X, l)$, $\overline{\free}(X, l-e)$, and $\overline{\free}(X,l-2e)$ are irreducible and nonempty.  Projection to $\mathbb{P}^2$ demonstrates that $l-2e$ is a separating class.  Theorem \ref{reducible fibers: 4 author result} proves this is the only separating class in $\mathscr{C}_X$.
\end{proof}

\begin{lem}
Up to pseudosymmetry, relations in the monoid $\mathbb{N}\mathscr{C}_X$ are generated by:
\begin{enumerate}
    \item $(3l-8e)+(l) = 4(l-2e)$
    \item $(2l-5e) +(l) = 2(l-2e)+(l-e)$
    \item $(2l-5e)+(l-e) = 3(l-2e)$
    \item $(l-2e)+(l)=2(l-e)$

\end{enumerate}
For each relation $\sum \alpha_i = \sum \beta_j$, a main component of $\prod_X \overline{\free}_2(X,\alpha_i)$ lies in the same component of free curves as a main component of $\prod_X \overline{\free}_2(X, \beta_j)$.
\end{lem}
\begin{proof}
Lemma \ref{Relations} may be used to identify the a generating set of relations.  In addition to the above relations, these include $(3l-8e)+(l-e)= (2l-5e)+2(l-2e)$ and $(3l-8e)+(l-2e) = 2(2l-5e)$, which are pseudosymmetric to (2) and (4). Irreducibility of $\overline{\free}(X, 2l-2e)$ follows from Lemma \ref{lines_in_P3}.  $\overline{\free}(X, 3l-5e)$ is irreducible and generically parameterizes twisted cubics, as any family of planar, geometrically rational cubics of class $3l -5e$ has dimension at most 6 by Theorem \ref{Monodromy}.  There are two components of $\overline{\free}(X, 3l -6e)$: one parameterizes triple covers of conics of class $l-2e$; the other generically parameterizes twisted cubics.  This latter component contains free chains of type $(l-2e, l-2e, l-2e)$ and of type $(2l-5e, l-e)$.  Lastly, consider a nodal curve $g: C_1 \cup C_2 \cup C_3 \rightarrow X$ of type $(3l-8e, e, l-e)$.  A general such curve is clearly a smooth point of $\overline{\mathcal{M}}_{0,0}(X)$, and smooths to free chains of type $(3l-7e, l-e)$ and of type $(3l-8e, l)$.  By \ref{MovableBB}, we may break $3l-7e$ into a free chain of type $(l-2e, 2l-5e)$, and apply previous arguments.
\end{proof}

\begin{lem}
For each nonzero $\alpha \in \Nef_1(X)_\mathbb{Z}$, $\overline{\free}(X,\alpha)$ is nonempty.
\end{lem}

\begin{proof}
The generators of $\Nef_1(X)$ are elements in $\mathscr{C}_X$ by \ref{Representability of Free Curves}.  It follows from \ref{Gordan's Lemma} that classes in $\mathscr{C}_X$ generate the monoid of integer points in $\Nef_1(X)$.
\end{proof}

\subsection*{2.15}

\textbf{Blow-up of an Intersection of a Quadric and a Cubic:} 
Let $X$ be the blow-up of $\mathbb{P}^3$ with center a complete intersection $c$ of a quadric with a cubic. This case satisfies the hypothesis of Proposition \ref{Quadric Curve}.

\begin{thm}
For all $\alpha \in l+\Nef_1(X)_\mathbb{Z}$, $\overline{\free}(X,\alpha)$ is nonempty and irreducible.  Moreover, $\overline{\free}^{bir}(X, \alpha)$ is irreducible for all nonzero $\alpha \in \Nef_1(X)_\mathbb{Z}$.
\end{thm}

\begin{proof}
By \ref{Quadric Curve}, the set $\mathscr{C}_X$ consisting of the classes $l$, $l-e$, and $l-2e$ is a core for $X$. However, $l-2e$ is a separating curve class. Indeed, in case $c$ is smooth, $\omega_c = \mathcal{O}_c(2+3-4) = \mathcal{O}_c(1)$ so genus $g(C)$ of $C$ is $4$ because $c$ is not planar. Moreover, the image of $c$ under the projection from a point $q \in \mathbb{P}^3 - c$ is a degree six planar curve. Thus, it has $\binom{6-1}{2}-4=6$ double points, meaning that there are $6$ curves of class $l-2e$ passing through $q$.
\end{proof}

\subsection*{2.17}
\textbf{Blow-up of an Elliptic Curve in a Quadric:} $X$ may be realized as a blow-up $f:X\rightarrow \mathbb{P}^3$ along an elliptic curve $c$ of degree $5$ using \ref{blow-up numbers}.  Note that this elliptic curve is contained in a cubic, but not a quadric in $\mathbb{P}^3$.  Contracting the other extreme ray of $\NE(X)$ realizes $X$ as the blow up of $Q\subset \mathbb{P}^4$ along an elliptic curve of degree $5$.

\begin{thm}\label{2.17thm}
For all nonzero $\alpha \in \Nef_1(X)_\mathbb{Z}$, $\overline{\free}^{bir}(X,\alpha)$ is nonempty and irreducible.
\end{thm}

\textbf{Effective Divisors:} $E$ and $5H -2E$ span the effective cone of divisors on $X$.  The exceptional divisor of the contraction $X\rightarrow Q$ has class $5H-2E$.

\textbf{Effective Curves:} The extreme rays of $\NE(X)$ are $e$ and $l-3e$.

\begin{lem}
A core of free curves on $X$ is given by
\begin{align*}
   \mathscr{C}_X = &\{ l-2e, \ l -e, \ 2l -5e, \ l  \} 
\end{align*}
The only separating class in $\mathscr{C}_X$ is $l -2e$.  $\overline{\free}^{bir}(X, 2l-4e)$ is irreducible.
\end{lem}
\begin{proof}
\textbf{Nef Curve Classes of Anticanonical Degree Between $2$ and $4$:}
The only nef curve classes of appropriate degree are those in $\mathscr{C}_X$ and $2l -4e$.  Theorem \ref{Monodromy} proves irreducibility of $\overline{\free}^{bir}(X, 2l-4e)$.  Hence $2l-4e$ is freely breakable.  Since the only anticanonical lines on $X$ are the extreme rays of $\NE(X)$, by \ref{nonfree curves} a general, irreducible curve of class $2l -5e$ is free.  It is clear that $\overline{\free}(X, \alpha)$ is irreducible and nonempty for each $\alpha \in \mathscr{C}_X$, and that $l-2e$ is a separating class.  Theorem \ref{reducible fibers: 4 author result} proves this is the only separating class in $\mathscr{C}_X$.
\end{proof}

\begin{lem}
Relations in the monoid $\mathbb{N}\mathscr{C}_X$ are generated by:
\begin{itemize}
    \item $l + (2l -5e) = (l-e) + 2(l-2e)$,
    \item $l + (l-2e) = 2(l-e)$,
    \item $3(l-2e) = (l-e) + (2l -5e)$.
\end{itemize}
For each relation $\sum \alpha_i = \sum \beta_j$, a main component of $\prod_X \overline{\free}_2(X,\alpha_i)$ lies in the same component of free curves as a main component of $\prod_X \overline{\free}_2(X, \beta_j)$.
\end{lem}
\begin{proof}
Lemma \ref{Relations} may be used to identify the above generating set of relations.  Moreover, $\overline{\free}(X, 3l-5e)$ is irreducible and generically parameterizes twisted cubics, as any family of planar, geometrically rational cubics of class $3l-5e$ has dimension at most 6.  Likewise, $\overline{\free}(X, 2l-2e)$ is irreducible, but there are two components of $\overline{\free}(X, 3l -6e)$.  The unique component which generically parameterizes twisted cubics contains free chains of type $(l-2e, l-2e, l-2e)$ and of type $(2l-5e, l-e)$.
\end{proof}

\begin{lem}
For each nonzero $\alpha \in \Nef_1(X)_\mathbb{Z}$, $\overline{\free}(X,\alpha)$ is nonempty.
\end{lem}

\begin{proof}
The generators of $\Nef_1(X)$ are elements in $\mathscr{C}_X$ by \ref{Representability of Free Curves}.  It follows from \ref{Gordan's Lemma} that classes in $\mathscr{C}_X$ generate the monoid of integer points in $\Nef_1(X)$.
\end{proof}

\subsection*{2.19}

\textbf{Blow-up of a Genus Two Curve in a Quadric:} 
The variety $X$ is described in Mori-Mukai as the blow-up of $V_4 \subset \mathbb{P}^5$ along a conic, where $V_4$ is a Fano threefold of index $2$ with $b_2=1$ and $(-\frac{1}{2}K_{V_5})^3 = 4$. Using \cite[Table~1 and Lemma~2.4]{weakfano}, we instead view $X$ as the blow-up of $\mathbb{P}^3$ along a curve $c$ of degree $5$ and genus $2$ lying in a quadric. This case then satisfies the hypothesis of \ref{Quadric Curve}.

\begin{thm}
For all $\alpha \in l+\Nef_1(X)_\mathbb{Z}$, $\overline{\free}(X,\alpha)$ is nonempty and irreducible.  Moreover, $\overline{\free}^{bir}(X, \alpha)$ is irreducible for all $\alpha \in \Nef_1(X)_\mathbb{Z}$.
\end{thm}

\begin{proof}
By \ref{Quadric Curve}, the set $\mathscr{C}_X$ consisting of the classes $l$, $l-e$, and $l-2e$ is a core for $X$. However, $l-2e$ is a separating curve class. Indeed, there are $\binom{5-1}{2}-2=4$ such curves through a general point.
\end{proof}

\subsection*{2.21}

\textbf{Blow-up of a Twisted Quartic in $Q$:}
Let $f:X \rightarrow Q$ be the blow-up of a smooth quadric $Q \subset \mathbb{P}^4$ with center a twisted quartic $c$.

\begin{thm}\label{2.21thm}
For all nonzero $\alpha \in \Nef_1(X)_\mathbb{Z}$, $\overline{\free}^{bir}(X,\alpha)$ is nonempty and irreducible.
\end{thm}

\textbf{Generators for $N^1(X)$ and $N_1(X)$:}

\begin{center}
\begin{tabular}{ll}
 $H$ = a hyperplane in $\mathbb{P}^4$ & $l$ = a line in $\mathbb{P}^4$ \\ 
 $E$ = the exceptional divisor $f^{-1}(c)$ & $e$ = an $f$-fiber over a point on $c$  
\end{tabular}
\end{center}

\textbf{Intersection Pairing:}
\begin{center}
\begin{tabular}{lll}
    $H \cdot l = 1$ &  $H \cdot e = 0$, \\
    $E \cdot l = 0$ &  $E \cdot e = -1$
\end{tabular}
\end{center}

\textbf{Anticanonical Divisor:}
\begin{align*}
    -K_X = 3H-E
\end{align*}

\textbf{Effective Divisors:} 
The divisors $2H-E$ and $E$ are effective. We will show indirectly, however, that these divisors do not generate the effective cone.

\textbf{Effective Curves:} $\NE(X)$ is generated by $e$ and $l - 2e$.  

\textbf{Pseudosymmetry:}  By \cite[Table~2.21]{mori1981classification}, the contraction of $l -2e$ also realizes $X$ as a blow-up of $Q$ along a smooth curve $c'$.  Applying Lemma \ref{blow-up numbers}, it follows that $c'$ has the same degree and genus as $c$, so that $c'$ must also be a twisted quartic.  Thus we obtain a pseudoaction on $X$ that swaps $e$ and $l-2e$.  The action on $N_1(X)$ is given by $e \mapsto l-2e$, $l \mapsto 2l -3e$. 

\begin{lem}\label{2.21core}
A core set of free curves on $X$ is given by
\begin{align*}
    \mathscr{C}_X = \{ l, l-e, 2l-3e \}.
\end{align*}

The only separating curve in $\mathscr{C}_X$ is $l-e$.  $\overline{\free}^{bir}(X, 2l-2e)$ is irreducible.
\end{lem}

\begin{proof}

\textbf{Nef Curve Classes of Anticanonical degree between $2$ and $4$:}

Such nef curve classes $\alpha = a l - b e$ satisfy the inequaltities
\begin{align*}
    0 \leq b \leq 2a, \hspace{.5cm} 2 \leq 3a-b \leq 4.
\end{align*}

\noindent The solutions are
\begin{align*}
    l, \hspace{.5cm} l-e, \hspace{.5cm} 2l-2e, \hspace{.5cm} 2l-3e, \hspace{.5cm} 2l-4e, \hspace{.5cm} 3l-5e, \hspace{.5cm} 3l-6e, \hspace{.5cm} 4l-8e.
\end{align*}

With respect to the pseudosymmetry, we see $4l-8e \mapsto 4e$, $3l-6e \mapsto 3e$, $3l-5e \mapsto l+e$, and $2l-4e \mapsto 2e$. Therefore, these classes are not nef after all. Note moreover that $2l-3e \mapsto l$.

The remaining solutions are those classes in $\mathscr{C}_X$, as well as $2l -2e$. This additional class is freely breakable.  Indeed, there are two components of $\overline{\free}(X, 2l-2e)$.  One parameterizes double covers of conics of class $l-e$.  Any component component of $\overline{\free}^{bir}(X, 2l-2e)$ must generically parameterize very free curves embedded in $X$. Such curves are precisely the residual intersection with $Q$ of those planes in $\mathbb{P}^4$ which intersect $c$ in two points. This moduli space is a $\mathbb{P}^2$-bundle over $\text{Sym}^2(c)$. The map associated to this family has image which dominates $\overline{\free}^{bir}(X,2l -2e)$.  As the family contains free chains of type $(l-e, l-e)$, this proves $2l -2e$ is freely breakable, and $\overline{\free}^{bir}(X, 2l-2e)$ is irreducible.

\textbf{Irreducible Spaces and Fibers:}
The curves of class $l$ are parameterized by the lines in $Q$. By \cite[Corollary~6.33]{eisenbud20163264}, this space is smooth and of dimension three. Moreover, the intersection of $Q$ with the projective tangent space of a point $p \in Q$ is the cone over a smooth conic. In particular, $\overline{\free}(X,l)$ is irreducible and $l$ is not a separating curve class.  Via pseudosymmetry, the same argument applies to curves of class $2l - 3e$.  

By this same description, the curves of class $l-e$ are parameterized by a bundle of conics over $c$. Consider the curves of class $l-e$ through a general point $p \in Q \subset \mathbb{P}^4$. These correspond precisely to the intersection points of $c$ with the tangent space $T_p Q \subset \mathbb{P}^4$. Since $c$ is a degree four curve, there are four such points. Hence, $l - e$ is a separating curve class.
\end{proof}

\begin{lem}\label{2.21relations}
The only relation in $\mathbb{N} \mathscr{C}_X$ is $(2l-3e)+(l)=3(l-e)$. A chain of type $(2l-3e, l)$ may be freely deformed to a chain of type $(l-e, l-e, l-e)$.
\end{lem}

\begin{proof}
The first assertion is clear. For the second, we may find an immersed chain of type $(l-e, l - 2e, e, l - e)$ that corresponds to a smooth point of the Kontsevich space.  This chain smooths to free chains of type $(2l - 3e, l)$ and $(l-e, l - e, l-e)$, proving our claim.  
\end{proof}

\begin{lem}
For each nonzero $\alpha \in \Nef_1(X)_\mathbb{Z}$, $\overline{\free}(X,\alpha)$ is nonempty.
\end{lem}

\begin{proof}
The extreme rays of $\Nef_1(X)$ are $l$ and $2l-3e$. Every class $\alpha \in \Nef_1(X)_\mathbb{Z}$ may be written as sums of positive multiples of these extremal rays plus either $l-e$, $2l-e$, or $2l-2e$. Hence, applying \ref{Representability of Free Curves}, it suffices to show that these three classes are representable by a free curve. This is clear.
\end{proof}

\subsection*{2.22}

\textbf{Blow-up along a Rational Curve in a Quadric:} 
The variety $X$ is described in Mori-Mukai is the blow-up of $V_5 \subset \mathbb{P}^6$ along a conic, where $V_5$ is a Fano threefold of index $2$ with $b_2=1$ and $(-\frac{1}{2}K_{V_5})^3 = 5$. However, using \cite[Table~1 and Lemma~2.4]{weakfano}, we may alternatively view $X$ as the blow-up of $\mathbb{P}^3$ along a curve $c$ of degree $4$ and genus $0$ lying in a quadric. We will utilize this alternative description, so that the case satisfies the hypothesis of \ref{Quadric Curve}.

\begin{thm}
For all $\alpha \in l+\Nef_1(X)_\mathbb{Z}$, $\overline{\free}(X,\alpha)$ is nonempty and irreducible.  Moreover $\overline{\free}^{bir}(X,\alpha)$ is irreducible for all nonzero $\alpha \in \Nef_1(X)_\mathbb{Z}$.
\end{thm}

\begin{proof}
By \ref{Quadric Curve}, the set $\mathscr{C}_X$ consisting of the classes $l$, $l-e$, and $l-2e$ is a core for $X$. However, $l-2e$ is a separating curve class. Indeed, there are $\binom{4-1}{2}-0=3$ such curves through a general point.
\end{proof}

\subsection*{2.23}

\textbf{Blow-up of a Degree $4$ Curve in $Q$:}
Let $f:X \rightarrow Q$ be the blow-up of a smooth quadric $Q \subset \mathbb{P}^4$ with center a curve $c$ which may be realized as an intersection of some $A \in |\mathcal{O}_Q(1)|$ with $B \in |\mathcal{O}_Q(2)|$.

\begin{thm}\label{2.23thm}
For all $\alpha \in l+\Nef_1(X)_\mathbb{Z}$, $\overline{\free}(X,\alpha)$ is nonempty and irreducible.  Moreover $\overline{\free}^{bir}(X,\alpha)$ is irreducible for all nonzero $\alpha \in \Nef_1(X)_\mathbb{Z}$.
\end{thm}

\textbf{Generators for $N^1(X)$ and $N_1(X)$:}

\begin{center}
\begin{tabular}{ll}
 $H$ = a hyperplane in $\mathbb{P}^4$ & $l$ = a line in $\mathbb{P}^4$ \\ 
 $E$ = the exceptional divisor $f^{-1}(c)$ & $e_1$ = an $f$-fiber over a point on $c$  
\end{tabular}
\end{center}

\textbf{Intersection Pairing:}
\begin{center}
\begin{tabular}{lll}
    $H \cdot l = 1$ &  $H \cdot e = 0$, \\
    $E \cdot l = 0$ &  $E \cdot e = -1$
\end{tabular}
\end{center}

\textbf{Anticanonical Divisor:}
\begin{align*}
    -K_X = 3H-E
\end{align*}

\textbf{Effective Divisors:} 
The divisors $H-E$ and $E$ are effective.

\begin{lem}\label{2.23core}
A core set of free curves on $X$ is given by 
\begin{align*}
    \mathscr{C}_X= \{ l, l-e \}.
\end{align*}
The only separating curve $\alpha \in \mathscr{C}_X$ is $l-e$.  $\overline{\free}^{bir}(X, 2l-2e)$ is irreducible.
\end{lem}

\begin{proof}

\textbf{Nef Curve Classes of Anticanonical degree between $2$ and $4$:}
Such nef curve classes $\alpha = al - b e$ satisfy the inequalities
\begin{align*}
    0 \leq b \leq a, \hspace{.5cm} 2 \leq 3a-b \leq 4.
\end{align*}
\noindent The solutions are those classes in $\mathscr{C}_X$, as well as $2l -2e$.  This additional class is freely breakable.  Indeed, there is just one component of $\overline{\free}(X, 2l-2e)$ which parameterizes double covers of curves of class $l-e$.  Any other irreducible curve of class $2l -2e$ meets $c$ twice.  Since the curve is not a double cover, we may parameterize a family of such curves by $\text{Sym}^2(c) \times Q$ by associating the plane spanned by three points to its intersection with $Q$.  The map associated to this family has image which dominates $\overline{\free}^{bir}(X,2l -2e)$, proving its irreducibility.  As the family contains free chains of type $(l-e, l-e)$, this proves $2l -2e$ is freely breakable.

\textbf{Irreducible Spaces and Fibers:}
The curves of class $l$ are parameterized by the lines in $Q$. By \cite[Corollary~6.33]{eisenbud20163264}, this space is smooth and of dimension three. Moreover, the intersection of $Q$ with the projective tangent space of a point $p \in Q$ is the cone over a smooth conic. In particular, the curves of class $l$ are not separating.

By this same description, the curves of class $l-e$ are parameterized by a bundle of conics over $c$. Consider the curves of class $l-e$ through a general point $p \in Q \subset \mathbb{P}^4$. These correspond precisely to the intersection points of $c$ with the tangent space $T_p Q \subset \mathbb{P}^4$. Since $c$ is a degree four curve, there are four such points. Hence, $l - e$ is a separating curve class.
\end{proof}

\begin{lem}
For each nonzero $\alpha \in \Nef_1(X)_\mathbb{Z}$, $\overline{\free}(X,\alpha)$ is nonempty.
\end{lem}

\begin{proof}
The generators of $\Nef_1(X)$ are elements in $\mathscr{C}_X$ by \ref{Representability of Free Curves}.  It follows from \ref{Gordan's Lemma} that classes in $\mathscr{C}_X$ generate the monoid of integer points in $\Nef_1(X)$.
\end{proof}

\begin{proof}[Proof of Theorem \ref{2.23thm}]
There are no relations in $\mathbb{N} \mathscr{C}_X$ because the classes $l$ and $l-e$ are linearly independent. %
\end{proof}

\subsection*{2.24}

\textbf{A Divisor of Bidegree $(1,2)$ on $\mathbb{P}^2 \times \mathbb{P}^2$:}
Let $X$ be a smooth divisor of bidegree $(1,2)$ on $\mathbb{P}^2 \times \mathbb{P}^2$. We may \ref{blowup} to case $3.8$.%

\subsection*{2.25}

\textbf{Blow-up of a Complete Intersection of Quadrics:} 
Let $X$ be the blow-up of $\mathbb{P}^3$ with center an elliptic curve $c$ which is the complete intersection of two quadrics. We may apply \ref{blowup} to case $3.6$.%

\subsection*{2.26}
\textbf{Description of Variety $X$:} 
The variety $X$ is described in \cite{mori1981classification} as the blow-up of $V_5 \subset \mathbb{P}^6$ along a line, where $V_5$ is a Fano threefold of index $2$ with $b_2=1$ and $(-\frac{1}{2}K_{V_5})^3 = 5$. However, the table in \cite{mori1981classification} also shows that $X$ may be viewed as the blow-up of a smooth quadric $Q \subset \mathbb{P}^4$ along a smooth curve $c$. Using Lemma \ref{blow-up numbers}, we see that $c$ must be a twisted cubic. We will henceforth use this description of $X$ as a blow-up via the map $\pi:X \rightarrow Q$.

\begin{thm}\label{2.26thm}
For all $\alpha \in l + \Nef_1(X)_\mathbb{Z}$, $\overline{\free}(X,\alpha)$ is nonempty and irreducible.  Moreover $\overline{\free}^{bir}(X, \alpha)$ is irreducible for all nonzero $\alpha \in \Nef_1(X)_\mathbb{Z}$.
\end{thm}

\textbf{Generators for $N^1(X)$ and $N_1(X)$:}
\begin{center}
\begin{tabular}{ll}
 $H$ = the class of a hyperplane & $l$ = a line in $\mathbb{P}^4$ \\ 
 $E$ = the exceptional divisor $\pi^{-1}(c)$ & $e$ = the $\pi$-fiber over a point on $c$  \\
\end{tabular}
\end{center}

\textbf{Intersection Pairing:} $H.l = 1$, $H.e = E.l = 0$, $E.e = -1$

\textbf{Anticanonical Divisor:}
\begin{align*}
    -K_X = 3H - E
\end{align*}

\textbf{Effective Divisors:}
The divisors $H-E$ and $E$ are effective.

\begin{lem}\label{2.26core}
A core of free curves on $X$ is given by
\begin{align*}
    \mathscr{C}_X = \{ l, l-e\}
\end{align*}

\noindent  %
The only separating class in $\mathscr{C}_X$ is $l -e$.  $\overline{\free}^{bir}(X, 2l-2e)$ is irreducible.
\end{lem}

\begin{proof}
\textbf{Nef Curve Classes of Anticanonical degree between $2$ and $4$:}
Such nef curve classes $\alpha = al - b e$ satisfy the inequalities
\begin{align*}
    0 \leq b \leq a, \hspace{.5cm} 2 \leq 3a-b \leq 4.
\end{align*}
\noindent The solutions are those classes in $\mathscr{C}_X$, as well as $2l -2e$.  This additional class is freely breakable.  Indeed, any irreducible curve of class $2l -2e$ meets $c$ twice.  Unless the curve is a double cover, we may parameterize a family of such curves by $\text{Sym}^2(c) \times Q$ by associating the plane spanned by three points to its intersection with $Q$.  The map associated to this family has image which dominates $\overline{\free}^{bir}(X,2l -2e)$.  This proves $\overline{\free}^{bir}(X, 2l-2e)$ is irreducible.  As the family contains free chains of type $(l-e, l-e)$, this proves $2l -2e$ is freely breakable.

\textbf{Irreducible Spaces and Fibers:}  It is well known that the space of lines in $Q$ is irreducible.  Moreover, through any point on $Q$, there exists a one parameter family of lines spanning a quadric cone, which is the intersection of $Q$ with its tangent plane.  Thus, $\overline{\free}(X, l-e)$ is irreducible; however, this also shows that the general fiber of $\text{ev}: \overline{\free}_1(X,l-e)\rightarrow X$ consists of three points.  More specifically, the tangent plane intersects $c$ at three distinct points.  Generally, these three points will lie on distinct lines.
\end{proof}

\begin{proof}[Proof of Theorem \ref{2.26thm}]
There are no relations in the monoid $\mathbb{N} \mathscr{C}_X$. %
\end{proof}

\subsection*{2.27}

\textbf{Blow-up of a Twisted Cubic:} 
Let $X$ be the blow-up of $\mathbb{P}^3$ along a twisted cubic. We may apply \ref{blowup} to case $3.22$.%

\subsection*{2.28} See Section \ref{E5 cases}.

\subsection*{2.29}

\textbf{Blow-up of a Quadric along a Conic:} 
Let $X$ be the blow-up of a smooth quadric $Q \subset \mathbb{P}^4$ with center a conic on it. We apply \ref{blowup} to case $3.18$.%

\subsection*{2.30}
\textbf{Blow-up of a Conic in $\mathbb{P}^3$:} 
Let $X$ be the blow-up of a smooth quadric $Q \subset \mathbb{P}^4$ with center a conic on it. We apply \ref{blowup} to case $3.18$.%

\subsection*{2.31}
\textbf{Blow-up of a Quadric along a Line:} 
Let $X$ be the blow-up of a smooth quadric $Q \subset \mathbb{P}^4$ with center a line on it. We apply \ref{blowup} to case $3.20$.%

\subsection*{2.32}

\textbf{Blow-up of a Divisor on $\mathbb{P}^2 \times \mathbb{P}^2$:} 
Let $X$ be a smooth divisor on $\mathbb{P}^2 \times \mathbb{P}^2$ of bidegree $(1,1)$. We apply \ref{blowup} to case $3.13$.%

\subsection*{2.33}
\textbf{Blow-up of a line in $\mathbb{P}^3$:} 
Let $X$ be the blow-up of a line in $\mathbb{P}^3$. We apply \ref{blowup} to case $3.18$.%

\subsection*{2.34}

\textbf{$\mathbb{P}^1 \times \mathbb{P}^2$:} 
Let $X = \mathbb{P}^1 \times \mathbb{P}^2$. We may apply \ref{blowup} to case $3.22$.%

\subsection*{2.35}

\textbf{Blow-up of a Point:} 
Let $X=V_7$ where $f:V_7 \rightarrow \mathbb{P}^3$ is the blow-up of $\mathbb{P}^3$ at a point. We may apply \ref{blowup} to case $3.11$.%

\subsection*{2.36} See Section \ref{E5 cases}.

\section{Picard Rank 1}\label{Section 13}
As mentioned in the introduction, \cite{2019} and \cite{lastDelPezzoThreefold} verify Conjecture \ref{GMC Tanimoto} for general smooth Fano threefolds of Picard rank one and index two.  The following theorem is a consequence of their results, Theorem \ref{main thm: conics} (proven in Theorem \ref{thm: conics}), and Proposition \ref{low degree curves degeneration}.

\begin{thm}
    Let $X$ be a smooth Fano threefold of Picard rank one and index at least two.  For all $\alpha \in \NE(X)_{\mathbb{Z}}$, $\overline{\free}^{bir}(X, \alpha)$ is irreducible.
\end{thm}
\begin{proof}
    When $-K_X . \alpha = 2$, our claim is proven in Theorem \ref{main thm: conics}.  Similarly, when $X$ is general in moduli, our statement follows directly from \cite{2019} and \cite{lastDelPezzoThreefold}.  To extend these results to arbitrary $X$, we may first use Proposition \ref{low degree curves degeneration} to show $\overline{\free}^{bir}(X, \alpha)$ is irreducible when $-K_X . \alpha \leq 4$.  The hypothesis of Proposition \ref{low degree curves degeneration} follows the classification of $a$-covers in Theorem \ref{classification a-covers}.  Movable Bend-and-Break (Theorem \ref{MovableBB}) then proves our claim through application of Corollary \ref{bir Main Method}.  
\end{proof}

Consider instead a Fano threefold $X$ of Picard rank one, index one, and of genus $g(X)$.  Although \cite{Fanoindex1rank1} addresses the case when $-K_X$ is very ample, $X$ is general in moduli, and $3 \leq g(X) \leq 10$, the proofs of \cite[Theorems~7.4,7.6]{Fanoindex1rank1}, which demonstrate irreducibility of the spaces of free anticanonical cubics and very free anticanonical quartics, are flawed.  We rectify their arguments below and extend the result to arbitrary Fano threefolds of Picard rank and index one.  We use the following classification theorem.

\begin{thm}[\cite{FanoV_Shafarevich}]\label{classification Picard rank 1}
Let $X$ be a smooth Fano threefold of Picard rank 1 and index 1.  Let $g := \frac{1}{2}(-K_X)^3 + 1$ be the genus of $X$.  Then $g \neq 11$ and $2 \leq g \leq 12$.  
\begin{enumerate}[(1)]
\item If $g \leq 10$, $X$ is the transverse intersection of a variety $\mathbb{G} \subset \mathbb{P}^n$ with a complete intersection $V \subset \mathbb{P}^n$ of hypersurfaces:
\begin{enumerate}[a.]
    \item $(g = 2)$: $\mathbb{G} = \mathbb{P}(1^4, 3) \subset \mathbb{P}^{20}$ is the cone over the Veronese embedding of $\mathbb{P}^3$ by cubics and $V$ is a quadric, i.e.\ $X \subset \mathbb{G}$ is a hypersurface of degree 6;
     \item $(g = 3)$: $\mathbb{G} = \mathbb{P}(1^5, 2) \subset \mathbb{P}^{15}$ is the cone over the Veronese embedding of $\mathbb{P}^4$ by quadrics and $V$ is a the intersection of a quadric and a hyperplane;
        \item $(g = 4)$: $V \subset \mathbb{G} = \mathbb{P}^{5}$ is a complete intersection of a quadric and a cubic;
        \item $(g = 5)$: $V \subset \mathbb{G} = \mathbb{P}^{6}$ is a complete intersection of three quadrics;
        \item $(g = 6)$: $\mathbb{G} \subset \mathbb{P}^{10}$ is the cone over the Grassmannian $Gr(2,5) \subset \mathbb{P}^9$ in the Pl\"{u}cker embedding and $V$ is the intersection of a quadric and a linear subspace of codimension 3;
    \item $(g = 7)$: $\mathbb{G} = OGr_+(5,10) \subset \mathbb{P}^{15}$ is the orthogonal Grassmanian in the embedding induced by the half-spinor representation and $V$ is a linear subspace of codimension 7;
    \item $(g = 8)$: $\mathbb{G} = Gr(2,6) \subset \mathbb{P}^{14}$ is the Grassmannian in the Pl\"{u}cker embedding and $V$ is a linear subspace of codimension 5;
    \item $(g = 9)$: $\mathbb{G} = LGr(3,6) \subset \mathbb{P}^{13}$ is the Lagrangian Grassmannian in the Pl\"{u}cker embedding and $V$ is a linear subspace of codimension 3;
    \item $(g = 10)$: $\mathbb{G} = G_2/P \subset \mathbb{P}^{13}$ is a closed orbit of the adjoint representation of $G_2$ and $V$ is a linear subspace of codimension 2.
\end{enumerate}
\item If $g = 12$, then $X \subset Gr(3,7)$ is the common zero locus of three global sections of the rank three bundle $\bigwedge^2 \mathcal{U}^\vee$, where $\mathcal{U}$ is the universal subbundle on $Gr(3,7)$.
\end{enumerate}
\end{thm}

\subsection*{General Complete Intersections}  We study Fano threefolds of Picard rank one, index one, and genus $g \leq 10$ as complete intersections of very ample divisors.  Our technique appears in a revision \cite{JLT2023Err} of \cite{Fanoindex1rank1} by the second author with Lehmann and Tanimoto.

For a fixed genus $g \leq 10$, let $\mathbb{G}$ be the variety from Theorem \ref{classification Picard rank 1} containing all prime Fano threefolds of genus $g$ as complete intersections. 
 We may express each complete intersection as the zero locus of some global section of a completely reducible, very ample vector bundle $\mathcal{V}$ on $\mathbb{G}$.   The complete intersections which are disjoint from the singular locus of $\mathbb{G}$ are parameterized by an open subset $B \subset H^0(\mathbb{G},\mathcal{V})$.  Consider the universal complete intersection over $B$: %
\[\begin{tikzcd}
	{\mathcal{X} \subset \mathbb{G} \times B} \\
	B
	\arrow["\pi", from=1-1, to=2-1]
\end{tikzcd}\]
The relative cone of curves $\NE(\pi)$ is canonically isomorphic to the Mori cone $\NE(\mathcal{X}_b)$ of a smooth fiber.  We identify curve classes $\alpha \in \NE(\mathcal{X}_b)$ with the corresponding class in $\NE(\pi)$.  Note that every component of $\overline{\free}(\mathcal{X}_b,\alpha)$ lies in a component of $\overline{\free}(\mathcal{X},\alpha)$.  Proposition \ref{prop: terminal away from codim two} uses the following properties of $\pi : \mathcal{X} \rightarrow B$ to prove irreducibility of $\overline{\free}^{bir}(\mathcal{X}_b, \alpha)$ from irreducibility of $\overline{\free}^{bir}(\mathcal{X}, \alpha)$. %
\begin{lem}
    $B$ is smooth, simply connected, and for $b\in B$ outside a codimension two subset, $\mathcal{X}_b$ is a terminal Gorenstein Fano threefold with basepoint free anticanonical linear series.
\end{lem}
\begin{proof}
    $B$ is smooth because it is an open subset of a vector space $\overline{B}$.  When $g > 2$, $\overline{B} \setminus B$ has codimension at least two, as the rank of $\mathcal{V}$ is at least two.  When $g = 2$, $\mathcal{V} \cong \mathcal{O}_{\mathbb{G}}(6)$ and $B \subset H^0(\mathbb{G},\mathcal{V})$ is the complement of a hyperplane.  Thus, in all cases $B$ is simply connected.

    We claim the discriminant locus $D = \{ b \in B : \mathcal{X}_b \text{ is singular}\}$ is irreducible.  Let $\text{ev} : \mathcal{X} \rightarrow \mathbb{G}$ denote the projection to $\mathbb{G}$.  For any smooth point $p \in \mathbb{G}$, the fiber $\text{ev}^{-1}(p)$ is an open subset of %
    the linear subspace $L \subset H^0(\mathbb{G},\mathcal{V})$ parameterizing complete intersections containing $p$.  Complete intersections which are singular at $p$ form an irreducible subset.  Indeed, the space of matrices $A : \mathcal{T}_p \mathbb{G} \rightarrow k^{\dim G - 3}$ is a linear quotient of $L$ since $\mathcal{V}$ is a direct sum of very ample line bundles.  By smoothness of $p \in \mathbb{G}$, a complete intersection is singular at $p$ if and only if the corresponding matrix $A$ is not of full-rank.  As the locus of such matrices is irreducible, varying $p \in \mathbb{G}^{sm}$ proves irreducibility of $D$.

    When $g = 2$, the argument for \cite[Proposition~7.1(b)]{eisenbud20163264} shows the singular locus of a general singular weighted hypersurface $\mathcal{X}_b \subset \mathbb{G}^{sm}$ is a simple double point.  By definition, $\mathcal{X}_b$ is terminal, Gorenstein, and has base point free anticanonical linear series.  When $g \geq 3$, we construct terminal Gorenstein Fano threefolds $\overline{X}$ through elementary flops of other Fano threefolds.  
    
    Let $\tilde{X} \rightarrow X$ be the blow-up of a general prime Fano threefold $X$ of genus $g(X) \geq 5$ along a general line.  By \cite[Proposition~4.3.1]{FanoV_Shafarevich}, the anticanonical model $\overline{X}$ of $\tilde{X}$ is a terminal Gorenstein Fano threefold of Picard rank 1, index 1, and genus $g = g(X) - 2$.  To obtain a terminal Gorenstein Fano threefold of genus $g = 9$, by \cite[Lemma~4.1.1, Proposition~4.4.1]{FanoV_Shafarevich}, we may instead let $\tilde{X} \rightarrow X$ be the blow-up of a general conic on a Fano threefold $X$ of genus 12.  In each case, $\overline{X}$ is the anticanonical model of a smooth weak Fano threefold $\tilde{X}$ whose anticanonical divisor is not expressible as the sum of two moving Weil divisors.  This implies each $\overline{X}$ is BN-general in the sense of \cite[Proposition~7.8]{mukaiBNgeneral}.  Hence, the anticanonical linear series embeds $\overline{X}$ as a complete intersection of the desired type \cite[Theorem~6.5(2)]{mukaiBNgeneral}.  Since having Gorenstein terminal singularities is an open condition on threefolds in flat families \cite[Corollary~5.44]{kollarMMP}, the general singular singlar $\mathcal{X}_b$ is a terminal Gorenstein Fano threefold with base point free anticanonical linear system.  
\end{proof}

We obtain the following corollary of the preceding lemma and Proposition \ref{prop: terminal away from codim two}.

\begin{cor}\label{cor: Picard rank 1 irred cubics and quartics}
    Let $X$ be a general Fano threefold of Picard rank 1, index 1, and genus $g \leq 10$.  For $\alpha \in \NE(X)$ with $-K_X . \alpha \in \{3,4\}$, $\overline{\free}^{bir}(X, \alpha)$ is irreducible and generically parameterizes embedded curves with balanced normal bundles.
\end{cor}
\begin{proof}
    Let $\pi : \mathcal{X} \rightarrow B$ be the family of complete intersections to which $X$ belongs.  To apply Proposition \ref{prop: terminal away from codim two}, it remains necessary to prove $\overline{\free}^{bir}(\mathcal{X}, \alpha)$ is irreducible and generically parameterizes curves with balanced normal bundle.  To this end, let $f : C \rightarrow X$ be a general map parameterized by a component of $\overline{\free}^{bir}(X, \alpha)$.
    
    When $g \geq 3$, \cite[Lemma~7.3~and~Lemma~7.5]{Fanoindex1rank1} prove $f$ is the embedding of a rational normal curve in $\mathbb{G} \subset \mathbb{P}^n$.  We claim this remains true for $g = 2$.  In this case, the family of rational normal curves in $\mathbb{G}^{sm}$ is irreducible \cite{thompsen98, KP01}, and each curve is contained in an irreducible family of complete intersections parameterized by $B$ with the expected codimension.  Thus $\overline{\free}^{bir}(\mathcal{X}, \alpha)$ is irreducible.

    When $g = 2$, recall that $\mathbb{G} = \mathbb{P}(1^4,3)$ is the cone over the Veronese embedding of $\mathbb{P}^3$ by cubics.  We prove $f$ is the embedding of a rational normal curve in $\mathbb{G} \subset \mathbb{P}^{20}$ by considering the projection $\phi : \mathbb{G}^{sm} \rightarrow \mathbb{P}^3$.  Note that $\phi \circ f$ is birational onto its image by generality of $f$.  Indeed, $X$ is covered by a family $\mathcal{F}$ of curves whose generic points are not stabilized by the Galois action of $\phi|_X$.  Thus, some deformation of $f$ meets a general curve parameterized by $\mathcal{F}$ but not its Galois conjugate.
    
    Suppose $f : C \rightarrow \mathbb{G} \subset \mathbb{P}^{20}$ is not an embedding by the complete linear system.  In this case, we claim $\phi \circ f$ is not an embedding.  This is clear when $-K_X . \alpha = 3$, as $-K_X$ is the pullback of the hyperplane class on $\mathbb{P}^3$.  When $-K_X . \alpha = 4$, each embedded rational quartic curve in $\mathbb{P}^3$ is a curve of bidegree $(1,3)$ on some smooth quadric $\mathbb{P}^1 \times \mathbb{P}^1 \subset \mathbb{P}^3$.  Any such curve is contained in a seven-dimensional vector space of cubics, and thus spans a $\mathbb{P}^{12}$ under the Veronese embedding by cubics, contrary to our hypothesis.

    Recall the restriction of $\phi$ to $X$ is a double covering branched along a sextic hypersurface $S$.  Each sextic hypersurface $S$ is the branch locus for an $\text{Aut}(\mathbb{G})$-equivalent family of deformations of $X$.  However, the double cover of $C$ branched along $(\phi \circ f)^*S$ may be reducible only if $(\phi \circ f)^*S$ is everywhere non-reduced.  For a given map $\phi \circ f$, this imposes $-3K_X . \alpha$ conditions on $S$ by \cite{curvesRegularity} and the Riemann-Roch Theorem for singular curves, since the image $\overline{C}$ of $\phi \circ f$ has arithmetic genus at most 3 and $\phi \circ f : C \rightarrow \overline{C}$ is birational.  As the locus of maps $\mathbb{P}^1 \rightarrow \mathbb{P}^3$ which are not embeddings has codimension one in $\Mor(\mathbb{P}^1, \mathbb{P}^3)$, by counting dimensions we conclude $\phi \circ f$ is an embedding for general $f$.

    We conclude that $\overline{\free}^{bir}(\mathcal{X}, \alpha)$ is irreducible.  When $-K_X . \alpha = 3$, Proposition \ref{prop: terminal away from codim two} implies $\overline{\free}(X, \alpha)$ is irreducible.  When $-K_X . \alpha = 4$, we must further show for a general map $f$ the normal bundle $\mathcal{N}_f$ is balanced.  This follows from irreducibility of $\overline{\free}^{bir}(\mathcal{X}, \alpha)$ and a degeneration argument.  Indeed, let $\ell \subset X$ be a general anticanonical line.  Since $|-K_X|$ is base point free, by Lemma \ref{lem: normal bundle lines} the normal bundle of $\ell$ satisfies $h^1(\ell, N_{\ell/X}) \leq 1$.  Let $D_\ell$ be the divisor swept out by deformations of $\ell$.  Consider a family of free cubics $\mathcal{C}$ on $X$, and let $D_p$ be the divisor swept out by members of $\mathcal{C}$ meeting a general point $p \in X$.  As general cubics $C \in \mathcal{C}$ meet $D_\ell$ transversely, the intersection $D_p \cap D_\ell$ is reduced.  Since $D_p . \ell > 0$, general deformations of $\ell$ meet $D_p$ transversely.  Therefore general deformations of $C \cup \ell$ are very free curves.  
    Proposition \ref{prop: terminal away from codim two} thus shows $\overline{\free}^{bir}(X, \alpha)$ is irreducible.
\end{proof}
\begin{rem}\label{rem: Ceresa-Verra}
    When $X$ has genus $g = 2$, the argument of \cite[Proposition~2.35]{ceresa86} would provide an alternative, simpler proof of the irreducibility of $\overline{\free}^{bir}(X, \alpha)$.  However, their argument asserts that if the Picard group of a double cover $A \rightarrow \mathbb{P}^2$  is generated by the hyperplane class, then there are no irreducible curves in $\mathbb{P}^2$ with reducible preimage.  Unfortunately, this reasoning is flawed: if the branch locus of $\pi : A \rightarrow \mathbb{P}^2$ has degree $2k$, the general element of $C \in |\pi^*\mathcal{O}(d)|$ for $d \geq k$ is not a double cover of some curve in $\mathbb{P}^2$.  Instead, the restriction $\pi : C \rightarrow \mathbb{P}^2$ will be the normalization of a degree $2d$ curve with $d(d-k)$ nodes.  This erroneous reasoning also leaves a gap in their proof of irreducibility of the space of free conics on $X$ mapping birationally to $\mathbb{P}^3$ under the double-cover map.  This gap is remedied by instead arguing as in the proof of Corollary \ref{cor: Picard rank 1 irred cubics and quartics}.
\end{rem}

\subsection*{Genus 12 Threefolds} To describe spaces of free anticanonical cubics and quartics on a prime Fano threefold $X$ of genus 12, we use another birational model of $X$ obtained through an elementary Sarkisov link.  We then complete the proof of Theorem \ref{Main Result} for arbitrary smooth Fano threefolds by proving Theorem \ref{thm: Picard1index1}.

\begin{lem}\label{lem: genus 12 Fano curves}
    Let $X$ be a Fano threefold of Picard rank 1, index 1, and genus $g = 12$.  For $\alpha \in \NE(X)$ with $-K_X . \alpha \in \{3,4\}$, $\overline{\free}^{bir}(X, \alpha)$ is irreducible.
\end{lem}
\begin{proof}
    Consider the blow-up $\pi : \tilde{X} \rightarrow X$ of a general point in $X$.  Let $E$ be the exceptional divisor of $\pi$ and $e \in \NE(\tilde{X})$ be the class of a general fiber of $\pi|_E$.  Let $H = \pi^*(-K_X)$ and $l \in \NE(\tilde{X})$ be the total transform of a general $(-K_X)$-line in $X$.   We claim that for $d \in \{3,4\}$, $\Mor(\mathbb{P}^1, \tilde{X}, dl - e)$ is irreducible.  This implies irreducibility of $\overline{\free}^{bir}(X, dl)$.  To prove our claim, we describe a birational contraction $\tilde{X} \dashrightarrow \mathbb{P}^3$.
    
    The class $2l - e$ is a $K_{\tilde{X}}$-trivial extreme ray of $\NE(\tilde{X})$ which corresponds to a small contraction.  Let $\phi : \tilde{X} \dashrightarrow \tilde{X}^+$ be the corresponding flip.  There is a unique $K_{\tilde{X}^+}$-negative birational contraction $\tilde{\pi} : \tilde{X}^+ \rightarrow Y$.  By \cite[Theorem~4.5.8]{FanoV_Shafarevich}, $Y\cong \mathbb{P}^3$ and $\tilde{\pi} : \tilde{X}^+ \rightarrow Y$ is the blow-up of a smooth genus 0 curve $\Gamma$ of degree 6.  
    
    The $\phi$-strict transforms of subvarieties of $\tilde{X}$ with classes $E, H, e$, and $l$ still generate $N^1(\tilde{X}^+)$ and $N_1(\tilde{X}^+)$ as vector spaces, and their intersection pairings are preserved.  In this basis, the exceptional divisor $D$ of $\tilde{\pi}$ has class $3H - 10E$ by \cite[Theorem~4.5.8]{FanoV_Shafarevich}.  This implies $\tilde{\pi}^*\mathcal{O}_{\mathbb{P}^3}(1) \sim H - 3E$.  We find that $\Mor(\mathbb{P}^1, \tilde{X}^+, 3l - e)$ parameterizes fibers of the contraction $D \rightarrow \Gamma$, while $\Mor(\mathbb{P}^1, \tilde{X}^+, 4l - e)$ generically parameterizes strict transforms of lines bisecant to $\Gamma$.  It follows that for $d \in \{3,4\}$, $\Mor(\mathbb{P}^1, \tilde{X}, dl - e)$ is birational to $\Mor(\mathbb{P}^1, \tilde{X}^+, dl - e)$, and thus irreducible.
\end{proof}

\begin{thm}\label{thm: Picard1index1}
     Let $X$ be a smooth Fano threefold of Picard rank one and index one.  For all $\alpha \in \NE(X)_{\mathbb{Z}}$ with $-K_X . \alpha \geq 3$, $\overline{\free}^{bir}(X, \alpha)$ is irreducible.
\end{thm}

\begin{proof}
    Let $l \in \NE(X)$ be the class of an anticanonical line on $X$, so that $\alpha = dl$.  When $X$ is general in moduli and $d\leq 4$, our statement follows directly from Corollary \ref{cor: Picard rank 1 irred cubics and quartics} and Lemma \ref{lem: genus 12 Fano curves}.  By Theorem \ref{classification a-covers} and Proposition \ref{low degree curves degeneration}, this implies irreducibilty of $\overline{\free}^{bir}(X, dl)$ for $3\leq d \leq 4$ and $X$ is arbitrary.  Movable Bend-and-Break (Theorem \ref{MovableBB}) proves a weak core of free curves on $X$ is $\mathscr{C}_X = \{2l, 3l\}$.

    The only relation in the monoid $\mathbb{N}\mathscr{C}_X$ is $3 \cdot (2l) = 2 \cdot (3l)$.  To see that a component of $\overline{\free}(X, 6l)$ contains free chains of each type, consider %
    a general immersed chain of type $(3l, l, 2l)$.  By Lemma \ref{lem: normal bundle lines}, the corresponding map is a smooth point of $\overline{\free}(X, 6l)$ which generalizes to immersed free chains of type $(3l, 3l)$ and $(4l, 2l)$ upon smoothing the corresponding node.  Since $\overline{\free}^{bir}(X, 4l)$ is irreducible, it contains free chains of type $(2l, 2l)$.  Thus, we may specialize the free chain of type $(4l, 2l)$ to one of type $(2l, 2l, 2l)$ while remaining in the same component of $\overline{\free}(X, 6l)$.  Corollaries \ref{bir Main Method} and \ref{connected fibers fixall} finish our proof.
\end{proof}

\bibliographystyle{alpha}
\bibliography{maincleanedbiber}
%\printbibliography

\end{document}